\documentclass[12pt,reqno]{amsart}
\usepackage[T1]{fontenc}
\usepackage{lmodern}
\usepackage[margin=1.05in]{geometry}
\usepackage{mathtools,amssymb,mathrsfs}
\usepackage{tikz-cd}
\usepackage{microtype}
\usepackage{setspace}
\usepackage{enumitem}
\usepackage[hidelinks]{hyperref}
\hypersetup{pdftitle={Classification of reduced free products of nuclear C*-algebras},pdfauthor={George A. Elliott, Chun Guang Li, Zhuang Niu}}

\newcommand{\C}{\mathbb C}
\newcommand{\E}{\mathbb E}
\newcommand{\N}{\mathbb N}
\newcommand{\Z}{\mathbb Z}
\newcommand{\R}{\mathbb R}
\newcommand{\fr}{{*_{\mathrm r}}}
\newcommand{\id}{\operatorname{id}}
\newcommand{\Ad}{\operatorname{Ad}}
\newcommand{\tr}{\operatorname{tr}}
\newcommand{\Ell}{\operatorname{Ell}}
\newcommand{\Uquot}[1]{\mathrm U_\infty(#1)/\mathrm{CU}_\infty(#1)}

\newcommand{\Tor}{\operatorname{Tor}}
\newcommand{\Hom}{\operatorname{Hom}}
\newcommand{\KK}{\mathrm{KK}}
\newcommand{\KL}{\mathrm{KL}}
\newcommand{\Cu}{\mathrm{Cu}}
\newcommand{\red}{\mathrm r}
\newcommand{\ulK}{\underline{\mathrm K}}
\newcommand{\unit}{\mathbf1}

\newcommand{\au}{\mathrel{\approx_{\mathrm u}}}
\DeclareMathOperator{\Aff}{Aff}
\DeclareMathOperator{\Tr}{Tr}
\DeclareMathOperator{\Ext}{Ext}

\DeclareMathOperator{\ev}{ev}
\DeclareMathOperator{\dist}{dist}
\DeclareMathOperator{\CU}{CU}
\DeclareMathOperator{\cel}{cel}
\DeclareMathOperator{\cer}{cer}
\newcommand{\cucu}{\cel_{\mathrm{CU}}}
\newtheorem{introtheorem}{Theorem}

\newtheorem{introcorollary}[introtheorem]{Corollary}
\newtheorem{theorem}{Theorem}[section]
\newtheorem{proposition}[theorem]{Proposition}
\newtheorem{lemma}[theorem]{Lemma}
\newtheorem{corollary}[theorem]{Corollary}
\theoremstyle{definition}
\newtheorem{definition}[theorem]{Definition}
\newtheorem{example}[theorem]{Example}
\theoremstyle{remark}
\newtheorem{remark}[theorem]{Remark}
\numberwithin{equation}{section}
\setlist[enumerate]{label=\textup{(\roman*)},leftmargin=2.2em,itemsep=3pt}
\makeatletter
\newcommand{\keep@runin@heading}{%
  \if@noskipsec
    \leavevmode\par\nobreak
    \@nobreaktrue
  \fi}
\AddToHook{env/theorem/before}{\keep@runin@heading}
\AddToHook{env/proposition/before}{\keep@runin@heading}
\AddToHook{env/lemma/before}{\keep@runin@heading}
\AddToHook{env/corollary/before}{\keep@runin@heading}
\AddToHook{env/definition/before}{\keep@runin@heading}
\AddToHook{env/example/before}{\keep@runin@heading}
\AddToHook{env/remark/before}{\keep@runin@heading}
\AddToHook{env/proof/before}{\keep@runin@heading}
\makeatother

\title[Classification of reduced free products]{Classification of reduced free products\\of nuclear $\mathrm{C}^{*}$-algebras}
\author[G. A. Elliott]{George A. Elliott}
\address{Department of Mathematics, University of Toronto,
Toronto, Ontario M5S 2E4, Canada}
\email{elliott@math.toronto.edu}
\author[C. G. Li]{Chun Guang Li}
\address{School of Mathematics and Statistics, Northeast Normal University,
Changchun 130024, P. R. China}
\email{licg864@nenu.edu.cn}
\author[Z. Niu]{Zhuang Niu}
\address{Department of Mathematics and Statistics, University of Wyoming,
Laramie, Wyoming 82071, USA}
\email{zhuangniu@icloud.com}
\begin{document}
\raggedbottom
\begin{abstract}
Consider the countably infinite reduced free product $D$ of copies of $(\mathrm C([0,1]),\lambda)$, where $\lambda$ is Lebesgue measure. It is proved that finite and countable reduced free products of unital separable nuclear UCT $\mathrm C^*$-algebras with respect to a chosen family of faithful traces are classified, after adjoining $D$, by the Elliott invariant. In other words, an isomorphism between the Elliott invariants of two such algebras is induced by an isomorphism between the algebras.

\end{abstract}
\maketitle

\tableofcontents

\section{Introduction}\label{sec:introduction}

In \cite{HP}, using techniques from the classification of nuclear
$\mathrm C^*$-algebras, Hirshberg and Phillips show that certain
nuclear free factors can be adjoined to infinite reduced free products
without changing the isomorphism class. In particular, they identify
the infinite reduced free product of Lebesgue interval algebras with
that of Jiang--Su algebras. Their proofs combine approximate
intertwining with existence and uniqueness results for homomorphisms
from nuclear algebras. Motivated by their work (in particular, their
circle--annulus isomorphism questions), we extend these techniques
further to classify a class of nonnuclear algebras obtained by reduced
free products.

The structural study of reduced free products goes back to Avitzour's
simplicity and unique-trace criteria \cite{Avitzour}.
Dykema and R\o rdam developed the theory of projections and ordered
$\mathrm{K}$-groups in free products \cite{DR}, while Dykema, Haagerup and
R\o rdam proved stable-rank-one results \cite{DHR}.
More recently, Robert's theory of selfless $\mathrm{C}^*$-algebras provided
a route to strict comparison \cite{RobertSelfless}.
Amrutam, Gao, Kunnawalkam Elayavalli and Patchell proved strict
comparison for $\mathrm{C}^*_{\mathrm r}(\mathbb F_n)$, $2\leq n<\infty$,
using this framework \cite{AGKEP}.
Hayes, Kunnawalkam Elayavalli and Robert developed selflessness
and further regularity results for reduced free products \cite{HKR}.
These results provide the structural ingredients for the Elliott
invariant in this setting.

On the $\mathrm{K}$-theoretic side, the calculation of Pimsner and Voiculescu
\cite{PV} gives $\mathrm{K}_1(\mathrm{C}^*_{\mathrm r}(\mathbb F_n))\cong\Z^n$,
distinguishing reduced free group algebras of different finite ranks.
The $\mathrm{K}$- and $\mathrm{KK}$-theory of amalgamated free products
developed by Fima and Germain \cite{FG} provides the corresponding
calculations for general factors.

In this paper, we consider
\[
 (D,\tau_D)=\fr_{j=1}^{\infty}
       (\mathrm C([0,1]),\lambda),
\]
where $\lambda$ is Lebesgue measure. For a unital algebra $E$ with
a faithful selected tracial state $\tau_E$, we call the operation
\[
 (E,\tau_E)\longmapsto(E,\tau_E)\fr(D,\tau_D)
\]
its \emph{$D$-stabilization}. We say that $E$ is $D$-stable if
$E\fr D\cong E$ by a trace-preserving isomorphism.
Reindexing the interval factors gives $D\fr D\cong D$, so every
$D$-stabilized algebra is $D$-stable.

The absorption results of \cite{HP} motivate viewing this operation
as a free-product analogue of tensorial $\mathcal Z$-stabilization.
Separable simple unital nuclear UCT $\mathrm C^*$-algebras
that absorb the Jiang--Su algebra tensorially are classified by
the Elliott invariant, in the purely infinite case
\cite{KirchbergPI,PhillipsPI} and in the stably finite case
\cite{CETWW,EGLN,GLN1,GLN2,TWW}.

The following Theorem~\ref{thm:nuclear-classification} gives a free-product
counterpart: after adjoining $D$, finite and countable reduced free
products of nuclear UCT algebras are classified by their Elliott
invariants. The factors may be nonsimple or multitracial, and only
their selected traces enter the construction. The classification and
morphism-realization assertions are restated and proved in
Theorem~\ref{thm:nuclear-classification-main}.

Write $\Ell$ for the Elliott invariant, consisting of ordered
$\mathrm K_0$ with its distinguished unit, $\mathrm K_1$, the
tracial state simplex, and the pairing of traces with $\mathrm K_0$.
A unital morphism of Elliott invariants consists of homomorphisms
on the two $\mathrm K$-groups, positive and unit preserving in degree
zero, and a continuous affine map of trace spaces in the reverse
direction, compatible with the pairing.
For a nonempty finite or countable family
$\mathcal A=((A_i,\tau_i))_{i\in S}$ of unital algebras with
faithful selected tracial states, write
\[
 P_{\mathcal A}=\fr_{i\in S}(A_i,\tau_i).
\]

\begin{introtheorem}[Nuclear factors: classification, range and models]
\label{thm:nuclear-classification}\label{thm:main}
Let $\mathcal A$ and $\mathcal B$ be nonempty finite or countable
families of separable unital nuclear UCT algebras equipped with
faithful tracial states.
Put $A=P_{\mathcal A}\fr D$ and $B=P_{\mathcal B}\fr D$. Then
each of $A$ and $B$ is isomorphic to $F\fr D$ for a nuclear UCT
core $F$ carrying its aggregate selected pointed data. In particular,
\[
 A\cong B\quad\Longleftrightarrow\quad\Ell(A)\cong\Ell(B).
\]
Every unital morphism $\Ell(A)\to\Ell(B)$ is induced by a unital
embedding $A\to B$. If the invariant morphism is an isomorphism,
the realizing map may be chosen to be a $^*$-isomorphism.
If each family contains two distinct factors admitting Haar unitaries,
the same conclusion holds for $P_{\mathcal A}$ and $P_{\mathcal B}$
without adjoining $D$.

The class of $D$-stabilized products in this theorem is closed under
strong Morita equivalence among unital $\mathrm C^*$-algebras.
More precisely, every nonzero corner $p\mathrm M_n(A)p$ is isomorphic
to $E\fr D$ for a simple, infinite-dimensional, unital, monotracial,
nuclear UCT, $\mathcal Z$-stable core $E$
(Corollary~\ref{cor:unital-morita-closure}).

Moreover, the range of the $D$-stabilized products consists exactly
of monotracial Elliott invariants with countable groups
$\mathrm K_0,\mathrm K_1$, a distinguished unit $u$, and a homomorphism
$\rho:\mathrm K_0\to\R$ satisfying
\[
 \rho(u)=1,\qquad
 \mathrm K_0^+=\{0\}\cup\{g:\rho(g)>0\}.
\]
Every such invariant is realized by $F\fr\mathcal Z$, where $F$ is
simple, infinite-dimensional, unital, monotracial, nuclear, UCT and
$\mathcal Z$-stable.
Every product $A$ also has an inductive-limit model by finite reduced
free products of unital type~I algebras with faithful traces, with
unital trace-preserving injective connecting homomorphisms
(Theorem~\ref{thm:standard-models}).
If $\mathrm K_0(A)$ is finitely generated with free rank $r$, at most
$\max\{2,r\}$ type~I factors suffice, even when $\mathrm K_1(A)$
is not finitely generated (Corollary~\ref{cor:finite-standard-model}).
When both groups are finitely generated, a finite product of the
commutative standard blocks and homogeneous Moore-space corners
also suffices.
The commutative subclass has the sharper range
and commutative standard models described in
Corollary~\ref{thm:classification}.
\end{introtheorem}

In the commutative factor case, the invariant range and the
standard models admit a more explicit description.
For a nonempty finite or countable family
$\mathcal X=((X_i,\mu_i))_{i\in S}$ of measured compact metrizable spaces,
put
\[
 P_{\mathcal X}=\fr_{i\in S}(\mathrm{C}(X_i),\tau_{\mu_i}),
 \qquad \tau_{\mu_i}(f)=\int_{X_i}f\,d\mu_i.
\]
All spaces are nonempty and all measures have full support; atoms
are allowed. Evaluation at a point splits every commutative unit class.

\begin{introcorollary}[Commutative factors: range and standard models]\label{thm:classification}
Let $\mathcal X$ and $\mathcal Y$ be nonempty finite or countable families
of nonempty compact metrizable spaces with full-support Borel probability
measures, and put $A=P_{\mathcal X}\fr D$ and $B=P_{\mathcal Y}\fr D$.
Then
\[
 A\cong B\quad\Longleftrightarrow\quad\Ell(A)\cong\Ell(B).
\]
Every unital morphism $\Ell(A)\to\Ell(B)$ is induced by a unital
embedding $A\to B$, which preserves the unique traces. If the
invariant morphism is an isomorphism, the realizing map may be
chosen to be a $^*$-isomorphism.

The range consists exactly of monotracial Elliott invariants with
trace pairing $\rho:\mathrm{K}_0\to\R$ and
\[
 \mathrm{K}_0=F\oplus H,\qquad [1]=u\in F,\qquad
 \rho(u)=1,\qquad \rho(H)=0,
\]
where $F$ is a countable free abelian group with a basis containing $u$,
$H$ and $\mathrm{K}_1$ are arbitrary countable abelian groups, and
$\mathrm{K}_0^+=\{0\}\cup\{g:\rho(g)>0\}$. The restriction of $\rho$ to $F$
is otherwise arbitrary. Every such invariant is realized in this class.

Moreover, each of $A$ and $B$ is isomorphic to an inductive limit of
finite reduced free products of measured commutative algebras with
full-support probability measures and spectra from the list
\[
 [0,1],\quad [0,1]\sqcup[0,1],\quad S^1,\quad S^2,
 \quad T_{\mathrm{II},m},\quad T_{\mathrm{III},m}\quad(m\ge2).
\]
Here $T_{\mathrm{II},m}$ and $T_{\mathrm{III},m}$ are the two- and
three-dimensional generalized projective spaces defined in
Section~\ref{sec:models}. The connecting maps are unital
trace-preserving embeddings. If both $\mathrm{K}$-groups of either algebra
are finitely generated, that algebra is isomorphic to a single such
finite free product.
With arbitrary compact metrizable spectra, finite generation of
$\mathrm K_0$ alone suffices for a finite commutative free-product
model, by Corollary~\ref{cor:finite-standard-model}.
If each family contains two distinct factors admitting Haar unitaries,
the same assertions hold for $P_{\mathcal X}$ and $P_{\mathcal Y}$
without adjoining $D$.
In particular, this last conclusion applies when each family contains
at least two nonatomic measures.
\end{introcorollary}

The classification and morphism-realization assertions are restated
and proved in Corollary~\ref{cor:commutative-classification-main}.
The range and model assertions of Theorem~\ref{thm:nuclear-classification}
and Corollary~\ref{thm:classification} are proved in
Section~\ref{sec:range}.

The preceding existence statements do not give uniqueness of
embeddings up to approximate unitary equivalence.
Example~\ref{ex:reservoir-nonuniqueness} gives two automorphisms of the
interval reservoir inducing the same full totally compatible triple
that are not approximately unitarily equivalent.
The uniqueness theorem for embeddings below requires the source to
be nuclear.

For the commutative subclass there is no dimension, connectedness, torsion, or finite-generation
restriction on the spaces or their $\mathrm{K}$-groups.
The algebras $A$ and $B$ in Corollary~\ref{thm:classification}
are simple and monotracial,
with strict comparison. In particular,
\[
 \mathrm{K}_0(A)^+=\{0\}\cup\{g\in \mathrm{K}_0(A):\tau_{A*}(g)>0\},
 \qquad \mathrm{T}(A)=\{\tau_A\}.
\]
Thus a unital morphism of their Elliott invariants is precisely a pair
of group homomorphisms $\chi_i:\mathrm{K}_i(A)\to \mathrm{K}_i(B)$, $i=0,1$, with
\[
 \chi_0([1_A])=[1_B],\qquad \tau_{B*}\chi_0=\tau_{A*}.
\]
Positivity follows from the strict trace cones. Such a morphism is
an isomorphism exactly when both group maps are isomorphisms.
The invariant is intrinsic to the product algebra. Its calculation
from the individual factors is given in
Proposition~\ref{prop:product-elliott}.
The Haar-unitary condition has a separate role: it permits removal of
$D$ by Shlyakhtenko's free-absorption theorem \cite{Shlyakhtenko}.
Corollary~\ref{cor:classification-d-stable} applies whenever the
original products absorb $D$. Theorem~\ref{thm:absorption-without-reservoir}
gives further absorption criteria.

Shlyakhtenko \cite{Shlyakhtenko} proved that the $\mathrm{C}^*$-algebras
generated by free semicircular systems of cardinality
$n\in\{2,3,\ldots,\infty\}$ are mutually isomorphic.
His absorption theorem shows that a reduced free product of two
faithfully tracial algebras,
each containing a Haar unitary, absorbs any finite or countable
free semicircular family; the factors need not be nuclear.
Combined with Hirshberg--Phillips, this
also identifies all reduced free powers of the Jiang--Su algebra
with at least two factors. The semicircular absorption theorem is
the input used here to remove $D$ under the Haar-unitary hypothesis
\cite[Theorem~5.2]{Shlyakhtenko}.
In particular, two simple infinite-dimensional factors satisfy
the Haar-unitary hypothesis by \cite[Corollary~B]{ThielHaar}.

The classification also allows more general reservoirs
\begin{equation}\label{eq:reservoir-intro}
 (D,\tau_D)=\fr_{j=1}^{\infty}(R_j,\sigma_j).
\end{equation}
Here the blocks are separable unital nuclear UCT algebras with
faithful selected traces, the unit maps induce
$\mathrm K_0(R_j)=\Z[1_{R_j}]$ and $\mathrm K_1(R_j)=0$, and the
traces satisfy the atomic nonconcentration condition in
Definition~\ref{def:reservoir}. For example, the blocks may all be
Jiang--Su algebras. The nuclear-core argument is proved directly
for these reservoirs; Corollary~\ref{cor:reservoir-model} subsequently
identifies every such reservoir with the interval product.
The later fusion constructions use this interval presentation.

The proof of Theorem~\ref{thm:nuclear-classification} begins with
the case of a single nuclear core. We establish this case for the
general reservoirs above before passing to several nuclear factors
by fusion.

\begin{introtheorem}[$D$-stabilization of a nuclear core]
\label{thm:core-classification-intro}
Let $D$ be a reservoir in the sense of Definition~\ref{def:reservoir}, and let
$E,F$ be separable unital nuclear UCT algebras with faithful selected
tracial states $\tau_E,\tau_F$. Put $A=E\fr D$ and $B=F\fr D$.
Let
\[
 \chi_i:\mathrm K_i(E)\longrightarrow\mathrm K_i(F),\qquad i=0,1,
\]
be group homomorphisms satisfying
\[
 \chi_0([1_E])=[1_F],\qquad \tau_{F*}\chi_0=\tau_{E*}.
\]
This pair is induced, through the canonical core inclusions,
by a unital trace-preserving embedding $A\to B$. If both group
maps are isomorphisms, the realizing map may be chosen to be a
$^*$-isomorphism.
Equivalently, every unital morphism $\Ell(A)\to\Ell(B)$ is realized
by a unital embedding, and every isomorphism is realized by a
$^*$-isomorphism. In particular,
$A\cong B$ if and only if $\Ell(A)\cong\Ell(B)$.
No splitting condition on either unit class is required.

Moreover, the invariant range is exactly that described in
Theorem~\ref{thm:nuclear-classification}.
Every such invariant is realized by $F_0\fr D$, where $F_0$ is simple,
infinite-dimensional, unital, monotracial, nuclear, UCT and
$\mathcal Z$-stable.
\end{introtheorem}

The invariant here is that of the $D$-stabilized product.  It is
computed from the graded $\mathrm K$-groups of the core, its unit,
and its selected trace pairing; the original core may have other
traces.  The classification assertion follows from the marked
$D$-stabilization theorem, Theorem~\ref{thm:marked-core}.
Its marked form realizes every prescribed unital totally compatible
full-invariant morphism, with an isomorphism for an invertible morphism,
and can fix any identified finite reservoir-block prefix pointwise.
The source order and its other traces are not part of the selected data.

A necessary condition for an algebra to belong to the class of
Theorems~\ref{thm:nuclear-classification} and~\ref{thm:core-classification-intro}
is that its tracial von Neumann closure be isomorphic to $L(\mathbb F_\infty)$.
Indeed, the simple infinite-dimensional monotracial nuclear core $F_0$
has tracial closure the hyperfinite $\mathrm{II}_1$ factor $R$ by
\cite{Connes}. The interval presentation of $D$ gives
$D''\cong L(\mathbb F_\infty)$, so the tracial closure of $F_0\fr D$
is $R*L(\mathbb F_\infty)\cong L(\mathbb F_\infty)$ by
\cite[Theorem~4.1]{Dykema}.

The methodological point of \cite{HP} is essential for the proof of
Theorem~\ref{thm:core-classification-intro}: Hirshberg and Phillips
apply uniqueness to nuclear cores and extend the correcting
automorphisms over free complements before approximate intertwining.
Our marked $D$-stabilization argument follows this strategy for
nuclear UCT cores with arbitrary $\mathrm K$-groups, preserving
the selected trace pairing and full-invariant markings.
Their first absorption theorem also allows nonexact background
factors; our reservoirs have nuclear UCT blocks with scalar pointed
$\mathrm K$-theory.
We answer their circle--annulus isomorphism questions in
Example~\ref{ex:hp-annulus} and, for exact targets, their uniqueness
question in Remark~\ref{rem:hp-uniqueness}.

The analytic foundation of
Theorem~\ref{thm:core-classification-intro} is the following
classification theorem for embeddings with nuclear source.
It applies to every separable
simple unital infinite-dimensional exact monotracial target with
stable rank one and strict comparison. Neither a free-product
presentation nor nuclearity or the UCT is required for the target.
At this level, ordinary $\mathrm{K}$-groups and traces must be supplemented by
total $\mathrm{K}$-theory and Hausdorffized stable unitary algebraic $\mathrm{K}_1$
(see also \cite{CGSTW}).
We express the full invariant through the totally compatible triples of
\cite[Section~6]{GLNHom}, including the coefficient-to-unitary maps
$\zeta_n^A$.

\begin{introtheorem}[Embeddings]\label{thm:embedding-intro}
Let $E$ be a separable unital nuclear UCT algebra with a faithful tracial
state $\tau_E$. Let $T$ be a separable simple unital infinite-dimensional
exact algebra with unique tracial state $\tau_T$, stable rank one, and
strict comparison.
Every unital totally compatible triple $(\kappa,\kappa_T,\kappa_\gamma)$
from $E$ to $T$ satisfying $\kappa_T(\tau_T)=\tau_E$ is induced by a
unital embedding $E\to T$.
Two such embeddings are approximately unitarily equivalent if and only
if their full-invariant morphisms agree.
\end{introtheorem}

The precise statement is Theorem~\ref{thm:embedding}.
Its proof combines the trace-kernel ideal methods of \cite{KES},
the invariant and lifting framework of \cite{CGSTW}, and the
$\mathrm{KK}$-uniqueness theorem of \cite{Szabo}.
Working in $\mathrm{KL}$ accounts for pure-extension classes when the source
$\mathrm{K}$-groups are infinitely generated.
The maps $\zeta_n$ retain the compatibility between coefficient
classes and unitary data in the presence of torsion.
Coefficient homotopies and commutator witnesses are preserved through
descent from the norm ultrapower to the original target.

These existence and uniqueness results supply the embeddings and
unitary corrections for marked nuclear-core replacement.
The passage from approximation to isomorphism is an instance of
Elliott's one-sided approximate intertwining argument. Its
inner-automorphism form is already present in the proof of
$\mathcal O_2\otimes\mathcal O_2\cong\mathcal O_2$ recorded by
R\o rdam \cite{RordamO2}; see also
\cite[Proposition~2.3.5]{RordamClassification} and
\cite[Section~2]{TomsWinter}.
For infinite reduced free products, Hirshberg and Phillips already
implement the local automorphism approximation mechanism, using
nuclear-source uniqueness and extension of the correcting conjugacies
over free complements \cite[proof of Theorem~3.1]{HP}; they complete
their argument by Elliott approximate intertwining
\cite[Theorem~2.2]{HP}.
OpenAI's von Neumann algebra argument uses the analogous mechanism
\cite[Section~5]{OpenAIFreeFactors}.
Shlyakhtenko explicitly formulates and names CAAP
\cite[Proposition~1.2 and Definition~1.4]{Shlyakhtenko}.
Remark~\ref{rem:caap-history} explains these connections.
We use the
$\mathrm C^*$-automorphism approximation property~\eqref{eq:caap} and its
isomorphism criterion \cite[Proposition~1.2]{Shlyakhtenko}:
trace-preserving automorphisms nearly fix a prescribed finite subset
of the source image while moving that image close to prescribed
target elements. Our marked version requires every correcting
automorphism to act identically on the full invariant.
Two applications of nuclear-source uniqueness, with the corrections
extended over free complements, verify this property for core
replacement. The resulting marked isomorphisms give the isomorphism
assertion in Theorem~\ref{thm:core-classification-intro}.
For a general invariant morphism, replace the source and target
separately by stabilizations of simple monotracial nuclear range cores.
Their native core inclusions identify the full invariants.
Nuclear-source existence realizes the transported morphism between
the two range cores; free extension by $\id_D$ and conjugation by
the replacement isomorphisms give the required embedding.

Fusion then provides the passage from
Theorem~\ref{thm:core-classification-intro} to
Theorem~\ref{thm:nuclear-classification}: it replaces a finite or
countable family of nuclear factors by one nuclear core with the
same pointed graded $\mathrm K$-groups and selected pairing.

For the fusion of two factors, we first treat the case in which
their $\mathrm K_1$-groups vanish. A matrix enlargement gives
\[
 \mathrm M_2(A_1\fr A_2)
 \cong (A_1\oplus A_2)*_{\C^2}\mathrm M_2,
\]
with the canonical traces and expectations.
The direct sum $A_1\oplus A_2$ is nuclear, and the remaining
attachment is finite dimensional. The embedding theorem supplies
the coefficient maps needed to compare this amalgamated product
with a matrix algebra over a nuclear range model.

To turn this comparison into an isomorphism, we construct a map
back to the source and a correcting automorphism of the target.
The same map must satisfy two estimates: it must approximately
invert the comparison embedding on a prescribed source set,
and its composition with the corrected embedding must approximate
a prescribed target set. Independent parts of the reservoir make
these requirements compatible. Marked automorphism approximation
then gives the amplified isomorphism; correction and compression
of the distinguished corner yield fusion of the original factors.

The split-unit argument of Section~\ref{app:split-unit}, including
its Gaussian matrix model \cite{Dykema,NSS}, separates arbitrary
$\mathrm K_1$ and reduces the general case to the preceding
construction. Iteration gives finite fusion. For countably many
factors, each successive fusion is chosen to approximate the
preceding identification on an increasing finite set.
Summable estimates for both the maps and their inverses ensure
that the limiting map is onto.

Throughout the construction we retain compatible total
$\mathrm K$-theory and unitary data and can fix a finite external
reservoir head; Theorem~\ref{gf:general-fusion} states this marked
form of fusion. Such control of the maps does not enlarge the
final algebra invariant, since unit- and trace-preserving graded
$\mathrm K$-isomorphisms admit compatible full extensions
\cite{CGSTW}. Section~\ref{sec:fusion-outline} gives a fuller
outline of the construction.
Concrete AF and standard Moore-space models are discussed in
Section~\ref{sec:range}.

The general range assertions in
Theorems~\ref{thm:nuclear-classification} and~\ref{thm:core-classification-intro}
are proved in Theorem~\ref{thm:nuclear-range}.
For the commutative subclass, Theorem~\ref{thm:invariant-range}
imposes a condition on the pairing,
beyond the splitting of the unit: the pairing must factor through a
free abelian direct summand containing the unit as a basis element.
The type~I models and their standard-block refinements are proved in
Theorem~\ref{thm:standard-models} and
Corollary~\ref{cor:finite-standard-model}. The building-block spectra have dimension
at most three, and the interval-component masses record the trace
pairing. For general nuclear factors, homogeneous corners over finite
wedges of two-dimensional Moore spaces also record arbitrary unit
classes, including their torsion components. The connecting embeddings
are obtained by fusion and nuclear existence.

As a further application, the selected trace pairing determines the
exponential length of these free products. For a unital algebra $B$,
write $\cel(B)$ for the supremum of the infimal lengths of unitary
paths from $1$ to elements of $\mathrm{U}_0(B)$, and $\cucu(B)$ for the same
supremum restricted to the closed commutator subgroup. Exponential
rank $1+\varepsilon$ means that single exponentials are dense in
$\mathrm{U}_0(B)$ but do not exhaust it; see Section~\ref{sec:exponential}.

\begin{introtheorem}[Exponential rank and exponential length]\label{thm:exponential-intro}
Let $\mathcal A$ be any nonempty finite or countable family of
separable unital nuclear UCT algebras with faithful selected traces,
put $A=P_{\mathcal A}\fr D$, and let
$H=\tau_{A*}(\mathrm{K}_0(A))\subseteq\R$.
Then
\[
 \cel(A)=\cucu(A)=
 \begin{cases}
 2\pi,&H\text{ is discrete},\\
 \pi,&H\text{ is dense in }\R,
 \end{cases}
 \qquad \cer(A)=1+\varepsilon.
\]
These are the only two possibilities for the length.
No splitting condition on the unit classes is needed.
The same assertions hold for $P_{\mathcal A}$ whenever two distinct
factors admit Haar unitaries, and all the displayed values are unchanged
by matrix amplification.
\end{introtheorem}

In the commutative subclass, connected spectra always give length $2\pi$.
For disconnected spectra, irrational values of the trace pairing can
give length $\pi$. The rank remains $1+\varepsilon$ in both cases.
The $2\pi$ upper bound follows from the embedding theorem by lifting
the spectral distribution of a unitary to an interval and matching its
determinant. Dense trace range gives real rank zero and hence the
sharper $\pi$ bound. Two interval factors provide a unitary whose ambient
commutant rules out an exact logarithm.

Section~\ref{sec:prelim} fixes the notation and invariants.
Sections~\ref{sec:local} and~\ref{sec:marked-core} establish the
embedding and marked replacement theorems.
Section~\ref{sec:fusion-outline} outlines the fusion argument.
Sections~\ref{sec:relative-logs} and~\ref{sec:fd-relative} establish
the logarithm comparison and finite-dimensional relative reservoirs.
Section~\ref{app:split-unit} proves split-unit fusion independently
and establishes Gaussian matrix and full-corner identifications.
Section~\ref{sec:general-fusion} proves general fusion and the
classification assertions of Theorem~\ref{thm:nuclear-classification},
including closure under unital strong Morita equivalence.
Section~\ref{sec:range} gives the general and commutative ranges,
models built from AF factors, and standard free-product models. Exponential rank and exponential
length are computed in Section~\ref{sec:exponential}; examples follow
in Section~\ref{sec:examples}.

\medskip
\noindent\textbf{Acknowledgements.}
The work was carried out using ChatGPT.
The first author is supported by an NSERC Discovery Grant,
the second author is supported by an NNSF Grant of China
(No.~12671150), and the third author is supported by a
Simons Foundation Grant (MP-TSM-00002606).

The main classification theorem of the paper was obtained on October 20,
and then the paper was posted on arXiv on October 27, and it appeared on
the 28th. Stuart White kindly informed us on October 29 that, that morning
(the 28th), he and a group of seven collaborators had completed a proof
of the same result, which they had been working on since May.

\section{Preliminaries}\label{sec:prelim}

All algebras are separable and unital, unless stated otherwise.
Selected traces are faithful unless explicitly stated otherwise.
Free products are reduced free
products with respect to the displayed traces, and are amalgamated
over the scalar units unless an amalgamation algebra is indicated.
All isomorphisms of traced algebras preserve their selected traces.
For a unital algebra $B$, let $\mathrm{T}(B)$ denote its tracial state space and put
\[
 \rho_B:\mathrm{K}_0(B)\longrightarrow\Aff \mathrm{T}(B),\qquad
 \rho_B([p]-[q])(\tau)=(\tau\otimes\Tr)(p)-(\tau\otimes\Tr)(q).
\]
In $\mathrm K_0$-pairings, matrix traces are unnormalized.
For a selected trace $\tau_B$, the selected state on $\mathrm M_n(B)$
is $\tau_B\otimes\tr_n$, where $\tr_n=n^{-1}\Tr$.
Conditional expectations are denoted
by $\E$, with a subscript specifying the range when needed.
We use the Elliott invariant
\[
 \Ell(B)=(\mathrm{K}_0(B),\mathrm{K}_0(B)^+,[1_B],\mathrm{K}_1(B),\mathrm{T}(B),\rho_B).
\]
For a faithfully tracial core $(E,\tau_E)$, the comparisons below
use its pointed groups and selected trace pairing
\begin{equation}\label{eq:selected-data}
 (\mathrm{K}_0(E),[1_E],\mathrm{K}_1(E),\tau_{E*}).
\end{equation}
Proposition~\ref{prop:product-elliott} explains how these data
compute $\Ell(P\fr D)$ from its nuclear factors.
The free-product trace and its pairing are always the selected ones.
The reservoir $D$ is as in \eqref{eq:reservoir-intro}, with its
chosen block presentation. Write
\[
 D_{\le r}=\fr_{j=1}^{r}(R_j,\sigma_j),\qquad
 D_{>r}=\fr_{j>r}(R_j,\sigma_j),\qquad D_{\le0}=\C.
\]
A finite block head and its complementary tail always refer to this
presentation until the interval convention following
Corollary~\ref{cor:reservoir-model} is adopted.
We also write $I=(\mathrm C([0,1]),\lambda)$ and $D_I=I^{\fr\infty}$.

\begin{definition}[Reservoirs]\label{def:reservoir}
For a unital algebra $R$ with faithful tracial state $\sigma$, let
$\mu_{\sigma,a}$ denote the spectral distribution of $a=a^*\in R$, and put
\[
 \operatorname{At}(R,\sigma)
 =\inf_{a=a^*\in R}\sup_{t\in\operatorname{sp}(a)}
       \mu_{\sigma,a}(\{t\}).
\]
A reservoir in this paper is the reduced free product
\eqref{eq:reservoir-intro} of separable unital nuclear UCT blocks with
faithful selected traces, such that the unit maps
$\Z\to\mathrm K_0(R_j)$ are isomorphisms, $\mathrm K_1(R_j)=0$, and
\begin{equation}\label{eq:reservoir-atom}
 \sum_j\bigl(1-\operatorname{At}(R_j,\sigma_j)\bigr)=\infty.
\end{equation}
The atom condition is that of \cite[Definition~2.5 and Lemma~2.6]{HP}.
It persists after removing finitely many blocks.
\end{definition}

For the support-shifting argument we use a splitting
\begin{equation}\label{eq:split-unit}
 \mathrm{K}_0(E)=\Z[1_E]\oplus G_E.
\end{equation}
Equivalently, the unit class admits a group retraction to $\Z$.
Evaluation at a point supplies this splitting for every $\mathrm{C}(X)$.
The condition concerns a direct summand, and does not require the
remaining group to be free or finitely generated.
For a unital algebra $B$, let $\mathrm{DU}(B)$ denote the subgroup
generated by commutators in $\mathrm U_0(B)$, and let $\CU(B)$ be
its norm closure, as in \cite[Definition~2.1]{GLNHom}.

\subsection{\texorpdfstring{Total \(\mathrm{K}\)-theory and unitary quotients}{Total K-theory and unitary quotients}}
We use total \(\mathrm K\)-theory and its Bockstein operations as in
\cite[Section~4.1]{DG}; see also \cite[Definition~6.1]{GLNHom} for the
coefficient groups. For the coefficient algebras, we use the convention
of \cite[Sections~2.3 and~3.1]{CGSTW}. Thus
\[
 \ulK(B)=\mathrm{K}_0(B)\oplus \mathrm{K}_1(B)
 \oplus\bigoplus_{m\geq2}\bigl(\mathrm{K}_0(B;\Z/m)\oplus \mathrm{K}_1(B;\Z/m)\bigr).
\]
A \(\Lambda\)-morphism respects the grading and all coefficient and
Bockstein maps. For a separable UCT source \(A\), the universal
multicoefficient theorem \cite{DL} identifies
\begin{equation}\label{eq:umct}
 \KL(A,B)\cong\Hom_\Lambda(\ulK(A),\ulK(B)).
\end{equation}
When a target ideal is not \(\sigma\)-unital, we use the
separable-subalgebra formulation of \(\KK\) and \(\KL\) in
\cite[Appendix~B]{CGSTW}; the corresponding multicoefficient theorem
is \cite[Theorem~8.5]{CGSTW}. Every class and each of its countably many
witnesses is represented inside an appropriate separable subalgebra.
We specify the enlargements needed when using this convention.

Give \(\mathrm U_\infty(B)=\bigcup_m\mathrm U(\mathrm M_m(B))\) its
topological inductive-limit topology, and put
\[
 \mathrm{DU}_\infty(B)=\bigcup_m\mathrm{DU}(\mathrm M_m(B)),\qquad
 \mathrm{CU}_\infty(B)=\overline{\mathrm{DU}_\infty(B)}.
\]
The closure is taken in that topology. After stabilization the subgroup
generated by connected-unitary commutators agrees with the subgroup
generated by commutators of all stable unitaries
\cite[Section~2.2]{CGSTW}. We use the Hausdorffized stable
unitary algebraic \(\mathrm K_1\)-group \(\Uquot{B}\) of
\cite[Section~3]{ThomsenTraces}, with the notation of
\cite[Definitions~2.4 and~2.7]{GLNHom}, and write \(\bar u\) for the
class of a stable unitary \(u\). A unital homomorphism \(\phi:A\to B\)
induces the map
\[
 \phi^\ddagger:\Uquot{A}\longrightarrow\Uquot{B}.
\]
The determinant exact sequence \cite[Corollary~3.3]{ThomsenTraces} is
\begin{equation}\label{eq:thomsen}
 0\longrightarrow\Aff\mathrm T(B)/\overline{\rho_B(\mathrm K_0(B))}
 \xrightarrow{\mathrm{Th}_B}\Uquot{B}
 \xrightarrow{\Pi_B^{\mathrm{cu}}}\mathrm K_1(B)\longrightarrow0.
\end{equation}
Here \(\Pi_B^{\mathrm{cu}}\) is the ordinary \(\mathrm K_1\)-class map;
we also write \(\mathrm{Th}_B\) for its determinant-kernel inclusion
composed with the quotient map from \(\Aff\mathrm T(B)\).
If \(B\) has stable rank one, the natural map
\begin{equation}\label{eq:unstable-alg}
 \mathrm U(B)/\CU(B)\longrightarrow\Uquot{B}
\end{equation}
is bijective, and the same holds for every matrix algebra over \(B\)
\cite[discussion following Definition~2.6]{GLNHom}.
The fixed matrix size in this observation
is important in the descent argument.

Following \cite[Section~6]{GLNHom}, we describe a morphism of the full
invariant from \(A\) to \(B\) by a \emph{totally compatible triple}
\[
 (\kappa,\kappa_T,\kappa_\gamma).
\]
Here \(\kappa:\ulK(A)\to\ulK(B)\) is a unital \(\Lambda\)-morphism,
\(\kappa_T:\mathrm T(B)\to\mathrm T(A)\) is continuous affine, and
\[
 \kappa_\gamma:\Uquot{A}\longrightarrow\Uquot{B}
\]
is a group homomorphism. We write
\(\kappa_\sharp:\Aff\mathrm T(A)\to\Aff\mathrm T(B)\) for the unital
positive map \(f\mapsto f\circ\kappa_T\). Let \(\kappa_i\) and
\(\kappa_{i,m}\) denote the restrictions of \(\kappa\) to ordinary
and coefficient \(\mathrm K_i\)-groups, respectively. Total compatibility
means
\begin{align*}
 \kappa_\sharp\rho_A&=\rho_B\kappa_0,
 &\kappa_\gamma\mathrm{Th}_A&=\mathrm{Th}_B\kappa_\sharp,\\
 \Pi_B^{\mathrm{cu}}\kappa_\gamma&=\kappa_1\Pi_A^{\mathrm{cu}},
 &\kappa_\gamma\zeta_m^A&=\zeta_m^B\kappa_{0,m}\quad(m\ge2),
\end{align*}
where
\[
 \zeta_m^A:\mathrm K_0(A;\Z/m)\longrightarrow\Uquot{A}
\]
are the natural coefficient-to-unitary maps of
\cite[Remark~6.6]{GLNHom}. Write
\(\mu_m^A:\mathrm K_0(A)\to\mathrm K_0(A;\Z/m)\) for the
coefficient map and
\(\nu_m^A:\mathrm K_0(A;\Z/m)\to\mathrm K_1(A)\) for the
Bockstein map. With our orientation convention, these maps satisfy
\begin{equation}\label{eq:coefficient-normalization}
 \Pi_A^{\mathrm{cu}}\zeta_m^A=\nu_m^A,
 \qquad
 \zeta_m^A\mu_m^A
   =\mathrm{Th}_A\circ\bigl(m^{-1}\rho_A\bigr);
\end{equation}
see \cite[Propositions~3.1 and~3.2]{CGSTW}.
A compatible \(\kappa_\gamma\) is
automatically continuous \cite[Proposition~2.14]{CGSTW}.
Under \eqref{eq:umct}, we identify \(\kappa\) with its class in
\(\KL(A,B)\) whenever the source satisfies the UCT.
For a unital homomorphism \(\phi:A\to B\), the induced triple is
\[
 ([\phi],\phi_T,\phi^\ddagger),
\]
where \([\phi]\) denotes the induced \(\Lambda\)-morphism, identified
with \(\KL(\phi)\) when \eqref{eq:umct} applies. Composition of triples
uses the usual order on total \(\mathrm K\)-theory and unitary quotients,
and the reverse order on trace spaces. Componentwise injectivity or
surjectivity refers to the covariant maps on total \(\mathrm K\)-theory,
unitary quotients, and affine trace functions.

\begin{remark}[Total compatibility]\label{rem:total-invariant-gln}
In \cite[Definition~6.3]{GLNHom}, the enlarged triple
\((\kappa,\kappa_T,\kappa_\Gamma)\) also records the unitary maps for
the coefficient algebras. These additional maps are determined by the
basic triple under total compatibility
\cite[Proposition~6.4(3)]{GLNHom}. By \cite[Remark~6.6]{GLNHom},
total compatibility is equivalent to ordinary compatibility together
with the displayed \(\zeta_m\)-relations. Thus we use the basic triple
with these relations throughout.
Our coefficient-algebra convention is that of \cite{CGSTW}; the
coefficient-to-unitary maps agree with those of \cite{GLNHom} after
the interval-orientation identification recorded in
\cite[Remark~3.4]{CGSTW}.
\end{remark}

\subsection{The regularity class of targets}
Call \(T\) an admissible target if it is separable, unital, simple, infinite dimensional, exact and monotracial, and has strict comparison and stable rank one. Its unique trace is denoted \(\tau\). A homomorphism is trace preserving when its pullback of \(\tau\) equals the specified source trace.

We use \(A_\omega\) for the norm ultrapower associated with a free ultrafilter \(\omega\) on \(\N\), and
\[
 A_\infty=\ell^\infty(\N,A)/c_0(\N,A)
\]
for the norm sequence algebra. These two constructions are kept distinct.
For an admissible target $T$, its unique trace is extremal, so its
GNS von Neumann algebra is a finite factor, as is its tracial ultrapower.
We first work in the norm ultrapower $T_\omega$ and its tracial quotient,
and use the sequence algebra only for the final reparameterization step.

\begin{lemma}\label{lem:positive-cone}
For an admissible target \(T\),
\[
 \mathrm{K}_0(T)^+=\{0\}\cup\{g\in \mathrm{K}_0(T):\tau_*(g)>0\}.
\]
If \(0<\tau_*(g)<k\), the class \(g\) is represented by a projection in \(\mathrm{M}_k(T)\).
\end{lemma}
\begin{proof}
Write \(g=[p]-[q]\), with $p,q$ in matrix algebras over $T$.
If \(\tau_*(g)>0\), strict comparison gives \(q\precsim p\).
For projections, this gives a partial isometry $v$ in the stabilization
with $v^*v=q$ and $vv^*\le p$. Thus $r=p-vv^*$ is a projection with
$[r]=g$. If $\tau_*(g)<k$, compare $r$ with
$1_{\mathrm M_k(T)}$. The resulting subprojection of
$1_{\mathrm M_k(T)}$ represents $g$ in $\mathrm M_k(T)$.
Conversely, a nonzero projection has strictly positive trace, since
$\tau$ is faithful.
\end{proof}

\begin{lemma}[Regularity and $\mathrm K$-theory of reservoir tails]
\label{lem:interval-regularity}\label{lem:reservoir-regularity}
Let $D$ be a reservoir and let $P$ be a separable unital exact algebra
with a faithful tracial state. For every $r\ge0$, $P\fr D_{>r}$ is an
admissible target. The canonical inclusion
\[
 P\fr D_{>r}\longrightarrow P\fr D
\]
induces graded and total $\mathrm K$-isomorphisms preserving the unit
and selected pairing. It therefore induces a full-invariant
isomorphism. If $P$ is nuclear, its canonical inclusion into
$P\fr D_{>r}$ also induces graded and total $\mathrm K$-isomorphisms.
In particular,
\[
 \mathrm K_0(D)=\Z[1_D],\qquad \mathrm K_1(D)=0.
\]
\end{lemma}
\begin{proof}
Apply \cite[Lemma~2.6]{HP} to $P$ together with the reservoir blocks
remaining after the finite deletion. Condition~\eqref{eq:reservoir-atom}
gives simplicity, unique trace, stable rank one and strict comparison;
the same lemma gives infinite dimensionality. Exactness follows from
reduced free-product permanence \cite{DykemaExact}.

The faithful traces give GNS-faithful scalar expectations.
Hence the ordinary reduced free products here are covered by
\cite[Theorem~3.1 and Corollary~4.12]{FG}, even when a partial
product is nonnuclear. For every finite head $H=D_{\le r}$, its
scalar free-product $\mathrm K$-sequence and the assumptions on the blocks give
$\mathrm K_0(H)=\Z[1_H]$ and $\mathrm K_1(H)=0$.
Applying the same sequence to
$(P\fr D_{>r})\fr H$ proves the assertion about the tail inclusion.
Here the unit-relation map is injective because the faithful
normalized trace forces every unit class to have infinite order.
For nuclear $P$, iteration and inductive-limit continuity give the
assertion about the core inclusion. Coefficient exact sequences give
the assertions for total $\mathrm K$-theory.
The tail and full algebras are monotracial, their pairing ranges
agree, and the inclusion preserves their unique traces. The Thomsen
sequence therefore makes its algebraic unitary map an isomorphism
as well; naturality retains the coefficient-to-unitary maps.
This proves the full-invariant assertion.
\end{proof}

\subsection{Computation of the Elliott invariant}\label{subsec:product-elliott}

\begin{proposition}\label{prop:product-elliott}
Let
\[
 P=\fr_{i\in S}(A_i,\tau_i)
\]
be a nonempty finite or countable reduced free product of separable
unital nuclear UCT algebras with faithful selected traces, and put
$A=P\fr D$. No condition on the unit classes is needed for these
calculations.
The canonical inclusions identify
\begin{align*}
 \mathrm{K}_0(P)&\cong
 \left(\bigoplus_{i\in S}\mathrm{K}_0(A_i)\right)
 \big/\left\langle[1_{A_i}]-[1_{A_j}]:i,j\in S\right\rangle,
 \\
 \mathrm{K}_1(P)&\cong\bigoplus_{i\in S}\mathrm{K}_1(A_i).
\end{align*}
All direct sums are algebraic. The common factor-unit class corresponds
to $[1_P]$, and the selected trace pairing is
\begin{equation}\label{eq:product-trace-pairing}
 \tau_{P*}([(g_i)_i])=\sum_i(\tau_i)_*(g_i).
\end{equation}
The inclusion $P\to A$ induces isomorphisms on $\mathrm{K}_0$ and $\mathrm{K}_1$,
preserving the unit and this pairing. Moreover,
\[
 \mathrm{K}_0(A)^+=\{0\}\cup\{g:\tau_{A*}(g)>0\},
 \qquad \mathrm{T}(A)=\{\tau_A\}.
\]
Together these formulas compute $\Ell(A)$ from the factors.
\end{proposition}
\begin{proof}
The faithful factor traces give GNS-faithful scalar expectations.
The scalar-amalgamated free-product $\mathrm K$-theory sequence
\cite[Theorem~3.1 and Corollary~4.12]{FG} therefore applies at every
binary step, including steps whose first factor is a nonnuclear
partial product. Its unit-relation map is
\[
 \Z\longrightarrow\mathrm K_0(P_0)\oplus\mathrm K_0(A_i),
 \qquad n\longmapsto(n[1_{P_0}],-n[1_{A_i}]).
\]
Applying a normalized trace shows that this map is injective.
Thus each step contributes no additional $\mathrm K_1$-term,
without any primitivity assumption on the units. Iteration gives
the two group formulas for finite products. The canonical finite-prefix
inclusions are injective and have dense union, so inductive-limit
continuity gives the formulas for countable products. Each factor trace is
normalized, so the sum in \eqref{eq:product-trace-pairing} vanishes
on the displayed unit relations and gives the free-product pairing.
Since $\mathrm{K}_0(D)=\Z[1_D]$ and $\mathrm{K}_1(D)=0$, the same sequence shows that
adjoining $D$ preserves both groups, the unit, and the pairing.
Finally, Lemma~\ref{lem:interval-regularity} gives a unique trace
and strict comparison for $A$, and Lemma~\ref{lem:positive-cone}
gives its positive cone. This supplies every component of $\Ell(A)$.
\end{proof}

\subsection{Haar unitaries}

\begin{lemma}[A Haar unitary]\label{lem:continuous-Haar}
If $M$ is compact metrizable and $\mu$ is a nonatomic probability
measure, then $\mathrm{C}(M)$ contains a Haar unitary for $\tau_\mu$.
\end{lemma}

\begin{proof}
Choose real-valued continuous functions $f_k$ with $\|f_k\|\leq1$
which separate points.  For $a=(a_k)\in[0,1]^{\N}$ let
\[
 h_a(x)=\sum_{k=1}^\infty 2^{-k}a_k f_k(x).
\]
The series converges uniformly in both $a$ and $x$, and the resulting
function of $(a,x)$ is jointly continuous. Give the parameter space its product
Lebesgue measure.  If $x\ne y$, choose $k$ with $f_k(x)\ne f_k(y)$.
Conditioning on all other coordinates shows that the equation
$h_a(x)=h_a(y)$ has parameter probability zero.  Because $\mu$ has
no atoms, the diagonal has $(\mu\times\mu)$-measure zero.  Fubini's
theorem therefore gives a parameter $a$ such that
\[
 (\mu\times\mu)\{(x,y):h_a(x)=h_a(y)\}=0.
\]
The pushforward distribution of this continuous function has no atoms:
an atom of mass $c>0$ would contribute at least $c^2$ to the displayed
collision probability.  Its cumulative distribution function $F$ is
continuous.  The probability integral transform shows that the
continuous function $g=F\circ h_a$ has uniform distribution on
$[0,1]$.  Therefore $v=\exp(2\pi i g)\in \mathrm{C}(M)$ satisfies
$\tau_\mu(v^k)=0$ for every nonzero integer $k$, as required.
\end{proof}

\section{Embeddings into admissible targets}\label{sec:local}

We extend the trace-kernel approach to the admissible targets of
Section~\ref{sec:prelim}. The source is nuclear and satisfies the UCT;
the target may be nonnuclear. The proof uses $\mathrm{KL}$-uniqueness for
absorbing representations and retains coefficient and unitary data
through descent to the original target.

The final Theorem~B of \cite{CGSTW} assumes a nuclear,
$\mathcal Z$-stable codomain and is used only for nuclear range
models below. The theorem for possibly nonnuclear admissible
targets proved here instead uses the abstract lifting results
of \cite{CGSTW}, the trace-kernel regularity of \cite{KES}, and
the absorbing-pair uniqueness theorem of \cite{Szabo}.
See also \cite[Addendum~1.5]{CGSTW} for the role of these advances.

\subsection{Classification in the norm ultrapower}
\label{sec:ultrapower-classification}

Throughout this section, $T$ is a separable, simple, unital, infinite-dimensional,
exact $\mathrm{C}^*$-algebra with unique tracial state $\tau$, stable rank one, and strict
comparison.  We use strict comparison with respect to the canonical
unnormalized extensions of $\tau$ to matrix algebras.  Exactness ensures that
there are no additional normalized quasitraces to consider.  Fix a free
ultrafilter $\omega$ on $\N$, and put
\[
 B=T_\omega,\qquad Q=T^\omega,\qquad J=\ker(q:B\longrightarrow Q).
\]
The canonical trace on $B$ and its quotient trace on $Q$ are both denoted by
$\tau$.  Thus the extension under consideration is
\begin{equation}\label{eq:trace-kernel-extension}
 0\longrightarrow J\xrightarrow{j}B\xrightarrow{q}Q\longrightarrow0.
\end{equation}
Here $Q$ is the tracial ultrapower of the finite factor
$\pi_\tau(T)''$, and hence is a $\mathrm{II}_1$ factor.  The identification of
these two descriptions of $Q$ follows from Kaplansky density; see
\cite[Section~2.3]{KES}. In particular,
\[
\mathrm K_0(Q)=\R,\qquad \mathrm K_1(Q)=0,
\]
where the first identification uses the unnormalized trace on matrix
projections. The second assertion follows because every unitary in a
finite von Neumann algebra has a bounded Borel self-adjoint logarithm.

We use the total $\mathrm K$-theory and stabilized unitary quotients
of Section~\ref{sec:prelim}. The closure defining
$\mathrm{CU}_\infty(C)$ is taken in the topological inductive-limit
topology. The canonical map and determinant inclusion are
\[
 \Pi_C^{\mathrm{cu}}:\Uquot{C}\longrightarrow\mathrm K_1(C),\qquad
 \mathrm{Th}_C:\Aff\mathrm T(C)\longrightarrow\Uquot{C}.
\]
In particular,
\begin{equation}\label{eq:thomsen-general}
 \mathrm K_0(C)\xrightarrow{\rho_C}\Aff\mathrm T(C)
 \xrightarrow{\mathrm{Th}_C}\Uquot{C}
 \xrightarrow{\Pi_C^{\mathrm{cu}}}\mathrm K_1(C)\longrightarrow0
\end{equation}
is exact when the kernel at $\Aff\mathrm T(C)$ is interpreted as
$\overline{\rho_C(\mathrm K_0(C))}$.
We retain totally compatible triples in the notation of
Section~\ref{sec:prelim}, including compatibility with the maps $\zeta_n^C$.
For a monotracial target, a triple is faithful if its trace component
pulls the unique target trace back to a faithful source trace.

The ideal regularity established below permits the use of abstract
existence and uniqueness for absorbing representations.

\begin{lemma}\label{lem:trace-kernel-regularity}
The following assertions hold.
\begin{enumerate}
\item $B$ has stable rank one and strict comparison with respect to its
 canonical trace.
\item Every tracial state on $B$ vanishes on $J$.  Consequently $B$ has unique
 tracial state.
\item $J$ is separably stable and has almost unperforated Cuntz semigroup.
 Every separable subalgebra of $J$ is contained in a separable, stable
 subalgebra of $J$ having almost unperforated Cuntz semigroup, and therefore
 having the corona factorization property.
\item Every trace-preserving homomorphism from a unital algebra with faithful
 trace to $B$ is full.
\end{enumerate}
\end{lemma}

\begin{proof}
Stable rank one passes to norm ultrapowers with a uniform inverse bound.
For a contraction $x_n\in T$ and $\varepsilon>0$, choose an invertible
$y_n$ with $\|y_n-x_n\|<\varepsilon/2$, and write
$y_n=v_n|y_n|$ with $v_n$ unitary. Then
\[
z_n=v_n(|y_n|+\varepsilon/2)
\]
satisfies $\|z_n-x_n\|<\varepsilon$ and
$\|z_n^{-1}\|\leq2/\varepsilon$, uniformly in $n$.
Thus $(z_n)_\omega$ is invertible in $B$ and approximates $(x_n)_\omega$.
The same argument works in each fixed matrix size.
We spell out the uniform bound needed for strict comparison, as in
\cite[proof of Proposition~2.5]{KES}. Let $a=(a_n)$ and $b=(b_n)$ be
positive contractions in a fixed matrix algebra over $B$, and suppose
$d_\tau(a)<d_\tau(b)$. For each $\varepsilon>0$, continuous spectral
cutoffs give $0<\delta<1$ such that
\[
 d_\tau((a_n-\varepsilon)_+)
 <d_\tau((b_n-\delta)_+)
\]
on a set in $\omega$. Indeed, dominate the support of
$(a_n-\varepsilon)_+$ by a continuous function of $a_n$ supported away
from zero, and minorize the support of $(b_n-\delta)_+$ by a continuous
function of $b_n$ whose limiting trace exceeds $d_\tau(a)$.
Coordinate strict comparison and the usual cutdown factorization give
$v_n$ with $\|v_n\|\leq1$,
\[
 v_n^*v_n=(a_n-2\varepsilon)_+,
 \qquad v_nv_n^*\in\operatorname{her}((b_n-\delta)_+).
\]
Choose a bounded continuous function $f_\delta$ equal to $t^{-1/2}$
on $[\delta,1]$, with $\|f_\delta\|\leq\delta^{-1/2}$. Then
$w_n=f_\delta(b_n)v_n$ satisfies
$w_n^*b_nw_n=(a_n-2\varepsilon)_+$ and
$\|w_n\|\leq\delta^{-1/2}$; set $w_n=0$ off the chosen set.
These bounded sequences give $(a-2\varepsilon)_+\precsim b$ in the
ultrapower. Letting $\varepsilon\downarrow0$ proves $a\precsim b$.
The same argument uses unnormalized traces in every matrix size.
Nuclearity or Jiang--Su stability of $T$ is not needed.

If $a\in J_+$ is a contraction, then $d_\tau(a)=0$, since every positive
root of $a$ still belongs to $J$.  Strict comparison in $B$ gives
$a^{\oplus k}\precsim1_B$ for every positive integer $k$.  For any tracial
state $\lambda$ on $B$, it follows that $k d_\lambda(a)\leq1$.  Thus
$\lambda(a)=0$, and $\lambda$ factors through $Q$.  A finite factor has unique
tracial state, by Dixmier averaging, so $\lambda=\tau$.

The assertions about separable stability and almost unperforation are
\cite[Propositions~2.4 and~2.5]{KES}.  One alternates the two separable
enlargements and passes to their inductive limit; almost unperforation and
stability pass to this limit.  A separable stable algebra with almost
unperforated Cuntz semigroup has the corona factorization property by
\cite[Proposition~2.17 and Theorem~5.11]{OPR}; this consequence is also recorded
in \cite[Section~2.4]{KES}.

For the final assertion, let $\phi:A\to B$ preserve a faithful trace, and let
$a\in A_+$ be nonzero.  Then $d_\tau(\phi(a))>0$.  For sufficiently large $m$,
strict comparison gives $1_B\precsim\phi(a)^{\oplus m}$.  Thus $\phi(a)$ is
full in $B$.
\end{proof}

\begin{lemma}\label{lem:closed-trace-range}
Put $H=\overline{\tau_*(\mathrm{K}_0(T))}\subseteq\R$.  Then
\[
 \tau_*(\mathrm{K}_0(B))=H.
\]
In particular, this range is closed.
\end{lemma}

\begin{proof}
A closed subgroup of $\R$ containing $1$ is either $\R$ or a discrete cyclic
group.  If $H$ is discrete, a projection in $\mathrm{M}_k(B)$ lifts to a sequence of
projections in $\mathrm{M}_k(T)$.  Its trace is an ultralimit of elements of the bounded
set $H\cap[0,k]$, and therefore belongs to $H$.  This proves one inclusion;
the reverse inclusion follows because a subgroup with discrete closure equals its closure, and constant classes realize the original trace image.

Suppose instead that $H=\R$, and fix $s\in\R$.  Choose $g_m\in \mathrm{K}_0(T)$ with
$\tau_*(g_m)\to s$.  Choose fixed positive integers $k,l$ such that
$0<\tau_*(g_m)+k<l$ for all $m$, after discarding finitely many terms.  Strict
comparison implies that every class of strictly positive trace in $\mathrm{K}_0(T)$ is
represented by a projection: if $g=[p]-[r]$ and $\tau(p)>\tau(r)$, first
subequate $r$ to $p$ and take the complementary projection.  A second comparison
with $1_{\mathrm{M}_l(T)}$ shows that a class of trace in $(0,l)$ has a projection
representative in $\mathrm{M}_l(T)$.  Hence choose $p_m\in \mathrm{M}_l(T)$ with
$[p_m]=g_m+k[1_T]$.  The class
$[(p_m)_\omega]-k[1_B]$ has trace $s$.
\end{proof}

\begin{proposition}\label{prop:kernel-algebraic-k1}
Let $s:J^\dagger\to\C1_B$ be the scalar map on the forced unitization.
There is a natural isomorphism
\[
 \Omega_T:\mathrm{K}_1(J)\xrightarrow{\cong}\Uquot{B},\qquad
 \Omega_T([u]_1)=\overline{u}-\overline{s(u)}.
\]
It identifies $\ker \mathrm{K}_1(j)$ with $\ker\Pi_B^{cu}$, both identified with
$\R/H$.
\end{proposition}

\begin{proof}
The argument is the proof of \cite[Theorem~6.8]{CGSTW}, with its hypotheses
replaced by the conclusions just proved.  We include the details needed here.
The formula is additive on matrix unitaries.  To see that it vanishes on
null-homotopic unitaries, it suffices to check exponentials.  For
$h=h^*\in \mathrm{M}_k(J^\dagger)$, the difference between $h$ and its scalar part has
zero value at every trace on $B$, by Lemma~\ref{lem:trace-kernel-regularity}.
The Thomsen map, or the Lie--Trotter formula modulo closed commutators, therefore
gives
\[
 \overline{e^{2\pi i h}}
 -\overline{e^{2\pi i s(h)}}=0.
\]
This proves that the formula descends to $\mathrm{K}_1(J)$.

Under $\mathrm{K}_0(Q)=\R$, the six-term sequence and
\eqref{eq:thomsen-general} give the commuting diagram
\[
\begin{array}{ccccccccc}
 \mathrm{K}_0(B)&\longrightarrow&\R&\xrightarrow{\partial}&\mathrm{K}_1(J)
 &\xrightarrow{\mathrm{K}_1(j)}&\mathrm{K}_1(B)&\longrightarrow&0\\
 \Vert&&\Vert&&\downarrow\Omega_T&&\Vert\\
 \mathrm{K}_0(B)&\xrightarrow{\tau_*}&\R&\xrightarrow{\mathrm{Th}_B}&
 \Uquot{B}&\xrightarrow{\Pi_B^{cu}}&\mathrm{K}_1(B)&\longrightarrow&0.
\end{array}
\]
The lower row is exact because $\tau_*(\mathrm{K}_0(B))=H$ is closed.  For the middle
square, a projection $p\in \mathrm{M}_k(Q)$ has a positive contraction lift
$a\in \mathrm{M}_k(B)$ and
$\partial[p]=[e^{2\pi i a}]_1$; its image under the displayed formula is
$\mathrm{Th}_B(\tau(a))$.  The other squares commute directly.  Equivalently, this gives a morphism of short exact sequences
\[
\begin{array}{ccccccccc}
0&\longrightarrow&\R/H&\longrightarrow&\mathrm K_1(J)&\longrightarrow&
\mathrm K_1(B)&\longrightarrow&0\\
&&\Vert&&\downarrow\Omega_T&&\Vert\\
0&\longrightarrow&\R/H&\longrightarrow&\Uquot{B}&\longrightarrow&
\mathrm K_1(B)&\longrightarrow&0 .
\end{array}
\]
The short five lemma proves that $\Omega_T$ is an isomorphism.
Its restriction to the two kernels is the identity under their
identification with $\R/H$.
\end{proof}

We recall explicitly the two abstract tools used below.  First,
\cite[Theorem~5.16(i),(ii)]{CGSTW} is an existence theorem for a separable source
and a $\sigma$-unital stable ideal $I$.  Its part (ii) requires the source to be
nuclear; it does not require a nuclear or Jiang--Su stable codomain.  Second,
\cite[Theorem~2.6]{Gabe} says that, for a separable stable ideal with
the corona factorization property, a Busby map is nuclearly absorbing exactly
when it is unitizably full.  For a nuclear source, nuclear absorption is
absorption.  The multiplier representation is absorbing if and only if its
Busby map is absorbing, by \cite[Corollary~5.11]{CGSTW}.

\begin{lemma}\label{lem:separable-subextension}
In \eqref{eq:trace-kernel-extension}, or in any finite matrix amplification of
that extension (with $B,J,Q$ replaced accordingly), one may pass to a unital
separable subextension
\[
 0\longrightarrow I\longrightarrow S\xrightarrow{q_S}D\longrightarrow0,
 \qquad I=S\cap J,
\]
with $I$ stable and $\Cu(I)$ almost unperforated, subject to the following
additional requirements whenever the indicated data are present.
\begin{enumerate}
\item A prescribed unitizably full map $\theta:A\to Q$ with separable source
 corestricts to a unitizably full map $\theta_D:A\to D$.
\item If $\widetilde\kappa\in \KK(A,B)$ satisfies $q_*(\widetilde\kappa)=\KK(\theta)$, one may retain
 $\widetilde\kappa_S\in \KK(A,S)$ with image $\widetilde\kappa$ and
 $(q_S)_*(\widetilde\kappa_S)=\KK(\theta_D)$.
\item Prescribed separable sets of homomorphism images may be included in $S$.
 A prescribed class in $\KK(A,J)$ may be retained in $\KK(A,I)$.
\item If a prescribed Cuntz pair has zero class in $\KL(A,J)$, the subextension
 may be chosen so that its class already vanishes in $\KL(A,I)$.
\end{enumerate}
In particular, $I$ has the corona factorization property.
\end{lemma}

\begin{proof}
This is the separabilization argument of
\cite[Lemma~7.8 and its proof]{CGSTW}; we spell out why Jiang--Su stability is
unnecessary here. To retain unitizable fullness, choose a countable dense
family in $(A^\dagger)_+$ and all its nonzero rational spectral cutdowns.
For each such cutdown, adjoin finitely many witnesses that its image under
$\theta^\dagger$ generates the quotient unit as a closed ideal.
Use the same quotient unit throughout. These witnesses retain fullness
of the forced-unitization map, including its complementary scalar summand
when $A$ is unital. The factorization of a KK-class and an equality
between its quotient and $\KK(\theta)$ are retained by the direct-limit
description
\[
 \KK(A,C)=\varinjlim_{C_0\subseteq C\ \mathrm{separable}}\KK(A,C_0)
\]
of \cite[Definition~5.4 and Proposition~B.6]{CGSTW}.  The analogous
direct-limit description for KL, \cite[Proposition~5.7]{CGSTW}, retains a zero
KL-class after a further separable enlargement.  These requirements persist
under subsequent enlargement of the subextension.

Starting from such a separable subextension, alternately enlarge its intersection
with $J$ to a separable stable subalgebra and to one with almost unperforated
Cuntz semigroup, using Lemma~\ref{lem:trace-kernel-regularity}.  Include this
enlargement in the middle algebra, take its intersection with $J$ again, and
continue.  If $S$ is the closure of the increasing middle algebras, the closure
$I$ of the intervening ideal enlargements is exactly $S\cap J$.  Indeed, for
$x\in S\cap J$ approximate $x$ by $s$ in an earlier middle algebra.  The
quotient norm formula gives an element in that earlier algebra's intersection
with $J$ within $\|q(s)\|$ plus an arbitrarily small error of $s$.  This
intersection was included in the next ideal enlargement.  Thus $x\in I$.
Stability and almost unperforation pass to the inductive limit.  This replaces
the separable Jiang--Su stable enlargement in the cited proof; no other step
of that proof uses it.  The final assertion follows from
Lemma~\ref{lem:trace-kernel-regularity}.
\end{proof}

\begin{lemma}\label{lem:abstract-full-lift}
Let $A$ be separable and nuclear, let $\theta:A\to Q$ be unitizably full, and
suppose $\widetilde\kappa\in \KK(A,B)$ satisfies $q_*(\widetilde\kappa)=\KK(\theta)$.  There is a
homomorphism $\phi:A\to B$ such that $q\phi=\theta$ and $\KK(\phi)=\widetilde\kappa$.
Moreover, given a lift $\psi:A\to B$ of $\theta$ and a class
$\delta\in \KK(A,J)$, there is another lift $\phi$ with
$[\phi,\psi]_{\KK(A,J)}=\delta$.
\end{lemma}

\begin{proof}
Choose a subextension as in Lemma~\ref{lem:separable-subextension}, retaining
$\theta_D$ and $\widetilde\kappa_S$.  Let $\mu:S\to M(I)$ be the multiplier map and
$\overline\mu:D\to Q(I)$ its Busby map. The multiplier map need not
be faithful on all of $S$, but the extension is the pullback
\[
S\cong\{(m,d)\in M(I)\oplus D:
          \pi_I(m)=\overline\mu(d)\}.
\]
Indeed, lift $d$ to $s\in S$ and correct $m-\mu(s)\in I$.
Injectivity follows because $\mu$ is faithful on $I$.
The Busby map is unital. Since $\theta_D$ is already unitizably
full, its forced unitization is full, and composition with the unital
$\overline\mu$ preserves fullness. Hence
$\overline\mu\theta_D$ is unitizably full.
It is absorbing by the corona
factorization property and the absorption criterion recalled above.  The
commuting pullback square shows that its KK-class factors through
$\KK(A,M(I))$, which is zero for stable $I$; see
\cite[proof of Theorem~5.16(ii)]{CGSTW}.  Theorem~5.16(ii) therefore lifts it to
an absorbing homomorphism into $M(I)$.  Pullback gives a homomorphism
$\psi:A\to S$ lifting $\theta_D$.

One must still correct its KK-class.  In this same separable extension,
\[
 \widetilde\kappa_S-\KK(\psi)\in\ker\bigl(\KK(A,q_S)\bigr).
\]
Half-exactness in the second variable for nuclear $A$ gives a preimage
$\delta'\in \KK(A,I)$; see \cite[Proposition~5.5(v)]{CGSTW}.  Since $\mu\psi$
is absorbing, Theorem~5.16(i) gives an absorbing $\overline\phi:A\to M(I)$
with $(\overline\phi,\mu\psi)$ a Cuntz pair of class $\delta'$.  It has the
same Busby quotient as $\mu\psi$, and therefore pulls back to
$\phi:A\to S$.  Naturality of Cuntz-pair classes,
\cite[Proposition~5.5(iv)]{CGSTW}, gives
$\KK(\phi)-\KK(\psi)=\widetilde\kappa_S-\KK(\psi)$.  Thus $\KK(\phi)=\widetilde\kappa_S$.

For the relative assertion, separabilize while retaining $\psi$ and $\delta$.
The representation $\mu\psi$ is absorbing by the same criterion.  Apply
Theorem~5.16(i) directly and pull back.  Naturality gives the specified
relative class after inclusion into $J$.
\end{proof}

\begin{lemma}\label{lem:kl-uniqueness-abstract}
Let $A$ and $I$ be separable, with $I$ stable.  If two absorbing representations
$\phi,\psi:A\to M(I)$ form a Cuntz pair with zero class in $\KL(A,I)$, then
there are unitaries $v_n\in1+I$ such that
\[
 \|v_n\phi(a)v_n^*-\psi(a)\|\longrightarrow0\qquad(a\in A).
\]
\end{lemma}

\begin{proof}
This follows from \cite[Theorem~A]{Szabo} by the formal proof of
\cite[Definition~5.6 and proof of Theorem~5.17(ii)]{CGSTW}; unlike the theorem stated there, the conclusion
here has no Jiang--Su tensor factor.  For completeness, KL-vanishing supplies
$\eta\in \KK(A,\mathrm{C}(\overline\N,I))$, where
$\overline\N=\N\cup\{\infty\}$, whose evaluations at finite integers are zero
and whose evaluation at infinity is $[\phi,\psi]$.  The constant representation
$\Psi(a)(n)=\psi(a)$ is absorbing into
$M(\mathrm C(\overline\N,I))=\mathrm C_{\mathrm{strict}}(\overline\N,M(I))$
by \cite[Proposition~3.2]{DadarlatTopology}, applied to the absorbing
representation $\psi$ and the separable unital nuclear algebra
$\mathrm C(\overline\N)$.
Abstract KK existence gives an absorbing $\Phi$ with Cuntz-pair class $\eta$
relative to $\Psi$.  Write $\phi_n=\mathrm{ev}_n\Phi$.  Since
$\Phi(a)-\Psi(a)\in \mathrm{C}(\overline\N,I)$, one has
$\|\phi_n(a)-\phi_\infty(a)\|\to0$, not merely strict convergence.
Each $\phi_n$ is absorbing, since its corona map equals that of $\psi$.
Naturality and additivity of Cuntz-pair classes give, for finite $n$,
\[
 [\phi_n,\psi]_{\KK(A,I)}=(\mathrm{ev}_n)_*(\eta)=0,
 \qquad
 [\phi,\phi_\infty]_{\KK(A,I)}
   =[\phi,\psi]_{\KK(A,I)}-(\mathrm{ev}_\infty)_*(\eta)=0.
\]
These two $\KK$-equalities are the inputs to the uniqueness theorem
\cite[Theorem~A]{Szabo}, which gives approximate
conjugacies by unitaries in $1+I$ from
$\phi_n$ to $\psi$ for finite $n$, and from $\phi$ to $\phi_\infty$, since the
corresponding KK-classes vanish.  For any fixed finite set and tolerance, first
choose $n$ so that $\phi_n$ and $\phi_\infty$ are close, and then compose these
two conjugacies.  Exhausting the source proves the assertion.
\end{proof}

\begin{proposition}\label{prop:relative-rotation}
Let $A$ be separable and satisfy the UCT.  There is a natural isomorphism
\begin{equation}\label{eq:rotation-isomorphism}
 \mathcal R_{A,T}:\ker\bigl(\KL(A,j):\KL(A,J)\to \KL(A,B)\bigr)
 \xrightarrow{\cong}
 \Hom\bigl(\mathrm{K}_1(A)/\Tor \mathrm{K}_1(A),\R/H\bigr).
\end{equation}
For a unital Cuntz pair $\phi,\psi:A\to B$ with relative class $\lambda$ in
this kernel, it satisfies
\begin{equation}\label{eq:rotation-pair-formula}
 \bigl(\phi^\ddagger-\psi^\ddagger\bigr)(x)
 =\mathcal R_{A,T}(\lambda)
       \bigl(\Pi_A^{cu}(x)+\Tor \mathrm{K}_1(A)\bigr).
\end{equation}
\end{proposition}

\begin{proof}
Apply the UMCT, \cite[Theorem~8.5]{CGSTW}, to identify $\KL$ with
morphisms of total $\mathrm K$-theory, naturally in the second variable.
This applies to the nonseparable ideals here through the
separable-subalgebra convention of Appendix~B of that paper.
The quotient $Q$ has
$\mathrm{K}_0(Q)=\R$, $\mathrm{K}_1(Q)=0$, and all its finite-coefficient $\mathrm{K}$-groups vanish.
Here is the relevant algebraic kernel calculation explicitly.
The six-term sequence gives injectivity of $\mathrm K_0(j)$.
Since every finite-coefficient group of $Q$ is zero, every
finite-coefficient map induced by $j$ is an isomorphism.
Moreover
\[
V:=\ker\mathrm K_1(j)\cong\R/H
\]
is divisible. A total-$\mathrm K$ morphism killed by $j$ therefore
has zero $\mathrm K_0$ component and zero coefficient components;
only its ordinary $\mathrm K_1$ component can remain, with image in $V$.
Bockstein naturality forces that component to vanish on torsion,
since each $n$-torsion class is in a Bockstein image.
Conversely a map $\mathrm K_1(A)/\Tor\mathrm K_1(A)\to V$,
together with zero maps on all other components, respects reduction
modulo $n$ because $V$ is divisible, and respects Bocksteins because
it kills torsion. All other coefficient squares are zero.
Thus the kernel calculation of \cite[Lemma~2.18]{CGSTW} gives
\[
 \ker\Hom_\Lambda(\ulK(A),\ulK(j))
 \cong
 \Hom\bigl(\mathrm K_1(A)/\Tor\mathrm K_1(A),
                 \ker\mathrm K_1(j)\bigr),
\]
by the ordinary $\mathrm K_1$ component.  Compose with
Proposition~\ref{prop:kernel-algebraic-k1} to obtain
\eqref{eq:rotation-isomorphism}.  This is exactly the proof of
\cite[Theorem~8.7]{CGSTW}; the nuclear and Jiang--Su stable codomain in that
statement is used there to obtain the quotient $\mathrm{K}$-theory and the isomorphism
$\Omega_T$, which have been proved independently here.

For a matrix unitary $u$ in $A$, the $\mathrm{K}_1$ component of the relative Cuntz-pair
class is $[\phi(u)\psi(u)^*]_1\in \mathrm{K}_1(J)$, using the scalar-unit convention;
see \cite[Proposition~8.3]{CGSTW}.  Applying $\Omega_T$ gives
$\overline{\phi(u)}-\overline{\psi(u)}$.  This proves
\eqref{eq:rotation-pair-formula}.
\end{proof}

\begin{theorem}\label{thm:ultrapower-full-invariant}
Let $T$ be an admissible target, let $\omega$ be a free ultrafilter
on $\N$, and put $B=T_\omega$.
Let $A$ be separable, unital, nuclear, and satisfy the UCT, and let $\lambda$ be
a faithful tracial state on $A$. Then:
\begin{enumerate}
\item Every unital totally compatible triple
$(\kappa,\kappa_T,\kappa_\gamma)$ from $A$ to $B$ satisfying
$\kappa_T(\tau)=\lambda$ is realized by a unital embedding $A\to B$.
\item Two unital embeddings $A\to B$ inducing the same triple and the
 trace $\lambda$ are unitarily equivalent.
\end{enumerate}
No vanishing, torsion, or finite-generation assumption is imposed on either
$\mathrm{K}$-group of $A$ or $T$.
\end{theorem}

\begin{proof}
We first record the quotient-level input.  The finite von Neumann algebra
$\pi_\lambda(A)''$ is hyperfinite, by nuclearity and the hyperfiniteness
theorem of \cite{Connes}. It admits a normal
trace-preserving embedding into the $\mathrm{II}_1$ factor $Q$.  One can obtain
this by successively embedding its increasing finite-dimensional
approximating algebras, using unitary conjugacy of trace-preserving embeddings
of finite-dimensional algebras into a factor, and then taking bounded
$2$-norm limits.  Its restriction gives a unital embedding
$\sigma:A\to Q$ with trace $\lambda$.
Moreover, any two embeddings $A\to Q$ inducing $\lambda$ extend
normally to trace-preserving embeddings of $\pi_\lambda(A)''$:
the trace identities give $\mathrm L^2$ isometries, and bounded
$\mathrm L^2$ approximation gives the extensions.
They need not have factorial domain.
The finite-dimensional approximation just described first gives
approximate conjugacy in $2$-norm on finite sets.  Countable saturation of the
tracial ultrapower, equivalently the usual $2$-norm epsilon test, gives one
unitary implementing the conjugacy on all of $A$.

\emph{Existence of total $\mathrm{K}$-theory and trace.}
Via the UMCT, regard $\kappa$ as a class in $\KL(A,B)$ and choose a
representative $\widetilde\kappa\in\KK(A,B)$.
Its image in $\KK(A,Q)$ equals $\KK(\sigma)$.
Indeed, $Q$ is unital and hence $\sigma$-unital, so the UCT applies.
Since $\mathrm K_0(Q)=\R$ is divisible and $\mathrm K_1(Q)=0$,
the extension terms vanish and
\[
\KK(A,Q)\cong\Hom(\mathrm K_0(A),\R).
\]
Trace compatibility gives equality in this group, not merely equality
after passage to $\KL$.
The map $\sigma$ itself is unital, so it is not unitizably full as a map from
the unital source $A$.  To apply Lemma~\ref{lem:abstract-full-lift}, use instead
\[
 \sigma_0(a)=\begin{pmatrix}\sigma(a)&0\\0&0\end{pmatrix}
       :A\longrightarrow \mathrm{M}_2(Q),\qquad
 \widetilde\kappa_0=(e_{11})_*(\widetilde\kappa).
\]
The forced unitization of $\sigma_0$ maps $A\oplus\C$ to $\mathrm{M}_2(Q)$ by
$(a,z)\mapsto\operatorname{diag}(\sigma(a),z)$, so it is full.  Here and
below, the displayed coordinates identify the forced unitization of a unital
algebra with its direct sum with $\C$.
The lemma gives $\rho:A\to \mathrm{M}_2(B)$ lifting $\sigma_0$ with class
$\widetilde\kappa_0$.
Set $p=\rho(1_A)$.  Unitality of the prescribed class gives
$[p]=[e_{11}]\in \mathrm{K}_0(\mathrm{M}_2(B))$.  Stable rank one gives cancellation, hence
Murray--von Neumann equivalence of these projections and also of their
complements.  Thus they are unitarily equivalent in $\mathrm{M}_2(B)$.  Conjugate
$\rho$ into the $e_{11}$ corner and identify that corner with $B$.
The resulting unital homomorphism $\psi:A\to B$ has total $\mathrm{K}$-theory map
$\kappa$ and trace $\lambda$.  Faithfulness of $\lambda$ makes it injective.

\emph{Correction of algebraic $\mathrm{K}_1$.}
Put $d=\kappa_\gamma-\psi^\ddagger$.  Compatibility with the Thomsen maps
shows that $d$ vanishes on $\mathrm{im}(\mathrm{Th}_A)$, and compatibility with
$\Pi^{cu}$ shows that its image is contained in $\ker\Pi_B^{cu}=\R/H$.
Hence it factors as a map $r:\mathrm{K}_1(A)\to\R/H$.  It vanishes on torsion.
Indeed, an element killed by $n$ is in the image of the Bockstein map
$\mathrm{K}_0(A;\Z/n)\to \mathrm{K}_1(A)$; the factorization of that Bockstein map through
$\zeta_n^A$, and total compatibility of both triples,
give $d\zeta_n^A=0$.  This is the algebraic argument in
\cite[proof of Theorem~8.9(ii)]{CGSTW}.

Write $\overline r:\mathrm K_1(A)/\Tor\mathrm K_1(A)\to\R/H$
for the induced map.
Choose $\eta\in\ker\KL(A,j)$ with
$\mathcal R_{A,T}(\eta)=\overline r$, and choose a
representative $\delta\in \KK(A,J)$ of $\eta$.  Apply the relative assertion of
Lemma~\ref{lem:abstract-full-lift} to
$\psi_0=\operatorname{diag}(\psi,0)$ and the corner image of $\delta$.
This gives $\rho_0:A\to \mathrm{M}_2(B)$ with the same quotient as $\psi_0$ and
relative class equal to that corner image.  Its absolute total $\mathrm{K}$-theory map
is unchanged, since $\eta$ maps to zero in $\KL(A,B)$.  In particular,
$[\rho_0(1_A)]=[e_{11}]$.

For clarity, the algebraic $\mathrm{K}_1$ calculation does not treat a nonunital image of
a unitary as a unitary.  Write $p_0=\rho_0(1_A)$ and extend the two maps to
unital maps on $A^\dagger=A\oplus\C$:
\[
 \rho_0^\dagger(a,z)=\rho_0(a)+z(1-p_0),\qquad
 \psi_0^\dagger(a,z)=\operatorname{diag}(\psi(a),z).
\]
For $u\in \mathrm{U}_k(A)$, the unitary $(u,1_k)\in \mathrm{U}_k(A\oplus\C)$ is sent to
$\rho_0^{(k)}(u)+1-p_0^{(k)}$ and
$\operatorname{diag}(\psi^{(k)}(u),1_k)$, respectively.  The $\mathrm{K}_1$ component
of the relative class of these unitized maps is the $\mathrm{K}_1$ component of
$\delta$, under $\mathrm{K}_1(A^\dagger)=\mathrm{K}_1(A)$.  The formula defining $\Omega_T$, or
the same calculation as \eqref{eq:rotation-pair-formula}, therefore gives the
difference $r([u]_1)$ for their algebraic $\mathrm{K}_1$ classes.
Now conjugate $p_0$ to $e_{11}$ by a unitary in $\mathrm{M}_2(B)$, as in the preceding
paragraph, and compress $\rho_0$ to a unital map $\phi:A\to B$.  Under this
conjugacy, the first of the displayed unitary images becomes
$\operatorname{diag}(\phi^{(k)}(u),1_k)$.  Inner conjugacy and matrix
stabilization preserve the algebraic $\mathrm{K}_1$ class.  Consequently
$\phi^\ddagger=\psi^\ddagger+d=\kappa_\gamma$.  Trace and total
$\mathrm{K}$-theory are unchanged.  This proves existence.

\emph{Uniqueness.}
Let $\phi,\psi:A\to B$ induce the same prescribed triple.  Their quotient
maps are unitarily conjugate by the quotient-level assertion above.  Every
unitary in the finite von Neumann algebra $Q$ is an exponential of a bounded
self-adjoint element, and therefore lifts to a unitary in $B$.  After
conjugating $\psi$, we may assume $q\phi=q\psi$.  Their relative class
$\eta=[\phi,\psi]_{\KL(A,J)}$ maps to zero in $\KL(A,B)$, by the UMCT and
equality of total $\mathrm{K}$-theory.  Equality on algebraic $\mathrm{K}_1$ and
\eqref{eq:rotation-pair-formula} imply
$\mathcal R_{A,T}(\eta)=0$.  Thus $\eta=0$.

Stabilize the two maps to $\phi_0=\operatorname{diag}(\phi,0)$ and
$\psi_0=\operatorname{diag}(\psi,0)$.  Their quotient maps are unitizably full,
and their relative KL-class is zero.  Lemma~\ref{lem:separable-subextension}
produces a subextension retaining this zero class and their images, whose
ideal $I$ is stable and has the corona factorization property.  The two
multiplier representations are absorbing.  Apply
Lemma~\ref{lem:kl-uniqueness-abstract} to obtain unitaries $w_n\in1+I$ giving
approximate conjugacy of those representations.  This is also approximate
conjugacy in the middle algebra: the difference after conjugation belongs to
$I$, and the multiplier map is isometric on $I$.

Evaluation at $1_A$ gives $w_ne_{11}w_n^*\to e_{11}$.  The standard
close-projection unitary, formed by polar decomposition of
$e_{11}(w_ne_{11}w_n^*)+(1-e_{11})(1-w_ne_{11}w_n^*)$, belongs to $1+I$ and
converges to $1$.  Multiplying $w_n$ by this correcting unitary makes it
commute with $e_{11}$ exactly without changing the limiting conjugacy.  Its
upper-left corner is then a unitary in $B$ approximately conjugating $\phi$
to $\psi$.  Finally, a norm-ultrapower epsilon test on a countable dense subset
of $A$ turns approximate unitary equivalence into unitary equivalence in $B$.
\end{proof}

\begin{remark}\label{rem:ultrapower-scope}
Theorem~\ref{thm:ultrapower-full-invariant} is an ultrapower theorem.  Obtaining
an embedding into $T$ itself additionally requires the coordinate control and
reparameterization argument proved separately.  In particular, neither
coordinate convergence in a naively stabilized norm quotient nor an
unproved interchange of total $\mathrm{K}$-theory with ultraproducts is used in the
proof above.
\end{remark}

\subsection{Coefficient groups and descent to the original target}
\label{subsec:coefficient-descent}

Throughout this section, $T$ is an admissible target, with unique tracial
state $\tau$.  Put
\[
 G=\mathrm{K}_0(T),\qquad H_1=\mathrm{K}_1(T),\qquad
 \sigma=\tau_*:G\longrightarrow\R,\qquad u=[1_T]_0.
\]
For a free ultrafilter $\omega$ on $\N$, write
\[
 T_\omega=\ell^\infty(T)/c_\omega(T),\qquad
 c_\omega(T)=\{(x_m):\lim_\omega\|x_m\|=0\}.
\]
We use the facts, established above, that $T_\omega$ has stable rank one
and a unique tracial state $\tau_\omega$.  We also use the computation
\begin{equation}\label{eq:descent-k1-ultrapower}
 \mathrm{K}_1(T_\omega)\ \cong\ (H_1)_\omega
\end{equation}
from \cite[Corollary~4.2]{KES}. Here and below, an ultrapower of an abelian
group is the algebraic ultrapower, with equality on an $\omega$-large set.
The uniform homotopy control behind this computation is supplied by
\cite[Theorem~1.2]{KES}. For injectivity, stable rank one represents a
class in $\mathrm K_1(T_\omega)$ by a unitary $u\in T_\omega$.
Lift $u$ to coordinate unitaries $u_m\in T$. If its coordinate class
vanishes, replace $u_m$ by $1_T$ outside an $\omega$-large set, so
that $[u_m]_1=0$ for every $m$.
The cited theorem gives $h_m,k_m\in\mathrm M_2(T)_{\mathrm{sa}}$
with $\|h_m\|,\|k_m\|\le\pi$ and
\[
 \|u_m\oplus1_T-e^{ih_m}e^{ik_m}\|<1/m.
\]
Consequently $u\oplus1_T=e^{ih}e^{ik}$ in $\mathrm M_2(T_\omega)$,
where $h=(h_m)_\omega$ and $k=(k_m)_\omega$. The path
$t\mapsto e^{ith}e^{itk}$ proves injectivity of the coordinate
$\mathrm K_1$-map. Stable rank one also permits
representatives in $\mathrm U(T)$ for arbitrary coordinate classes,
which gives surjectivity. The uniform estimate is an external input
to the present argument.

\begin{lemma}[Bounded-trace description of $\mathrm{K}_0$]
\label{lem:descent-bounded-k0}
The coordinate map identifies $\mathrm{K}_0(T_\omega)$ with the subgroup
\[
 G_b=\left\{[(g_m)]\in G_\omega:
       \sup_m|\sigma(g_m)|<\infty
       \text{ for some representative }(g_m)\right\}.
\]
Under this identification,
\[
 (\tau_\omega)_*([(g_m)])=\lim_\omega\sigma(g_m).
\]
In particular,
\begin{equation}\label{eq:descent-trace-range-closure}
 (\tau_\omega)_*(\mathrm{K}_0(T_\omega))
       =\overline{\sigma(G)}.
\end{equation}
\end{lemma}

\begin{proof}
Projection representatives in a fixed matrix algebra over $T_\omega$
lift to sequences of projections in the same matrix algebra over $T$.
Consequently, a difference of projection classes in $\mathrm{K}_0(T_\omega)$
has coordinate traces bounded uniformly in absolute value.

The coordinate map is injective.  Indeed, suppose that projections
$p,q\in \mathrm{M}_d(T_\omega)$ lift to projections $p_m,q_m\in \mathrm{M}_d(T)$ and that
$[p_m]_0=[q_m]_0$ for $\omega$-many $m$.  Cancellation, which follows
from stable rank one, gives Murray--von Neumann equivalence of $p_m$
and $q_m$ in $\mathrm{M}_d(T)$ for those $m$.  Choose implementing partial
isometries in $\mathrm{M}_d(T)$; their norms are at most one.  Their ultrapower
therefore implements $p\sim q$.  Padding with zero projections first
handles differences represented in different matrix sizes.

For surjectivity, let $|\sigma(g_m)|\leq C$ and choose an integer
$k>C+1$.  Strict comparison of projections implies
\[
 \{g\in G:\sigma(g)>0\}\subseteq \mathrm{K}_0(T)_+.
\]
For completeness, if $g=[p]_0-[q]_0$ has strictly positive trace, then
$q\precsim p$, and the difference is represented by the complement of
a subprojection of $p$ equivalent to $q$.  Thus $g_m+ku$ is represented
by a projection $r_m$ in some matrix algebra over $T$.  Its trace is
strictly between zero and $2k$.  Comparing $r_m$ with $1_{2k}$, or with
$1_{2k+1}$ if desired to keep an unused rank, moves it by equivalence
into the fixed matrix algebra $\mathrm{M}_{2k+1}(T)$.  The difference
$[(r_m)]_0-[1_k]_0$ consequently represents $[(g_m)]$ in $\mathrm{K}_0(T_\omega)$.

The trace formula follows from projection representatives.  Its range
is contained in the right-hand side of
\eqref{eq:descent-trace-range-closure}, because a bounded sequence of
points in the closed set $\overline{\sigma(G)}$ has its ultralimit in
that set.  Conversely, if $t\in\overline{\sigma(G)}$, choose
$g_m\in G$ such that $\sigma(g_m)\to t$.  This sequence has bounded
trace and represents a class with trace $t$.
\end{proof}

We next justify all uses of coefficient classes in coordinate maps.
For $n\geq2$, fix the usual nuclear coefficient algebra $D_n$ with
$\mathrm{K}_0(D_n)=\Z/n$ and $\mathrm{K}_1(D_n)=0$, obtained by suspending the mapping
cone of the unital inclusion $\C\to \mathrm{M}_n$.  Write
$\mathrm{K}_i(B;\Z/n)=\mathrm{K}_i(B\otimes D_n)$.  This convention gives the usual
Bockstein exact sequence.

\begin{lemma}[Finite coefficients commute with the ultrapower]
\label{lem:descent-coefficients}
For every $n\geq2$ and $i=0,1$, the canonical coordinate map is an
isomorphism
\begin{equation}\label{eq:descent-coefficient-isomorphism}
 \eta_{i,n}:\mathrm{K}_i(T_\omega;\Z/n)
       \longrightarrow \mathrm{K}_i(T;\Z/n)_\omega.
\end{equation}
The ordinary and coefficient coordinate maps commute with the
Bockstein operations and the changes of coefficient.
\end{lemma}

\begin{proof}
We first describe the map rather than assuming that tensor products
commute with ultrapowers.  Exactness of $D_n$ gives an exact sequence
\[
 0\longrightarrow c_\omega(T)\otimes D_n
 \longrightarrow\ell^\infty(T)\otimes D_n
 \longrightarrow T_\omega\otimes D_n\longrightarrow0.
\]
Coordinate evaluation induces a $*$-homomorphism from the middle
algebra to $\ell^\infty(T\otimes D_n)$.  It sends the left-hand
ideal into $c_\omega(T\otimes D_n)$: this follows first for finite
sums of elementary tensors and then by norm approximation.  It
therefore induces a canonical $*$-homomorphism
\begin{equation}\label{eq:descent-coefficient-tensor-map}
 T_\omega\otimes D_n\longrightarrow (T\otimes D_n)_\omega.
\end{equation}
We do not assert that this map is surjective as a map of algebras.

For any $\mathrm{C}^*$-algebra $B$, there are natural coordinate maps
$\mathrm{K}_i(B_\omega)\to \mathrm{K}_i(B)_\omega$.  Here is the required elementary
verification.  Lift a projection, or a unitary, in a fixed matrix
algebra over the appropriate unitization of $B_\omega$ to a bounded
sequence of self-adjoint elements, or arbitrary elements, over $B$.
The projection or unitary defects tend to zero along $\omega$.
Spectral cut-down, or polar decomposition, gives coordinate
projections or unitaries on an $\omega$-large set, in the same matrix
size.  The scalar part can be kept fixed when $B$ is nonunital.

Independence of representatives and homotopy invariance require no
uniform homotopy-length assertion.  For a given continuous homotopy
in a fixed matrix algebra over $B_\omega$, choose a finite subdivision
so that successive values differ in norm by less than $1/8$.
Lift these finitely many values and round them as above.  On the
intersection of finitely many $\omega$-large sets, successive
coordinate projections, or unitaries, differ by less than $1/2$.
They can be joined by the usual short projection homotopies, or
short logarithm paths.  The endpoints are the chosen coordinate
representatives.  Stabilizations used in a given $\mathrm{K}$-theory
equality have a fixed finite matrix size and are treated in exactly
the same way.  This proves that the coordinate maps are well defined.
Applying this construction after
\eqref{eq:descent-coefficient-tensor-map} defines $\eta_{i,n}$.
We spell out the boundary compatibility used below.
Let $C_{\iota_n}$ be the cone of $\iota_n:\C\to\mathrm M_n$,
so $D_n=SC_{\iota_n}$.
For every algebra $X$ the cone extension is
\[
0\longrightarrow S\mathrm M_n(X)\longrightarrow
X\otimes C_{\iota_n}\longrightarrow X\longrightarrow0 .
\]
Apply this first with $X=T_\omega$.
There is a morphism from this extension to the norm ultraproduct
of the corresponding extensions for $T$, using the same elementary
tensor maps as above. The lower extension is exact: surjectivity
follows from the uniformly contractive cone lifts
$x\mapsto[t\mapsto\chi(t)(x\otimes1_n)]$, where
$\chi(0)=1$ and $\chi$ vanishes near $1$.
If a bounded coordinate sequence has quotient norms tending to zero along
$\omega$, lift it into the coordinate ideals with errors tending
to zero and uniformly bounded norms.
The exponential and index formulas for the boundary maps commute
with these coordinate maps. Each formula for a particular class
uses one fixed matrix size, finitely many lifts, and functional
calculus with a fixed spectral cutoff, so it may be evaluated
coordinatewise. Suspension and the standard Bott identification
also commute with these maps; a fixed Bott representative and
finite approximation of its paths suffice.
Thus the induced diagram of Bockstein sequences commutes.
The standard cone diagrams defining changes of coefficient,
with their fixed matrix amplifications and suspension
identifications, give the other required naturality squares
in the same way. This argument does not assert surjectivity
of any algebra-level tensor map.

For the group calculation, Lemma~\ref{lem:descent-bounded-k0} gives
\begin{equation}\label{eq:descent-k0-mod-torsion}
 G_b/nG_b\cong(G/nG)_\omega,
 \qquad \Tor_n(G_b)\cong(\Tor_nG)_\omega.
\end{equation}
For the first surjectivity, replace an arbitrary representative
$g_m$ of a class in $(G/nG)_\omega$ by
\[
 g_m-n\left\lfloor\frac{\sigma(g_m)}n\right\rfloor u.
\]
Its trace belongs to $[0,n)$.  For injectivity, if $g_m=nh_m$ on an
$\omega$-large set and $(\sigma(g_m))$ is bounded, then
$\sigma(h_m)=\sigma(g_m)/n$ there; modifying the sequence outside
that set makes $(h_m)$ a bounded-trace representative.  The second
isomorphism follows because all torsion classes have trace zero.

For any abelian group $L$, coordinate selection gives
\[
 L_\omega/nL_\omega\cong(L/nL)_\omega,
 \qquad \Tor_n(L_\omega)\cong(\Tor_nL)_\omega.
\]
Use this with $L=H_1$, together with
\eqref{eq:descent-k1-ultrapower}.  The canonical maps now give a
commutative diagram between the short exact Bockstein sequence
\[
 0\longrightarrow \mathrm{K}_i(T_\omega)/nK_i(T_\omega)
 \longrightarrow \mathrm{K}_i(T_\omega;\Z/n)
 \longrightarrow\Tor_n(\mathrm{K}_{i-1}(T_\omega))\longrightarrow0
\]
and the algebraic ultrapower of the corresponding sequence for $T$.
Algebraic ultrapowers preserve exactness of this finite sequence,
as preimages can be chosen coordinatewise.  The two outside
vertical maps are isomorphisms by the preceding calculations.
The middle map is consequently an isomorphism.
\end{proof}

\begin{lemma}[Witnesses in a fixed matrix algebra]
\label{lem:descent-fixed-commutators}
Let $d\geq1$, and let $w\in \mathrm{U}_d(T_\omega)$.  If
$\overline{w}=0$ in $\Uquot{T_\omega}$, then, for each $k\geq1$,
there are an integer $L_k$ and unitaries
$a_{k,\ell},b_{k,\ell}\in \mathrm{U}_d(T_\omega)$, $1\leq\ell\leq L_k$,
such that
\[
 \left\|w-\prod_{\ell=1}^{L_k}
       a_{k,\ell}b_{k,\ell}a_{k,\ell}^*b_{k,\ell}^*\right\|
       <2^{-k-1}.
\]
The matrix size $d$ is independent of $k$; the number $L_k$ need not
be bounded.  If $w=[(w_m)]$ with $w_m\in \mathrm{U}_d(T)$, these witnesses
can be lifted to coordinate unitaries in $\mathrm{U}_d(T)$ so that the
corresponding inequality with $2^{-k}$ holds on an $\omega$-large set.
\end{lemma}

\begin{proof}
Apply the stable-rank-one identification
\[
 \mathrm{U}(\mathrm{M}_d(T_\omega))/\overline{\mathrm{DU}(\mathrm{M}_d(T_\omega))}
       \cong \Uquot{\mathrm{M}_d(T_\omega)}
       \cong \Uquot{T_\omega};
\]
see \cite[Remark~2.8]{CGSTW}.  Thus $w$ lies in the norm closure
of the commutator subgroup of $\mathrm{U}_d(T_\omega)$ itself.  This gives
all the approximations in the single matrix algebra $\mathrm{M}_d(T_\omega)$.
Recall that the stable Hausdorffized unitary quotient kills commutators
of arbitrary stable unitaries, consistently with the connected-commutator
convention of Section~\ref{sec:prelim}.  By injectivity of the displayed
stable-rank-one identification, the same is true in this fixed matrix
algebra modulo $\CU(\mathrm M_d(T_\omega))$.
Unitaries lift to coordinate unitaries in a norm ultrapower.
For fixed $k$, lift its finitely many witnesses and use the strict
norm inequality to obtain the coordinate assertion.
The coordinate witnesses need not be connected to the identity;
their commutators nevertheless have zero stable unitary-quotient class.
\end{proof}

\begin{remark}\label{rem:descent-fixed-size-warning}
The fixed matrix size in the preceding lemma is essential.  Merely
giving, for each accuracy, an approximation after some possibly
larger matrix amplification does not establish membership in the
closure of the commutator subgroup for the topological inductive
limit defining $\mathrm U_\infty/\mathrm{CU}_\infty$.  We retain a separate finite list of
witnesses for each accuracy, always in the same matrix algebra.
No uniform bound on commutator length is used.
\end{remark}

\begin{lemma}[The diagonal invariant is injective]
\label{lem:descent-diagonal-injective}
For every free ultrafilter $\omega$, the diagonal embedding
$\iota_\omega:T\to T_\omega$ induces injective maps on total
$\mathrm K$-theory, affine trace functions, and the stabilized unitary quotient.
\end{lemma}

\begin{proof}
Injectivity on ordinary $\mathrm{K}$-theory follows from
Lemma~\ref{lem:descent-bounded-k0} and
\eqref{eq:descent-k1-ultrapower}.  Injectivity with finite
coefficients follows from Lemma~\ref{lem:descent-coefficients}.
The affine trace map is the identity on $\R$, since both algebras
have the unique trace under consideration.

For algebraic $\mathrm{K}_1$, let $v\in \mathrm{U}_d(T)$ have diagonal image of zero
class.  Lemma~\ref{lem:descent-fixed-commutators} gives, for each
accuracy, finitely many commutator witnesses in $\mathrm{U}_d(T_\omega)$.
Lift them and choose one coordinate where the asserted
approximation holds.  Hence $v$ is approximable by products of
commutators in $\mathrm{U}_d(T)$ and $\overline{v}=0$.
\end{proof}

The following descent argument is deliberately phrased in terms of
the ultrapower existence and uniqueness results, rather than an
unstated finite-set classification theorem.  It explains all
diagonal choices, including the coefficient and unitary data.

\begin{proposition}[Descent of a prescribed totally compatible triple]
\label{prop:descent-prescribed-invariant}
Let $A$ be a unital separable nuclear UCT algebra, and let
$\Xi=(\kappa,\kappa_T,\kappa_\gamma)$ be a faithful unital totally
compatible triple from $A$ to $T$. Use the following two conclusions
of Theorem~\ref{thm:ultrapower-full-invariant}:
\begin{enumerate}
 \item for some free ultrafilter $\omega$, there is a unital
 $*$-homomorphism $\Phi:A\to T_\omega$ whose induced triple is
 the composite of $\Xi$ with the triple induced by the diagonal
 embedding $\iota_\omega:T\to T_\omega$;
 \item for every free ultrafilter $\nu$, unital full maps
 $A\to T_\nu$ inducing the same triple are unitarily equivalent.
\end{enumerate}
Then there is a unital embedding $\rho:A\to T$ satisfying
$([\rho],\rho_T,\rho^\ddagger)=\Xi$.
\end{proposition}

\begin{proof}
Since $T$ is monotracial, the trace component specifies
$s=\kappa_T(\tau)\in\mathrm T(A)$, and faithfulness of $\Xi$ says
that $s$ is faithful.  Nuclearity and the Choi--Effros lifting theorem give
ucp lifts $\phi_m:A\to T$ of $\Phi$.  A cpc lift can be made
unital by adding
$s(a)(1-\phi_m(1))$ if necessary; this does not alter the quotient
map.  Thus
\begin{equation}\label{eq:descent-initial-cp-lifts}
 \lim_\omega\|\phi_m(ab)-\phi_m(a)\phi_m(b)\|=0,
 \qquad
 \lim_\omega\tau(\phi_m(a))=s(a)
 \quad(a,b\in A).
\end{equation}

Choose countable dense test sets for $A$ and all its finite matrix
algebras.  Enumerate representatives of the countable total
$\mathrm{K}$-theory group $\underline{\mathrm{K}}(A)$.  For coefficient classes, these
representatives are projections and unitaries over $A\otimes D_n$
with their prescribed scalar parts.  Applying
$\phi_m\otimes\id_{D_n}$ and rounding gives their partial
coordinate classes whenever the relevant finite multiplicativity
conditions hold.  Lemma~\ref{lem:descent-coefficients}, and the
ordinary coordinate maps, show that every finite list of these
partial classes has its prescribed value under $\kappa$ on an
$\omega$-large set.

Choose a sequence $u_j\in \mathrm{U}_{d_j}(A)$ whose terms are dense in
each $\mathrm{U}_d(A)$.  Choose a fixed unitary
$v_j\in \mathrm{U}_{e_j}(T)$ representing
$\kappa_\gamma(\overline{u_j})$, and put
$D_j=\max\{d_j,e_j\}$.  Pad both unitaries with identity blocks
when needed.  The unitary
\[
 W_j=\Phi^{(D_j)}(u_j)\,\iota_\omega^{(D_j)}(v_j)^*
       \in \mathrm{U}_{D_j}(T_\omega)
\]
has zero algebraic $\mathrm{K}_1$-class.  For every pair $(j,k)$, apply
Lemma~\ref{lem:descent-fixed-commutators} to $W_j$, and fix
$L_{j,k}$ lifted commutator witnesses in this single matrix size
$D_j$.  The unitary obtained by polar correction of
$\phi_m^{(d_j)}(u_j)$ represents $\Phi^{(d_j)}(u_j)$ in the
ultrapower.  Consequently, on an $\omega$-large set its padded
product with $v_j^*$ is within $2^{-k}$ of a product of
$L_{j,k}$ coordinate commutators in $\mathrm{U}_{D_j}(T)$.

Select an increasing sequence $m(n)$ as follows.  At stage $n$,
impose multiplicativity and trace errors less than $1/n$ on the
first $n$ dense tests; impose the prescribed values on the first
$n$ total-$\mathrm{K}$ representatives; and retain all the commutator
witness inequalities indexed by $j,k\leq n$.  Each stage has
finitely many requirements, each valid on an $\omega$-large set,
so their intersection is infinite and one may choose
$m(n)>m(n-1)$.  Set $\psi_n=\phi_{m(n)}$.

The maps $\psi_n$ are asymptotically multiplicative in the usual
sense, and
\begin{equation}\label{eq:descent-selected-trace}
 \tau\psi_n(a)\longrightarrow s(a)\qquad(a\in A).
\end{equation}
They define a unital $*$-homomorphism
\[
 \psi_\infty:A\longrightarrow
 T_\infty:=\ell^\infty(T)/c_0(T).
\]
For each fixed $j,k$, all sufficiently large $n$ retain a
$2^{-k}$ commutator approximation to the difference between the
polar image of $u_j$ and $v_j$, with exactly $L_{j,k}$ factors
in the fixed matrix size $D_j$.

Now fix any map $r:\N\to\N$ such that $r(n)\to\infty$, and any
free ultrafilter $\nu$.  The sequences $\psi_n$ and $\psi_{r(n)}$
induce unital homomorphisms
$\Psi_\nu,\Psi_\nu^r:A\to T_\nu$.  Both have trace $s$ by
\eqref{eq:descent-selected-trace}.  Their ordinary and coefficient
$\mathrm{K}$-maps equal $[\iota_\nu]\kappa$, because the partial
coordinate classes have their prescribed constant values
eventually, and all the coordinate maps above are injective.

Their algebraic $\mathrm{K}_1$ maps also agree with the prescribed one.
Indeed, for fixed $j,k$, the retained $L_{j,k}$ witnesses form
unitaries in $\mathrm{U}_{D_j}(T_\nu)$ and give the same $2^{-k}$
approximation there.  Letting $k\to\infty$ shows that the image
of $u_j$ differs from the constant $v_j$ by an element of
$\overline{\mathrm{DU}(\mathrm{M}_{D_j}(T_\nu))}$.  The matrix size is $D_j$ for every $k$.
This argument applies to both $\psi_n$ and $\psi_{r(n)}$.

Equality on the dense unitary test sets implies equality on all
unitaries by an exact Thomsen-map calculation. For $u\in \mathrm{U}_d(A)$,
choose a tested $v\in \mathrm{U}_d(A)$ with $\|u-v\|<2$ and write
$v^*u=e^{ih}$ with $h=h^*\in \mathrm{M}_d(A)$. For each of the two induced
maps, compatibility with the Thomsen map sends
$\overline{u}-\overline{v}$ to
\[
 \operatorname{Th}_{T_\nu}
 \bigl(\kappa_{\sharp,\nu}(\widehat h)/(2\pi)\bigr),
\]
where $\kappa_{T,\nu}=\kappa_T\circ(\iota_\nu)_T$ is their common
map on tracial state spaces, and
$\kappa_{\sharp,\nu}=\Aff\mathrm T(\iota_\nu)\circ\kappa_\sharp$
is its dual map on affine functions. The prescribed map has exactly the same value
on this difference. Equality on $\overline{v}$ therefore gives equality
on $\overline{u}$, with no limiting argument in the quotient group.
We have proved
\begin{equation}\label{eq:descent-equal-full-invariants}
 \begin{aligned}
 [\Psi_\nu]&=[\Psi_\nu^r]=[\iota_\nu]\kappa,\\
 (\Psi_\nu)_T&=(\Psi_\nu^r)_T=\kappa_{T,\nu},\\
 \Psi_\nu^\ddagger&=(\Psi_\nu^r)^\ddagger
                   =\iota_\nu^\ddagger\kappa_\gamma.
 \end{aligned}
\end{equation}
The maps are full by Lemma~\ref{lem:trace-kernel-regularity}(iv).
Explicitly, for a nonzero positive $a\in A$ the image has positive
canonical trace. Choose $\delta>0$ so its spectral cut
$(\Psi_\nu(a)-\delta)_+$ has positive dimension function, and then
choose $m$ with $m d_\tau((\Psi_\nu(a)-\delta)_+)>1$.
Strict comparison puts the unit below the direct sum of these
$m$ cutdowns, so the image of $a$ is full.
The same argument applies to $\Psi_\nu^r$.  The ultrapower uniqueness therefore
gives a unitary conjugating $\Psi_\nu$ and $\Psi_\nu^r$.

We claim that, for every finite $F\subset A$ and $\varepsilon>0$,
for all sufficiently large $n$ there is $z_n\in \mathrm{U}(T)$ with
\begin{equation}\label{eq:descent-eventual-coordinate-unitaries}
 \|z_n\psi_n(a)z_n^*-\psi_{r(n)}(a)\|<\varepsilon
       \qquad(a\in F).
\end{equation}
Otherwise, the set of indices where no such unitary exists is
infinite.  Choose a free ultrafilter $\nu$ containing that set.
The unitary just obtained from
\eqref{eq:descent-equal-full-invariants} lifts to a sequence of
unitaries in $T$.  Its coordinate inequalities on $F$ contradict
the choice of $\nu$.  This proves the claim.  Diagonalizing
\eqref{eq:descent-eventual-coordinate-unitaries} over increasing
finite dense sets gives a unitary in $T_\infty$ conjugating
$\psi_\infty$ and $r^*\psi_\infty$.

This holds for every diverging reparameterization $r$.
Intertwining by reparameterization
\cite[Proposition~9.2]{CGSTW} now gives a homomorphism
$\rho:A\to T$ whose constant-sequence map is unitarily equivalent
to $\psi_\infty$.  That proposition requires only a separable
source and a unital target; it does not require nuclearity of
$T$.  Since $\psi_\infty$ is unital, so is $\rho$.

Finally, compose the unitary equivalence in $T_\infty$ with its
quotient to any $T_\nu$.  Equation
\eqref{eq:descent-equal-full-invariants} and
Lemma~\ref{lem:descent-diagonal-injective} give
\[
 \begin{aligned}
 [\iota_\nu][\rho]&=[\iota_\nu]\kappa,\\
 \rho_T(\tau)&=s,\\
 \iota_\nu^\ddagger\rho^\ddagger
     &=\iota_\nu^\ddagger\kappa_\gamma.
 \end{aligned}
\]
Diagonal injectivity and monotraciality therefore give
$([\rho],\rho_T,\rho^\ddagger)=\Xi$.
In particular, $\tau\rho=s$ is faithful, so $\rho$ is injective.
This last recovery uses the monotracial ultrapower, and requires
no assertion about all traces of the sequence algebra $T_\infty$.
\end{proof}

\begin{theorem}[Embeddings into an admissible target]\label{thm:embedding}
Let \(A\) be separable, unital, nuclear, and satisfy the UCT, and let
\(T\) be an admissible target. Every faithful unital totally compatible
triple \(\Xi=(\kappa,\kappa_T,\kappa_\gamma)\) from \(A\) to \(T\) is
realized by a unital embedding \(A\to T\). Two unital embeddings inducing
the same triple, whose common induced trace on \(A\) is faithful, are
approximately unitarily equivalent. No assumption on either
\(\mathrm{K}\)-group is needed.
\end{theorem}
\begin{proof}
For existence, compose \(\Xi\) with the triple induced by a diagonal
embedding into \(T_\omega\). Theorem~\ref{thm:ultrapower-full-invariant}
realizes this composite in the ultrapower. That theorem supplies the same
uniqueness statement for every free ultrafilter.
Proposition~\ref{prop:descent-prescribed-invariant} therefore gives an
embedding into \(T\) inducing \(\Xi\).

For uniqueness, apply Theorem~\ref{thm:ultrapower-full-invariant} to the two diagonal embeddings into \(T_\omega\). Lift an implementing unitary to a sequence of unitaries of \(T\). For a prescribed finite source set and tolerance, one of its coordinates gives the desired approximate conjugacy. Exhaust the separable source.
\end{proof}

\begin{remark}[A uniqueness question of Hirshberg and Phillips]
Let $X$ be a nonempty contractible compact metrizable space, let $T$
be an admissible target with unique trace $\tau$, and let
$\phi,\psi:\mathrm C(X)\to T$ be unital embeddings.\label{rem:hp-uniqueness}
Then
\[
 \tau\circ\phi=\tau\circ\psi
 \quad\Longleftrightarrow\quad
 \phi\text{ and }\psi\text{ are approximately unitarily equivalent}.
\]
Taking $X=[0,1]^n$ and $T=D_I$ answers the first part of
Hirshberg and Phillips's \cite[Question~4.8]{HP} affirmatively.
The same conclusion answers the exact-target case of their more
general target question. Their formulation of that second part does
not assume exactness. For a finite-dimensional simple target, the
assertion follows by simultaneous diagonalization. Equality of the
normalized traces means that the two lists of evaluation points agree
with multiplicity, so the embeddings are unitarily conjugate.

Contractibility gives
$\mathrm K_0(\mathrm C(X))=\Z[1]$ and
$\mathrm K_1(\mathrm C(X))=0$. Thus unitality determines the maps
on ordinary $\mathrm K$-theory and, by the coefficient exact
sequences, on total $\mathrm K$-theory. Equality of the induced
traces determines the trace component. Since the source
$\mathrm K_1$-group vanishes, the Thomsen sequence shows that its
Hausdorffized unitary group is the determinant kernel, so its map
is determined by the trace component as well. Hence $\phi$ and
$\psi$ induce the same triple. Their common induced trace
is faithful, and $\mathrm C(X)$ is nuclear and satisfies the UCT.
Theorem~\ref{thm:embedding} therefore gives approximate unitary
equivalence. The converse follows from traciality and norm
continuity. Finally, $D_I$ is admissible by
Lemma~\ref{lem:interval-regularity}, so the stated specialization
applies.
\end{remark}

\begin{remark}
Theorem~\ref{thm:embedding} is the analytic input used for replacement. Its hypotheses concern the nuclearity and UCT of the source, not nuclearity of the target. The target may be any of the free products \(\mathrm{C}(M)*_{\red}D\) or \(A*_{\red}D\) used below. The entire triple is needed for uniqueness even when the theorem to be deduced is stated using ordinary \(\mathrm{K}\)-groups and one trace pairing.
\end{remark}

\section{Marked replacement of a nuclear core}\label{sec:marked-core}

Throughout this section $D$ is any reservoir of Definition~\ref{def:reservoir}.
The nuclear cores are separable, unital and UCT,
and their selected traces are faithful.  No finite-generation assumption is made.
No splitting of the core's unit class is required. The separate
support-shifting construction of Section~\ref{app:split-unit} uses
a splitting only for the factor being moved.
Write $S_E=E\fr D$, with its selected free-product trace.
The targets $S_E$ and their infinite reservoir tails have the regularity
required in Theorem~\ref{thm:embedding}, by
Lemma~\ref{lem:interval-regularity}. Their canonical inclusions
$j_E:E\to S_E$ induce isomorphisms on graded and total $\mathrm{K}$-theory.
We retain totally compatible triples, including their total
$\mathrm K$-theory, trace and Hausdorffized unitary maps and the
coefficient-to-unitary compatibility. The selected data determine the
invariants of the monotracial $D$-stabilizations; we do not identify
the trace simplices of the original cores. In particular, selected
traces on multitracial cores are allowed.
Whenever $f:L\to L'$ is a trace-preserving unital map between
monotracial algebras and $f_*$ is a graded $\mathrm K$-isomorphism,
the induced triple $([f],f_T,f^\ddagger)$ is invertible.
The coefficient exact sequences prove this on total $\mathrm K$-theory.
The trace map identifies the one-point trace simplices, and its
induced map on affine functions is $\id_\R$. The two determinant
quotients agree because the pairing ranges agree, so the Thomsen
sequence proves invertibility on the unitary quotients as well.
We shall use this observation for the reservoir-tail inclusions
without further mention.

\begin{lemma}[A monotracial range model]\label{lem:marked-range}
Let $E$ be a nuclear UCT core with faithful selected trace.
There is a simple, infinite-dimensional,
unital, monotracial, nuclear UCT, $\mathcal Z$-stable algebra $A$
with the same selected graded $\mathrm{K}$-groups, unit and trace pairing as $E$.
There is a unital trace-preserving embedding $\theta:E\to A$
inducing graded and total $\mathrm{K}$-isomorphisms.
\end{lemma}
\begin{proof}
Put $G_i=\mathrm{K}_i(E)$, $u=[1_E]$ and $\sigma=\tau_{E*}$, and equip $G_0$
with the strict trace cone
\[
 G_0^+=\{0\}\cup\{g:\sigma(g)>0\}.
\]
This is a simple weakly unperforated scaled ordered group: adding a sufficiently
large multiple of $u$ makes any element positive, every
nonzero positive element is an order unit, and rational comparisons
with $u$ show that its only normalized state is $\sigma$.
The nuclear range theorem \cite[Theorem~2.8]{EGLN} supplies a
classifiable monotracial range model for these data, with one-point
trace simplex.  Choose it infinite dimensional and $\mathcal Z$-stable.
For clarity, a finite-dimensional simple model would have
$(G_0,u,G_1)=(\Z,m,0)$ with $m\geq1$. For these data one can
instead take the infinite-dimensional model $\mathrm{M}_m(\mathcal Z)$,
with its normalized trace. Thus no splitting of $u$ is needed.
The identity on the selected graded data is compatible with the
trace map $\mathrm T(A)\to\mathrm T(E)$ sending the unique trace of
$A$ to $\tau_E$. Its induced map on affine functions is evaluation
at $\tau_E$.
A nonzero source projection has strictly positive selected trace,
so its image is positive in the displayed cone.
The nuclear embedding theorem \cite[Theorem~B and Corollary~C]{CGSTW}
therefore gives $\theta$; its version with totally compatible
triples permits any chosen compatible extension of these data.  Coefficient exact
sequences give the asserted total $\mathrm{K}$-isomorphism.
\end{proof}

\begin{lemma}[Factorization of totally compatible triples]\label{lem:marked-factor}
Let $E$ have faithful selected trace and let
$\theta:E\to A$, $j:E\to T$ be trace-preserving unital embeddings
inducing graded $\mathrm K$-isomorphisms, where $A,T$ are monotracial.
Their total $\mathrm K$-theory maps, Hausdorffized unitary maps and
induced maps on affine functions are surjective. There is a unique
totally compatible isomorphism triple
$\lambda=(\lambda_{\mathrm K},\lambda_T,\lambda_\gamma)$ from $A$ to $T$
such that
\[
 \lambda([\theta],\theta_T,\theta^\ddagger)
       =([j],j_T,j^\ddagger).
\]
\end{lemma}
\begin{proof}
The total $\mathrm K$-theory maps are isomorphisms; set
$\lambda_{\mathrm K}=[j][\theta]^{-1}$.
Let $\lambda_T:\mathrm T(T)\to\mathrm T(A)$ be the unique affine map
between the one-point trace simplices. Its induced map
$\lambda_\sharp:\Aff\mathrm T(A)\to\Aff\mathrm T(T)$ is $\id_\R$.
Put $H=\overline{\tau_{E*}\mathrm K_0(E)}$ and use the Thomsen sequence
\[
 0\longrightarrow \Aff\mathrm T(L)/\overline{\rho_L\mathrm K_0(L)}
 \longrightarrow\Uquot{L}\xrightarrow{\Pi_L^{\mathrm{cu}}}\mathrm K_1(L)
 \longrightarrow0.
\]
The determinant-kernel maps induced by $\theta,j$ are the same
surjection, namely evaluation at $\tau_E$ modulo $H$.
Together with the $\mathrm K_1$-isomorphisms this makes
$\theta^\ddagger$ and $j^\ddagger$ surjective. Their kernels coincide:
an element in either kernel has zero ordinary $\mathrm K_1$-class,
hence is $\mathrm{Th}_E(f)$, and its image vanishes precisely when
$f(\tau_E)\in H$. Thus there is a unique unitary isomorphism
$\lambda_\gamma$ with
$\lambda_\gamma\theta^\ddagger=j^\ddagger$.
The induced maps on $\Aff\mathrm T$ are the same surjection,
namely evaluation at the selected trace. This is not a surjectivity
assertion about $\theta_T$ or $j_T$ when $E$ has several traces;
both maps have the single selected trace as their image.

The compatibility squares follow by cancellation on total
$\mathrm K$-theory, unitary quotients and affine functions.
In particular, for every $m\geq2$, naturality gives
\[
 \lambda_\gamma\zeta_m^A\ulK_0(\theta;\Z/m)
 =\zeta_m^T\lambda_{\mathrm K}^{(m)}
                    \ulK_0(\theta;\Z/m),
\]
and the last factor is an isomorphism. This proves the
coefficient-to-unitary compatibility, including torsion, and proves
the lemma.
\end{proof}

Following \cite[Definition~1.4]{Shlyakhtenko}, we use CAAP to abbreviate
the \emph{$\mathrm C^*$-automorphism approximation property}.
For a trace-preserving unital inclusion
$B\subset C$, it asks
for trace-preserving automorphisms of $C$ that almost fix a
prescribed finite subset of $B$ and move $B$ close to a prescribed
element of $C$. We use the following form of the one-sided Elliott
approximate intertwining argument, retaining control of the full
invariant; compare \cite[Proposition~1.2]{Shlyakhtenko}.

\begin{lemma}[One-sided approximate intertwining with an exact invariant]\label{lem:marked-caap}
Let $B\subset C$ be a unital inclusion of faithfully tracial
separable algebras, with $\tau_B=\tau_C|_B$.
Suppose a subgroup $\mathcal G\subset\operatorname{Aut}_\tau(C)$
induces the identity on total $\mathrm K$-theory, the trace simplex
and the Hausdorffized unitary quotient of $C$, and has the following
property:
given finite $F\subset B$, $z\in C$ and $\eta,\varepsilon>0$, some
$\alpha\in\mathcal G$ and $y\in B$ satisfy
\begin{equation}\tag{CAAP}\label{eq:caap}
 \max_{b\in F}\|\alpha(b)-b\|<\eta,
 \qquad \|z-\alpha(y)\|<\varepsilon.
\end{equation}
Then there is a trace-preserving isomorphism $\Phi:B\to C$,
arbitrarily close to inclusion on any prescribed finite set, whose
induced triple is that of the inclusion $\iota:B\hookrightarrow C$:
\[
 ([\Phi],\Phi_T,\Phi^\ddagger)
       =([\iota],\iota_T,\iota^\ddagger).
\]
\end{lemma}
\begin{proof}
Fix a finite set $F_0\subset B$ and $\delta>0$, and choose
$\varepsilon_n>0$ with $\sum_n\varepsilon_n<\delta$.
Let $(b_j)$ be dense in $B$, and let $(z_n)$ be a sequence in $C$
in which each member of a fixed countable dense subset occurs
infinitely often. Put $\gamma_0=\id_C$ and $\phi_0=\iota$.
Inductively, having chosen $\gamma_{n-1}$ and $y_1,\ldots,y_{n-1}\in B$,
put
\[
 S_n=F_0\cup\{b_1,\ldots,b_n\}\cup\{y_1,\ldots,y_{n-1}\}.
\]
Apply~\eqref{eq:caap} to $S_n$ and $\gamma_{n-1}^{-1}(z_n)$ to find
$\alpha_n\in\mathcal G$ and $y_n\in B$ such that
\[
 \max_{b\in S_n}\|\alpha_n(b)-b\|<\varepsilon_n,
 \qquad
 \|\gamma_{n-1}^{-1}(z_n)-\alpha_n(y_n)\|<\varepsilon_n.
\]
Set $\gamma_n=\gamma_{n-1}\alpha_n$ and $\phi_n=\gamma_n\iota$.
Since $\gamma_{n-1}$ is isometric, these estimates give
\[
 \|\phi_n(b)-\phi_{n-1}(b)\|<\varepsilon_n\quad(b\in S_n),
 \qquad \|z_n-\phi_n(y_n)\|<\varepsilon_n.
\]
This is the one-sided approximate intertwining
\[
\begin{tikzcd}[column sep=large,row sep=large]
 B \arrow[r,"\id_B"] \arrow[d,hook,"\iota"']
   & B \arrow[r,"\id_B"] \arrow[d,hook,"\iota"']
   & B \arrow[r] \arrow[d,hook,"\iota"']
   & \cdots \arrow[r] & B \arrow[d,dashed,"\Phi"] \\
 C \arrow[r,"\alpha_1^{-1}"']
   & C \arrow[r,"\alpha_2^{-1}"']
   & C \arrow[r] & \cdots \arrow[r] & C .
\end{tikzcd}
\]
The square containing $\alpha_n^{-1}$ commutes within
$\varepsilon_n$ on $S_n$, because
$\|\alpha_n^{-1}(b)-b\|=\|\alpha_n(b)-b\|$.
The bottom limit is identified with $C$ by the stage maps $\gamma_n$,
which satisfy $\gamma_n\alpha_n^{-1}=\gamma_{n-1}$.
Under these identifications the vertical maps are $\phi_n$.

The summable estimates on the dense sequence $(b_j)$ show that
$\phi_n$ converges pointwise in norm to a unital trace-preserving
embedding $\Phi:B\to C$. It is isometric as a limit of isometric maps,
and
\[
 \max_{b\in F_0}\|\Phi(b)-b\|
 \leq\sum_n\varepsilon_n<\delta.
\]
Every chosen preimage $y_n$ belongs to all subsequent sets $S_m$,
$m>n$. Hence
\[
 \|\Phi(y_n)-z_n\|
 \leq \|\phi_n(y_n)-z_n\|
       +\sum_{m>n}\|\phi_m(y_n)-\phi_{m-1}(y_n)\|
 <\sum_{m\geq n}\varepsilon_m.
\]
The right side tends to zero. Since every member of the chosen dense
subset of $C$ occurs infinitely often among the $z_n$, the range of
$\Phi$ is dense. It is also norm closed, so $\Phi$ is surjective.

Finally, every $\gamma_n\in\mathcal G$ acts identically on the full
invariant, so each $\phi_n$ induces the inclusion triple.
For any fixed projection or unitary representative over $B$ or
$B\otimes D_m$, the images converge in norm; eventual closeness gives
equality of its $\mathrm K$-class. Tensoring by a fixed coefficient
algebra preserves point-norm convergence, so this also treats total
$\mathrm K$-theory.
For $v\in\mathrm U_d(B)$, the images converge in the single matrix
algebra $\mathrm M_d(C)$, and their common algebraic unitary class
passes to the limit because
$\mathrm U_d(C)\to\Uquot{C}$ is continuous and $\Uquot{C}$ is Hausdorff.
For every $\sigma\in\mathrm T(C)$, we have
$\sigma\circ\phi_n=\sigma\circ\iota$ for every $n$, and the equality
passes to the limit. Thus
$([\Phi],\Phi_T,\Phi^\ddagger)=([\iota],\iota_T,\iota^\ddagger)$.
This argument treats each individual class and requires neither
finite generation nor a uniform matrix bound across different classes.
\end{proof}

\begin{remark}[Approximate intertwining and the history of CAAP]
\label{rem:caap-history}
The limiting argument in Lemma~\ref{lem:marked-caap} is the one-sided
Elliott approximate intertwining argument. A direct antecedent of its
local approximation condition is Elliott's proof of
$\mathcal O_2\otimes\mathcal O_2\cong\mathcal O_2$, as presented in
\cite{RordamO2}. That presentation uses two-sided approximate
intertwining. To see the corresponding inner-automorphism condition,
suppose $i:A\hookrightarrow B$ and $j:B\to A$ are unital
$^*$-homomorphisms with
$ji\au\id_A$ and $ij\au\id_B$.
Given finite $F\subset A$, $z\in B$ and $\eta,\varepsilon>0$,
first conjugate $j$ by a unitary of $A$ to obtain $j'$ with
$\|j'i(a)-a\|<\eta/2$ for $a\in F$.
Since $ij'\au\id_B$, there is a unitary $u\in B$ such that
\[
 \|\Ad(u)(ij'(i(a)))-i(a)\|<\eta/2\quad(a\in F),
 \qquad \|\Ad(u)(ij'(z))-z\|<\varepsilon.
\]
Thus $\alpha=\Ad(u)$ almost fixes $i(F)$ within $\eta$ and
moves $i(j'(z))$ within $\varepsilon$ of $z$.
These are precisely the local estimates in~\eqref{eq:caap}, with
inner automorphisms and with the trace requirement omitted.
The omission matters here, since $\mathcal O_2$ has no tracial state.
The general inner-automorphism criterion is recorded in
\cite[Proposition~2.3.5]{RordamClassification};
\cite[Section~2 and Proposition~2.6]{TomsWinter} explicitly connects
this absorption method with Elliott's proof.

Hirshberg and Phillips already implement the same local automorphism
approximation mechanism for infinite reduced free-product absorption
\cite[proof of Theorem~3.1]{HP}.
Under their hypotheses, their construction verifies CAAP for the
canonical inclusion $D\subset A\fr D$.
Indeed, after their regrouping, write $D=\mathop{\fr}_{n\geq1}C_n$.
Given finite $S\subset D$, $z\in A\fr D$ and tolerances
$\eta,\varepsilon>0$, approximate $S$ and $z$ by word polynomials
using $C_1,\ldots,C_r$ and finitely many coefficients from $A$.
Choose $n>r$. They construct a trace-preserving embedding
$\phi_n:A\to C_n$ and, by nuclear-source uniqueness, a unitary
$u\in A\fr C_n$ such that $u\phi_n(a)u^*$ is as close to $a$
as needed on those coefficients. Put
\[
 D_{\ne n}=\mathop{\fr}_{k\ne n}C_k,
 \qquad
 \alpha=\Ad(u)\fr\id_{D_{\ne n}}
       \in\operatorname{Aut}_\tau\bigl((A\fr C_n)\fr D_{\ne n}\bigr).
\]
Under this identification, $\alpha$ is a trace-preserving automorphism
of $A\fr D$ fixing the first $r$ blocks pointwise. Its image
$\alpha(D)$ approximately contains the required coefficients from $A$.
Choosing the polynomial and coefficient errors sufficiently small gives
\[
 \max_{b\in S}\|\alpha(b)-b\|<\eta,
 \qquad \dist(z,\alpha(D))<\varepsilon.
\]
Thus their proof yields exactly the local condition~\eqref{eq:caap}.
They do not use the name CAAP; their printed proof uses the
two-sided Elliott approximate intertwining of \cite[Theorem~2.2]{HP}.
The extension over the free complement is also emphasized in their
introduction; it need not be inner on the entire free product.

The name CAAP and the trace-preserving norm formulation used here
are from \cite[Proposition~1.2 and Definition~1.4]{Shlyakhtenko}.
Its von Neumann algebra counterpart, AAP, describes the mechanism
of \cite[Proposition~5.1 and Section~5.2]{OpenAIFreeFactors}:
automorphisms approximately fix the existing generators and move
elements of their algebra close to the additional generator,
followed by a summable iteration. The original argument does not
use the names AAP or CAAP, and its von Neumann algebra
automorphisms are not asserted to preserve the canonical reduced
group $\mathrm C^*$-algebra. Our use of the classical intertwining
argument additionally records the prescribed full invariant:
the subgroup in Lemma~\ref{lem:marked-caap} acts trivially on that
invariant, so the limiting isomorphism induces the inclusion triple.
The subsequent core-replacement and fusion arguments verify the
required local approximation condition in the settings of this paper.
\end{remark}

\begin{proposition}[Replacement with the inclusion marking]
\label{prop:marked-replacement}
Let $\theta:E\to A$ be as in Lemma~\ref{lem:marked-factor}, with
$E,A$ nuclear UCT and $A$ monotracial. Then there is a
trace-preserving isomorphism $E\fr D\to A\fr D$ inducing the same
totally compatible triple as $\theta\fr\id_D$.
\end{proposition}
\begin{proof}
Regard $B=\mathrm{C}^*(\theta(E),D)$ as a subalgebra of $C=A\fr D$.
For every finite reservoir-block head $D_{\leq r}$, the tail inclusion
$k:C_0=A\fr D_{>r}\to C$ is a total $\mathrm{K}$-isomorphism.
Its induced triple is invertible, since both algebras are
monotracial with the same pairing, by the Thomsen sequence.
Consequently every automorphism
$\beta=\id_{D_{\leq r}}\fr\Ad(u)$, $u\in\mathrm U(C_0)$,
induces the identity triple on $C$: the relation
$\beta k=k\Ad(u)$ and triviality of the triple of an inner
automorphism permit cancellation of the invertible triple of $k$.

We verify~\eqref{eq:caap} using precisely these automorphisms and their inverses.
Approximate finite $F\subset B$ and $z\in C$ by polynomials with a
common finite reservoir-block head.  Put $T=E\fr D_{>r}$ and let
$j:E\to T$, $j_A:A\to C_0$ be canonical.
Lemma~\ref{lem:marked-factor} gives a totally compatible
isomorphism triple $\lambda$ from $A$ to $T$ with
\[
 \lambda([\theta],\theta_T,\theta^\ddagger)
       =([j],j_T,j^\ddagger).
\]
Theorem~\ref{thm:embedding} realizes it by $\rho:A\to T$.
Nuclear-source uniqueness gives $\rho\theta\au j$; conjugate $\rho$
to make this approximation as accurate as needed on the finitely
many $E$-coefficients of $F$.
Set $L=\theta\fr\id_{D_{>r}}$ and $\chi=L\rho:A\to C_0$.
Since $Lj=j_A\theta$, the triples induced by $\chi\theta$ and
$j_A\theta$ agree. Surjectivity of $[\theta]$ and
$\theta^\ddagger$ permits cancellation on total $\mathrm K$-theory
and unitary quotients; the trace maps agree because $A$ is
monotracial. Thus
\[
 ([\chi],\chi_T,\chi^\ddagger)
       =([j_A],(j_A)_T,j_A^\ddagger).
\]
A second use of nuclear-source uniqueness
gives $u\in \mathrm{U}(C_0)$ with $\Ad(u)j_A$ close to $\chi$ on the
$A$-coefficients of $z$ and the $\theta$-images of the selected
$E$-coefficients of $F$.
For every selected coefficient $e\in E$, the two uniqueness estimates give
\[
 \|\beta(\theta(e))-\theta(e)\|
 \leq\|u\theta(e)u^*-\chi\theta(e)\|
       +\|L\rho\theta(e)-Lj(e)\|<2\delta.
\]
The head is fixed exactly.  Replacing each $A$-coefficient of the
polynomial for $z$ by its image under $\chi$ gives an element
of $B$, because $\chi(A)\subset B$; its distance from $\beta(z)$
is arbitrarily small.  Polynomial continuity gives the specified
errors, first choosing the polynomial approximations and then
$\delta$.  Taking $\alpha=\beta^{-1}$ gives the orientation
in Lemma~\ref{lem:marked-caap}.  The subgroup generated by these
automorphisms induces the identity triple on $C$, so that lemma
gives the isomorphism with exactly the inclusion marking.
\end{proof}

\begin{remark}[Comparison with Hirshberg--Phillips: replacement]
\label{rem:hp-replacement}
Proposition~\ref{prop:marked-replacement} uses the same CAAP mechanism
as Hirshberg--Phillips's absorption argument, as explained in
Remark~\ref{rem:caap-history}. In both arguments, nuclear-source
uniqueness supplies correcting conjugacies, which are extended by
the identity over free complements.
Their uniqueness input \cite[Corollary~3.4]{HP} treats unital
one-dimensional NCCW algebras and their unital countable inductive
limits with vanishing $\mathrm K_1$; they apply it inside an exact
intermediate product with $\mathcal Z$ and transfer the conjugacies
through their diagram~(3.4).
The additional input here is Theorem~\ref{thm:embedding}, which
permits general nuclear UCT sources with arbitrary $\mathrm K$-groups
and prescribed total $\mathrm K$-theory and algebraic unitary data.
Cancellation through the invertible triple induced by $k$ verifies
that each correction preserves the full invariant exactly, so
Lemma~\ref{lem:marked-caap} yields the prescribed inclusion marking.
\end{remark}

\begin{theorem}[Marked nuclear-core $D$-stabilization]\label{thm:marked-core}
Let $D$ be a reservoir. Let $E,F$ be nuclear UCT cores with faithful
selected traces. Every prescribed unital totally compatible morphism
$\xi$ between the full invariants of $E\fr D$ and $F\fr D$ is realized
by a unital trace-preserving embedding. If $\xi$ is an isomorphism of
full invariant triples, the realizing map may be chosen to be a
$^*$-isomorphism.

Every unit-preserving homomorphism of graded $\mathrm K$-groups
that preserves the selected pairing admits a compatible full extension
and is therefore realized. If both group maps are isomorphisms,
the realizing map may be chosen to be an isomorphism.
The prescribed $\xi$ may be any unital totally compatible extension,
as provided by \cite[Theorem~3.9]{CGSTW}; no canonical extension is
asserted. No preservation of the original ordered
$\mathrm K_0$ or of the other source traces is required.
The realizing map may fix any identified finite reservoir-block prefix
$H=D_{\leq r}$ pointwise and carry $E\fr D_{>r}$ into
$F\fr D_{>r}$, and onto it in the isomorphism case.
In particular the old core stays in the complementary tail.
We call $H$ a \emph{protected head}: it consists of the first $r$
reservoir blocks, and the realizing map fixes every element of $H$.
Only the number of blocks is finite; $H$ need not be finite dimensional.
\end{theorem}
\begin{proof}
We first suppose that $\xi$ is an isomorphism of full invariant triples,
that is, an invertible morphism, and take $r=0$, so $H=D_{\leq0}=\C1$.
Choose a monotracial range model $A$
for $F$ and an embedding $\theta_F:F\to A$ as in
Lemma~\ref{lem:marked-range}. Factorization gives a totally
compatible isomorphism triple $\lambda_F$ from $A$ to $F\fr D$ with
\[
 \lambda_F([\theta_F],(\theta_F)_T,\theta_F^\ddagger)
       =([j_F],(j_F)_T,j_F^\ddagger).
\]
The faithful totally compatible triple
\[
 \lambda_F^{-1}\xi([j_E],(j_E)_T,j_E^\ddagger)
\]
from $E$ to $A$ is realized by a trace-preserving embedding
$\theta_E:E\to A$, using the nuclear embedding theorem
\cite[Theorem~B]{CGSTW}. It induces graded $\mathrm K$-isomorphisms.
Proposition~\ref{prop:marked-replacement} gives isomorphisms
$\Phi_E:E\fr D\to A\fr D$ and $\Phi_F:F\fr D\to A\fr D$.
Writing $j_A:A\to A\fr D$, their exact markings are
\[
 \begin{aligned}
 ([\Phi_F],(\Phi_F)_T,\Phi_F^\ddagger)
   &=([j_A],(j_A)_T,j_A^\ddagger)\lambda_F^{-1},\\
 ([\Phi_E],(\Phi_E)_T,\Phi_E^\ddagger)
   &=([j_A],(j_A)_T,j_A^\ddagger)\lambda_F^{-1}\xi.
 \end{aligned}
\]
For the first identity, compose on the right with the triple of
$j_F$; for the second, use that of $j_E$. The resulting equalities
are exactly the two inclusion markings of
Proposition~\ref{prop:marked-replacement} and the definition of
$\theta_E$. The total $\mathrm K$-theory and unitary maps induced
by $j_E,j_F$ are surjective by Lemma~\ref{lem:marked-factor}, so
cancellation proves these two components of each identity. All
stabilizations are monotracial, so the trace components agree as
well. Hence $\Phi_F^{-1}\Phi_E$ realizes $\xi$ exactly.

Now let $\xi$ be an arbitrary unital totally compatible morphism.
Choose simple monotracial nuclear range models $A_E,A_F$ by
Lemma~\ref{lem:marked-range}. Choose compatible isomorphism triples
between their stabilizations and the respective original stabilizations
by \cite[Theorem~3.9]{CGSTW}. Apply the invertible case just proved
separately to the two matching pairs of cores, obtaining marked
trace-preserving isomorphisms
\[
 \Theta_E:E\fr D\longrightarrow A_E\fr D,\qquad
 \Theta_F:F\fr D\longrightarrow A_F\fr D.
\]
Let $\Lambda_E,\Lambda_F$ be their induced triples, and let
$J_E,J_F$ be the triples of the native inclusions
\[
 j_E:A_E\longrightarrow A_E\fr D,\qquad
 j_F:A_F\longrightarrow A_F\fr D.
\]
Both $J_E$ and $J_F$ are invertible: they are total-$\mathrm K$
isomorphisms, both their source and target are monotracial, and
the selected trace ranges agree, so the Thomsen sequence gives the
unitary isomorphisms as well. The triple
\[
 \zeta=J_F^{-1}\Lambda_F\xi\Lambda_E^{-1}J_E
\]
is a unital totally compatible morphism from $A_E$ to $A_F$.
Its trace map sends the unique trace to the unique faithful trace.
Nuclear existence \cite[Theorem~B]{CGSTW} realizes $\zeta$ by a unital
trace-preserving embedding $\theta:A_E\to A_F$.
Reduced free-product functoriality gives an actual embedding
\[
 L=\theta\fr\id_D:A_E\fr D\longrightarrow A_F\fr D.
\]
Write $\operatorname{Inv}(L)$ for its induced triple. The identity
$Lj_E=j_F\theta$ gives
\[
 \operatorname{Inv}(L)J_E=J_F\zeta
      =\Lambda_F\xi\Lambda_E^{-1}J_E.
\]
Cancel $J_E$ to see that
\[
 \Theta_F^{-1}L\Theta_E:E\fr D\longrightarrow F\fr D
\]
induces exactly $\xi$. No invertibility of $\xi$ was used in this
construction.

For the assertion about graded group homomorphisms, transport the
given unit- and pairing-preserving maps to $A_E,A_F$.
They are positive for the strict trace cones: a nonzero positive
class has strictly positive selected trace, and the pairing is
preserved. The unique trace map is faithful.
The compatible-extension theorem \cite[Theorem~3.9]{CGSTW}
therefore supplies a full morphism, which is realized as above.
For group isomorphisms the extension is invertible and the first
part of the proof supplies an isomorphism.

For a prescribed head $H=D_{\leq r}$, write $D=H\fr D_{>r}$.
The inclusions of $E\fr D_{>r}$ and $F\fr D_{>r}$ into
$E\fr D$ and $F\fr D$, respectively, induce invertible triples.
Transport $\xi$ through these triples, apply the appropriate preceding
construction to $D_{>r}$,
and extend the resulting tail embedding by $\id_H$ using reduced
free products. If $\xi$ is an isomorphism of invariants, choose the
tail isomorphism.
Cancellation through the invertible triple of either tail inclusion
verifies its marking. This fixes $H$ pointwise and has the stated
containment, or equality, of complementary tails.
\end{proof}

However, uniqueness up to approximate unitary equivalence fails in
general for embeddings between these $D$-stabilized free products, even
when the full invariant maps agree.

\begin{example}[Failure of uniqueness for free-product embeddings]
\label{ex:reservoir-nonuniqueness}
Let
\[
 D_I=\fr_{j=1}^{\infty}(\mathrm C([0,1]),\lambda),
\]
where $\lambda$ is Lebesgue measure, and let $\alpha$ exchange the first
two free factors while fixing all later factors pointwise.
Then $\alpha$ and $\id_{D_I}$ induce the same full totally compatible
triple, but are not approximately unitarily equivalent.
Thus uniqueness fails even for automorphisms in the class of
Theorem~\ref{thm:marked-core}, with both nuclear cores
equal to $\C$. Such counterexamples can also fix any prescribed finite
interval prefix pointwise.
\end{example}
\begin{proof}
Permutation of the identically traced factors gives a unital
trace-preserving automorphism $\alpha$. The reservoir has
\[
 \mathrm K_0(D_I)=\Z[1],\qquad
 \mathrm K_1(D_I)=0,\qquad
 \mathrm T(D_I)=\{\tau\}.
\]
Hence unitality determines the identity map on ordinary
$\mathrm K$-theory. For every $m\ge2$, the coefficient exact sequences
give
\[
 \mathrm K_0(D_I;\Z/m)=\Z/m,\qquad
 \mathrm K_1(D_I;\Z/m)=0.
\]
The coefficient map from $\mathrm K_0(D_I)$ is surjective, so
$\alpha$ acts identically on total $\mathrm K$-theory as well.
The determinant sequence~\eqref{eq:thomsen} identifies
\[
 \Uquot{D_I}\cong\R/\Z.
\]
This identification is natural with respect to the trace, so
$\alpha^\ddagger=\id$. Its trace map is also the identity.
Thus all components of its totally compatible triple agree with
those of $\id_{D_I}$.

For the failure of approximate unitary equivalence, put
$M=\pi_\tau(D_I)''$, and let $N\subset M$ be the von Neumann algebra
generated by the factors with indices $j\ge3$.
The tracial free-product construction gives
\[
 M=\bigl(\mathrm L^\infty([0,1],\lambda)
       *\mathrm L^\infty([0,1],\lambda)\bigr)*N.
\]
Here $N$ is a type $\mathrm{II}_1$ factor, $M$ is full, and
$N'\cap M=\C$, by \cite[Corollary~3.2 and Theorems~3.4 and~3.7]{UedaFull}.
The normal extension $\bar\alpha$ fixes $N$ pointwise.
An implementing unitary for $\bar\alpha$ would therefore belong to
$N'\cap M=\C$, forcing $\bar\alpha=\id$. This is impossible because
$\alpha$ exchanges two distinct free factors. Thus $\bar\alpha$ is outer.

If $\alpha\au\id_{D_I}$, there would be unitaries $u_n\in D_I$ with
$u_n a u_n^*\to\alpha(a)$ in norm for every $a\in D_I$.
Density of $D_I$ in $M$ for $\|\cdot\|_2$ and trace preservation
would then give
\[
 \|u_n x u_n^*-\bar\alpha(x)\|_2\longrightarrow0
 \qquad(x\in M).
\]
For a full finite factor, the inner automorphism group is closed
in this topology. Hence $\bar\alpha$ would be inner, a contradiction.
Finally, exchanging two factors after a prescribed finite prefix
and fixing all other factors gives the same counterexample with
that prefix fixed pointwise.
\end{proof}

\begin{remark}[Comparison with Hirshberg--Phillips: the nuclear model]
\label{rem:hp-model}
The common nuclear model $A$ used in the invertible case of
Theorem~\ref{thm:marked-core} plays the role of the Jiang--Su model
in \cite[Proposition~2.13 and proof of Theorem~3.1]{HP}.
Its pointed graded $\mathrm K$-groups and selected pairing are chosen
to match the given core data, and the two embeddings into it retain
the prescribed totally compatible triple.
\end{remark}

\begin{remark}[The role of $\mathrm K$-theory and of the reservoir]
\label{rem:reservoir-roles}
Theorem~\ref{thm:marked-core} classifies the products using the pointed
graded $\mathrm K$-groups of their nuclear cores and the selected trace
pairing. The order on the original core and its full trace simplex
need not be preserved. Strict comparison and monotraciality recover
the order and trace simplex of the product from these selected data.
Totally compatible triples are needed in the embedding and
uniqueness steps, and compatible extensions of the selected group
isomorphisms exist.

The proof uses four properties of the reservoir: the core inclusions
and complementary-tail inclusions preserve $\mathrm K$-theory; every
infinite tail gives an admissible target; finite sets are approximated
using finitely many reservoir blocks; and a trace-preserving inner
automorphism of the complementary tail extends by the identity on
the finite head. The first two properties are supplied by
Lemma~\ref{lem:reservoir-regularity}. The last two follow from the
free-product presentation. No interval coordinate, semicircular
distribution or polynomial-flow construction is used in this proof.
\end{remark}

\begin{corollary}[Independence of the reservoir presentation]
\label{cor:reservoir-model}
Every reservoir $D$ of Definition~\ref{def:reservoir} is
trace-preservingly isomorphic to $D_I=I^{\fr\infty}$.
Consequently any two such reservoirs give the same class of
$D$-stabilized algebras, and $D\fr D\cong D$.
\end{corollary}
\begin{proof}
We use the elementary approximate-intertwining fact that, if the
connecting inclusions $A_n\hookrightarrow A_{n+1}$ are point-norm
limits of isomorphisms, then $A_0\cong\varinjlim A_n$.
To see this, compare the inclusion system with the system of
isomorphisms, choosing the latter with summable errors on increasing
finite sets containing dense subsets and the relevant images from
all earlier stages. The usual approximate intertwining identifies
their limits; the isomorphism system has limit $A_0$.
The same argument preserves selected traces.

Apply Theorem~\ref{thm:marked-core} to the nuclear cores $\C$ and $I$
with reservoir $D$. It gives isomorphisms $D\to I\fr D$ fixing any
prescribed finite block prefix of $D$. Since such prefixes have dense
union, the canonical inclusion is a point-norm limit of these
isomorphisms. Extending over an external finite interval head proves
the same assertion for the inclusions
\[
 I^{\fr n}\fr D\longrightarrow I^{\fr(n+1)}\fr D.
\]
The preceding fact gives $D\cong D_I\fr D$.

Conversely, $D_I$ itself satisfies Definition~\ref{def:reservoir}.
Apply Theorem~\ref{thm:marked-core}, with reservoir $D_I$, to $\C$
and each individual nuclear block $R_j$. It gives isomorphisms
$D_I\to R_j\fr D_I$ approximating the canonical inclusion by fixing
arbitrarily large interval prefixes. Extend them over the preceding
finite free head $R_1\fr\cdots\fr R_{j-1}$. Thus the inclusions in
the system
\[
 D_I\longrightarrow R_1\fr D_I\longrightarrow
 R_1\fr R_2\fr D_I\longrightarrow\cdots
\]
are point-norm limits of isomorphisms. Its limit is $D\fr D_I$, so
$D_I\cong D\fr D_I$. Associativity and interchange of the two free
factors now give $D\cong D_I$. No fusion theorem is used.
Finally, reindexing interval factors gives $D_I\fr D_I\cong D_I$.
\end{proof}

\begin{remark}[Finite old-core alignment]\label{rem:marked-alignment}
Suppose additionally that a unital embedding
$\kappa:E\to F\fr D_{>r}$ preserves the selected trace and has the same
induced triple as the restriction of the constructed tail
isomorphism.  Theorem~\ref{thm:embedding} makes the two restrictions
approximately unitarily equivalent inside $F\fr D_{>r}$.
For any finite core packet, postcompose by
$\id_H\fr\Ad(u)$ for the resulting tail unitary $u$.
This attains the desired approximation, fixes $H$ exactly, preserves
the complementary tail, and changes no component of $\xi$.
Thus exact invariant marking and finite old-core coherence are available
whenever the comparison embedding already lies in the same chosen
complementary tail. This statement does not prepare an arbitrary
core which is merely scalar free from a reservoir.
\end{remark}
\begin{remark}[An exactly protected common matrix frame]\label{rem:common-frame}
Suppose the two nuclear-source embeddings being compared agree on a
unital frame $C\cong\mathrm M_d$ inside the same complementary tail.
Choose an approximate conjugator $u$ in that tail, including its matrix
units among the uniqueness tests. Then
\[
y=\frac1d\sum_{i,j=1}^d e_{ij}u e_{ji}\in C'
\]
satisfies
$\|y-u\|\leq d\max_{i,j}\|[e_{ij},u]\|$.
For a sufficiently accurate initial conjugator, $y$ is invertible.
Its polar unitary lies in $C'$ and remains close to $u$.
Thus the frame can be fixed exactly, while extending by the identity
on an external protected reservoir head.

The invariant must mark the minimal-projection class of this frame.
Preservation of the total unit alone does not imply preservation of
$[e_{11}]$ in the presence of $d$-torsion.
For instance, in $\Z g\oplus\Z/2\,t$ with unit $2g$,
the automorphism $g\mapsto g+t$, $t\mapsto t$ preserves the unit
and the trace pairing but changes the chosen minimal class $g$.
\end{remark}

\begin{remark}[The geometric scope of replacement]
The preceding results classify nuclear-source embeddings and replace
abstract nuclear cores after adjoining a general reservoir of Definition~\ref{def:reservoir}.
They do not assert that approximate conjugating unitaries protect an
additional reservoir. The alignment in
Remark~\ref{rem:marked-alignment} requires the compared maps already
to lie in the same ambient complementary tail; scalar freeness from
the reservoir alone does not give this containment.
The finite-base construction in Section~\ref{sec:general-fusion}
supplies the additional relative compatibility needed for fusion.
Its correcting maps are constructed in the chosen complementary
tail, so this geometric requirement is verified explicitly.
\end{remark}

\begin{remark}[Interval coordinates for the subsequent constructions]
\label{rem:interval-convention}
From the next section onward, fix a trace-preserving identification
$D\cong D_I$ from Corollary~\ref{cor:reservoir-model} and use $D$ for
this interval presentation. Finite heads and complementary tails in
the matrix, fusion and exponential-length constructions refer to
these interval coordinates. A continuous change of variables also
gives the free standard semicircular presentation used there.
The extension $\id_P\fr(D\cong D_I)$ transfers the conclusions back
to every reservoir of Definition~\ref{def:reservoir}; it preserves the
canonical $P$-inclusion and its selected invariant data.
This choice does not identify an original block prefix with an
interval prefix. Protection of original block prefixes is the direct
conclusion of Theorem~\ref{thm:marked-core}.
\end{remark}

\section{Outline of the fusion argument}\label{sec:fusion-outline}

The nuclear-core theorem of Section~\ref{sec:marked-core} classifies
algebras of the form $E\fr D$. To apply it to several nuclear
factors, we must first identify their $D$-stabilized free product
with $E\fr D$ for a nuclear core $E$ carrying the same aggregate
selected data.
We outline this fusion construction before developing its analytic
tools. Section~\ref{sec:relative-logs} treats logarithm comparison
and interval automorphisms; Section~\ref{sec:fd-relative} supplies the
finite-dimensional relative reservoirs, and
Section~\ref{sec:general-fusion} assembles the general argument.
The independent split-unit construction in
Section~\ref{app:split-unit} supplies the reduction for arbitrary
$\mathrm K_1$.

For the binary step, let $A_1,A_2$ be nuclear UCT algebras with
faithful selected traces, and initially assume
$\mathrm K_1(A_1)=\mathrm K_1(A_2)=0$.
The goal is an isomorphism
$(A_1\fr A_2)\fr D\cong E_0\fr D$, where $E_0$ is nuclear UCT
and carries the aggregate pointed $\mathrm K$-groups and selected
trace pairing.

\begin{enumerate}
\item \emph{Matrix reduction.}
The exact reduced identity
\[
 \mathrm M_2(A_1\fr A_2)
 \cong (A_1\oplus A_2)*_{\C^2}\mathrm M_2
\]
puts the two original factors into the nuclear coefficient algebra
$A_1\oplus A_2$. The remaining attachment is
$\mathrm M_2$, over the common labelled diagonal $\C^2$.
The trace on $A_1\oplus A_2$ is half-weighted, and the selected traces on the
matrix algebras are normalized. Thus the problem becomes a comparison of
a nuclear coefficient algebra with a finite-dimensional attachment.

\item \emph{Choice of a nuclear model.}
Choose a simple monotracial nuclear UCT range core $E_0$ with
the aggregate selected data and put $E=\mathrm M_2(E_0)$.
The embedding theorem supplies maps $a_i:A_i\to E_0$ with the
canonical selected markings. Their diagonal sum gives
$a:A_1\oplus A_2\to E$, explicitly
$a(x_1\oplus x_2)=\operatorname{diag}(a_1(x_1),a_2(x_2))$,
and the standard matrix frame gives $b:\mathrm M_2\to E$.
These maps agree on $\C^2$. The core $E_0$ records the aggregate unit
and pairing, while the labelled frame in $E$ records the corner
to be compressed at the end.

\item \emph{The relative interval comparison.}
Let $I_h=(\mathrm C([0,1]),\lambda)$, with coordinate $h$, and put
\[
 \begin{aligned}
 S&=\bigl((A_1\oplus A_2)*_{\C^2}\mathrm M_2\bigr)\fr D,\\
 T&=\bigl(E*_{\C^2}(\C^2\otimes I_h)\bigr)\fr D,
 \qquad q=e^{2\pi ih}.
 \end{aligned}
\]
The relative Haar conjugation rule gives a faithful embedding
$\eta:S\to T$ defined by
\[
 \eta|_{A_1\oplus A_2}=a,\qquad
 \eta|_{\mathrm M_2}=\Ad(q)\circ b,\qquad
 \eta|_D=\id_D.
\]
The last prescription concerns the common external reservoir;
the interval $h$ is a separate relative coordinate commuting with
$\C^2$. Proposition~\ref{gf:zero-K1-binary-fusion} verifies that $\eta$ has invertible
full selected marking and that the working targets satisfy the
admissibility and zero-$\mathrm K_1$ hypotheses of the return
construction.

\item \emph{One return with source and target control.}
For prescribed finite sets $F\subset S$ and $L\subset T$ and
$\varepsilon>0$, construct an actual trace-preserving embedding
$\beta:T\to S$ and an actual trace-preserving automorphism
$\alpha$ of $T$ such that
\[
 \|\beta\eta(s)-s\|<\varepsilon\quad(s\in F),\qquad
 \|\alpha\eta\beta(t)-t\|<\varepsilon\quad(t\in L).
\]
The return has the inverse full selected marking, and $\alpha$ acts
identically on the full invariant.
The first estimate comes from nuclear-source uniqueness and
relative free transport. The second requires compatible corrections
of the nuclear coefficients and the whole finite matrix frame,
using independent parts of the reservoir.
Relative logarithm comparison, lifts of circle automorphisms to
interval automorphisms, and normalized corner determinants make
the two estimates hold for the same $\beta$.
These are the principal analytic ingredients developed in
Sections~\ref{sec:relative-logs}, \ref{sec:fd-relative}
and~\ref{sec:general-fusion}.

\item \emph{Isomorphism and compression.}
Include $\eta(F)$ in $L$. The two estimates then make $\alpha$
almost fix $\eta(F)$, while placing each target test element near
$\alpha\eta(S)$. More precisely,
\[
 \begin{aligned}
 \|\alpha\eta(s)-\eta(s)\|
 &\leq\|s-\beta\eta(s)\|
       +\|\alpha\eta\beta(\eta(s))-\eta(s)\|<2\varepsilon
       &&(s\in F),\\
 \dist(t,\alpha\eta(S))
 &\leq\|t-\alpha\eta\beta(t)\|<\varepsilon
       &&(t\in L).
 \end{aligned}
\]
Since the tolerances are arbitrary, Lemma~\ref{lem:marked-caap} gives
$S\cong T$ with the marking of $\eta$.
Finite-dimensional tag absorption identifies $T$ with $E\fr D$,
fixing $E$. Together with the matrix identities of
Proposition~\ref{gf:matrix-identities}, this yields
\[
 \mathrm M_2\bigl((A_1\fr A_2)\fr D\bigr)
 \cong \mathrm M_2(E_0\fr D).
\]
The marking tracks the exact class of the standard projection
$e_{11}$, including possible two-torsion: $\eta(e_{11})=e_{11}$,
because $e_{11}\in\C^2$ and $q$ commutes with $\C^2$.
Stable-rank-one cancellation for this projection and its complement
gives a unitary correction carrying it to the standard target
corner. Compression therefore gives the desired binary fusion.

\item \emph{Finite and countable families and arbitrary
$\mathrm K_1$.}
Iteration treats finite zero-$\mathrm K_1$ families. For a
countable zero-$\mathrm K_1$ family, each new fusion is chosen to approximate the
previous identification on increasing finite subsets of the old
$D$-stabilized product, using
Proposition~\ref{gf:tail-controlled-binary-fusion}.
The two-sided dense exhaustions and summable
estimates make both the stage isomorphisms and their inverses
converge, giving mutually inverse limiting maps as in
Proposition~\ref{gf:countable-zero-K1-fusion}.
The split-unit argument of Section~\ref{app:split-unit} then
separates arbitrary $\mathrm K_1$ and completes the general
fusion theorem. Marked nuclear-core replacement realizes any
prescribed compatible global marking after the fusion isomorphism
has been constructed.
\end{enumerate}

\section{Relative logarithms and interval automorphisms}
\label{sec:relative-logs}

We now develop the logarithm comparison needed for the simultaneous
estimates in the outline of Section~\ref{sec:fusion-outline}.
The coefficient algebra need not be
nuclear: proper unitary comparison uses a nuclear difference ideal
and a vanishing actual relative $\mathrm{KK}$ class.
We also lift suitable circle automorphisms to actual interval
automorphisms. These results will later be applied to a global
nuclear core together with finitely many distinguished intervals,
and extended over a finite diagonal in
Section~\ref{sec:general-fusion}.

\subsection{The endpoint state and one-turn interval relations}
\label{rf:subsec-foundation}

Throughout this subsection, all reduced free products use their
specified faithful tracial states.  We write
\[
 R=(F,\tau_F)\fr(\mathrm C(\mathbb T),m),\qquad
 Q=(F,\tau_F)\fr(\mathrm C([0,1]),\lambda),
\]
and denote the circle and interval coordinates by $w$ and $h$.
The canonical inclusion $j:R\to Q$ sends $w$ to $e^{2\pi ih}$
and fixes $F$.  No nuclearity assumption on $F$ is made in the
next three lemmas.

\begin{lemma}[The endpoint state]\label{rf:lem-endpoint}
The state $\omega=\operatorname{ev}_1\circ\E_{\mathbb T}$ is
the unique state on $R$ taking value one at $w$.  For every
$a\in R$,
\[
 \sup_f\|f(w)(a-\omega(a)1)f(w)\|\longrightarrow0
 \qquad(\delta\downarrow0),
\]
where the supremum runs over positive contractions $f$ supported
in $\{z:|z-1|\leq\delta\}$.  In any representation of $R$,
the spectral projection $P=1_{\{1\}}(w)$ satisfies
$PaP=\omega(a)P$.
\end{lemma}

\begin{proof}
Put $F_n=w^nFw^{-n}$ and $B=\mathrm C^*(F_n:n\in\Z)$.
Expanding centered words proves that the $F_n$ are free.
Finite sums $\sum_kx_kw^k$, with each $x_k$ supported in
finitely many $F_n$, are dense in $R$, and
$\E_{\mathbb T}(xw^k)=\tau_B(x)w^k$ for $x\in B$.
For centered elements $y_1,\ldots,y_N$ in free subalgebras,
the creation, annihilation and diagonal decomposition on the
reduced free-product Hilbert space gives
\[
 \Big\|\sum_jy_j\Big\|
 \leq\max_j\|y_j\|
  +\Big(\sum_j\tau(y_j^*y_j)\Big)^{1/2}
  +\Big(\sum_j\tau(y_jy_j^*)\Big)^{1/2}.
\]
If $x\in B$ is centered and has finite integer support of
width less than $L$, its translates by $w^{jL}$ lie in free
subalgebras.  Consequently
\[
 \Big\|N^{-1}\sum_{j=1}^Nw^{jL}xw^{-jL}\Big\|
 \leq(N^{-1}+2N^{-1/2})\|x\|.
\]
A state $\varphi$ with $\varphi(w)=1$ has $w$ in its
multiplicative domain.  It is invariant under these translates
and satisfies $\varphi(aw^k)=\varphi(a)$.  Averaging and then
using the displayed dense algebra gives
$\varphi(\sum_kx_kw^k)=\sum_k\tau_B(x_k)=\omega(\sum_kx_kw^k)$.
This proves uniqueness.

For the asserted uniform estimate,
$\|(w^n-1)f(w)\|\leq |n|\delta$ gives
\[
 \|f(w)xf(w)\|
 \leq(N^{-1}+2N^{-1/2}+2LN\delta)\|x\|.
\]
First choose $N$ large and then $\delta$ small.  For a local
sum $\sum_kx_kw^k$, apply this to $x_k-\tau_B(x_k)1$;
the remaining scalar error is at most
$\delta\sum_k|k|\,|\tau_B(x_k)|$.  Density proves the
statement for all $a$.  Finally, every unit vector in
$P\mathcal H$ defines a state taking value one at $w$.
Uniqueness and polarization give the compression formula.
\end{proof}

\begin{lemma}[The reduced interval pushout]
\label{rf:lem-pushout}
The canonical surjection
\[
 R*_{\mathrm C(\mathbb T)}^{\max}\mathrm C([0,1])\longrightarrow Q
\]
is an isomorphism.
\end{lemma}

\begin{proof}
For $0<\delta<1/4$, define continuous circle functions by
\[
 g^0_\delta(e^{2\pi it})=
 \begin{cases}t&0\leq t\leq1-\delta,\\
 (1-\delta)(1-t)/\delta&1-\delta\leq t\leq1,
 \end{cases}
\]
and
\[
 g^1_\delta(e^{2\pi it})=
 \begin{cases}1-(1-\delta)t/\delta&0\leq t\leq\delta,\\
 t&\delta\leq t\leq1.
 \end{cases}
\]
Their Haar laws are exactly normalized Lebesgue measure on
$[0,1-\delta]$ and $[\delta,1]$, respectively: the two linear
branches give the constant density
$1+\delta/(1-\delta)=1/(1-\delta)$ on the corresponding interval.
Each $g^s_\delta(w)$ is free from $F$.  Equality of the joint free
laws and faithfulness of the specified traces therefore give, for
every joint polynomial with finitely many coefficients in $F$,
\[
 \begin{split}
 \|p(F,g^0_\delta(w))\|
     &=\|p(F,(1-\delta)h)\|_Q,\\
 \|p(F,g^1_\delta(w))\|
     &=\|p(F,\delta+(1-\delta)h)\|_Q.
 \end{split}
\]
Both affine functions of $h$ converge to $h$ in norm.  Polynomial
continuity consequently gives
\begin{equation}\label{rf:eq-branch-norm}
 \|p(F,g^s_\delta(w))\|\longrightarrow\|p(F,h)\|_Q.
\end{equation}

Represent the full pushout on $\mathcal H$, and write $W,H$
for its circle and interval coordinates.  Set
$P_s=1_{\{s\}}(H)$ and $P=P_0+P_1=1_{\{1\}}(W)$.
By Lemma~\ref{rf:lem-endpoint}, $PaP=\omega(a)P$ for
$a\in R$.  The subspaces
$\mathcal H_s=\overline{RP_s\mathcal H}$ are orthogonal,
reduce $R$, and satisfy $P\mathcal H_s=P_s\mathcal H$.
Off $P\mathcal H$, the equation $e^{2\pi iH}=W$ uniquely
determines $H$ as a Borel argument of $W$.  On
$P_s\mathcal H$, it equals $s$.  Hence $\mathcal H_s$ also
reduces $H$.  Their orthogonal complement $\mathcal H_*$
reduces $R,H$ and has no $1$-eigenvectors for $W$.
The spectral theorem gives
$g^s_\delta(W)\to H$ strongly on $\mathcal H_s$, and either
branch converges strongly on $\mathcal H_*$.  Strong lower
semicontinuity of the norm, contractivity of the representation
of $R$, and \eqref{rf:eq-branch-norm} now give
$\|p(F,H)\|\leq\|p(F,h)\|_Q$ on each reducing subspace.
The canonical quotient gives the opposite inequality.
\end{proof}

\begin{corollary}[Strong one-turn logarithms]
\label{rf:cor-strong-logs}
Let $L$ have a faithful tracial state, let $F\subset L$, and
let $w\in L$ be Haar and free from $F$.  If positive
contractions $H_n\in L$ satisfy
$\|e^{2\pi iH_n}-w\|\to0$, then
\[
 \|p(F,H_n)\|\longrightarrow\|p(F,h)\|_Q
\]
for every joint polynomial, and all joint moments converge.
In particular, $h\mapsto(H_n)_n$ and the constant inclusion
of $F$ define a faithful homomorphism $\Theta:Q\to L_\infty$,
where $L_\infty=\ell^\infty(L)/c_0(L)$, and
$\Theta j=\iota_L|_R$.
\end{corollary}

\begin{proof}
In every norm ultrapower of every subsequence, $H=(H_n)_\omega$
has $e^{2\pi iH}=w$.  Lemma~\ref{rf:lem-pushout} gives
the upper bound for each joint polynomial.  Failure of that
bound along the original sequence would persist along a
subsequence and contradict this observation.

Put $w_n=e^{2\pi iH_n}$.  Its law converges to Haar measure.
The difference $H_n-g^0_\delta(w_n)$ is bounded by one and
is supported where $H_n\in[1-\delta,1]$.  The latter set
maps into a circle arc of Haar measure $\delta$.
Weak convergence of the circle laws therefore gives
$\limsup_n\|H_n-g^0_\delta(w_n)\|_2^2\leq\delta$.
For fixed $\delta$, $g^0_\delta(w_n)\to g^0_\delta(w)$ in
norm.  Letting $\delta\downarrow0$ proves convergence of
$H_n$ in $2$-norm to the Borel argument of $w$.  This argument
is uniform on bounded words and proves joint moment convergence.
Faithfulness of the trace on $Q$ then gives
\[
 \|p(F,h)\|
 =\lim_{k\to\infty}\tau_Q((p^*p)^k)^{1/(2k)}
 \leq\liminf_n\|p(F,H_n)\|.
\]
This supplies the lower bound and the asserted faithful map.
Every ultralimit trace on its image is the specified trace.
\end{proof}

\subsection{The actual two-logarithm comparison ideal}
\label{rf:subsec-two-log}

\begin{proposition}[The induced two-projection ideal]
\label{rf:prop-two-log}
Let $R$ be separable and unital, and let $w\in R$ have a unique
state $\omega$ taking value one at $w$.  Suppose
$T=\mathrm C^*(R,H_0,H_1)$, where $H_s$ are positive
contractions with $e^{2\pi iH_s}=w$.  Put
\[
 J=\operatorname{Ideal}_{\mathrm C^*(1,H_0,H_1)}(H_1-H_0),
 \qquad I=\operatorname{Ideal}_T(H_1-H_0).
\]
Then, with $\mathcal H_\omega$ the GNS space of $\omega$,
\begin{equation}\label{rf:eq-induced-ideal}
 I\cong\mathcal K(\mathcal H_\omega)\otimes J,
 \quad axb^*\longleftrightarrow
 |\pi_\omega(a)\Omega\rangle
 \langle\pi_\omega(b)\Omega|\otimes x.
\end{equation}
Moreover, $J$ is a quotient of
\[
 J_{\rm pair}=
 \{f\in\mathrm C([0,1],\mathrm M_2):
       f(0)=0,\ f(1)\text{ is diagonal}\}.
\]
Thus $I$ is separable, nuclear and type I, with nuclear
dimension at most one.  If $\dim\mathcal H_\omega=\infty$,
then $I$ is stable and has the corona factorisation property.
\end{proposition}

\begin{proof}
Represent $T$ faithfully on $\mathcal H$ and put
$P=1_{\{1\}}(w)$.  Each $H_s$ commutes with $w$; off
$P\mathcal H$ both are the same Borel argument, while on
$P\mathcal H$ they are projections.  Hence
$(w-1)J=J(w-1)=0$.  Uniqueness of $\omega$ gives
$PaP=\omega(a)P$, and therefore
\begin{equation}\label{rf:eq-sandwich}
 xay=\omega(a)xy\qquad(x,y\in J, a\in R).
\end{equation}
Let $\mathcal V=\overline{J\mathcal H}$ and
$\mathcal M=\overline{R\mathcal V}$.  The map
\[
 V:\mathcal H_\omega\otimes\mathcal V\to\mathcal M,
 \qquad \pi_\omega(a)\Omega\otimes\xi\mapsto a\xi
\]
is unitary, since its inner products are
$\omega(a^*b)\langle\xi,\eta\rangle$.  Here $R$ acts as
$\pi_\omega(R)\otimes1$, and $x\in J$ acts as
$|\Omega\rangle\langle\Omega|\otimes x|_{\mathcal V}$.
The representation of $J$ on $\mathcal V$ is faithful.
It follows that the closed span $K$ of $axb^*$ acts as the
right side of \eqref{rf:eq-induced-ideal} on $\mathcal M$
and as zero on $\mathcal M^\perp$.

We check invariance under the two logarithms.  The
$1$-eigenspace of $\pi_\omega(w)$ is exactly $\C\Omega$:
the endpoint compression gives
$1_{\{1\}}(\pi_\omega(w))\pi_\omega(a)\Omega
 =\omega(a)\Omega$, and these vectors are dense.
Also $P\mathcal M=\mathcal V$, and $\mathcal M$ reduces
each $H_s$ because its complement to $\mathcal V$ carries
the common argument of $w$, whereas $H_sJ\subset J$.
In the tensor coordinates above,
\[
 H_s=\operatorname{arg}_0(\pi_\omega(w))\otimes1
       +|\Omega\rangle\langle\Omega|\otimes m_s,
\]
where $m_s\in\mathcal M(J)$ is the multiplier induced by
$H_s$.  Both terms normalize $\mathcal K(\mathcal H_\omega)
\otimes J$.  Thus $K$ is an ideal of $T$ containing $J$,
and equals $I$.

The element $w$ is central in $D=\mathrm C^*(1,H_0,H_1)$.
Since $J$ annihilates $\overline{(w-1)D}$, its intersection
with that ideal is zero, so $J$ embeds as an ideal in the
fiber at $w=1$.  This fiber is generated by two projections.
The universal algebra of two projections is
\[
 \{f\in\mathrm C([0,1],\mathrm M_2):
             f(0),f(1)\text{ are diagonal}\}.
\]
See \cite[Theorem~1.3]{RaeburnSinclair}. With the parameter chosen
so that $p=q$ at zero, the standard generators satisfy
$(p-q)^2(t)=t1$. Thus the ideal generated by $p-q$ consists
exactly of the functions vanishing at zero, which is $J_{\rm pair}$.
Its one-dimensional subhomogeneous description, passage to
quotients and matrix stability give the nuclear dimension
bound.  Infinite-dimensional $\mathcal H_\omega$ gives
stability; finite nuclear dimension gives the corona
factorisation property \cite{RF-RobertCFP}.
\end{proof}

\begin{lemma}[Proper comparison]\label{rf:lem-proper}
Let $A$ be separable and unital, and let $I$ be separable,
stable, nuclear and have the corona factorisation property.
For unital full homomorphisms $\varphi_0,\varphi_1:A\to
\mathcal M(I)$ forming a Cuntz pair, assume
$[\varphi_0,\varphi_1]=0$ in $\KK(A,I)$.
Then there are unitaries $v_m\in1+I$ such that
$\Ad(v_m)\varphi_0(a)\to\varphi_1(a)$ for every $a\in A$.
\end{lemma}

\begin{proof}
Every compression $a\mapsto b^*\varphi_s(a)b$, $b\in I$,
takes values in the nuclear algebra $I$, so is nuclear.
Both maps are weakly nuclear.  The same argument for
$\mathrm C([0,1])\otimes I$ makes every ordinary Cuntz-pair
homotopy weakly nuclear: this homotopy ideal is nuclear as well.
Thus both the pairs and their homotopy relations agree, identifying
$\KK_{\rm nuc}(A,I)=\KK(A,I)$.
Fullness and the corona factorisation property give unital
nuclear absorption \cite[Theorem~1.11]{RF-HW}.
The corona relative commutant is $\mathrm K_1$-injective by
\cite[Corollary~2.9]{Szabo}.  Thus
\cite[Theorem~1.12(i)]{RF-HW} gives asymptotic conjugacy by
unitaries in the unitization of $I$.  Removing their scalar
parts puts them in $1+I$ without changing conjugation.
\end{proof}

\begin{proposition}[Reparameterization descent]
\label{rf:prop-descent}
Let $F$ be separable, unital and faithfully tracial, with
$\dim L^2(F,\tau_F)=\infty$, and suppose $Q=F\fr
\mathrm C([0,1])$ is simple.  Let $L$ be unital with a
specified faithful trace, and suppose the native
$R=F\fr\mathrm C(\mathbb T)$ embeds unitally and trace
preservingly in $L$.  If positive contractions $H_n\in L$ satisfy
$e^{2\pi iH_n}\to w$ in norm, there are a faithful unital
$f:Q\to L$ and unitaries $u_n\in L$ such that
\[
 \Ad(u_n)fj(x)\longrightarrow x\qquad(x\in R).
\]
The map $f$ preserves the specified trace.
\end{proposition}

\begin{proof}
Corollary~\ref{rf:cor-strong-logs} gives
$\Theta:Q\to L_\infty$ with $\Theta j=\iota_L|_R$.
For any diverging $r:\N\to\N$, form
$T=\mathrm C^*(R,H_0,H_1)\subset L_\infty$, with
$H_0=(H_n)_n$ and $H_1=(H_{r(n)})_n$.
Their exponentials equal the constant $w$.  By
Proposition~\ref{rf:prop-two-log}, the difference ideal
$I$ is stable and nuclear with the corona factorisation
property: the GNS space contains $L^2(F,\tau_F)$ isometrically.
If $I=0$ the two maps coincide.  Otherwise their multiplier
maps form a unital Cuntz pair and are full, since $Q$ is
simple.  They agree exactly on $F$.

The inclusion $F\to Q$ is a $\KK$-equivalence.  Indeed the
full free product with the interval retracts onto $F$ by
evaluation at zero, and $h\mapsto th$ gives the required
homotopy.  The full-to-reduced quotient is a
$\KK$-equivalence by
\cite[Proposition~2.15 and Theorem~3.1]{FG}: the faithful
expectations identify the vertex-reduced and ordinary reduced
products.
Restriction to $F$ therefore annihilates the actual
$\KK$-pair injectively.  Lemma~\ref{rf:lem-proper}
gives proper approximate unitary comparison of
$\Theta$ and $r^*\Theta$.

These multiplier estimates are estimates in $T$ itself:
their errors belong to $I$, on which the multiplier action
is isometric.  Essentiality of $I$ is unnecessary.
Separability and diagonal selection improve approximate
unitary equivalence in $L_\infty$ to unitary equivalence.
Explicitly, lift successively chosen implementing unitaries
to coordinate unitaries in $L$, and let the finite packet
index increase sufficiently slowly with the coordinate.
Such coordinate lifts exist by polar decomposition of an
asymptotically unitary representative in $\ell^\infty(L)$.
No homomorphic or completely positive lift of $\Theta$
is used.  Intertwining by reparameterization
\cite[Proposition~9.2]{CGSTW} now gives
$\Theta=\Ad(U)\iota_Lf$ for a homomorphism $f:Q\to L$.
It is faithful and unital because $\Theta$ is.
Lifting $U$ to coordinate unitaries proves the displayed
limit.  Applying any ultralimit of the specified trace to
this equality and using joint moment convergence proves
trace preservation.
\end{proof}

\begin{corollary}[Two logarithms in a sequence algebra]
\label{rf:two-log-comparison}
Let $F$ be separable, unital and faithfully tracial, with
$\dim L^2(F)=\infty$, and suppose $Q=F\fr\mathrm C([0,1])$ is
simple.  Put $q=e^{2\pi ih}$ and $R=\mathrm C^*(F,q)\subset Q$.
For a unital algebra $L$, let $\rho_0,\rho_1:Q\to L_\infty$ be
faithful unital homomorphisms agreeing on $R$.  Then there is
$V\in\mathrm U(L_\infty)$ with $\Ad(V)\rho_0=\rho_1$.
\end{corollary}
\begin{proof}
Identify $R$ with its common faithful image and form the separable
algebra generated by this image and $\rho_0(h),\rho_1(h)$.
Lemma~\ref{rf:lem-endpoint} and
Proposition~\ref{rf:prop-two-log} give the difference ideal
$I\cong\mathcal K(\mathcal H_\omega)\otimes J$.  The GNS space
contains $L^2(F)$, so $I$ is stable, nuclear and has the corona
factorisation property.  If $I=0$ the maps coincide.  Otherwise the
two multiplier maps are full because their source $Q$ is simple
and their common unit is the multiplier unit.  They agree on $F$.
As in Proposition~\ref{rf:prop-descent}, $F\hookrightarrow Q$ is
a $\KK$-equivalence, so their actual Cuntz-pair class is zero.
Lemma~\ref{rf:lem-proper} gives proper approximate conjugacies.
Their errors lie in $I$, where the multiplier action is isometric,
so these are norm estimates in $L_\infty$.  Lifting the implementing
unitaries to coordinate unitaries and diagonalizing on a countable
dense subset of $Q$ gives the asserted single unitary $V$.
\end{proof}

\begin{corollary}[Actual marked interval models]
\label{rf:actual-model}
Let $L$ be an admissible target containing a uniform interval
coordinate $k$.  Let $F\subset L$ be separable and unital with
$\dim L^2(F)=\infty$, and let $w\in L$ be Haar and free from $F$.
Suppose $Q=F\fr\mathrm C([0,1])$ is simple and the two embeddings
of $\mathrm C(\mathbb T)$ with coordinate images $w$ and
$e^{2\pi ik}$ have the same full invariant.
Then there are a trace-preserving embedding $f:Q\to L$ and
unitaries $u_n\in L$ such that
\[
 \Ad(u_n)f(a)\longrightarrow a\quad(a\in F),\qquad
 \Ad(u_n)f(e^{2\pi ih})\longrightarrow w.
\]
In particular $f|_F$ has the full invariant of its native inclusion.
\end{corollary}
\begin{proof}
Theorem~\ref{thm:embedding}, applied to the nuclear source
$\mathrm C(\mathbb T)$, supplies unitaries $z_n$ with
$z_ne^{2\pi ik}z_n^*\to w$.  The positive uniform coordinates
$H_n=z_nkz_n^*$ therefore satisfy the hypotheses of
Proposition~\ref{rf:prop-descent}.  Its conclusion gives the
embedding and unitaries.  Approximate unitary equivalence on $F$
retains ordinary and coefficient $\mathrm K$-theory, traces and
the continuous algebraic-unitary invariant.
\end{proof}

\subsection{Lifting circle automorphisms to actual interval automorphisms}
\label{rf:subsec-circle-auto}

\begin{theorem}[Relative interval lifting]
\label{rf:thm-circle-lift}
\label{rf:circle-lift}
Let $F$ be separable, unital and faithfully tracial, with
$\dim L^2(F,\tau_F)=\infty$.  Suppose
$Q=F\fr\mathrm C([0,1])$ is an admissible target, and put
$q=e^{2\pi ih}$ and $R=\mathrm C^*(F,q)\subset Q$.
Let $\gamma\in\operatorname{Aut}(R)$ preserve the specified
trace and satisfy $\gamma(F)=F$.
Assume that the three embeddings of $\mathrm C(\mathbb T)$
into $Q$ with coordinate images $q,\gamma(q),\gamma^{-1}(q)$
induce the same totally compatible triple.
Then there are trace-preserving automorphisms
$\alpha_n\in\operatorname{Aut}(Q)$ such that
\[
 \alpha_n(x)\longrightarrow\gamma(x)\qquad(x\in R).
\]
They may be chosen as inner conjugates of one fixed
automorphism $\alpha$.  For every $n$, the restriction
$\alpha_n|_F$ is approximately unitarily equivalent to
$\gamma|_F$, and therefore induces exactly its full
compatible invariant on $F$.
No convergence on the interval coordinate $h$ is asserted.
\end{theorem}

\begin{proof}
Nuclear-source uniqueness for $\mathrm C(\mathbb T)$,
Theorem~\ref{thm:embedding}, gives positive uniform
contractions $H_n$ with
$e^{2\pi iH_n}\to\gamma(q)$: conjugate the native $h$
by unitaries comparing $q$ with $\gamma(q)$.
The coefficient algebra is now $\gamma(F)$, and
$\gamma(q)$ is Haar and free from it.  Applying
Proposition~\ref{rf:prop-descent} with the identification
of $F$ supplied by $\gamma$ gives a fixed trace-preserving
embedding $f:Q\to Q$ and unitaries $u_n\in Q$ with
$f_n=\Ad(u_n)f$ satisfying $f_n|_R\to\gamma$.
Applying the same construction to $\gamma^{-1}$ gives a
fixed embedding $g$ and $g_n=\Ad(v_n)g$ with
$g_n|_R\to\gamma^{-1}$.

Contractivity and the two limits imply
$f_ng_n|_R\to\id_R$ and $g_nf_n|_R\to\id_R$.
For example, for $x\in R$,
\[
 \|f_ng_n(x)-x\|
 \leq\|g_n(x)-\gamma^{-1}(x)\|
      +\|f_n(\gamma^{-1}(x))-x\|\longrightarrow0.
\]
Define faithful unital homomorphisms
\[
 \rho_0(x)=(f_ng_n(x))_n,\qquad \rho_1(x)=\iota_Q(x)
 \quad(x\in Q).
\]
The preceding convergence gives $\rho_0|_R=\rho_1|_R$.
Corollary~\ref{rf:two-log-comparison}, with $L=Q$, supplies
$V\in\mathrm U(Q_\infty)$ such that $\Ad(V)\rho_0=\rho_1$.
Lift $V$ to coordinate unitaries $V_n\in Q$.  Since
\[
 f_ng_n=\Ad(u_nf(v_n))\,fg,
\]
we have $\Ad(V_nu_nf(v_n))fg(x)\to x$ for every $x\in Q$.
Thus the fixed composition $fg$ is approximately unitarily
equivalent to $\id_Q$.  The same argument gives
$gf\au\id_Q$. The two-sided approximate intertwining argument
\cite[Proposition~9.7]{CGSTW} gives an isomorphism
$\alpha:Q\to Q$ approximately unitarily equivalent to $f$.
For each successive finite packet in $R$, choose $n$ making
$f_n$ close to $\gamma$, and then an inner conjugate of
$\alpha$ close to $f_n$ on that packet.  A diagonal choice
gives the asserted automorphisms.  Their trace preservation
follows either from approximate unitary equivalence to $f$
or from uniqueness of the trace on $Q$.
Finally, $f|_F\au\gamma|_F$ by the original sequence $u_n$,
and $\alpha\au f$.  Approximate unitary equivalence is
unchanged by an inner conjugacy, so each $\alpha_n|_F$
has this same approximate-unitary class.  Its ordinary and
coefficient $\mathrm K$-maps and its continuous algebraic-unitary
map consequently agree with those of $\gamma|_F$; its trace
map agrees by trace preservation.
\end{proof}

\begin{remark}[Finite packets and moving coefficients]
\label{rf:rem-moving-packets}
The theorem applies separately at every coordinate.  Thus,
for a sequence $\gamma_n$ satisfying its hypotheses, finite
packets $G_n\subset R$, and positive tolerances
$\varepsilon_n$, one may choose actual automorphisms
$\alpha_n$ with
$\|\alpha_n(x)-\gamma_n(x)\|<\varepsilon_n$ for $x\in G_n$.
The packets may include moving coefficients or multipliers
belonging to $R$.  This conclusion supplies precisely finite
packet control on $R$; it adds no control on $h$.
The full marked hypothesis in the theorem must be verified.
Vanishing of the ordinary $\mathrm K_1$ class alone does
not replace it.  An infinite free interval tail in $F$
supplies the infinite GNS space and simplicity required
above; the admissible-target and marked uniqueness
hypotheses remain part of the assertion.
\end{remark}


\subsection{Comparison of finitely many common free circles}
\label{af:multi-logs}

We first extend the logarithm comparison in the preceding one-circle construction
to finitely many independent circle coordinates.  The independence is
essential: it separates their endpoint defects.  Throughout this
subsection, $F$ is separable, unital and faithfully tracial, and
\[
 R=F\fr\mathop{\fr}_{i=1}^r
       (\mathrm C(\mathbb T_{w_i}),m),\qquad
 Q=F\fr\mathop{\fr}_{i=1}^r
       (\mathrm C([0,1]_{h_i}),\lambda),
 \qquad 1\leq r<\infty.
\]
The native inclusion $j:R\to Q$ fixes $F$ and sends
$w_i$ to $e^{2\pi ih_i}$.  Put
\[
 \omega_i=\ev_1\circ\E_{\mathbb T_{w_i}},\qquad
 (\pi_i,\mathcal H_i,\Omega_i)=\operatorname{GNS}(\omega_i).
\]
Applying Lemma~\ref{rf:lem-endpoint} with all the other free
factors included in the coefficient algebra shows that $\omega_i$
is the unique state of $R$ taking value one at $w_i$.

\begin{lemma}[Separation of endpoint representations]
\label{af:endpoint-mixing}
If $i\ne j$, then
\[
 \pi_i(w_j)^n\longrightarrow0
 \quad\hbox{weakly as }|n|\longrightarrow\infty.
\]
In particular $\pi_i(w_j)$ has no $1$-eigenvectors.  In any
representation of $R$, put $P_i=1_{\{1\}}(w_i)$ and
$\mathcal M_i=\overline{RP_i\mathcal H}$.  The spaces
$\mathcal M_i$ are pairwise orthogonal and reduce $R$.
\end{lemma}

\begin{proof}
Let $a,b$ belong to the dense algebra of finite reduced words,
using Laurent polynomials in the circle coordinates.  For sufficiently
large $|n|$, reduction of $a^*w_j^nb$ leaves a nonzero power of
$w_j$: the fixed words can cancel only finitely many powers at
its two boundaries.  Since $j\ne i$, its expectation onto
$\mathrm C(\mathbb T_{w_i})$ is zero.  Thus
\[
 \langle\pi_i(w_j)^n\pi_i(b)\Omega_i,
                  \pi_i(a)\Omega_i\rangle
 =\omega_i(a^*w_j^nb)=0
\]
eventually.  Density and boundedness give weak convergence on
$\mathcal H_i$, and therefore absence of $1$-eigenvectors.

The endpoint compression formula gives
$P_i aP_i=\omega_i(a)P_i$.  Consequently the representation
of $R$ on $\mathcal M_i$ is an amplification of $\pi_i$,
under the unitary
\[
 \mathcal H_i\otimes P_i\mathcal H\longrightarrow\mathcal M_i,
 \qquad \pi_i(a)\Omega_i\otimes\xi\longmapsto a\xi.
\]
The first assertion implies $P_j\mathcal M_i=0$ for $j\ne i$.
Since $\mathcal M_i$ reduces $R$, it is then orthogonal to
$\overline{RP_j\mathcal H}=\mathcal M_j$.
\end{proof}

\begin{proposition}[The finite logarithm difference ideal]
\label{af:multi-ideal}
Let
\[
 \mathcal T=\mathrm C^*(R,H_i^s:1\leq i\leq r, s=0,1),
 \qquad 0\leq H_i^s\leq1,\qquad e^{2\pi iH_i^s}=w_i.
\]
Write
\[
 J_i=\operatorname{Ideal}_{\mathrm C^*(1,H_i^0,H_i^1)}
                         (H_i^1-H_i^0),\qquad
 I=\operatorname{Ideal}_{\mathcal T}
                         (H_i^1-H_i^0:1\leq i\leq r).
\]
Then
\begin{equation}\label{af:eq-multi-ideal}
 I=\bigoplus_{i=1}^r I_i,
 \qquad
 I_i\cong\mathcal K(\mathcal H_i)\otimes J_i.
\end{equation}
Each $J_i$ is a quotient of
\[
 J_{\rm pair}=
 \{f\in\mathrm C([0,1],\mathrm M_2):
                f(0)=0,\ f(1)\text{ is diagonal}\}.
\]
Thus $I$ is separable, nuclear and type I, of nuclear dimension
at most one.  If $\dim L^2(F,\tau_F)=\infty$, it is stable
and has the corona factorisation property.
\end{proposition}

\begin{proof}
Represent $\mathcal T$ faithfully on $\mathcal H$.  For a fixed
$i$, Proposition~\ref{rf:prop-two-log}, applied first to
$\mathrm C^*(R,H_i^0,H_i^1)$, identifies the closed span
\[
 K_i=\overline{\operatorname{span}}\{axb^*:a,b\in R, x\in J_i\}
       \cong\mathcal K(\mathcal H_i)\otimes J_i.
\]
Its support space is
$\mathcal N_i=\overline{RJ_i\mathcal H}\subset\mathcal M_i$.
The proof of that proposition shows that $\mathcal N_i$ reduces
$R,H_i^0,H_i^1$, and that these two logarithms normalize $K_i$.
We verify the additional invariance required here.

For $j\ne i$, Lemma~\ref{af:endpoint-mixing} shows that
$w_j|_{\mathcal N_i}$ has no $1$-eigenvectors.  Both $H_j^s$
therefore restrict to the same Borel argument of $w_j$ on
$\mathcal N_i$.  More explicitly, this subspace reduces $w_j$,
its spectral projection at one is zero there, and off that
projection the relation $e^{2\pi iH_j^s}=w_j$ determines
$H_j^s$ uniquely.  Thus $\mathcal N_i$ reduces both logarithms,
and in the tensor representation they act as
\[
 \operatorname{arg}_0(\pi_i(w_j))\otimes1.
\]
These bounded operators normalize
$\mathcal K(\mathcal H_i)\otimes J_i$.
On $\mathcal N_i^\perp$ the ideal $K_i$ is zero, so this
proves that $K_i$ is an ideal of the whole algebra $\mathcal T$.
It consequently equals the ideal $I_i$ generated there by
$H_i^1-H_i^0$.

The spaces $\mathcal N_i\subset\mathcal M_i$ are pairwise
orthogonal.  Hence $I_iI_j=0$ for $i\ne j$, and the finite
sum of the $I_i$ is closed and equals $I$.  The description of
$J_i$ and its nuclear dimension are those of
Proposition~\ref{rf:prop-two-log}.  Moreover
$L^2(F,\tau_F)$ embeds isometrically in every $\mathcal H_i$,
since $\omega_i|_F=\tau_F$.  Infinite dimensionality of this
space therefore gives stability of every nonzero $I_i$.
The finite direct sum is stable, and its finite nuclear dimension
gives the corona factorisation property \cite{RF-RobertCFP}.
\end{proof}

\begin{proposition}[Simultaneous strong one-turn logarithms]
\label{af:strong-logs}
The canonical map from the iterated full pushout
\begin{equation}\label{af:eq-full-pushout}
 R*_{\mathrm C(\mathbb T_{w_1})}^{\max}\mathrm C([0,1]_{h_1})
  *_{\mathrm C(\mathbb T_{w_2})}^{\max}\cdots
  *_{\mathrm C(\mathbb T_{w_r})}^{\max}\mathrm C([0,1]_{h_r})
 \longrightarrow Q
\end{equation}
is an isomorphism, where every amalgamation sends
$w_i$ to $e^{2\pi ih_i}$.

Let $L$ be unital with a faithful specified trace, and suppose
the native $R$ embeds unitally and trace preservingly in $L$.
If positive contractions $H_{i,n}\in L$ satisfy
\[
 \|e^{2\pi iH_{i,n}}-w_i\|\longrightarrow0
 \qquad(1\leq i\leq r),
\]
then, for every joint polynomial with coefficients in $F$,
\[
 \|p(F,H_{1,n},\ldots,H_{r,n})\|
       \longrightarrow\|p(F,h_1,\ldots,h_r)\|_Q,
\]
and all joint moments converge to those in $Q$.  In particular
the constant inclusion of $F$ and the assignments
$h_i\mapsto(H_{i,n})_n$ define a faithful unital homomorphism
$\Theta:Q\to L_\infty$ with $\Theta j=\iota_L|_R$.
\end{proposition}

\begin{proof}
We establish the full pushout assertion first.  At the first
step, Lemma~\ref{rf:lem-pushout} applies with coefficient
algebra $F\fr\mathop{\fr}_{i=2}^r\mathrm C(\mathbb T_{w_i})$.
Its conclusion replaces just the first circle by its uniform
interval.  Repeat the argument, putting the intervals already
obtained and all the remaining circles into the coefficient
algebra at each step.  These are reduced products with faithful
specified traces.  After $r$ steps the resulting reduced product
is $Q$, proving \eqref{af:eq-full-pushout}.

In any norm ultrapower of any subsequence, the classes
$H_i=(H_{i,n})_\omega$ satisfy $e^{2\pi iH_i}=w_i$.
The full pushout just proved gives the upper bound for every
polynomial norm.  A failure of this upper bound on the original
sequence would persist on a subsequence and contradict that
conclusion.

For the moments, use the continuous argument functions
$g^0_\delta$ from the proof of Lemma~\ref{rf:lem-pushout}.
As in Corollary~\ref{rf:cor-strong-logs}, the circle laws
converge to Haar measure and
\[
 \limsup_n
 \|H_{i,n}-g^0_\delta(e^{2\pi iH_{i,n}})\|_2^2\leq\delta.
\]
For fixed $\delta$ the continuous functions of the exponentials
converge in norm to $g^0_\delta(w_i)$.  Letting $\delta$ tend
to zero shows that $H_{i,n}$ converges in $2$-norm to the Borel
argument of $w_i$ in the tracial von Neumann algebra of $L$.
There are only finitely many coordinates.  Telescoping bounded
words therefore gives simultaneous joint moment convergence.
The Borel arguments have uniform laws and remain free from each
other and from $F$, so these are exactly the moments of $Q$.
Faithfulness of its trace gives the lower norm bound through
\[
 \|p(F,h_1,\ldots,h_r)\|
  =\lim_{k\to\infty}\tau_Q((p^*p)^k)^{1/(2k)}
  \leq\liminf_n\|p(F,H_{1,n},\ldots,H_{r,n})\|.
\]
This proves the norm assertion and the faithful sequence-algebra
map.  Every ultralimit of the specified trace induces the native
trace on its image.
\end{proof}

\begin{corollary}[Finite common-circle comparison]
\label{af:multi-comparison}
Suppose $\dim L^2(F,\tau_F)=\infty$ and $Q$ is simple.  Let
$L$ be unital, and let $\rho_0,\rho_1:Q\to L_\infty$ be
faithful unital homomorphisms agreeing on $j(R)$.  Then there
is $V\in\mathrm U(L_\infty)$ with
\[
 \Ad(V)\rho_0=\rho_1.
\]
\end{corollary}

\begin{proof}
Identify $R$ with its common faithful image, and form the
separable algebra generated by that image and the finitely many
pairs $\rho_0(h_i),\rho_1(h_i)$.  Proposition~\ref{af:multi-ideal}
gives its difference ideal $I$, which is stable and nuclear and
has the corona factorisation property.  If $I=0$, the maps
coincide.  Otherwise their multiplier maps are unital and full:
the simple source $Q$ makes the image of every nonzero positive
element generate the multiplier unit.  Their differences lie
in $I$, so they form a Cuntz pair.

The inclusion $F\to Q$ is a $\KK$-equivalence.  Indeed, the
full free product of $F$ with finitely many intervals retracts
onto $F$ by evaluation at zero, and simultaneous contraction of
the interval coordinates gives a homotopy inverse.  Its
full-to-reduced quotient is a $\KK$-equivalence by
\cite[Proposition~2.15 and Theorem~3.1]{FG}, using the faithful
expectations to identify the reduced models.  Since the two multiplier maps agree
exactly on $F$, restriction annihilates their actual Cuntz-pair
class in $\KK(Q,I)$ injectively.  That class is therefore zero.
Lemma~\ref{rf:lem-proper} gives proper approximate unitary
comparison by unitaries in $1+I$.

These are norm estimates in the generated algebra, not merely
in its multiplier image: their errors belong to $I$, where
the multiplier action is isometric.  Lift the implementing
unitaries to coordinate unitaries in $L$, and diagonalize on
a countable dense subset of $Q$.  The resulting single unitary
in $L_\infty$ gives the asserted equality.  Essentiality of
$I$ is not needed.
\end{proof}

\begin{corollary}[Actual finite interval models]
\label{af:actual-model}
Suppose $\dim L^2(F,\tau_F)=\infty$ and $Q$ is simple.
Let $L$ be an admissible target containing the native $R$
unitally and trace preservingly.  Suppose that $L$ contains a
uniform interval coordinate $k$ such that, for every $i$, the
circle embeddings with coordinate images $w_i$ and
$e^{2\pi ik}$ induce the same totally compatible triple.
Then there are a trace-preserving embedding $f:Q\to L$ and
unitaries $u_n\in L$ such that
\[
 \Ad(u_n)f(a)\longrightarrow a\quad(a\in F),\qquad
 \Ad(u_n)f(e^{2\pi ih_i})\longrightarrow w_i
                    \quad(1\leq i\leq r).
\]
In particular, the restriction $f|_F$ is approximately unitarily
equivalent to its native inclusion and has exactly its full
classification invariant.
\end{corollary}

\begin{proof}
Nuclear-source uniqueness for $\mathrm C(\mathbb T)$,
Theorem~\ref{thm:embedding}, gives unitaries $z_{i,n}\in L$
with
$z_{i,n}e^{2\pi ik}z_{i,n}^*\to w_i$.  Put
$H_{i,n}=z_{i,n}kz_{i,n}^*$.  Proposition~\ref{af:strong-logs}
gives a faithful unital map $\Theta:Q\to L_\infty$ agreeing
with the native inclusion on $R$.

For every diverging map $s:\N\to\N$, the reparameterization
$s^*\Theta$ also agrees with $\Theta$ on $R$.
Corollary~\ref{af:multi-comparison} therefore gives unitary
equivalence of these two maps.  Intertwining by reparameterization
\cite[Proposition~9.2]{CGSTW}, exactly as in
Proposition~\ref{rf:prop-descent}, yields a homomorphism
$f:Q\to L$ and a unitary $U\in L_\infty$ with
\[
 \Theta=\Ad(U)\iota_Lf.
\]
The map $f$ is faithful and unital because $\Theta$ is.
Lifting $U$ to coordinate unitaries gives the displayed limits.
Joint moment convergence in
Proposition~\ref{af:strong-logs}, applied to any ultralimit
of the specified trace, gives trace preservation.
Approximate unitary equivalence on $F$ preserves ordinary and
coefficient $\mathrm K$-theory, traces, and the continuous
algebraic-unitary invariant.  Thus it preserves the complete
classification marking on that coefficient algebra.
\end{proof}

\begin{remark}
The vanishing used in Corollary~\ref{af:multi-comparison}
concerns the actual $\KK(Q,I)$ class of a specified Cuntz pair;
it follows from exact agreement on $F$ and the $\KK$-equivalence
$F\to Q$.  The marking assertion in
Corollary~\ref{af:actual-model} instead follows from
approximate unitary equivalence of the coefficient restrictions.
It does not identify arbitrary $\KK$ classes from their induced
$\mathrm K$-group maps when $F$ is nonnuclear.

The finite-family construction below requires only
one-circle automorphism lifting.  At each such lift, all other
whole interval factors belong to the coefficient algebra in
Theorem~\ref{rf:circle-lift}.  Its full marked hypothesis
is checked before applying the theorem; the resulting coefficient
restriction retains its global approximate-unitary class.
\end{remark}

\section{Finite-dimensional relative reservoirs}
\label{sec:fd-relative}

All traces in this section are faithful.  All conditional expectations
onto finite-dimensional subalgebras preserve the specified trace, and
all amalgamated free products use these expectations.  We write
$I=(\mathrm C([0,1]),\int_0^1\cdot\,dt)$ and
$D_\infty=\mathop{*}_{j\geq1}I_j$.  A uniform interval coordinate
can equivalently be replaced by a standard semicircular coordinate:
the semicircle distribution function is a homeomorphism from
$[-2,2]$ onto $[0,1]$.  We use this identification when calculating
operator-valued moments.  In particular, all Gaussian calculations
below concern actual elements of reduced $\mathrm C^*$-algebras.

If
\[
 F=\bigoplus_{s=1}^r\mathrm M_{d_s}\subset H,
\]
we fix matrix units $e^{(s)}_{ij}$ and put
\[
 p_s=\sum_{i=1}^{d_s}e^{(s)}_{ii},\qquad
 q_s=e^{(s)}_{11},\qquad
 \alpha_s=\tau(p_s),\qquad \lambda_s=\tau(q_s)=\alpha_s/d_s.
\]
The $p_s$ need not be central in $H$.  The trace-preserving
expectation onto $F$ is
\begin{equation}\label{eq:fd-expectation}
 \E_F(b)=\sum_{s=1}^r\frac{d_s}{\alpha_s}
       \sum_{i,l=1}^{d_s}\tau(e^{(s)}_{li}b)e^{(s)}_{il}.
\end{equation}
A label-preserving map between two such inclusions preserves the
chosen matrix units as well as the traces; consequently it preserves
all the numbers $\alpha_s$ and $\lambda_s$.

\begin{lemma}[An actual finite-dimensional relative tag]
\label{lem:fd-tag-embedding}
Let $F\subset H$ be a unital finite-dimensional subalgebra of a
unital faithfully tracial $\mathrm C^*$-algebra.  There is a unital
trace-preserving embedding
\[
 H*_F(F\otimes I)\longrightarrow H*I
\]
which fixes $H$.  Its interval coordinate commutes with $F$.
The construction is natural for label-preserving trace-preserving
embeddings and uses only the displayed fresh interval.  Additional
external scalar-free factors are therefore unchanged.
\end{lemma}

\begin{proof}
Let $s$ be the standard semicircular generator of the fresh interval,
free from $H$, and set
\[
 Y=\sum_{s_0=1}^r\alpha_{s_0}^{-1/2}
       \sum_{i,j=1}^{d_{s_0}}
       e^{(s_0)}_{ij}s e^{(s_0)}_{ji}.
\]
The index $s_0$ denotes a summand and is distinct from the
semicircular variable $s$.  Matrix multiplication gives $[Y,F]=0$.
With $\E_H:H*I\to H$ the canonical expectation, the second
operator-valued cumulant is
\[
 \E_H(YbY)=\sum_{s_0=1}^r\frac{d_{s_0}}{\alpha_{s_0}}
       \sum_{i,l=1}^{d_{s_0}}
       \tau(e^{(s_0)}_{li}b)e^{(s_0)}_{il}=\E_F(b).
\]
For example, the scalar contraction of two terms from summands
$s_0,t_0$ contains
$\tau(e^{(s_0)}_{ji}b e^{(t_0)}_{lr})$; cyclicity of the trace
and the outside matrix product force $s_0=t_0$ and give the
factor $d_{s_0}$.  Multilinearity of operator-valued cumulants
shows that every cumulant of $Y$ of order different from two
vanishes.  Thus $Y$ is $H$-valued semicircular with covariance
$\E_F$.

The central standard semicircle in $F\otimes I$, regarded in
$H*_F(F\otimes I)$, has exactly this $H$-valued Wick formula.
In particular its scalar law is the standard semicircle law,
since $\E_F(1)=1$.  Equality of all scalar mixed moments, together
with faithfulness of both reduced traces, identifies the
generated $\mathrm C^*$-algebras isometrically.  Applying the
semicircle distribution function gives the claimed uniform
interval.  The formula for $Y$ commutes with every
label-preserving embedding, proving naturality.  The assertion
about external factors follows by reduced free-product
functoriality.
\end{proof}

\begin{lemma}[Reassociation and decomposition into block tags]
\label{lem:fd-reassociation}
Let $F_1,\ldots,F_m$ be unital finite-dimensional subalgebras of
a faithfully tracial algebra $R$.  Relative interval extensions
over these subalgebras can be performed in any order, with
canonical trace-preserving isomorphisms fixing $R$ and all the
labelled interval coordinates.

For a single algebra $F=\bigoplus_s\mathrm M_{d_s}$, its
$F$-central interval extension is also the successive extension
by one block-supported interval for each $s$.  The $s$th
extension is over
\[
 F_s^+=\mathrm M_{d_s}\oplus\mathbb C(1-p_s)
\]
and adjoins
\[
 G_s=(\mathrm M_{d_s}\otimes I_s)
                      \oplus\mathbb C(1-p_s).
\]
Zero complementary summands are omitted.  The added interval
is supported on $p_s$; every other block tag belongs to the
remaining coefficient algebra when this tag is placed last.
\end{lemma}

\begin{proof}
For the first assertion replace the intervals by their
semicircular coordinates.  Over $R$ the joint Wick formula
has covariance maps $E_{F_i}^R$ and zero mixed cumulants.
To check this in an iterated extension, use at each step the
expectation identity
\[
 E_{F_i}^{R'}=E_{F_i}^R \E_R^{R'},
\]
where $R'$ is the algebra already constructed.  This identity
follows from trace preservation and uniqueness of the
expectation onto $F_i$.  Inserting it in the Wick recursion
gives the same $R$-valued formula in every order.  Equivalently,
the recursion is the Fock formula for the direct sum of the
covariance correspondences.  Faithfulness of the reduced
traces upgrades equality of moments to the asserted
$\mathrm C^*$-isomorphisms.

For the second assertion let $t$ be the $F$-central standard
semicircle and put $t_s=p_st$.  Its block variables have
covariances
\[
 \eta_s(b)=p_s\E_F(b)p_s,
\]
zero mixed cumulants, and vanishing higher cumulants.  The
supported semicircle in $G_s$ has precisely covariance
$\eta_s$; the complementary scalar summand contributes no
variable.  The preceding Wick recursion therefore identifies
the successive construction with the original extension.
The block tags generate the original factor because
$t=\sum_st_s$ and $t_s=p_st$.  Transforming back to uniform
intervals uses functional calculus separately in the unital
corners $p_sRp_s$.
\end{proof}

\begin{lemma}[The native relative-extension class]
\label{lem:fd-extension-kk}
Let $R$ be separable and faithfully tracial.  The native inclusion
\[
 R\longrightarrow R*_F(F\otimes I)
\]
is a $\mathrm{KK}$-equivalence for every unital
finite-dimensional $F\subset R$.  The same is true for a
block-supported extension and for any finite iteration of
these extensions.  Consequently the native inclusion, rather
than just an unmarked abstract absorption isomorphism, gives
the canonical identification of total $\mathrm K$-theory.
\end{lemma}

\begin{proof}
First use the full amalgamated free product.  Evaluation of
the interval coordinate at zero gives a retraction onto $R$.
The homotopy $h\mapsto th$, $0\leq t\leq1$, in the uniform
coordinate, fixes $R$ and the amalgamated $F$ and joins the
retracted map to the identity.  Thus the full native
inclusion is a homotopy equivalence.  For a supported
extension the identical homotopy is used in the supported
matrix block and is the identity in the complementary
constant block.

The full-to-reduced comparison is a $\mathrm{KK}$-equivalence
by \cite[Proposition~2.15 and Theorem~3.1]{FG}:
the specified trace-preserving expectations are faithful,
since the traces are faithful, and hence are GNS-faithful.
Proposition~2.15 identifies the vertex-reduced and ordinary
reduced products, and Theorem~3.1 applies to the canonical
quotient.  That quotient
commutes with the two native inclusions of $R$.  Therefore
the reduced native inclusion is a $\mathrm{KK}$-equivalence
as well.  Composition proves the assertion for finite
iterations.  Applying coefficient $\mathrm K$-theory to
these $\mathrm{KK}$-equivalences preserves all Bockstein
operations and proves the final assertion.
\end{proof}

When both algebras in a native inclusion above are admissible
targets, the inclusion also identifies their full compatible
invariants for the selected traces.  Indeed its total
$\mathrm K$-theory map is an isomorphism, its normalized
trace map is the identity, and the trace ranges of
$\mathrm K_0$ therefore agree.  The determinant exact
sequence
\[
 0\longrightarrow
 \mathbb R/\overline{\tau_*(\mathrm K_0(\,\cdot\,))}
 \longrightarrow \Uquot{\,\cdot\,}
 \longrightarrow \mathrm K_1(\,\cdot\,)\longrightarrow0
\]
then shows that its Hausdorffized algebraic unitary map is
an isomorphism.  The same argument applies in compatible
normalized full corners.  This observation does not assume
that either $\mathrm K_1$ group vanishes.

\begin{lemma}[Onto absorption with all finite-dimensional labels fixed]
\label{lem:fd-tag-absorption}
Let $R=H*D_\infty$ and let
$F_1,\ldots,F_m\subset H$ be unital finite-dimensional
subalgebras with fixed matrix units.  There is a
trace-preserving isomorphism
\[
 J_R:R*_{F_1}(F_1\otimes I_1)\cdots
             *_{F_m}(F_m\otimes I_m)\longrightarrow R
\]
which fixes $H$ pointwise.  It can fix any prescribed finite
ordinary reservoir prefix.  The choices are natural for
label-preserving trace-preserving embeddings of $H$ that use
the same ordinary reservoir coordinates.
The same conclusion holds for block-supported tags: in any
one relative addition, interval coordinates may be added on
only some of the summands of the finite-dimensional base,
with the other summands left constant.
\end{lemma}

\begin{proof}
We first treat one algebra $F=\bigoplus_s\mathrm M_{d_s}$.
Let
\[
 p=\sum_sq_s,\qquad v_{s,i}=e^{(s)}_{i1}.
\]
Use a single letter $a=(s,i)$ for a pair of indices, and write
$s(a)=s$.  Then
\begin{equation}\label{eq:fd-full-frame}
 v_a^*v_b=\delta_{ab}q_{s(a)},\qquad
 \sum_av_av_a^*=1.
\end{equation}
In particular $p$ is full.  Let $s_j$ be the standard free
semicircular coordinates of $D_\infty$, and put
\[
 z^{(j)}_{ab}=v_a^*s_jv_b\in q_{s(a)}Rq_{s(b)}.
\]
Together with $pHp$, these entries generate $pRp$.  This follows
by inserting the second identity in
\eqref{eq:fd-full-frame} on either side of every reservoir
letter and between successive coefficient letters.

The operator-valued Wick recursion gives, for every
$x\in pHp$,
\begin{equation}\label{eq:fd-entry-covariance}
 \kappa_2\bigl(z^{(j)}_{ab},x z^{(l)}_{cd}\bigr)
 =\delta_{jl}\delta_{ad}\delta_{bc}
       \tau(q_{s(b)}x)q_{s(a)}.
\end{equation}
All higher cumulants vanish.  For clarity the identity is valid
for offdiagonal coefficients as well: its scalar contraction is
\[
 \tau(v_b x v_c^*)
       =\tau(v_c^*v_b x)
       =\delta_{bc}\tau(q_{s(b)}x).
\]
The remaining outside product is
$v_a^*v_d=\delta_{ad}q_{s(a)}$.
This also proves the Wick recursion directly by contracting
innermost pairs; no centrality of the $p_s$ in $H$ has been used.

Consequently the self-adjoint diagonal entries
\[
 \lambda_s^{-1/2}z^{(j)}_{(s,i),(s,i)}
       \quad (j\geq1,\ 1\leq i\leq d_s)
\]
are mutually free over $pHp$, are free over $pHp$ from the
entire offdiagonal entry family, and have covariance
\[
 b\longmapsto\lambda_s^{-1}\tau(q_sb)q_s.
\]
Here freeness refers to the operator-valued Gaussian law;
the offdiagonal entries retain their prescribed adjoint
relations and their joint covariances.

In $R^+=R*_F(F\otimes I)$ let $t$ be the new central
standard semicircle and set $t_s=q_stq_s$.  Its compressed
covariances are
\[
 \kappa_2(t_s,b t_l)
       =\delta_{sl}\lambda_s^{-1}\tau(q_sb)q_s.
\]
All higher cumulants vanish, and these variables are free
over $pHp$ from the preceding entry family.  One way to see
the last assertion without a compression assumption is to
use the joint $H$-valued Wick formula of $s_j,t$: its
covariances are $b\mapsto\tau(b)1$ and $\E_F^H$, with zero
mixed cumulants, by the expectation tower in
Lemma~\ref{lem:fd-reassociation}.  Compressing this recursion
gives exactly the displayed formulas.

For each $s$ reindex the countably infinite diagonal family
indexed by $(j,i)$ together with the one additional variable
$t_s$ onto the native diagonal family.  Fix $pHp$ and all
offdiagonal entries.  This prescription preserves all joint
Wick moments, has an inverse given by the inverse reindexing,
and therefore gives a trace-preserving isomorphism
\[
 \phi:pR^+p\longrightarrow pRp.
\]
The corner is generated by the listed entries and the $t_s$,
because $t=\sum_{s,i}v_{s,i}t_sv_{s,i}^*$.

Reconstruct the full algebra by the finite frame:
\begin{equation}\label{eq:fd-reconstruction}
 J_R(x)=\sum_{a,b}v_a\phi(v_a^*xv_b)v_b^*.
\end{equation}
Using \eqref{eq:fd-full-frame} proves that this map is unital
and multiplicative.  The same formula with $\phi^{-1}$ is
its inverse.  It fixes $H$, since $\phi$ fixes $pHp$, and
trace preservation follows by applying cyclicity to the
diagonal terms.  Thus this is an onto isomorphism, rather
than merely a relative embedding.

For finitely many tags partition the unused reservoir into
$m$ infinite banks.  Absorb the tags successively, using
one bank for each tag and treating all preceding banks and
tag images as coefficients.  Lemma~\ref{lem:fd-reassociation}
permits this order and leaves the other banks as external
free factors.  A finite ordinary reservoir prefix is first
included in $H$.  Finally every entry formula, normalization,
partition and reindexing can be chosen with the same labels
on the two sides of a label-preserving embedding.  Formula
\eqref{eq:fd-reconstruction} then commutes with that
embedding, proving the exact naturality assertion.
For a block-supported tag the same proof simply adds a
variable $t_s$ at each selected vertex and adds no variable
at the other vertices.  The native diagonal families at
unselected vertices are left unchanged.  This proves the
stated supported version as well.
\end{proof}

\begin{lemma}[A free implementing circle for a finite-dimensional frame]
\label{lem:fd-free-implementer}
Let $N$ be faithfully tracial and let $F_0,B\subset N$ be
unital subalgebras free over a common finite-dimensional subalgebra $G$.
Let $G\subset C\subset F_0$ be finite dimensional, and suppose that a
unitary $w_0\in N\cap G'$ satisfies $w_0Cw_0^*=B$. In
\[
 \widetilde N=N*_C(C\otimes\mathrm C(\mathbb T))
\]
let $v$ be the $C$-central Haar generator. Then $W=w_0v$ commutes
with $G$, satisfies $WcW^*=w_0cw_0^*$ for $c\in C$, and is a
$G$-central Haar unitary free from $F_0$ over $G$. In particular,
\[
 \mathrm C^*(F_0,W)\cong F_0*_G(G\otimes\mathrm C(\mathbb T)).
\]
When $G=\C1$, these are the scalar Haar and scalar freeness
conclusions. Multiplying $W$ by a scalar phase preserves the conclusions.
\end{lemma}

\begin{proof}
The commutation and frame identities follow from centrality of $v$
over $C$. Let $\theta=\Ad(w_0)|_C$ and $H_0=\mathrm C^*(F_0,B)$.
The expectation onto $N$ in $\widetilde N$ annihilates a word
in $W,W^*$ with coefficients in $H_0$ under the following reduced-word rule:
coefficients at $W\,\cdot\,W^*$ must have zero expectation onto $C$,
and those at $W^*\,\cdot\,W$ must have zero expectation onto $B$.
Indeed, substitute $W=w_0v$ and $W^*=v^*w_0^*$. At the first
type of junction the coefficient between $v$ and $v^*$ is
already $C$-centered. At the second it is $w_0^*xw_0$, which
is $C$-centered because $\E_C(w_0^*xw_0)=w_0^*\E_B(x)w_0$.
At each same-sign junction split the intervening coefficient into
its $C$-part and its centered part. The $C$-parts commute with $v$
and combine powers of the same sign. By $C$-bimodularity,
the intervening centered coefficients stay centered. The amalgamated
reduced-word rule for the nonzero powers of $v$ proves the assertion.
In particular all nonzero powers of $W$ have zero expectation onto $N$.

Now consider a word in nonzero powers of $W$ with $G$-centered
$F_0$-letters between them and optional $F_0$ end letters.
Expand each power into single letters of the same sign, inserting
units between them. At each $WfW^*$ junction split $f$ into
$f-\E_C(f)$ and $\E_C(f)$, replacing the second choice by
$b=\theta(\E_C(f))\in B$. The expectation tower and
$\theta|_G=\id_G$ give $\E_G(b)=0$. These junctions are disjoint.
If deletions merge intervening coefficients, the merged coefficient
is an alternating $F_0$--$B$ word with centered $B$-letters and
centered internal $F_0$-letters. Its expectation onto $F_0$, and
hence onto $C$, is zero by freeness over $G$. If it now occurs
between $W^*$ and $W$, both outside $F_0$-letters are also centered:
they originally occurred at opposite-sign junctions. Its expectation
onto $B$ is therefore zero. A surviving $WfW^*$ junction without
mergers is $C$-centered by the split, while an original
$W^*fW$ junction is $B$-centered because $\E_B(f)=\E_G(f)=0$.
Thus every term still containing a $W$-letter satisfies the
reduced-word rule just proved. Terms with all stable letters
deleted are alternating words with centered $B$-letters and
centered internal $F_0$-letters, so their expectation onto $G$
is also zero. This proves freeness over $G$. The faithful reduced
traces give the asserted isomorphism.
\end{proof}

We will also use retained-frame liberation for direct sums.
If $A,B\subset M$ are scalar-free unital subalgebras,
$A$ is finite dimensional, $F_0\subset M$ contains $A$,
and $u$ is an $A$-central Haar unitary free from $M$
over $A$, then $uF_0u^*$ is scalar-free from $B$ and
contains the original $A$.  Indeed $uMu^*$ and $M$
are free over $A$ by the amalgamated reduced-word rule.
Their restricted product is
\[
 (uF_0u^*)*_A(A*B)\cong(uF_0u^*)*B,
\]
since $\E_A(b)=\tau(b)1$ on $B$.  This proves the
claim with its reduced norm.

Exact finite-dimensional frame conjugacy also requires
all summands to be retained.  If two unital embeddings of
$F=\bigoplus_s\mathrm M_{d_s}$ into a unital
stable-rank-one algebra have the same $\mathrm K_0$
class on each labelled $q_s$, cancellation gives
partial isometries between the corresponding $q_s$.
Writing these as $a_s$, with initial projection
$\varphi(q_s)$ and final projection $\psi(q_s)$, the unitary
\[
 \sum_{s,i}\psi(e^{(s)}_{i1})a_s\varphi(e^{(s)}_{1i})
\]
conjugates the entire labelled $\varphi$-frame onto the
$\psi$-frame.  Equality of a single ambient unit class
would not suffice.

\begin{lemma}[Relative Nielsen moves as actual ambient automorphisms]
\label{lem:fd-relative-nielsen}
Let $T=H*D_\infty$ be an admissible target, and let
$F_1,\ldots,F_m\subset H$ be unital finite-dimensional
subalgebras.  Form the finitely enlarged algebra
\[
 T^+=T*_{F_1}(F_1\otimes I_{H_1})\cdots
                  *_{F_m}(F_m\otimes I_{H_m}).
\]
Fix $a\in\{1,\ldots,m\}$, write
$F_a=\bigoplus_s\mathrm M_{d_s}$, and put
$Z=\exp(2\pi iH_a)$.  Suppose
$u\in\mathrm U(F_a'\cap T)$ and, for every labelled
minimal projection $q_s\in F_a$,
\begin{equation}\label{eq:fd-corner-marking}
 q_suq_s\in\CU(q_sTq_s).
\end{equation}
There are actual trace-preserving automorphisms of $T^+$
which, on arbitrary prescribed finite packets, approximate
the identity on $T$ and on every other tag logarithm,
and approximate the prescription
\[
 Z\longmapsto Zu.
\]
They induce the identity on the full compatible invariant.
All packets may contain moving coefficient elements and
are specified before the approximation is chosen.  No
approximation assertion is made on $H_a$.

In \eqref{eq:fd-corner-marking} one may equivalently impose
zero $\mathrm K_1$ class and zero Hausdorffized determinant
in each normalized corner.  Ambient determinant zero
alone is not the hypothesis.  No assumption
$\mathrm K_1(T)=0$ is needed: the multiplier classes in
\eqref{eq:fd-corner-marking} are required to vanish
individually.
\end{lemma}

\begin{proof}
Split the selected tag into its block-supported tags by
Lemma~\ref{lem:fd-reassociation}.  We perform a move on
one block at a time, putting that block outermost and
including all other block tags among the coefficients.
Fix a block and abbreviate its support by $p_0$, its
minimal projection by $q$, its matrix units by $e_{ij}$,
and $v_i=e_{i1}$.  Let $M$ denote the full remaining
coefficient algebra, including every other tag.  The
whole algebra, in this ordering, is
\[
 R=M*_{\mathrm M_d\oplus\mathbb C(1-p_0)}
       ((\mathrm M_d\otimes I)\oplus\mathbb C(1-p_0)).
\]
Write $h$ for the supported interval and $h_q=qhq$,
with functional calculus in the unital corner $qRq$.
Then
\begin{equation}\label{eq:fd-supported-corner}
 qRq\cong(qMq)*I_{h_q}.
\end{equation}
To prove generation, use
$h=\sum_iv_i h_qv_i^*$ to expand every occurrence of
the supported tag in a corner word.  The coefficients
between consecutive $h_q$'s then lie in $qMq$.
For the free law, a centered element of $qMq$ has
expectation zero onto
$\mathrm M_d\oplus\mathbb C(1-p_0)$: its expectation
is its normalized corner trace times $q$.
A centered function of $h_q$ has the same property.
The amalgamated reduced-word rule therefore gives
scalar freeness in the normalized corner trace.
Faithfulness proves \eqref{eq:fd-supported-corner}.
This argument uses the outermost block, rather than
compressing an algebra in which other block tags
are still treated as added variables.

Set $z=\exp(2\pi ih_q)$ and $u_q=quq$.
The corner circle free product has the trace-preserving
Nielsen automorphism
\[
 \gamma|_{qMq}=\id,\qquad \gamma(z)=zu_q,
\]
whose inverse sends $z$ to $zu_q^*$.
The multiplier $u_q$ belongs to the coefficient corner.
Its class in the Hausdorffized algebraic unitary group
is zero by \eqref{eq:fd-corner-marking}, also after
inclusion in $qMq$.  Thus $z,zu_q,zu_q^*$ have
identical full marked circle invariants: their
$\mathrm K_1$ classes and their algebraic unitary
classes agree, and all three have Haar spectral
measure.  The unit determines their $\mathrm K_0$ maps;
since both ordinary $\mathrm K$-groups of $\mathrm C(\mathbb T)$
are free cyclic, the coefficient maps are determined by the
ordinary maps as well.  This checks the full marking required by
Theorem~\ref{rf:circle-lift}.

All corner hypotheses of that theorem hold here.
By Lemma~\ref{lem:fd-tag-absorption}, $M$ and $R$
are isomorphic to $T$, with the original $H$ fixed.
The nonzero projection $q\in T$ is full in the
original simple algebra $T$, hence in both $M$ and
$R$.  Admissibility passes to full corners.
The coefficient corner has infinite-dimensional
GNS space since it contains $qTq$.
Theorem~\ref{rf:circle-lift} and its moving-packet
formulation in Remark~\ref{rf:rem-moving-packets}
therefore supply actual trace-preserving corner
automorphisms $\theta$ approximating $\gamma$ on
any chosen finite coefficient-and-circle packet.

We now extend this finite-packet assertion to the whole
ambient algebra without asserting an ordinary matrix
amplification by $q$.  Fullness in $T$ gives finitely
many $x_j\in Tq\subset Mq$ such that
\[
 \sum_{j=1}^N x_jx_j^*=1.
\]
For instance choose a finite positive sum of terms
$y_jqy_j^*$ which is invertible and normalize the
$y_jq$ by its inverse square root.  Set
\[
 P=(x_i^*x_j)_{i,j}\in\mathrm M_N(qMq),\qquad
 \Psi(r)=(x_i^*rx_j)_{i,j}.
\]
Then $P$ is a projection and
\[
 \Psi:R\longrightarrow P\mathrm M_N(qRq)P
\]
is a unital isomorphism; its inverse is
$(b_{ij})\mapsto\sum_{i,j}x_i b_{ij}x_j^*$.
Include every entry of $P$ in the protected corner
coefficient packet.  The projection
$Q=\theta_N(P)$ is then as close to $P$ as desired.
When $\|P-Q\|<1$, the polar unitary of
\[
 PQ+(1-P)(1-Q)
\]
is a unitary $w$ satisfying $wQw^*=P$, with
$\|w-1\|\to0$ as $\|P-Q\|\to0$.  Thus
\begin{equation}\label{eq:fd-ambient-morita}
 \beta=\Psi^{-1}\Ad(w)\theta_N\Psi
\end{equation}
is an actual automorphism of $R$.

For each protected coefficient $r\in M$, include all
entries $x_i^*rx_j$ in the corner packet.  Equation
\eqref{eq:fd-ambient-morita} then makes $\beta(r)$
as close to $r$ as desired.  To control the selected
supported circle
\[
 Z_0=\sum_{k=1}^dv_kzv_k^*,\qquad
 U_0=p_0up_0=\sum_{k=1}^dv_ku_qv_k^*,
\]
also include the finitely many coefficients
$x_i^*v_k$ and $v_k^*x_j$.  Since
\[
 x_i^*Z_0x_j
       =\sum_k(x_i^*v_k)z(v_k^*x_j),
\]
the corner estimate $\theta(z)\approx zu_q$
then gives $\beta(Z_0)\approx Z_0U_0$.
Every other block tag belongs to $M$ and is protected,
including its logarithm.  These estimates involve
only finitely many elements and finite sums; choosing
the corner tolerance sufficiently small gives any
specified ambient tolerance.

The map \eqref{eq:fd-ambient-morita} preserves trace:
with $\tau_q=\tau(q)^{-1}\tau|_{qRq}$,
\[
 \tau(r)=\tau(q)(\Tr_N\otimes\tau_q)(\Psi(r)),
\]
and both $\theta_N$ and $\Ad(w)$ preserve the
matrix trace.  It also preserves the full marking.
The corner coefficient inclusion in
\eqref{eq:fd-supported-corner} is a $\mathrm{KK}$-
equivalence, and hence a total-$\mathrm K$-theory
isomorphism, by Lemma~\ref{lem:fd-extension-kk}
with scalar amalgamation.  Its trace range
on $\mathrm K_0$ is unchanged.  The determinant
exact sequence therefore shows that its map on
Hausdorffized algebraic unitary groups is also an
isomorphism.  The corner lifts can be chosen with the identity
marking on this coefficient inclusion.  More precisely,
the proof of Theorem~\ref{rf:circle-lift} first
constructs a fixed embedding whose coefficient
restriction has the prescribed full approximate-unitary
class; it then uses inner conjugates of a fixed
isomorphism approximately unitarily equivalent to that
embedding.  Consequently each chosen automorphism,
and not only its point-norm limit on finite packets,
has this full coefficient marking.  It therefore has
identity full marking on the corner.  Matrix
stabilization, the full-corner identification and
the inner correction in
\eqref{eq:fd-ambient-morita} preserve that conclusion.

Finally perform these moves for the finitely many
summands.  At each step protect the other tags and
all preceding moving images needed in the finite
request.  Their individual errors can be made
arbitrarily small, so the finite composition fixes
the requested coefficient packets and sends
$Z=\sum_s p_sZ$ as close as desired to $Zu$.
Each factor of the composition is an actual
trace-preserving automorphism with identity full
marking.  Stable rank one and the determinant
exact sequence identify
\eqref{eq:fd-corner-marking} with the equivalent
$\mathrm K_1$ and normalized-corner determinant
conditions stated in the lemma.
\end{proof}

\section{Split-unit fusion}\label{app:split-unit}

The support-shifting argument below gives a direct proof of fusion
when only one unit class may be nonsplit. It is also used in the
main construction to separate the root's $\mathrm K_1$-group and to
incorporate the split-unit leaves.

\begin{theorem}[Fusion with split-unit leaves]\label{thm:split-fusion}
Let $\mathcal A=((A_i,\tau_i))$ be a nonempty finite or countable
family of separable unital nuclear UCT algebras with faithful selected
traces, and suppose all but at most one factor have split unit class.
If $F$ is a nuclear UCT core with faithful selected trace and the
same selected pointed graded $\mathrm K$-groups and pairing as
$P_{\mathcal A}$, then, for every reservoir $D$,
\[
 P_{\mathcal A}\fr D\cong F\fr D.
\]
Every prescribed compatible full marking may be realized, and a
finite reservoir prefix may be fixed pointwise. Such products are
classified by their Elliott invariants, with realization of every
invariant isomorphism. For two such families, if each contains two
distinct factors admitting Haar unitaries, the classification also
holds without adjoining $D$.
\end{theorem}

We first prove the matrix-completion and support-shifting lemmas;
the countable fusion argument then proves the theorem.

\subsection{Matrix completion}\label{sec:matrix}

For a faithfully tracial algebra $(P,\tau_P)$ define
\begin{equation}\label{eq:note-2-1}
 \mathcal F(P)=
 (P\oplus\C,\tfrac12\tau_P\oplus\tfrac12\id)\fr D,
 \qquad r=(1_P,0),\quad q=1-r,\quad
 R=\C r\oplus\C q.
\end{equation}
The trace on $R$ is the uniform two-point trace.
All expectations onto $R$ are trace preserving.

\begin{lemma}[Mixed associativity]\label{lem:assoc}
Let $A_1,A_2$ contain a common finite-dimensional tracial algebra
$R$, with trace-preserving expectations onto $R$, and let $Q$ be
faithfully tracial. Then
\begin{equation}\label{eq:note-2-2}
 A_1*_R(A_2\fr Q)
 \cong (A_1*_R A_2)\fr Q
\end{equation}
by an isomorphism respecting the canonical copies of $A_1,A_2,Q$.
The expectation from $A_2\fr Q$ to $R$ is the composition of its
canonical expectation onto $A_2$ with the expectation onto $R$.
\end{lemma}
\begin{proof}
Work inside $C=(A_1*_R A_2)\fr Q$, and put
$B=\mathrm{C}^*(A_2,Q)$. Its trace and expectation onto $R$ have the stated
form. We verify that $A_1$ and $B$ are $R$-free.

Expand an $R$-centered algebraic coefficient of $B$ as a sum of
an $R$-centered coefficient from $A_2$ and scalar-reduced
$A_2$--$Q$ words containing centered $Q$ letters.
In an alternating word from $\ker \E_R|_{A_1}$ and
$\ker \E_R|_B$, either no $Q$ letter occurs, in which case use
$R$-freeness of $A_1,A_2$, or $Q$ letters occur. Each block
between successive centered $Q$ letters is scalar centered:
inside one $B$ coefficient this follows from scalar reduction;
across coefficients the block contains an $R$-centered alternating
$A_1$--$A_2$ word. Testing the resulting word against $R$ shows
that its $R$-expectation is zero, by scalar freeness of $Q$.
The indicated algebras generate $C$. Equality of their faithful
tracial joint laws identifies the reduced GNS completions.
\end{proof}

\begin{lemma}[Completing a support projection]\label{lem:complete}
For every faithfully tracial $P$,
\begin{equation}\label{eq:note-2-3}
 (P\oplus\C)*_R \mathrm{M}_2(\C)\cong \mathrm{M}_2(P),
\end{equation}
where $r=e_{11}$, $q=e_{22}$, and the copy of $P\oplus\C$ is
$\operatorname{diag}(P,\C1_P)$.
\end{lemma}
\begin{proof}
The two algebras on the right generate $\mathrm{M}_2(P)$: the diagonal
$p e_{11}$ and the constant matrix units generate all $p e_{ij}$.
The $R$-centered elements of $P\oplus\C$ have the form
$\operatorname{diag}(a,0)$, $\tau_P(a)=0$, while those of
$\mathrm{M}_2(\C)$ are off diagonal. An alternating centered word with
at least two $P\oplus\C$ letters vanishes because
\[
 \operatorname{diag}(a,0)\,m\,\operatorname{diag}(b,0)=0
 \quad\text{for off-diagonal }m.
\]
With one such letter, its $R$-expectation is zero: multiplication
by at most two off-diagonal letters gives either an off-diagonal
matrix or a diagonal entry proportional to $a$.
This proves $R$-freeness. The faithful normalized matrix trace
then gives exactly the reduced completion.
\end{proof}

\begin{lemma}[Gaussian matrix amplification]\label{lem:gaussian}
For every faithfully tracial $P$ there is an isomorphism
\begin{equation}\label{eq:note-2-4}
 (\mathrm{M}_2(P),\tr_2\otimes\tau_P)\fr D
 \cong \mathrm{M}_2(P\fr D)
\end{equation}
fixing the canonical $\mathrm{M}_2(P)$.
\end{lemma}
\begin{proof}
In $P\fr D_{\mathrm{ent}}$ choose standard semicirculars
$a_n,b_n,x_n,y_n$, $n\geq1$, all freely independent and free
from $P$. Set
\begin{equation}\label{eq:note-2-5}
 c_n=(x_n+i y_n)/\sqrt2,\qquad
 S_n=\frac1{\sqrt2}
 \begin{pmatrix}a_n&c_n\\c_n^*&b_n\end{pmatrix}.
\end{equation}
Here $D_{\mathrm{ent}}=\mathrm{C}^*(a_n,b_n,x_n,y_n:n\geq1)\cong D$.

We check the full distribution, including coefficients from
the possibly nonnuclear algebra $P$.
Put $B=\mathrm{M}_2(P)$, $\tau_B=\tr_2\otimes\tau_P$, and
$\E_B=\id_{\mathrm{M}_2}\otimes \E_P$. Scalar semicircularity and freeness
from $P$ give $P$-valued cumulants zero except at order two,
where
\begin{equation}\label{eq:note-2-6}
 \kappa_2^P(z_\alpha,pz_\beta)
 =\delta_{\alpha\beta}\tau_P(p)1_P .
\end{equation}
This is the cumulant factorization criterion in
\cite[Theorem 3.5]{NSS}. Entrywise matrix amplification preserves
the cumulant formula. Equivalently, write $S_n$ as the linear
combination of its four real Gaussian entries with coefficients
\[
 C_1=e_{11}/\sqrt2,\quad C_2=e_{22}/\sqrt2,\quad
 C_3=(e_{12}+e_{21})/2,\quad
 C_4=i(e_{12}-e_{21})/2.
\]
These coefficients give exactly \eqref{eq:note-2-5}.
All higher $B$-valued cumulants are zero, and
\begin{equation}\label{eq:note-2-7}
 \begin{split}
 \kappa_2^B(S_n,bS_m)
 &=\delta_{nm}\sum_{\alpha=1}^4
 C_\alpha(\id_{\mathrm{M}_2}\otimes\tau_P)(b)C_\alpha\\
 &=\delta_{nm}\tau_B(b)1_B .
 \end{split}
\end{equation}
Thus the $B$-valued Wick recursion is precisely that of a
standard scalar semicircular family free from $B$.
In particular the family $(S_n)$ has that joint law with $B$.
The scalar matrix version is the classical Gaussian matrix
construction in \cite[Section 3]{Dykema}.

Consequently, the canonical $B$ and the $S_n$ generate a copy
of $B\fr D$. It is the entire $\mathrm{M}_2(P\fr D_{\mathrm{ent}})$:
constant matrix units recover every matrix entry of $S_n$,
and their real and imaginary parts recover $a_n,b_n,x_n,y_n$.
This proves surjectivity at the $\mathrm{C}^*$-algebraic level.

Faithfulness proves equality of the $\mathrm{C}^*$-norms, not only equality
of tracial moments. For any $z$ in a faithfully tracial
$\mathrm{C}^*$-algebra,
\[
 \|z\|=\lim_{k\to\infty}
 \tau((z^*z)^k)^{1/(2k)}.
\]
Finally reindex the countable entry family to identify
$D_{\mathrm{ent}}$ with $D$.
\end{proof}

\begin{lemma}[Full-corner Gaussian lemma]
\label{lem:gaussian-full-corner}
Let $(N,\tau)$ be a separable unital faithfully tracial $\mathrm C^*$-algebra,
let $p\in N$ be a nonzero projection, and present the reservoir $D$ by a
countable standard free semicircular family $(S_j)_j$. Put $t=\tau(p)$,
$\tau_p=t^{-1}\tau|_{pNp}$ and
\[
 N[p,D]=\mathrm C^*(N,t^{-1/2}pS_jp:j\geq1)\subset N\fr D.
\]
For $p\ne1$, this is the reduced amalgamated product
$N*_{\C p\oplus\C(1-p)}\bigl((\C p\otimes D)\oplus\C(1-p)\bigr)$,
with trace weights $t,1-t$ on the two summands of the second factor.
Its $p$-corner is $(pNp,\tau_p)\fr D$ over $pNp$.
If $p$ is full, there is a trace-preserving isomorphism
$N[p,D]\cong N\fr D$ fixing $N$ pointwise. In particular,
$p(N\fr D)p\cong(pNp,\tau_p)\fr D$ over $pNp$ for the normalized
corner traces.
\end{lemma}
\begin{proof}
By Corollary~\ref{cor:reservoir-model}, we may take $D$ to be
the countable free semicircular reservoir.
The case $p=1$ is immediate.  Suppose $p\ne1$, put
$F=\C p\oplus\C(1-p)$, and form the supported extension
\[
 S=N*_F G,\qquad
 G=(\C p\otimes D)\oplus\C(1-p).
\]
The trace of $G$ is $t\tau_D\oplus(1-t)\id_\C$, and the
expectations onto $F$ preserve the specified traces.  If $h_j$ are
the supported standard semicircular coordinates of $G$, their
$N$-valued joint Wick formula has zero higher cumulants and
covariances
\[
 \kappa_2^N(h_j,bh_k)
       =\delta_{jk}\,t^{-1}\tau(bp)p\qquad(b\in N).
\]
Inserting $p$ in the Wick formula for $(S_j)_j$ gives the same
$N$-valued Wick moments for the compressed family
$(t^{-1/2}pS_jp)_j$. The faithful reduced traces identify $S$ with
$N[p,D]$ over $N$. Moreover,
\begin{equation}\label{eq:full-corner-supported}
 pSp\cong(pNp,\tau_p)\fr D.
\end{equation}
Indeed every $h_j$ is supported on $p$, so compression of a word
leaves coefficients in $pNp$.  Centered corner coefficients and
centered elements of $pGp$ have expectation zero onto $F$.
The amalgamated reduced-word rule proves scalar freeness for the
normalized corner trace, and faithfulness proves the asserted
$\mathrm C^*$-isomorphism.  This is the supported-corner argument
of~\eqref{eq:fd-supported-corner}, with countably many coordinates.

Assume now that $p$ is full. To identify $S$ with $N\fr D$ over $N$,
use the coarse correspondence and its commuting coefficient algebra
\[
 H=L^2(N,\tau)\otimes N,\qquad
 B=N^{\mathrm{op}}\otimes_{\min}N.
\]
The right Hilbert $N$-module inner product is
$\langle a\otimes b,c\otimes d\rangle=\tau(a^*c)b^*d$,
and $N$ acts on the left tensor factor from the left.  The
coefficient action is
$(a^{\mathrm{op}}\otimes b)(c\otimes d)=ca\otimes bd$.
For a faithful representation $\pi:N\to\mathcal B(V)$, the induced
action on $H\otimes_\pi V=L^2(N,\tau)\otimes V$ is the spatial action
$R_a\otimes\pi(b)$, where $R_a$ denotes right multiplication by $a$.
Induction through $\pi$ is faithful on the adjointable operators,
so the coefficient action is continuous for the minimal norm.
It commutes with the outer $N$-bimodule actions.  Set
\[
 P=p^{\mathrm{op}}\otimes p,\qquad
 \zeta=1\otimes1,\qquad \xi=t^{-1/2}p\otimes p.
\]
The closed bimodule generated by $\xi$ is $PH$, and
\[
 \langle\xi,b\xi\rangle=t^{-1}\tau(bp)p.
\]
Thus the tracial Gaussian constructions for $H^{\oplus\infty}$
with seeds $\zeta_j$, and for $(PH)^{\oplus\infty}$ with seeds
$\xi_j$, are respectively $N\fr D$ and $S$. To justify these
identifications in norm, the seeds generate their correspondence,
and products of seed fields with coefficients produce elementary
Fock tensors modulo tensors of shorter length. Thus the span of
$A\Omega N$ is dense in the Fock module, where $A$ is the algebra
generated by the coefficient and seed fields and $\Omega$ is the
vacuum vector. The joint Wick formulas give the reduced-product
moments, so the scalar vacuum state on $A$ is tracial. If it
vanishes on $a^*a$, traciality makes it vanish on $b^*a^*ab$
for every $b\in A$. Faithfulness of $\tau$ on $N$ then gives
$ab\Omega=0$ for all $b\in A$, and cyclicity gives $a=0$.
The faithful scalar moments therefore identify the generated
$\mathrm C^*$-algebras with the indicated reduced products.

On the algebraic coarse correspondence put
$J(c\otimes d)=d^*\otimes c^*$, and on $B$ put
\[
 \sigma(a^{\mathrm{op}}\otimes b)
       =(b^*)^{\mathrm{op}}\otimes a^*.
\]
Then $\sigma$ is a conjugate-linear $^*$-automorphism,
$\sigma(P)=P$, and $J(b\eta)=\sigma(b)J\eta$ on algebraic
vectors.  Both seeds are $J$-symmetric.  Since $p$ is full in $N$,
$P$ is full in $B$.  It is full in the real algebra
$B_{\mathbb R}=B^\sigma$ as well.  To see this explicitly, choose
a finite complex fullness witness, average its positive sum under
$\sigma$, and write each coefficient as $a+ib$ with
$a,b\in B_{\mathbb R}$.  The averaged sum is then a finite sum
of terms $aPa^*$ and $bPb^*$.  Normalizing by its positive inverse
square root gives $a_1,\ldots,a_m\in B_{\mathbb R}$ with
\[
                  \sum_{i=1}^m a_iPa_i^*=1.
\]

For completeness, this finite witness gives the required real
stabilization.  Write $E=PB_{\mathbb R}$.  The map
$b\mapsto(Pa_i^*b)_{i=1}^m$ embeds $B_{\mathbb R}$ as an
orthogonally complemented summand of $E^m$, so
$E^m\cong B_{\mathbb R}\oplus Q$ for a finitely generated
projective real module $Q$.  Since $E^m$ is a summand of
$B_{\mathbb R}^m$, so is $Q$.  Hence
\[
 E^{\oplus\infty}
 \cong B_{\mathbb R}^{\oplus\infty}\oplus Q^{\oplus\infty}
 \cong B_{\mathbb R}^{\oplus\infty}.
\]
For the last isomorphism, choose $Q'$ with
$Q\oplus Q'\cong B_{\mathbb R}^m$, decompose the standard
countable module as $Q^{\oplus\infty}\oplus(Q')^{\oplus\infty}$,
and reindex the two countable copies of $Q$.  Complexification and
the action on $H$ give an $N$-bimodule unitary
\[
 U:(PH)^{\oplus\infty}\longrightarrow H^{\oplus\infty}
\]
whose coefficient matrix has entries fixed by $\sigma$.

We check that second quantization gives an isomorphism of the
Gaussian $\mathrm C^*$-algebras.  For a symmetric seed $e$ and
an algebraic vector $\eta=\sum_r a_reb_r$, the field operator is
\[
 s(\eta)=\ell(\eta)+\ell(J\eta)^*
          =\sum_r a_r[\ell(e)+\ell(e)^*]b_r.
\]
In particular, if $J\eta=\eta$, then
$\ell(\eta)+\ell(\eta)^*$ belongs to the algebra generated by
$N$ and the seed field.  The same observation applies to finite
sums over different seeds.  Each column of the real module
unitary is a norm-convergent column over $B$.  Truncate it to
finitely many coordinates and approximate its entries in norm by
algebraic tensor coefficients fixed by $\sigma$, using
symmetrization.  Applying these columns to $\xi_j$ produces
algebraic vectors $\eta_l$ with $J\eta_l=\eta_l$ and
$\eta_l\to U\xi_j$ in the correspondence norm.  Their Gaussian
operators belong to $N\fr D$ and converge in norm, since
\[
 \|[\ell(\eta_l)+\ell(\eta_l)^*]
          -[\ell(\eta_k)+\ell(\eta_k)^*]\|
       \leq2\|\eta_l-\eta_k\|.
\]
Here $\ell$ is the creation operator on the full Fock module.
The same argument applies to $U^*$ and each $\zeta_j$, multiplying
the symmetrized algebraic approximating coefficients on the left
by $P$.  Since $\sigma(P)=P$, this preserves their symmetry.  Its
algebraic approximants lie in $PH$ and are finite sums of vectors
$a\xi b$, so their Gaussian operators belong to $S$.
Consequently the Fock unitary induced by $U$ carries $S$ onto
$N\fr D$, with its inverse carrying the latter algebra onto $S$.
It fixes the vacuum coefficient algebra $N$ and preserves the
vacuum expectation and the trace.  This norm argument uses only
symmetric algebraic approximants and makes no boundedness claim
for $J$ in the correspondence norm.

Compress this $N$-fixing isomorphism by $p$ and use
\eqref{eq:full-corner-supported}.
\end{proof}

\begin{theorem}[Matrix completion removes the scalar summand]
\label{thm:bridge}
Let $P,Q$ be arbitrary faithfully tracial unital algebras.
If
\begin{equation}\label{eq:note-2-8}
 \Phi:\mathcal F(P)\longrightarrow\mathcal F(Q)
\end{equation}
is a trace-preserving isomorphism fixing $r$, then
\begin{equation}\label{eq:note-2-9}
 P\fr D\cong Q\fr D.
\end{equation}
No nuclearity or $\mathrm{K}$-theory assumption on $P,Q$ is needed here.
\end{theorem}
\begin{proof}
Trace preservation and $\Phi(r)=r$ imply that $\Phi$ preserves
the expectation onto $R$. Explicitly, that expectation is
\[
 \E_R(a)=2\tau(rar)r+2\tau(qaq)q.
\]
Reduced amalgamated functoriality therefore gives an isomorphism
\[
 \widehat\Phi=\Phi*_R\id_{\mathrm{M}_2}:
 \mathcal F(P)*_R \mathrm{M}_2\longrightarrow
 \mathcal F(Q)*_R \mathrm{M}_2
\]
fixing every added matrix unit. Lemmas~\ref{lem:assoc}--%
\ref{lem:gaussian} give identifications
\begin{equation}\label{eq:note-2-10}
 \begin{aligned}
 \mathcal F(P)*_R \mathrm{M}_2
 &\cong ((P\oplus\C)*_R \mathrm{M}_2)\fr D\\
 &\cong \mathrm{M}_2(P)\fr D
 \cong \mathrm{M}_2(P\fr D),
 \end{aligned}
\end{equation}
and the corresponding identifications for $Q$.
They all preserve the matrix units and take $r$ to $e_{11}$.
Compressing the resulting isomorphism by $e_{11}$ proves \eqref{eq:note-2-9}.
\end{proof}

\begin{remark}[Invariant and reservoir control]\label{rem:control}
The completed isomorphism fixes all matrix units, so it equals
$\mathrm{M}_2(\Theta)$ for its corner isomorphism $\Theta$.
The prescribed total $\mathrm{K}$-theory and matrix-stabilized unitary data therefore
descend through the canonical corner maps.
For $u\in\mathrm U_0(P\fr D)$, using a path from $1$ to $u$
and its diagonal amplification,
\[
 \Delta_{\tr_2\otimes\tau}(\operatorname{diag}(u,1))
 =\tfrac12\Delta_\tau(u);
\]
normalizing the corner trace restores the original determinant.
This formula concerns the determinant kernel. The ordinary
$\mathrm K_1$ component of a general unitary is retained by the
same corner identification.

The class $[q]-[r]$ is killed by the actual matrix-unit relation
$r\sim q$. Other specified reduced classes survive. This is
a functor applied to the exact isomorphism \eqref{eq:note-2-8}.

Finite reservoir control survives as well. If $\Phi$ fixes
finitely many decorated semicircular generators and $r$,
then $\widehat\Phi$ fixes these generators and all matrix units.
It consequently fixes every Gaussian entry extracted from
them. By choosing the entry enumeration, any specified finite
collection of ordinary reservoir interval factors can be
protected pointwise.
\end{remark}

\subsection{Finite support shifting and fusion}\label{sec:fusion}

\begin{proposition}[Support shifting at equal weights]
\label{prop:shift}
Let $E_1,E_2$ be nuclear UCT cores with faithful selected traces.
Suppose that $E_2$ has a chosen splitting
$\mathrm{K}_0(E_2)=\Z[1_{E_2}]\oplus G_2$.
There is a trace-preserving isomorphism fixing $r$,
\begin{equation}\label{eq:note-3-1}
 \mathcal F(E_1\fr E_2)
 \cong (E_1\oplus E_2,
        \tfrac12\tau_1\oplus\tfrac12\tau_2)\fr D.
\end{equation}
No splitting of $[1_{E_1}]$ is required.
\end{proposition}
\begin{proof}
Put $u_2=[1_{E_2}]$. Compare the nuclear cores
$E_2\oplus\C$ and $\C\oplus E_2$, with half weights on both
components. On $\mathrm{K}_0$ use the isomorphism
\[
 (ku_2+g,\ell)\longmapsto(k,\ell u_2+g),
 \qquad g\in G_2,
\]
and on $\mathrm{K}_1$ move the $E_2$ summand to the second component.
This fixes the two component-unit classes $r,q$ and preserves
the selected pairing: both pairings on $G_2$ equal
$\frac12\tau_{2*}$. Theorem~\ref{thm:marked-core} gives an
isomorphism after adjoining $D$.
Stable-rank-one cancellation adjusts it by a unitary to fix
$r$ exactly. It then preserves the expectation onto $R$.

Amalgamate this isomorphism with $\id_{E_1\oplus\C}$ over $R$.
Mixed associativity gives
\[
 \begin{aligned}
 (E_1\oplus\C)*_R((E_2\oplus\C)\fr D)
 &\cong((E_1\fr E_2)\oplus\C)\fr D,\\
 (E_1\oplus\C)*_R((\C\oplus E_2)\fr D)
 &\cong(E_1\oplus E_2)\fr D.
 \end{aligned}
\]
The first line places both nonscalar cores in the $r$-corner;
the second places them in opposite corners.
This proves \eqref{eq:note-3-1}, and the construction fixes the canonical
$E_1\oplus\C$ at this stage. All replacements may be carried out
in a complementary reservoir tail before extending by the identity
on a protected prefix.
\end{proof}

\begin{theorem}[Connected commutative fusion]\label{thm:fusion}
Let $E_i=(\mathrm{C}(X_i),\mu_i)$, $i=1,2$, where $X_i$ are connected
compact metrizable spaces and the measures are faithful. No assumption
on either $\mathrm{K}$-group is made.
Choose points $x_i\in X_i$, and let $X_1\vee X_2$ denote the wedge
obtained by identifying them.
For any faithful selected probability on $X_1\vee X_2$,
\begin{equation}\label{eq:note-3-2}
 (E_1\fr E_2)\fr D
 \cong \mathrm{C}(X_1\vee X_2)\fr D.
\end{equation}
The specified reduced $\mathrm{K}$-classes can be preserved.
The result holds for any finite number of such cores.
\end{theorem}
\begin{proof}
Let $W=\mathrm{C}(X_1\vee X_2)$. Use Proposition~\ref{prop:shift}.
The pullback description
\[
 W=\{(f_1,f_2)\in E_1\oplus E_2:f_1(x_1)=f_2(x_2)\}
\]
gives the split evaluation extension
\[
 0\longrightarrow
 \mathrm C_0(X_1\setminus\{x_1\})\oplus
 \mathrm C_0(X_2\setminus\{x_2\})
 \longrightarrow W\longrightarrow\C\longrightarrow0.
\]
The splitting is the constant functions. Consequently the reduced
$\mathrm K$-groups are the direct sums of the two wedge summands;
no CW or cofibration hypothesis on the spaces is needed.
The nuclear cores $E_1\oplus E_2$ and $W\oplus\C$ have
isomorphic selected data: map the first component unit to
$1_W\oplus0$, the second to $0\oplus1$, and map their reduced
$\mathrm{K}$-groups to the wedge summands. All reduced $\mathrm{K}_0$-classes
have trace zero by connectedness, so the trace pairing agrees
for any faithful selected wedge law.
Apply Theorem~\ref{thm:marked-core}, then adjust by a unitary to fix
$r$. We obtain $\mathcal F(E_1\fr E_2)\cong\mathcal F(W)$
over $R$. Theorem~\ref{thm:bridge} proves \eqref{eq:note-3-2}.
Repeat for finitely many factors.
\end{proof}

\begin{proposition}[Fusion into a nuclear range model]
\label{prop:model}
Let $E_1,E_2$ satisfy Proposition~\ref{prop:shift}, with the chosen
splitting $\mathrm{K}_0(E_2)=\Z[1_{E_2}]\oplus G_2$.
Put $P=E_1\fr E_2$, and give
$G_0=\mathrm{K}_0(P)$ the order
\begin{equation}\label{eq:note-3-4}
 (G_0)_+=\{g:\rho(g)>0\}\cup\{0\},
 \qquad \rho=\tau_{P*},\qquad u=[1_P].
\end{equation}
There is a simple, infinite-dimensional, monotracial nuclear UCT
$\mathcal Z$-stable model $F$ with pointed $\mathrm{K}$-groups and pairing
$(G_0,u,\mathrm{K}_1(P),\rho)$, and
\begin{equation}\label{eq:note-3-5}
 P\fr D\cong F\fr D .
\end{equation}
Any prescribed compatible isomorphism of the selected pointed data
may be used. The unit of $F$ need not split.
\end{proposition}
\begin{proof}
Write $u_i=[1_{E_i}]$. The free-product $\mathrm{K}$-sequence
\cite{FG} gives the identifications
\[
 G_0=\mathrm{K}_0(E_1)\oplus G_2,
 \qquad u=(u_1,0),\qquad
 \mathrm{K}_1(P)=\mathrm{K}_1(E_1)\oplus\mathrm{K}_1(E_2).
\]
Indeed, the unit relation in
$\mathrm{K}_0(E_1)\oplus(\Z u_2\oplus G_2)$ eliminates the
second unit coordinate: $(x,ku_2+g)$ represents $(x+ku_1,g)$.
On this group the selected pairing is
$\rho(x,g)=\tau_{1*}(x)+\tau_{2*}(g)$.

The group \eqref{eq:note-3-4} is simple, weakly unperforated and
directed, and $u$ is an order unit. Here weak unperforation means
that $ng>0$ for $n\geq1$ implies $g>0$; torsion is permitted.
Every normalized state equals $\rho$: rational upper and lower
bounds on $g$ using $mu-ng$ force its state value to be $\rho(g)$.
The range theorem \cite[Theorem~2.8]{EGLN} therefore supplies $F$.
It can be chosen infinite dimensional. In the cyclic case with
pointed data $(\Z,m,0)$, one can use $\mathrm{M}_m(\mathcal Z)$.

By Proposition~\ref{prop:shift}, $\mathcal F(P)$ is isomorphic
over $R$ to $(E_1\oplus E_2)\fr D$.
Compare that nuclear core with $F\oplus\C$ by the explicit map
\[
 \mathrm{K}_0(E_1)\oplus(\Z u_2\oplus G_2)
 \longrightarrow
 (\mathrm{K}_0(E_1)\oplus G_2)\oplus\Z,
 \qquad
 (x,ku_2+g)\longmapsto((x,g),k),
\]
and use the corresponding identity on the two
$\mathrm{K}_1$ summands. This sends the two support classes
$(u_1,0)$ and $(0,u_2)$ to $(u_F,0)$ and $(0,1)$,
respectively. At half weights it preserves the selected trace,
since both values are
$\frac12(\tau_{1*}(x)+k+\tau_{2*}(g))$.
Theorem~\ref{thm:marked-core} applies directly to these decorated
cores and gives an isomorphism to $\mathcal F(F)$; it requires
no splitting of their unit classes. The splitting needed in this
argument was used in Proposition~\ref{prop:shift}.
Adjust it to fix $r$, then apply Theorem~\ref{thm:bridge}.
The exact marking descends through matrix completion by
Remark~\ref{rem:control}. Composing the first coordinate of
the displayed map with a prescribed pointed isomorphism, and
using its $\mathrm{K}_1$ map, gives the stated freedom.

The core $\mathrm{K}_0$-map need not preserve the positive cone of the
core algebra $F\oplus\C$. The marked replacement theorem applies
to the simple $D$-stabilized targets, where positivity is
controlled by the faithful selected trace. No embedding into
the original core $F\oplus\C$ is asserted.
\end{proof}

\begin{remark}[Two nonsplit units]\label{rem:nonsplit-fusion}
The split-unit hypothesis on the shifted factor has an algebraic
role. For $E_2=\mathrm{M}_n(\C)$, $n>1$, its support class in
$\mathrm{K}_0(E_2\oplus\C)$ is $n$-divisible, whereas the intended
scalar support in $\mathrm{K}_0(\C\oplus E_2)$ is primitive.
Hence the support-preserving isomorphism used in
Proposition~\ref{prop:shift} cannot exist in this case.
The unit relation can also create torsion for infinite-dimensional
simple nuclear factors. If $\mathcal U$ is the CAR algebra, then
\[
 \begin{split}
 \mathrm{K}_0(\mathcal U\fr\mathcal U)
 &\cong
 \frac{\Z[1/2]\oplus\Z[1/2]}{\Z(1,-1)}\\
 &\cong\Z[1/2]\oplus\bigl(\Z[1/2]/\Z\bigr),
 \qquad
 \mathrm{K}_1(\mathcal U\fr\mathcal U)=0.
 \end{split}
\]
The second isomorphism is $[(x,y)]\mapsto(x+y,y+\Z)$;
the unit is $(1,0)$ and the selected trace is the first coordinate.
Products with two nonsplit unit classes therefore require an additional
fusion argument.  Section~\ref{sec:general-fusion}
provides it for arbitrary nuclear UCT factors, including this CAR
example.  The
support-shifting proof remains sufficient when all but at most one
factor have split unit class.
\end{remark}

\subsection{Countable fusion and invariant control}
\label{sec:classification}

We now prove the split-unit fusion theorem, retaining the full
markings through finite fusion and passage to a countable family.

\begin{proof}[Proof of Theorem~\ref{thm:split-fusion}]
Put $P=P_{\mathcal A}$ and $A=P\fr D$. To compare two such
products, also write $Q=P_{\mathcal B}$ and $B=Q\fr D$, and fix
the prescribed isomorphism $\Theta:\Ell(A)\to\Ell(B)$. Proposition~\ref{prop:product-elliott} identifies its group
maps with
$\alpha=(\chi_0,\chi_1):\mathrm{K}_*(P)\to\mathrm{K}_*(Q)$,
preserving the unit and selected trace pairing.
In each family place the possibly nonsplit factor first; if all units
split, choose any first factor. Write $E_j=A_j$ for the first family.
Choose splittings
$\mathrm{K}_0(E_j)=\Z[1_{E_j}]\oplus G_{E_j}$ for $j\geq2$.
No splitting of $[1_{E_1}]$ is assumed.

For a finite family with at least two factors, iterating
Proposition~\ref{prop:model} gives $P\fr D\cong F\fr D$, where
$F$ is a simple monotracial nuclear UCT $\mathcal Z$-stable range model
for the pointed groups and selected pairing of $P$, with the strict
trace cone. At each step the newly fused factor has split unit;
the old nuclear core need not have split unit. For one factor use
Theorem~\ref{thm:marked-core} with its monotracial range model directly.
Nuclear classification \cite[Theorems 2.7--2.8]{EGLN} identifies
models with isomorphic data, realizing a prescribed isomorphism.

Here is the countable construction, including the invariant alignment
for unrestricted groups. Choose simple monotracial nuclear UCT
$\mathcal Z$-stable models $F_n$, $n\geq1$, for the data of the first
$n$ factors. The free-product $\mathrm{K}$-sequence gives the actual
decompositions
\[
 \mathrm{K}_0(F_n)=\mathrm{K}_0(E_1)
       \oplus\bigoplus_{2\leq j\leq n}G_{E_j},
 \qquad
 \mathrm{K}_1(F_n)=\bigoplus_{j\leq n}\mathrm{K}_1(E_j),
\]
with unit $([1_{E_1}],0)$ and the aggregate selected pairing.
Give these groups the strict trace cone and choose trace-preserving
embeddings $\iota_n:F_n\to F_{n+1}$ inducing the natural inclusions,
using \cite{CGSTW}. Choose compatible full-invariant extensions when
realizing these embeddings. The limit
$F_\infty=\varinjlim(F_n,\iota_n)$ is simple, monotracial, nuclear,
UCT and $\mathcal Z$-stable. Its groups and pairing are the aggregate
data of $P$, and its order is the strict trace cone: a class of
strictly positive trace already has that property at a finite stage.

Put, for $n\geq1$,
\[
 A_n=F_n\fr\left(\fr_{j>n}E_j\right)\fr D,
 \qquad B_n=F_n\fr D,
 \qquad c_n=\iota_n\fr\id_D.
\]
Theorem~\ref{thm:marked-core}, followed by extension over the later
free factors, gives an initial trace-preserving isomorphism
$\eta_1:A\to A_1$ with the prescribed first-factor marking.
In particular $\eta_1(E_1)$ and $\eta_1(D)$ lie in $B_1$, and the
later factors are fixed.
Finite fusion on $F_n,E_{n+1}$, extended identically over the later
factors, gives isomorphisms $\theta_n:A_n\to A_{n+1}$ with
$\theta_n(B_n)\subseteq B_{n+1}$. We can arrange that
$b_n=\theta_n|_{B_n}$ approximates $c_n$ on any prescribed finite
packet. First approximate its elements by free-product polynomials
with error smaller than the prescribed tolerance. Only finitely many
reservoir intervals occur in these polynomials. Protect a finite
interval prefix containing them and run fusion in its complementary
tail. Contractivity controls the error on the original packet.
Then compare only the nuclear old core $F_n$ in that tail and apply
Theorem~\ref{thm:embedding}. The inner correction is made in the
tail and extended by the identity on the protected prefix.

To justify that comparison, it is necessary to align the full
$\mathrm{KL}$-class, equivalently total $\mathrm{K}$-theory, and the
unitary map on $F_n$. One may choose the nuclear replacement between
decorated cores with the following $\mathrm{KK}$ block. Let
$u:\C\to F_{n+1}$ be the unit map. The chosen splitting of the new
factor gives a retraction
$\epsilon_*:\mathrm{K}_0(E_{n+1})\to\Z$ with
$\epsilon_*([1_{E_{n+1}}])=1$. UCT surjectivity lifts it, together
with the zero $\mathrm{K}_1$-map, to
$\epsilon\in\mathrm{KK}(E_{n+1},\C)$.
Choose $\kappa\in\mathrm{KK}(E_{n+1},F_{n+1})$ inducing the
unit-preserving group map that includes the new reduced summand and
the new $\mathrm{K}_1$ summand. Use
\begin{equation}\label{eq:note-4-2}
 \begin{array}{c|cc}
  &F_{n+1}&\C\\ \hline
 F_n&[\iota_n]&0\\
 E_{n+1}&\kappa-u\epsilon&\epsilon
 \end{array}.
\end{equation}
UCT surjectivity supplies $\kappa$ for arbitrary countable groups.
Since the composition of $\epsilon$ with the unit of $E_{n+1}$
is $1\in\mathrm{KK}(\C,\C)$, this block fixes the two component
units $r,q$. It induces an ordinary $\mathrm{K}$-isomorphism and
retains the entire old class $[\iota_n]$.
At half weights its pairing on the new factor is
$\tfrac12(\rho\kappa-\epsilon+\epsilon)
=\tfrac12\tau_{E_{n+1},*}$.
The block is realized in the simple $D$-stabilized target; its separate
component $\kappa-u\epsilon$ need not be positive in $F_{n+1}$.
Theorem~\ref{thm:marked-core} retains the chosen total-$\mathrm{K}$ map
exactly. It need not retain a chosen absolute $\mathrm{KK}$ representative:
pure extensions are irrelevant to the $\mathrm{KL}$ comparison used here.

Choose the compatible unitary map with this total-$\mathrm{K}$ extension
using \cite[Theorem 3.9 and Remark 3.10]{CGSTW}. Any discrepancy
on the old core factors through
$\mathrm{K}_1(F_n)/\Tor\mathrm{K}_1(F_n)$ into the determinant
kernel, by \cite[Lemma 3.11]{CGSTW}.
For this application to supported maps, adjoin the complementary
scalar summand: compare the two unital maps from $F_n\oplus\C$
which send the added scalar unit to $q$. They have equal total
$\mathrm{K}$-theory and the same half-weight trace. Their $\mathrm{K}_1$
is just $\mathrm{K}_1(F_n)$, so the lemma gives exactly the stated
discrepancy. It extends by zero along the actual direct sum
\[
 \frac{\mathrm{K}_1(F_n\oplus E_{n+1})}{\Tor}
 =\frac{\mathrm{K}_1(F_n)}{\Tor}\oplus
   \frac{\mathrm{K}_1(E_{n+1})}{\Tor}.
\]
Adding this correction changes neither total $\mathrm{K}$-theory nor
the trace and preserves every coefficient-to-unitary square.
No freeness of either quotient and no real-valued lift are used.
The full-$\mathrm{KL}$ local theorem realizes the corrected data.
Corner equivalence and Remark~\ref{rem:control} retain them through
matrix completion. Thus the old-core maps have equal full invariant,
as required for the preceding nuclear-source uniqueness step.

Fix countable dense subsets of the $B_n$. Recursively choose finite
packets $K_n\subset B_n$ tracking larger and larger finite portions
of these subsets from every earlier stage under both connecting
systems, and containing
$b_{n-1}(K_{n-1})\cup c_{n-1}(K_{n-1})$ when $n>1$.
Choose the next fusion so that
$\|b_n(x)-c_n(x)\|<\varepsilon_n$ for $x\in K_n$, where
$\sum_n\varepsilon_n<\infty$. After $b_n$ is chosen, its finitely
many required images can be included in $K_{n+1}$.
Thus both images of every retained element belong to the next packet;
inductively, every subsequent mixed composition of the $b$- and
$c$-maps on such an element is retained.

Write $b_{n,m}$ and $c_{n,m}$ for the connecting compositions,
and $c_{m,\infty}$ for the canonical map to the $c$-limit.
For $x\in K_n$ and $m\geq n$,
\[
 \|c_{m+1,\infty}b_{n,m+1}(x)
       -c_{m,\infty}b_{n,m}(x)\|<\varepsilon_m.
\]
Thus these images are Cauchy, with errors bounded by the summable
tails. Reversing $b$ and $c$ gives the reverse limit map. The same
tail estimates show that, on a tracked element represented at
stage $m$, either composite differs from the identity by at most
$2\sum_{j\geq m}\varepsilon_j$. Its later connecting images remain
in the packets, so letting $m$ tend to infinity proves that the
two maps are inverse on the tracked dense sets, and hence everywhere.
Consequently
\begin{equation}\label{eq:note-4-3}
 \varinjlim(B_n,b_n)\cong
 \varinjlim(B_n,c_n)=F_\infty\fr D.
\end{equation}
Set $\eta_n=\theta_{n-1}\cdots\theta_1\eta_1:A\to A_n$
for $n\geq2$, retaining the preceding definition for $n=1$.
The pulled-back algebras $\eta_n^{-1}(B_n)$ increase, because
$\theta_n(B_n)\subseteq B_{n+1}$.
Each $\theta_n$ carries the whole
$F_n\fr E_{n+1}\fr D$ onto $B_{n+1}$ and fixes the later
factors; its tail correction preserves these properties.
Inductively $\eta_n(E_i)\subset B_n$ for $i\leq n$, the later
factors are fixed, and $\eta_n(D)\subset B_n$.
The pulled-back union therefore contains every original factor
and $D$, so it is dense in $A$.
This proves $P\fr D\cong F_\infty\fr D$.
The stage markings identify the induced pointed groups and pairing
with the prescribed aggregate data. For two products with isomorphic
data, nuclear classification of their $F_\infty$ models proves the
required classification, including realization of the prescribed
isomorphism $\Theta$ of Elliott invariants.
For an arbitrary prescribed compatible full marking
$\xi:A\to F\fr D$, let $\Lambda$ be the triple of the constructed
isomorphism $\Phi:A\to F_\infty\fr D$.
Apply Theorem~\ref{thm:marked-core} to $\xi\Lambda^{-1}$ between
the two nuclear-core stabilizations, and compose its realizing
isomorphism with $\Phi$. This gives the requested marking exactly.

To protect an original finite reservoir-block head $H$, first
write $D=H\fr D_{>r}$. The complementary tail is again a reservoir.
Run the whole construction with $D_{>r}$, using an interval
identification only inside that tail, and extend by $\id_H$.
The native tail inclusions induce invertible full triples by
Lemma~\ref{lem:reservoir-regularity}; transport the prescribed
marking through these triples before the construction and cancel
them afterwards, as in Theorem~\ref{thm:marked-core}.
This fixes $H$ pointwise without identifying an original block
prefix with an interval prefix.

Finally, suppose two distinct factors contain Haar unitaries.
Group one as the first factor and all the others as the second.
Both grouped algebras contain a Haar unitary. Semicircular absorption
\cite[Theorem 5.2]{Shlyakhtenko} gives $P\cong P\fr D$,
after identifying $D$ with its interval presentation by
Corollary~\ref{cor:reservoir-model}.
The same reasoning applies to $Q$. Since the $D$-stabilized algebras
are simple, monotracial and strictly comparing, their positive cone
is $\{0\}\cup\{g:\tau_*(g)>0\}$, so their Elliott invariant is
determined by the pointed groups and selected trace pairing.
To realize an isomorphism $\Theta_0:\Ell(P)\to\Ell(Q)$ after
removing $D$, choose any trace-preserving absorption isomorphisms
$a:P\to P\fr D$ and $b:Q\to Q\fr D$. Realize
$\Ell(b)\Theta_0\Ell(a)^{-1}$ between the $D$-stabilized products and
then compose with $b^{-1}$ and $a$. No canonical marking of the
absorption isomorphisms is required. The two-factor Haar condition
is used only for this removal step; the $D$-stabilized classification
also includes scalar cores and one-factor families.
\end{proof}

\section{General fusion and classification}\label{sec:general-fusion}

This section proves fusion for every nonempty finite or countable
family of separable unital nuclear UCT algebras with faithful selected
traces.  We first fuse zero-$\mathrm K_1$ cores over a finite diagonal.  Tail
control then gives a two-sided countable exhaustion.  The independent
split-unit theorem, Theorem~\ref{thm:split-fusion}, finally separates
and reincorporates arbitrary $\mathrm K_1$-groups.  In all constructions an
originally prescribed finite reservoir head may be kept external.
We first use the interval presentation $D=D_I$;
Corollary~\ref{cor:reservoir-model} supplies the general reservoir
at the end.

\subsection{The relative interval pushout over a finite diagonal}

Let $G\subset N$ be a unital inclusion into a separable unital
$\mathrm C^*$-algebra $N$, where $G$ is finite dimensional and
commutative. Fix faithful selected traces and the
trace-preserving expectation, and put
\[
 R_G=N*_G(G\otimes \mathrm C(\mathbb T_v)),\qquad
 Q_G=N*_G(G\otimes \mathrm C([0,1]_h)),
\]
with product Haar and uniform expectations.

\begin{proposition}[Relative one-turn pushout]\label{gf:relative-pushout}
The canonical map
\[
 R_G*^{\max}_{G\otimes \mathrm C(\mathbb T)}
                       (G\otimes \mathrm C([0,1]))\longrightarrow Q_G
\]
is an isomorphism.  In particular, a faithful representation of
\(R_G\) together with a positive contraction \(K\) commuting with
its represented \(G\) and satisfying \(e^{2\pi iK}=v\) extends to
a contractive homomorphism of \(Q_G\).  If its selected trace on
\(R_G\) is the native one, the extension preserves the native trace
and is faithful.
\end{proposition}

\begin{proof}
Use the continuous branch functions $g_\delta^0,g_\delta^1$
from the proof of Lemma~\ref{rf:lem-pushout}.  Their Haar distributions
are exactly uniform on \([0,1-\delta]\) and \([\delta,1]\),
respectively.  The reduced amalgamated free law therefore gives
\[
 \begin{split}
 \|P(N,g_\delta^0(v))\|&=\|P(N,(1-\delta)h)\|_{Q_G},\\
 \|P(N,g_\delta^1(v))\|&=\|P(N,\delta+(1-\delta)h)\|_{Q_G}.
 \end{split}
\]
Both norms tend to \(\|P(N,h)\|_{Q_G}\).

Here is the endpoint compression needed for arbitrary representations.
Let
\[
 \epsilon=\operatorname{ev}_1 \E_{G\otimes \mathrm C(\mathbb T)}:R_G\to G.
\]
If \(P=1_{\{1\}}(v)\), then
\begin{equation}\label{gf:endpoint}
                         PaP=\epsilon(a)P\qquad(a\in R_G).
\end{equation}
To check it, the conjugates \(v^kNv^{-k}\) are free over \(G\).
For a finite-support word centered under expectation onto \(G\),
widely spaced conjugates have disjoint support and are free over
\(G\).  The creation, annihilation and diagonal decomposition on
the reduced Hilbert \(G\)-module bounds the norm of their average by
\((m^{-1}+2m^{-1/2})\|a\|\).
Compression by \(P\) is invariant under these conjugations.
For a finite Laurent sum \(\sum_kx_kv^k\), this proves
\(P(\sum_kx_kv^k)P=(\sum_k\E_G(x_k))P\), which is
\eqref{gf:endpoint}; density completes the argument.

In a representation of the full pushout put
\(P_s=1_{\{s\}}(K)\), \(s=0,1\).
The subspaces \(\overline{R_GP_s\mathcal H}\) are orthogonal and
reduce \(R_G\), by \eqref{gf:endpoint} and the fact that
the \(P_s\)'s commute with \(G\). The same identity gives
\(P\overline{R_GP_s\mathcal H}=P_s\mathcal H\) for \(s=0,1\).
The two subspaces also reduce \(K\):
away from the \(v=1\) spectral subspace, \(K\) is the unique Borel
argument of \(v\), and on \(P_s\) it is \(s\).
Thus \(g_\delta^s(v)\to K\) strongly on the corresponding subspace;
either branch works on their orthogonal complement.
The preceding branch norm limits give the required upper bound for
every polynomial in \(N,K\).  The canonical reduced quotient gives
the reverse bound, proving the isomorphism.

For the trace assertion, a positive one-turn logarithm differs in
\(2\)-norm from a fixed branch only on an arc whose Haar measure
tends to zero.  Its joint moments with \(N\) therefore equal those
of the native relative interval.  Faithfulness of the latter trace
gives injectivity.
\end{proof}

\subsection{The finite-base endpoint ideal and comparison}

Continue with $R_G,Q_G$ above and finite-dimensional commutative $G$.  Assume that
\(N=N_0\fr D_0\), where \(G\subset N_0\) is expected and \(D_0\) is an
external infinite scalar interval reservoir.  This assumption provides
stability; it is separate from the uniform-logarithm criterion.

Let \(\mathcal H_\epsilon\) be the Hilbert \(G\)-module obtained by
completing \(R_G\) for
\(\langle a,b\rangle_G=\epsilon(a^*b)\), with
\(\epsilon\) as in \eqref{gf:endpoint}.
Its left representation has the useful decomposition
\begin{equation}\label{gf:endpoint-module}
 \mathcal H_\epsilon=G\oplus
       (\ell^2(\mathbb Z)\otimes\mathcal V),\qquad
 \pi_\epsilon(v)=1_G\oplus(S\otimes1),
\end{equation}
where \(S\) is bilateral shift and \(\mathcal V\) is a countably
generated Hilbert \(G\)-module.  To see this, identify the endpoint
GNS module with \(L^2(*_{k\in\mathbb Z,G}N_k,\E_G)\), where
\(N_k=v^kNv^{-k}\).  The vacuum is \(G\); on nonempty reduced
words \(v\) shifts all labels simultaneously.  Group each orbit by
requiring its first label to be zero.  Every such orbit is a copy
of \(\mathbb Z\), proving \eqref{gf:endpoint-module}.

The external reservoir supplies a standard \(\ell^2\otimes G\)
orthogonal summand in this module.  For example the usual free-product
decomposition for \(N_0\fr D_0\) contains
\((L^2(D_0)\ominus\mathbb C)\otimes G\) as a summand.  All remaining
summands are countably generated.  Hilbert-module stabilization
therefore gives \(\mathcal H_\epsilon\cong\ell^2\otimes G\).

\begin{proposition}[The relative two-logarithm ideal]
\label{gf:relative-ideal}
Represent \(R_G\) faithfully, and let \(H_0,H_1\) be positive
contractions commuting with the represented \(G\), with
\(e^{2\pi iH_0}=e^{2\pi iH_1}=v\).  In the generated algebra put
\[
 J_G=\operatorname{Ideal}_{\mathrm C^*(G,H_0,H_1)}(H_1-H_0),\qquad
 I=\operatorname{Ideal}_{\mathrm C^*(R_G,H_0,H_1)}(H_1-H_0).
\]
Then \(J_G\) is a quotient of \(G\otimes J_{\rm pair}\), where
\[
 J_{\rm pair}=\{f\in \mathrm C([0,1],\mathrm M_2):
                       f(0)=0,\ f(1)\text{ diagonal}\},
\]
and
\begin{equation}\label{gf:induced-ideal}
 I\cong\mathcal K_{J_G}(\mathcal H_\epsilon\otimes_G J_G)
                  \cong\mathcal K\otimes J_G.
\end{equation}
In particular \(I\) is nuclear, stable, type I and of finite nuclear
dimension, and has the corona factorization property.
\end{proposition}

\begin{proof}
Off \(P=1_{\{1\}}(v)\), both logarithms are the same Borel argument.
On \(P\) they are two projections commuting with \(G\).
The universal two-projection difference ideal is \(J_{\rm pair}\);
nuclearity of \(G\) gives the asserted tensor-product quotient.
By \eqref{gf:endpoint},
\[
              xay=x\epsilon(a)y\quad(x,y\in J_G,\ a\in R_G).
\]
Thus the rank-one-operator formula
\[
 \theta_{a\otimes x,b\otimes y}\longmapsto axy^*b^*
\]
identifies the compact operators in
\eqref{gf:induced-ideal} with the closed span of these
products.  Faithfulness follows by inducing a faithful representation
of \(J_G\); the endpoint compression identity gives the required
inner products.

Here is why the span is invariant under the logarithms, without
assuming in advance that it is an ideal.  In the faithful generated
representation put
\[
 \mathcal L=\overline{R_GJ_G\mathcal H},\qquad
 P=1_{\{1\}}(v).
\]
The induced-module identification shows that \(\mathcal L\) reduces
\(R_G\).  The shift in \eqref{gf:endpoint-module} has no eigenvector
at \(1\), so \(P\mathcal L=\overline{J_G\mathcal H}\) is precisely
its vacuum part.  This subspace is preserved by each \(H_s\) because
\(J_G\) is an ideal in \(\mathrm C^*(G,H_0,H_1)\).  On
\((1-P)\mathcal L\), both logarithms are the common Borel argument
of \(v\) and hence preserve \(\mathcal L\).  Self-adjointness makes
\(\mathcal L\) reducing for both logarithms.  In this induced
representation, they act on the shift summand
of \eqref{gf:endpoint-module} as the same bounded Borel argument of
\(S\), tensored with the identity.  This is an adjointable operator:
it is a bounded Hilbert-space operator on \(\ell^2(\mathbb Z)\).
On the vacuum summand they act as the two endpoint projections in
\(\mathcal M(J_G)\).  Consequently they normalize the compact-operator
algebra just identified; on \(\mathcal L^\perp\) those compact
operators vanish.  The span is therefore an ideal of the whole
generated algebra, and is exactly \(I\), rather
than merely a subalgebra containing the original difference.

The second isomorphism follows from standardness of
\(\mathcal H_\epsilon\).  The quotient $J_G$ is subhomogeneous of
topological dimension at most one; finite nuclear dimension passes
to its stabilization.  The finite-nuclear-dimension corona-factorisation result
\cite{RF-RobertCFP} applies.
\end{proof}

\begin{theorem}[Comparison over a fixed relative circle]
\label{gf:relative-comparison}
Assume \(Q_G\) is simple and the preceding external-reservoir
hypothesis holds.  Two faithful unital homomorphisms
\(\rho_0,\rho_1:Q_G\to L_\infty\) that agree on \(R_G\) are
unitarily equivalent.  The same holds in a norm ultrapower.

Consequently, if the native \(R_G\) embeds unitally and trace
preservingly in an actual faithfully tracial algebra \(L\) and positive contractions \(H_n\in L\) satisfy
\[
 [H_n,b]\longrightarrow0\quad(b\in G),\qquad
                    e^{2\pi iH_n}\longrightarrow v
\]
in norm, there is an actual faithful trace-preserving embedding
\(f:Q_G\to L\) whose restriction to \(R_G\) is approximately
unitarily equivalent to its native inclusion.
\end{theorem}

\begin{proof}
The difference ideal is the stable nuclear ideal in
Proposition~\ref{gf:relative-ideal}.  If it is zero, the two
maps coincide; suppose it is nonzero.  The two multiplier maps form
a unital full Cuntz pair, by simplicity of \(Q_G\), and agree on
\(N\).  The inclusion \(N\to Q_G\) is a \(\KK\)-equivalence:
contract the interval in the full amalgamated product and use the
full-to-reduced equivalence \cite[Theorem 3.1]{FG} for the specified
GNS-faithful expectations, using \cite[Proposition~2.15]{FG}
to identify the vertex-reduced and ordinary reduced products.
Restriction thus annihilates their actual
\(\KK\)-pair injectively.  Lemma~\ref{rf:lem-proper} gives approximate conjugacy by unitaries in \(1+I\).
The norm estimates hold in the generated algebra because their
errors lie in \(I\).  The sequence-algebra epsilon test gives one
implementing unitary.

For descent, Proposition~\ref{gf:relative-pushout} first gives
\(\Theta:Q_G\to L_\infty\) agreeing with the diagonal native map
on \(R_G\); the trace follows from the same branch estimate.
Every diverging reparameterization still agrees with \(\Theta\)
on this \emph{diagonal} circle algebra.  Apply the comparison just
proved, and then reparameterization descent
\cite[Proposition 9.2]{CGSTW}, as in
Proposition~\ref{rf:prop-descent}, to
obtain an actual \(f\) unitarily equivalent to \(\Theta\) after
diagonal inclusion.  This proves the asserted finite-packet limits
and trace preservation.  Diagonality of the common circle is an
essential hypothesis at this step.
\end{proof}

\subsection{Two exact matrix identities}

\begin{proposition}\label{gf:matrix-identities}
Let $A,B$ contain the same finite-dimensional tracial algebra $F$,
with faithful selected traces and their trace-preserving
expectations.  For $M=A*_F B$, there is a trace-preserving reduced
isomorphism
\begin{equation}
 \mathrm M_2(M)\cong (A\oplus B)*_{F\oplus F}\mathrm M_2(F),
 \label{gf:eq-MT}
\end{equation}
using the half-weighted trace on $A\oplus B$ and
$\tr_2\otimes\tau_F$ on $\mathrm M_2(F)$.
For every faithfully tracial $M$ and $k\geq1$ there is also an
isomorphism
\begin{equation}
 \mathrm M_k(M)\fr D\cong \mathrm M_k(M\fr D)
 \label{gf:eq-GA}
\end{equation}
fixing the labelled $\mathrm M_k(M)$ and preserving the normalized trace.
\end{proposition}
\begin{proof}
For \eqref{gf:eq-MT} embed the first factor diagonally in
$\mathrm M_2(M)$ and the second through the common $F$.
A centered letter from $\mathrm M_2(F)$ is off diagonal.  A centered
diagonal coefficient is $\operatorname{diag}(a,b)$ with
$\E_F(a)=\E_F(b)=0$.  Alternating centered products therefore
follow paths alternating between $A$ and $B$; the defining
$F$-reduced-word rule makes their trace zero.  Thus the displayed
copies have exactly the required reduced free law.  Their matrix
units and diagonal coefficients generate $\mathrm M_2(M)$.  Faithfulness
of the traces proves the reduced norm identity and surjectivity.

For \eqref{gf:eq-GA}, put $R=\mathrm M_k(M)$,
$\mathcal A=R\fr D$ and $P=e_{11}$, using the standard matrix units
$e_{ab}\in R$.  Write $\tau_R=\tr_k\otimes\tau_M$ and let
$\tau_{\mathcal A}$ be the free-product trace.
The canonical trace-preserving expectation $\E_R:\mathcal A\to R$
restricts to an expectation
$\E_M:P\mathcal AP\to PRP\cong M$.
The normalized corner trace is
$\tau_P=k\tau_{\mathcal A}|_{P\mathcal AP}$, with
$\tau_M\circ\E_M=\tau_P$.

Choose the standard semicircular reservoir coordinates $(S_i)_{i\geq1}$.
Their freeness from $R$ gives the $R$-valued Wick rule: conditional
moments are sums over noncrossing pairings, with each pair evaluated by
\[
 \E_R(S_i rS_j)=\delta_{ij}\tau_R(r)1_R\qquad(r\in R).
\]
In particular their first moments and all $R$-valued cumulants
of order greater than two vanish.
Set
\[
 X_{ab}^{(i)}=\sqrt{k}\,e_{1a}S_i e_{b1}\in P\mathcal AP
 \qquad(1\leq a,b\leq k).
\]
Then $(X_{ab}^{(i)})^*=X_{ba}^{(i)}$ and $\E_M(X_{ab}^{(i)})=0$.
For $m\in PRP\cong M$, bimodularity and the Wick rule give
\[
 \begin{aligned}
 \E_M\bigl(X_{ab}^{(i)}mX_{cd}^{(j)}\bigr)
 &=k e_{1a}\E_R\bigl(S_i(e_{b1}me_{1c})S_j\bigr)e_{d1}\\
 &=k\delta_{ij}\tau_R(e_{b1}me_{1c})e_{1a}e_{d1}\\
 &=\delta_{ij}\delta_{bc}\delta_{ad}\tau_M(m)P.
 \end{aligned}
\]
The last equality uses
$\tau_R(e_{b1}me_{1c})=\delta_{bc}\tau_M(m)/k$ and
$e_{1a}e_{d1}=\delta_{ad}P$; this explains the factor $\sqrt{k}$.
For a longer product with coefficients in $M$, substitute the same
matrix-unit expressions into the $R$-valued Wick rule.
Contract the innermost pairs first, using the displayed calculation.
Each contraction gives a coefficient in $PRP$, so the same covariance
applies recursively to every outer pair.  Only pairings occur.
Thus the $X$'s satisfy the $M$-valued Wick
recursion with this covariance; their higher $M$-valued cumulants
vanish as well.

For each $i$, take the self-adjoint variables
\[
 X_{aa}^{(i)}\ (1\leq a\leq k),\qquad
 \frac{X_{ab}^{(i)}+X_{ba}^{(i)}}{\sqrt2},\qquad
 \frac{X_{ab}^{(i)}-X_{ba}^{(i)}}{\mathrm i\sqrt2}\quad(a<b).
\]
There are $k+2\binom{k}{2}=k^2$ of them.
Denoting this whole countable family by $(Y_\lambda)$, the covariance
calculation gives
\[
 \E_M(Y_\lambda mY_\mu)=\delta_{\lambda\mu}\tau_M(m)P.
\]
The change of variables preserves the vanishing of higher cumulants.
Consequently their mixed moments with $M$ are exactly those of
standard scalar semicirculars free from one another and from $M$.
This also follows directly from the Wick recursion, since the
displayed covariance is precisely the covariance for that free family.

We next check that these variables generate the entire corner.
Let $C=\mathrm C^*(M,X_{ab}^{(i)}:i,a,b)\subset P\mathcal AP$;
the $X$'s and the $Y$'s generate the same algebra.
The standard matrix decomposition
\[
 \Psi:\mathcal A\longrightarrow\mathrm M_k(P\mathcal AP),
 \qquad \Psi(x)=[e_{1a}xe_{b1}]_{a,b},
 \qquad
 \Psi^{-1}([c_{ab}])=\sum_{a,b}e_{a1}c_{ab}e_{1b}
\]
is a $^*$-isomorphism, as follows by multiplying matrix units and
using $\sum_a e_{a1}e_{1a}=1$.
Its values on $R$ have entries in $M$, and
$\Psi(S_i)=k^{-1/2}[X_{ab}^{(i)}]_{a,b}$.
Since $R$ and the $S_i$ generate $\mathcal A$, it follows that
$\Psi(\mathcal A)\subset\mathrm M_k(C)$.
Taking the $(1,1)$-corner gives $P\mathcal AP=C$.

Reindexing the $k^2$ semicircular variables for each $i$ by
$\{1,\ldots,k^2\}\times\mathbb N\cong\mathbb N$ identifies this corner with
$M\fr D$, fixing $M$ and preserving $\tau_P$.
For completeness, equality of the mixed moments gives equality of
the reduced polynomial norms: for either faithful trace $\tau$,
\[
 \|x\|=\lim_{n\to\infty}\tau((x^*x)^n)^{1/(2n)}.
\]
Thus the polynomial correspondence extends isometrically to a
trace-preserving isomorphism $\theta:P\mathcal AP\to M\fr D$.
The desired isomorphism and its inverse are explicitly
\[
 x\longmapsto[\theta(e_{1a}xe_{b1})]_{a,b},
 \qquad
 [y_{ab}]\longmapsto
       \sum_{a,b}e_{a1}\theta^{-1}(y_{ab})e_{1b}.
\]
These maps fix the labelled $\mathrm M_k(M)$ entry by entry.
Finally, traciality and $\tau_P=k\tau_{\mathcal A}|_{P\mathcal AP}$
give
\[
 \tau_{\mathcal A}\!\left(\sum_{a,b}e_{a1}c_{ab}e_{1b}\right)
       =\frac1k\sum_a\tau_P(c_{aa}),
\]
which proves preservation of the normalized matrix trace and
completes \eqref{gf:eq-GA}.
The corresponding statement for a full projection is
Lemma~\ref{lem:gaussian-full-corner}.
\end{proof}

\subsection{Relative free transport and actual returns}

We use the following finite-base setup.  Let $B$ be
finite dimensional and $G\subset B$ its labelled full diagonal
algebra.  Let $N$ and $E$ be unital nuclear UCT algebras with
faithful selected traces and trace-preserving embeddings
\[
 a:N\to E,\qquad b:B\to E
\]
agreeing on the specified $G\subset N,B$.
Identify all these copies of $G$.  Put
\begin{equation}\label{gf:eq-ST}
 S=(N*_G B)\fr D,\qquad
 T=(E*_G(G\otimes I_h))\fr D,\qquad q=e^{2\pi ih}.
\end{equation}
The $G$-Haar conjugation rule gives a faithful reduced embedding
\begin{equation}
 \eta|_N=a,\qquad \eta|_B=\Ad(q)b,\qquad \eta|_D=\id.
 \label{gf:eq-FE}
\end{equation}
Indeed $E$ and $q\,b(B)q^*$ are free over $G$, and the ordinary
reservoir is external.

Assume in this setup that $[\eta]$ and $\eta^\ddagger$ are
isomorphisms, that $\eta_T$ is an affine homeomorphism, that
$\mathrm K_1(N)=\mathrm K_1(E)=\mathrm K_1(S)=0$, and that the
working free-reservoir targets are admissible.  They will be verified below for the binary fusion application.
Finite relative tags preserve the relevant $\KK$ groups by
Lemma~\ref{lem:fd-extension-kk}.  Whenever tag absorption is used,
we reserve a fresh external infinite scalar bank: the working
algebra has the form $H\fr D_{\rm fresh}$, with all labelled
finite-dimensional bases contained in $H$.  Thus
Lemma~\ref{lem:fd-tag-absorption} also preserves admissibility
in each such working algebra.
In particular $\mathrm K_1(T)=0$ by the marking of $\eta$;
the finite-tag $\KK$-equivalences and full-corner Morita
equivalence retain this vanishing in the working algebras and
normalized corners used below.

\begin{proposition}\label{gf:actual-returns}
There are actual trace-preserving embeddings $\beta_n:T\to S$
with inverse core marking and $\beta_n\eta\to\id_S$ pointwise.
They may be chosen to fix $G$ exactly.
\end{proposition}
\begin{proof}
Fix finite source coefficient packets and a finite ordinary
reservoir head $D_{\rm head}$.  Keep this head external throughout
the following construction, using independent infinite banks in
its complement. Choose an infinite working bank $D_0$ there and put
$S_{\mathrm{work}}=(N*_G B)\fr D_0$.
Theorem~\ref{thm:embedding} supplies an
inverse-marked $\theta:E\to S_{\mathrm{work}}$.
Theorem~\ref{thm:embedding}, applied to the whole nuclear $N$, gives inner-corrected
$\theta_n$ with $\theta_na\to\id_N$.
Correct their finite $G$-frames by unitaries tending to $1$,
so that $\theta_n|_G=\id_G$ exactly.

Reserve a $G$-central relative Haar unitary $v$, realized by the
finite-dimensional tag construction in a fresh scalar interval.
Conjugate the old $B$ by $v$, fix $N$, and move the old ordinary
working bank into a fresh external bank $D_1$.  Reduced
$G$-freeness gives an actual embedding $\sigma$ with these
prescriptions. To check the retained-frame assertion directly, put
$M=(N*_G B)\fr D_1$ and $A=\mathrm C^*(G,D_1)$.
Reduced reassociation makes $N,B,A$ free over $G$.
In the fresh $G$-central Haar extension, $M$ and $vMv^*$
are free over $G$: alternating centered words expand to reduced
words containing nonzero powers of $v$.
Thus $N,B,A,vBv^*$ form a free family over $G$.
Consequently $\mathrm C^*(N,vBv^*,D_1)$ has the prescribed
source law and is free over $G$ from the retained $B$.
Faithful traces give the asserted actual embedding.
Its image core
$E'=\sigma\theta_n(E)$ remains root-good and is free over $G$
from the retained old $B$.  On the nuclear $E$, its inverse
marking is unchanged: the $\mathrm K_0$ action and trace are unchanged,
and both source and ambient $\mathrm K_1$ vanish, eliminating the
degree-zero UCT Ext terms.

Let $C=\sigma\theta_nb(B)\subset E'$ and write the corresponding
labelled matrix units as $c_{ij}^{(\ell)},b_{ij}^{(\ell)}$.
Their diagonal projections already agree exactly in $G$.
Consequently the explicit element
\[
 w_0=\sum_{\ell,i}b_{i1}^{(\ell)}c_{1i}^{(\ell)}
       \in \mathrm C^*(B,C)
\]
is unitary, commutes with $G$, and satisfies
$w_0c_{ij}^{(\ell)}w_0^*=b_{ij}^{(\ell)}$.
These identities follow directly by multiplying matrix units.
No projection cancellation or additional $\mathrm K_0$-existence statement
is needed for this frame-matching step.

Here the free-implementer lemma has a relative version.
Over the whole working algebra containing $E'$, the old $B$
and $w_0$, take a fresh $C$-amalgamated extension with
$C$-central Haar $z$, free from that whole algebra over $C$, and
put $W=w_0z$.  Then
\begin{equation}
 W\in G',\qquad WCW^*=B,\qquad
 \mathrm C^*(E',W)\cong E' *_G(G\otimes \mathrm C(\mathbb T)).
 \label{gf:eq-RT}
\end{equation}
Indeed $E'$ and $B$ are free over $G$, and $w_0$ fixes $G$
pointwise by conjugation. Lemma~\ref{lem:fd-free-implementer}
applies in the reduced extension over $C$ with coefficient algebra
$\mathrm C^*(E',B)$. The trace-preserving expectation onto $C$
restricts correctly from the whole working extension, so the
subextension embeds there. The $C$-central tag is actual by
the finite-dimensional tag embedding.

Independently normalize the normalized determinant of $W$
in every minimal $G$-corner to $[1/2]$, by replacing $W$ by
$dW$, $d\in Z(G)$.  All corner $\mathrm K_1$ classes vanish.
This preserves the $G$-Haar free law, while rotating the old
labelled $B$-frame by $\Ad(d)$.  That rotation can be undone:
the source map which fixes $N,D$ and acts by $\Ad(d)$ on $B$
is an actual trace-preserving automorphism, since it fixes
$G$ pointwise.  A path of diagonal phases shows it is homotopic
to the identity, so its total $\mathrm K$-map is the identity.
Its trace preservation
and $\mathrm K_1(S)=0$ retain the full unitary marking as well.

Reserve an additional external infinite bank $D_3$.  The
coefficient $E'\fr D_3$ and the fixed actual $dW$-circle satisfy
the relative actual-model theorem
(Theorem~\ref{gf:relative-comparison}, applied to
$G=\mathrm C(\{\text{finitely many points}\})$).
The native marked $G$-central interval can be made in a fresh
bank by
\[
 Y=\sum_j\tau(e_j)^{-1/2}e_j s e_j
\]
and the semicircle distribution function.  Nuclear-source
uniqueness for $G\otimes \mathrm C(\mathbb T)$ supplies approximate
logarithms for the fixed circle, and the fixed-relative-circle
comparison descends them to actual embeddings of the full
interval domain.  In particular their whole coefficient
restrictions are approximately unitarily equivalent to the
native one.

Choose the root and frame errors tending to zero, and undo
the diagonal source gauge by postcomposing the return with its
inverse. The gauge fixes $N$ and the protected reservoir head
pointwise, so this retains the
native root estimates and restores the old whole $B$-frame.
Use $D_3$ for the unprotected target tail and extend the actual
return by the identity on $D_{\rm head}$.  Thus each return
approximates the native map on the chosen $N$-packet, the whole
finite $B$-frame, and this head.  Letting these packets and heads
increase gives $\beta_n\eta\to\id_S$ by word approximation.  Near-identity finite-frame corrections make
$G$ exact, and later make the whole returned finite $B$-frame
exact when adjoining its tags.  The inverse core marking is
retained by coefficient approximate unitary equivalence and
the homotopically trivial source gauges.
\end{proof}

\subsection{Simultaneous corner phases in target capture}

\begin{theorem}[Finite-base finite-dimensional attachment]
\label{gf:finite-base-fusion}
In the finite-base setup just specified, with $S,T$ as
in~\eqref{gf:eq-ST}, there is a marked trace-preserving isomorphism
\[
                       S\cong T\cong E\fr D.
\]
\end{theorem}
\begin{proof}
Add a $G$-central uniform root tag $H_Z$ and a whole
$B$-central uniform leaf tag $H_Y$ to the source; on the target
the latter tag is central over $B_0=q\,b(B)q^*$.
Write $S^+,T^+,\eta^+$ for the resulting diagrams and
$Z=e^{2\pi iH_Z}$, $Y=e^{2\pi iH_Y}$.
The actual returns extend, after the near-identity whole
$B$-frame correction, to
\begin{equation}
 \beta_n^+:T^+\to S^+,\qquad
 \beta_n^+\eta^+\to\id_{S^+},\qquad \beta_n^+(T)\subset S.
 \label{gf:eq-ER}
\end{equation}
All finite-frame corrections occur in the original $S$.
The tag extensions send each target $H_Z,H_Y$ to its source
counterpart exactly. Hence, with $\phi_n=\eta^+\beta_n^+$,
\[
 \phi_n(H_Z)=H_Z,\qquad \phi_n(H_Y)=H_Y,\qquad
 \phi_n|_{B_0}=\id_{B_0}.
\]
Put $q_{1,n}=\phi_n(q)$.
The core and circle images in this notation lie in the
original $T$, and $\phi_n|_G=\id_G$.

Theorem~\ref{thm:embedding}, on the nuclear core $E$, gives
core-alignment unitaries.
Correct on the finite diagonal to choose them in $G'$:
\[
 u_n\phi_n(e)u_n^*\longrightarrow e\quad(e\in E).
\]
Averaging $q_{1,n}u_n^*q^*$ over the unitary group of $B_0$
and taking polar parts gives
\begin{equation}
 v_n\in B_0'\cap T,\qquad
             v_n-q_{1,n}u_n^*q^*\longrightarrow0.
 \label{gf:eq-AL}
\end{equation}
The source-good whole leaf and the core alignment justify
this averaging. Indeed, for $A_n=q_{1,n}u_n^*q^*$,
the exact identity $\phi_n|_{B_0}=\id$ and core alignment imply
$\|[A_n,b]\|\to0$ on the finite matrix frame of $B_0$.
Finite dimensionality makes this uniform for $b\in\mathrm U(B_0)$;
the average differs from $A_n$ by at most that commutator bound.
It is therefore invertible for large $n$, and its polar unitary
lies in $B_0'\cap T$ and has the asserted norm estimate.

Here independent root corner phases cannot be discarded.
Write $G=\bigoplus_j\C e_j$ and let $b_\ell$ be the central
supports of $B_0$.  Since $G$ is its full diagonal, each
$b_\ell$ is a sum of some $e_j$'s.  Equivalent minimal
projections in a block have the same normalized determinant
for $v_n$, since $v_n$ commutes with the whole block.
Let $r_{\ell,n}$ represent this common determinant and put
\[
 c_n=\sum_\ell e^{2\pi i r_{\ell,n}}b_\ell
       \in Z(B_0)\subset G.
\]
Because $q$ commutes with $G$, these block supports are
also the central supports of $C_B=b(B)$; hence
$c_n\in Z(C_B)$.
Every normalized $B_0$-corner determinant of
$v_n'=v_nc_n^*$ is exactly zero.

Both $q$ and $q_{1,n}$ are exponentials of actual uniform
intervals in each normalized $G$-corner.  Their corner
determinants are $[1/2]$.  Equation \eqref{gf:eq-AL} therefore
gives
\[
 \Delta_{e_j}(u_n)+\Delta_{e_j}(v_n)\longrightarrow0.
\]
Thus every normalized $G$-corner determinant of $c_nu_n$
tends to zero.  Choose diagonal phases $d_n\in Z(G)$,
$d_n\to1$, so that
\[
 u_n''=d_nc_nu_n
\]
has exactly zero determinant in every $G$-corner.  This uses
only that a closed trace-range subgroup is either discrete
or all of $\mathbb R$.  Then
\begin{equation}
 u_n''\phi_n(e)(u_n'')^*-c_nec_n^*\longrightarrow0,\qquad
 v_n'q-q_{1,n}(u_n'')^*\longrightarrow0 .
 \label{gf:eq-CP}
\end{equation}
For completeness, with $\epsilon_n=\|v_n-A_n\|$ and
$\delta_n=\|d_n-1\|$, the two errors are bounded by
\[
 \begin{aligned}
 \|u_n''\phi_n(e)(u_n'')^*-c_nec_n^*\|
 &\leq\|u_n\phi_n(e)u_n^*-e\|+2\delta_n\|e\|,\\
 \|v_n'q-q_{1,n}(u_n'')^*\|&\leq\epsilon_n+\delta_n.
 \end{aligned}
\]
Stable rank one and zero corner $\mathrm K_1$ turn these zero
determinants into the full corner $\CU$ conditions.

Set $H_X=q^*H_Yq$ and $X=e^{2\pi iH_X}$.
Reassociation writes the tag coefficient as
\[
 \mathcal N=(E*_G(G\otimes I_{H_Z}))
                *_{C_B}(C_B\otimes I_{H_X}).
\]
On the faithful relative circle algebra
\[
 R^+=(\mathcal N*_G(G\otimes \mathrm C(\mathbb T_q)))\fr D
\]
define
\begin{equation}
 \Gamma|_{\mathcal N}=\Ad(Z),\qquad
 \Gamma(q)=qXZ^*=YqZ^*,\qquad \Gamma|_D=\id.
 \label{gf:eq-GC}
\end{equation}
This is a relative Nielsen automorphism.
In each normalized $G$-corner its multiplier has determinant
$\Delta(X)-\Delta(Z)=[1/2]-[1/2]=0$.
Lemma~\ref{lem:fd-relative-nielsen} therefore gives
actual automorphisms $\Gamma_n$ of $T^+$ approaching this
prescription on arbitrary finite packets in $R^+$, with full
identity marking.  One first extends $\Ad(Z)$ on the
coefficient and then performs the right-multiplier move.

There is a source automorphism $\gamma_S$ acting by $\Ad(Z)$
on $N$, by $\Ad(Y)$ on the whole $B$, and fixing $D$ and
the tag logarithms.  Reassociate the two relative tag
extensions, with the ordinary $D$ external, as
\[
 N*_G(B\otimes I_{H_Y})*_G(G\otimes I_{H_Z}).
\]
The map is partial conjugation of $N$ by the $G$-central $Z$;
its inverse is the corresponding conjugation by $Z^*$.
In this whole-finite-dimensional-leaf case $Y$ commutes with
$B$ itself. The path $Z_t=e^{2\pi itH_Z}$ gives trace-preserving
automorphisms from the identity to $\gamma_S$. Homotopy retains
total $\mathrm K$-theory; trace preservation and
$\mathrm K_1(S^+)=0$ retain the full unitary map.
The exact circle identity is
\begin{equation}
                    \Gamma\eta^+=\eta^+\gamma_S.
 \label{gf:eq-SC}
\end{equation}
In particular $\eta^+(S^+)\subset R^+$.

For the phase correction let $\delta_{c_n}$ be the actual
automorphism acting by $\Ad(c_n)$ on $\mathcal N$ and fixing
$q,D$. It fixes both tag logarithms, since $c_n\in G\cap Z(C_B)$.
The path
$c_n(t)=\sum_\ell e^{2\pi it r_{\ell,n}}b_\ell$
shows that its total $\mathrm K$-map is the identity; trace
preservation and $\mathrm K_1(T^+)=0$ give identity full marking.
Set $\Gamma_{c_n}=\delta_{c_n}\Gamma$ and take simply
\[
             \Gamma_{c_n,n}=\delta_{c_n}\Gamma_n.
\]
Postcomposition is isometric, so the same finite-packet accuracy
approximates $\Gamma_{c_n}$. Its circle formula in
\eqref{gf:eq-GC} is unchanged, because $\delta_{c_n}$ fixes
$X,Z,q$.

Use Lemma~\ref{lem:fd-relative-nielsen} once more:
let $\Theta_{Z,n}$ approximate $Z\mapsto Zu_n''$, and let
$\Theta_{Y,n}$ approximate $Y\mapsto Yv_n'$.
The exact corner conditions were checked above.
Choose all coefficient packets after the moving core and
circle images are known.  The actual maps
\[
 F_n=\Theta_{Z,n}\phi_n\Gamma_n|_T,\qquad
 G_n=\Theta_{Y,n}\Gamma_{c_n,n}|_T
\]
then have the following common images in the norm sequence
algebra:
\[
 \begin{aligned}
 F(e)&=Zc\,e\,c^*Z^*=G(e) &&(e\in E),\\
 F(q)&=Yq_1u^*c^*Z^*=Yv'qZ^*=G(q),\\
 F(d)&=d=G(d) &&(d\in D).
 \end{aligned}
\]
These equalities mean that the coordinate norm differences
tend to zero; $c_n$ is allowed to move.
No equality on the added $Y$ logarithm is needed, since the
comparison domain is the original $T$.

The common relative circle is the faithful copy of
\[
 (E\fr D)*_G(G\otimes \mathrm C(\mathbb T)).
\]
Only the two one-turn logarithms of its one interval differ.
Theorem~\ref{gf:relative-comparison} applies: its coefficient
has an external infinite reservoir, its source is simple by
finite-tag absorption, and the native coefficient inclusion
is a $\KK$-equivalence.  Thus the actual Cuntz-pair class
vanishes.  Here we use the theorem's first, comparison assertion,
which requires a common faithful relative circle.  Its separate
descent assertion is not being used, so the common image need
not be diagonal and the moving $c_n$ is harmless.
Lift the comparison unitary and diagonalize to get
actual $V_n\in T^+$ with
\[
 \Ad(V_n)\Theta_{Z,n}\eta^+\beta_n^+\Gamma_n(x)
       -\Theta_{Y,n}\Gamma_{c_n,n}(x)\longrightarrow0
       \qquad(x\in T).
\]
Put $H_n=\Theta_{Y,n}\Gamma_{c_n,n}$ and define the actual
trace-preserving automorphism
\[
 a_n=H_n^{-1}\Ad(V_n)\Theta_{Z,n}.
\]
Every factor has identity full marking.  The preceding comparison gives
\[
       a_n\eta^+\beta_n^+\Gamma_n(x)\longrightarrow x
       \qquad(x\in T).
\]

Now put $b_n=\gamma_S^{-1}\beta_n^+\Gamma_n$.
Equations \eqref{gf:eq-ER} and \eqref{gf:eq-SC} give
$b_n\eta^+\to\id$ on $S^+$.
For a fixed $n$ and finite target packet $L\subset T$,
choose a second actual lift $\Gamma_m$ accurately on the
already specified moving packet $\eta^+b_n(L)\subset R^+$.
Then
\[
 \Gamma_m\eta^+b_n(x)\approx
 \Gamma\eta^+b_n(x)
 =\eta^+\gamma_Sb_n(x)
 =\eta^+\beta_n^+\Gamma_n(x).
\]
Consequently $\alpha_n=a_n\Gamma_m$ supplies target capture
for that same source-good return.  Every inverse is an
actual coordinate inverse; none is passed through a limit.

The exactly natural finite-tag absorption of
Lemma~\ref{lem:fd-tag-absorption}
removes the added $Z,Y$ tags using common ordinary banks.
Given finite source and target packets, first approximate them
by polynomials using a common finite ordinary reservoir prefix
$D_{\rm head}$.  Put
\[
 H_S=(N*_G B)\fr D_{\rm head},\qquad
 H_T=(E*_G(G\otimes I_h))\fr D_{\rm head}.
\]
Then $\eta(H_S)\subset H_T$, the labelled bases $G,B$ map
onto $G,B_0$, and $\eta$ fixes the common unused scalar banks.
In particular the original interval $h$ belongs to $H_T$.
Choose the same bank partitions and reindexings in the absorption
lemma on both sides.  This gives label-natural isomorphisms
$J_S:S^+\to S$, $J_T:T^+\to T$ fixing $H_S,H_T$, respectively,
and satisfying
\[
                      J_T\eta^+=\eta J_S.
\]
Fix these isomorphisms, and apply the preceding construction
to the protected polynomial packets to choose an enlarged
return $b^+$ and enlarged correction $\alpha^+$.  Transport them by
\[
 b=J_Sb^+J_T^{-1},\qquad
 \alpha=J_T\alpha^+J_T^{-1}.
\]
For protected polynomials $s\in H_S$ and $t\in H_T$, the
commuting square gives
\[
 b\eta(s)=J_Sb^+\eta^+(s),\qquad
 \alpha\eta b(t)=J_T\alpha^+\eta^+b^+(t).
\]
Hence the source and target errors are transported isometrically.
Polynomial approximation now shows that the original $S,T$ have,
for every pair of finite
packets, a common actual return $b$ and actual target automorphism
$\alpha$ with
\[
 b\eta\approx\id_S,\qquad \alpha\eta b\approx\id_T.
\]
All maps have the required full markings: the returns have
the inverse core marking; the Nielsen lifts have identity
marking by their corner conditions; coefficient conjugacies
and inner comparisons preserve it.  The source gauges are
homotopic to the identity.  Native finite-tag inclusions
are $\KK$-equivalences, and zero $\mathrm K_1$ and the selected trace
retain the determinant compatibility.

To verify~\eqref{eq:caap} for $\eta(S)\subset T$, include
$\eta(F)$ in the target packet whenever $F\subset S$ is the
source packet.  Then
\[
 \|\alpha\eta(s)-\eta(s)\|
 \leq\|s-b\eta(s)\|
       +\|\alpha\eta b(\eta(s))-\eta(s)\|
                         \qquad(s\in F).
\]
The same $\alpha$ places each target test element near
$\alpha\eta(S)$.  Lemma~\ref{lem:marked-caap} therefore gives
$S\cong T$ with the full marking of $\eta$.  Finally finite-dimensional tag
absorption on $E\fr D$ gives $T\cong E\fr D$, fixing $E$.
\end{proof}

\subsection{\texorpdfstring{Zero-$\mathrm K_1$}{Zero-K1} binary reduction and exact-unit compression}

\begin{proposition}\label{gf:zero-K1-binary-fusion}
Let $A_1,A_2$ be separable unital nuclear UCT algebras with faithful
selected traces and $\mathrm K_1(A_i)=0$.  Their free product, after adjoining a reservoir, is
trace-preservingly isomorphic to a nuclear-core reservoir model
with the aggregate selected pointed data.
\end{proposition}
\begin{proof}
Write $u_i=[1_{A_i}]$ and
\[
 G_0=\frac{\mathrm K_0(A_1)\oplus \mathrm K_0(A_2)}
                 {\langle(u_1,-u_2)\rangle},\qquad
 u=[(u_1,0)]=[(0,u_2)].
\]
The aggregate pairing sends $[(x_1,x_2)]$ to
$\tau_{1*}(x_1)+\tau_{2*}(x_2)$ and sends $u$ to $1$.
Choose the simple monotracial nuclear UCT range
core $E_0$ with $\mathrm K_0(E_0)=G_0$, unit $u$, this pairing and
$\mathrm K_1(E_0)=0$, using its strict trace cone and \cite[Theorem 2.8]{EGLN}.
Theorem~\ref{thm:embedding} supplies actual unital trace-preserving embeddings
$a_i:A_i\to E_0$ with the canonical markings.  Positivity uses
the range core's strict trace cone; no identification of the
original cores' ordered $\mathrm K_0$-groups is required.

Apply the finite-base setup with
\[
 N=A_1\oplus A_2,\qquad B=\mathrm M_2,\qquad G=\C^2,\qquad
 E=\mathrm M_2(E_0),\qquad a=\operatorname{diag}(a_1,a_2),
\]
and the standard embedding $b:\mathrm M_2\to \mathrm M_2(E_0)$.
The trace on $N$ is half-weighted and $G$ is the common full
diagonal.  The map in the amalgamated $\mathrm K$-theory sequence has columns
\[
 (u_1,0,-h),\qquad (0,u_2,-h)
       \quad\text{in }\mathrm K_0(A_1)\oplus \mathrm K_0(A_2)\oplus\mathbb Zh,
\]
where $h$ is the minimal matrix class.  It is injective:
the selected trace on each $A_i$ forces the corresponding integer
coefficient to vanish.  Its cokernel is $G_0$, with $h=u$.
Thus $\mathrm K_1(S)=0$, the amplified unit is $2u$, and the liberated
embedding $\eta$ is a $\mathrm K$-isomorphism with the required selected
trace.  In the amplified diagram the pairing is half the aggregate
pairing, so it sends $2u$ to $1$ and $u$ to $1/2$.
The full-to-reduced $\KK$ equivalence and natural coefficient
sequences give its total $\mathrm K$-marking; zero $\mathrm K_1$ and the selected
trace give the unitary/determinant part of that marking.  Its full
selected marking is therefore an isomorphism.

The working source targets are admissible by the infinite external
reservoir and the standard matrix identity.  The target and its
finite tag extensions are admissible by finite-dimensional tag
absorption.  Theorem~\ref{gf:finite-base-fusion} applies.
Writing $P=A_1\fr A_2$, the exact identities in
Proposition~\ref{gf:matrix-identities} give
\[
 \mathrm M_2(P\fr D)\cong \mathrm M_2(E_0\fr D).
\]
Under the constructed marking the distinguished projection
$e_{11}$ has class \emph{exactly} $u$ on both sides.
Indeed $\eta(e_{11})=e_{11}$, since $e_{11}\in G$ and $q$
commutes with $G$. Its class
is not inferred just by halving the amplified unit $2u$; that
inference would lose possible two-torsion.  Stable-rank-one
cancellation in the target, for the projection and its complement,
gives a unitary correcting its image to the standard $e_{11}$.
Compressing this actual isomorphism gives
\[
                         (A_1\fr A_2)\fr D\cong E_0\fr D
\]
with the selected unamplified unit and pairing.
\end{proof}

\subsection{Binary fusion with control on the old stabilized core}

\begin{proposition}[Tail-controlled binary fusion]
\label{gf:tail-controlled-binary-fusion}
Let $R,A$ be nuclear UCT cores with faithful selected traces and
$\mathrm K_1(R)=\mathrm K_1(A)=0$.  Let $E$ be a simple monotracial nuclear UCT
range core for their aggregate selected pointed data, and choose
an actual trace-preserving unital embedding $\lambda:R\to E$
inducing the canonical old-core $\mathrm K_0$-map.
For an interval reservoir $D$, finite $\mathcal Y\subset R\fr D$
and $\varepsilon>0$, there is a trace-preserving isomorphism
\[
 \psi:(R\fr A)\fr D\longrightarrow E\fr D
\]
inducing the canonical aggregate marking and satisfying
\begin{equation}
 \|\psi(y)-(\lambda\fr \id_D)(y)\|<\varepsilon
                         \qquad(y\in\mathcal Y).
 \label{gf:eq-TB}
\end{equation}
It may fix any specified finite external reservoir head pointwise.
\end{proposition}
\begin{proof}
Approximate the elements of $\mathcal Y$ by free-product
polynomials using finitely many $R$-coefficients and finitely many
reservoir intervals.  Include those intervals and the prescribed
protected head in a finite head $H$, and write $D=H\fr D_{\rm tail}$
with an infinite complementary tail.  Run
Proposition~\ref{gf:zero-K1-binary-fusion} only in that tail, obtaining
\[
 \chi:(R\fr A)\fr D_{\rm tail}\overset{\cong}{\longrightarrow}
                          E\fr D_{\rm tail}
\]
with the canonical aggregate marking, using marked core replacement
(Theorem~\ref{thm:marked-core}) to select the displayed core $E$.

The maps $\chi|_R$ and the native $\lambda:R\to E\subset
E\fr D_{\rm tail}$ have the same full selected marking.  Here their
ordinary $\mathrm K_0$-maps and selected traces agree, while both source
and target have zero $\mathrm K_1$. Writing $U=E\fr D_{\rm tail}$,
the degree-zero UCT extension terms are
\[
 \Ext^1_{\Z}(\mathrm K_1(R),\mathrm K_0(U))
 \oplus\Ext^1_{\Z}(\mathrm K_0(R),\mathrm K_1(U))=0.
\]
Thus their $\KK$ classes agree, the coefficient maps are forced by
the Bockstein sequences, and the compatible unitary map is determined by trace
and $\mathrm K_0$.  There is no additional old-core rotation choice.
Nuclear-source uniqueness on $R$ gives a unitary
$w\in E\fr D_{\rm tail}$ such that $\Ad(w)\chi$ is as close to
$\lambda$ as required on those finitely many $R$-coefficients.

Freely extend $\Ad(w)\chi$ by $\id_H$ and reassociate the
scalar free products.  This is the required actual isomorphism
$\psi$.  It fixes $H$ exactly, and polynomial continuity gives
\eqref{gf:eq-TB}, including the initial polynomial approximation
errors.  The inner correction is made before extending over $H$;
it does not conjugate the protected reservoir coordinates.
\end{proof}

This is why finite-head protection suffices for countable coherence.
Although uniqueness is applied only to the nuclear algebra $R$,
it controls an arbitrary finite packet of $R\fr D$: its finitely many
reservoir coordinates are kept external to that uniqueness step.

\subsection{A two-sided finite-prefix construction}

\begin{proposition}[Countable zero-$\mathrm K_1$ fusion]
\label{gf:countable-zero-K1-fusion}
Let $((A_i,\tau_i))_{i\geq1}$ be nuclear UCT cores with faithful
selected traces and $\mathrm K_1(A_i)=0$.  For an interval reservoir $D$,
there is a nuclear UCT core
$E_\infty$ with their aggregate selected pointed data and an actual
trace-preserving isomorphism
\[
                  \left(\mathop{\fr}_{i\geq1}A_i\right)\fr D
                          \cong E_\infty\fr D.
\]
\end{proposition}
\begin{proof}
First use an interval reservoir.  Put $P_n=A_1\fr \cdots\fr A_n$.
The free-product $\mathrm K$-sequence gives
\[
 G_n=\mathrm K_0(P_n)=
 \frac{\bigoplus_{i\leq n}\mathrm K_0(A_i)}
      {\langle [1_{A_1}]-[1_{A_i}]:2\leq i\leq n\rangle},
 \qquad \mathrm K_1(P_n)=0.
\]
The canonical maps $G_n\to G_{n+1}$ are injective.  Indeed, in
$(G_n\oplus \mathrm K_0(A_{n+1}))/\langle(u_n,-u_{n+1})\rangle$,
a class $(x,0)$ can vanish only if
$(x,0)=k(u_n,-u_{n+1})$; applying the selected trace on the new
factor gives $k=0$.

Choose simple monotracial nuclear UCT $\mathcal Z$-stable range
cores $E_n$ for $(G_n,u_n,\rho_n)$, with their strict trace cones
and zero $\mathrm K_1$.  Theorem~\ref{thm:embedding} supplies trace-preserving unital embeddings
\[
 \lambda_n:E_n\longrightarrow E_{n+1}
\]
inducing the canonical maps of the selected pointed data.
Their limit $E_\infty=\varinjlim(E_n,\lambda_n)$ is simple,
monotracial, nuclear, UCT and $\mathcal Z$-stable.  Continuity of
ordinary $\mathrm K$-theory gives the aggregate pointed groups and pairing;
a class of positive trace is positive at a finite stage, so its
order is the strict trace cone.

Set
\[
 S_n=P_n\fr D,\qquad T_n=E_n\fr D,\qquad
 i_n:S_n\hookrightarrow S_{n+1},\qquad
 j_n=\lambda_n\fr \id_D:T_n\hookrightarrow T_{n+1}.
\]
These are actual trace-preserving embeddings.  Start with a
canonically marked isomorphism $\phi_1:S_1\to T_1$ given by
nuclear-core replacement.  Suppose $\phi_n$ has been constructed.
Freely extending it by the identity on $A_{n+1}$ gives an actual
isomorphism
\[
 \phi_n\fr \id_{A_{n+1}}:S_{n+1}\longrightarrow
                              (E_n\fr A_{n+1})\fr D.
\]
For any prescribed finite $\mathcal G_n\subset T_n$,
Proposition~\ref{gf:tail-controlled-binary-fusion} gives an
isomorphism
\[
 \psi_n:(E_n\fr A_{n+1})\fr D\longrightarrow T_{n+1}
\]
as close to $j_n$ as required on $\mathcal G_n$.  Define
\[
              \phi_{n+1}=\psi_n(\phi_n\fr \id_{A_{n+1}}).
\]
Thus every $\phi_n$ is an actual isomorphism onto its entire stage.

Here are explicit two-sided exhaustion conditions.  Fix positive
$\varepsilon_n$ with $\sum_n\varepsilon_n<\infty$.
After $\phi_n$ is known, choose finite
$\mathcal F_n\subset S_n$ and put
$\mathcal G_n=\phi_n(\mathcal F_n)$.  Require
\begin{equation}
 i_{n-1}(\mathcal F_{n-1})\subset\mathcal F_n,
 \qquad
 j_{n-1}(\mathcal G_{n-1})\subset\mathcal G_n.
 \label{gf:eq-EX}
\end{equation}
Also include successively more points from a fixed countable dense
subset of every earlier $S_k$, propagated by the $i$'s, and of
every earlier $T_k$, propagated by the $j$'s.  The second condition
and all target dense points can be included simply by adding their
$\phi_n$-preimages to $\mathcal F_n$.  Every stage has only
finitely many requirements.

Choose $\psi_n$ accurate within $\varepsilon_n$ on
$\mathcal G_n$.  Then
\begin{equation}
 \|\phi_{n+1}i_n(x)-j_n\phi_n(x)\|<\varepsilon_n
                           \quad(x\in\mathcal F_n).
 \label{gf:eq-AI-plus}
\end{equation}
Applying the actual isometric inverse $\phi_{n+1}^{-1}$ gives
the equally important reverse estimate
\begin{equation}
 \|\phi_{n+1}^{-1}j_n(y)-i_n\phi_n^{-1}(y)\|
       <\varepsilon_n\quad(y\in\mathcal G_n).
 \label{gf:eq-AI-minus}
\end{equation}
The propagation conditions \eqref{gf:eq-EX} and the summable
estimates make the forward and inverse maps Cauchy on the two
dense exhaustions.  The coordinate maps are contractive
homomorphisms, so convergence extends from these dense sets to
pointwise convergence and preserves the algebraic operations.
They extend to trace-preserving homomorphisms
between the two limits.  On a test point propagated to stage $n$,
either composition differs from the identity by at most
$2\sum_{k\geq n}\varepsilon_k$.  Letting $n\to\infty$ proves
that these maps are mutual inverses.  This proves surjectivity,
not just the existence of a limiting embedding.

Finally,
\[
 \varinjlim(S_n,i_n)=\left(\mathop{\fr}_{i\geq1}A_i\right)\fr D,
 \qquad
 \varinjlim(T_n,j_n)=E_\infty\fr D.
\]
These norm identifications follow directly from faithful reduced
embeddings and density of finite free-product words.  On the
target side, the copies of $E_n$ have dense union in $E_\infty$;
on the source side every word uses only finitely many factors.
The constructed isomorphism has the canonical pointed groups and
pairing, since all stage maps have those markings.
\end{proof}

\subsection{Arbitrary \texorpdfstring{$\mathrm K_1$}{K1} and the general fusion theorem}

\begin{theorem}[Fusion of finite and countable nuclear families]
\label{thm:general-fusion}\label{gf:general-fusion}
Let $((A_i,\tau_i))_{i\in J}$ be a nonempty finite or countable
family of separable unital nuclear UCT algebras with faithful
selected traces.  Let $F$ be a nuclear UCT core with faithful
selected trace having the same selected pointed graded $\mathrm K$-groups
and pairing as $P=\mathop{\fr}_{i\in J}A_i$.  For every reservoir
$D$ there is a trace-preserving isomorphism
\[
                              P\fr D\cong F\fr D.
\]
Every prescribed compatible full marking can be realized, and a
specified finite external reservoir head can be fixed pointwise.
\end{theorem}
\begin{proof}
For a finite zero-$\mathrm K_1$ family, iterate
Proposition~\ref{gf:zero-K1-binary-fusion} in independent banks.
For a countable zero-$\mathrm K_1$ family use
Proposition~\ref{gf:countable-zero-K1-fusion}.  We now remove the
$\mathrm K_1$ restriction, writing $J=\mathbb N$ in the countable case;
the same construction applies to a finite index set.  Partition a working interval reservoir into
mutually free infinite banks $D_i$ and one spare infinite bank.
For each $i$, choose nuclear range cores $A_i^0,C_i$ with
$A_i^0$ retaining the selected $\mathrm K_0$, unit and pairing of $A_i$,
$\mathrm K_1(A_i^0)=0$, and $\mathrm K_0(C_i)=\mathbb Z[1_{C_i}]$,
$\mathrm K_1(C_i)=\mathrm K_1(A_i)$.  Theorem~\ref{thm:split-fusion}, applied
to $A_i^0,C_i$ with nuclear target $A_i$, supplies an actual
marked isomorphism $(A_i^0\fr C_i)\fr D_i\cong A_i\fr D_i$.  Invert
it to obtain
\[
 A_i\fr D_i\cong(A_i^0\fr C_i)\fr D_i',\qquad
 \mathrm K_1(A_i^0)=0,\quad
 \mathrm K_0(C_i)=\mathbb Z[1_{C_i}],\quad \mathrm K_1(C_i)=\mathrm K_1(A_i),
\]
where $A_i^0$ retains the selected $\mathrm K_0$, unit and pairing of
$A_i$.  Their countable reduced free product is an actual
isomorphism: the trace-preserving isomorphisms give isometries
on each finite free subproduct, and these finite subproducts have
dense union.  Regrouping the reservoir banks therefore gives
\[
 P\fr D\cong
 \left(\mathop{\fr}_{i\in J}A_i^0\right)\fr 
 \left(\mathop{\fr}_{i\in J}C_i\right)\fr D.
\]
Fuse the $A_i^0$ in a spare bank, using
Proposition~\ref{gf:countable-zero-K1-fusion} in the countable case
and finite iteration of Proposition~\ref{gf:zero-K1-binary-fusion}
otherwise.  Extend the resulting actual isomorphism by the
identity over every $C_i$.  Theorem~\ref{thm:split-fusion}, including its countable assertion,
now applies to the
resulting one nuclear core and the family $(C_i)$: only that
one core may have nonsplit unit.  This gives a nuclear-core
reservoir model with the aggregate graded groups and pairing,
and marked nuclear-core replacement gives $F\fr D$.

For an arbitrary prescribed compatible global full marking,
first take the actual fusion just constructed.  Compose with
the residual full-marking automorphism of $F\fr D$, supplied by
Theorem~\ref{thm:marked-core}.  We do not
assume that a desired global marking splits factorwise, or that
an arbitrary absolute $\KK$ representative is retained through
the inductive limit.

Keep any originally protected finite reservoir head external
before partitioning banks.  All splittings, controlled binary
steps, reservoir changes and final marking corrections are made
in its complementary tail and freely extended by its identity.
It is therefore fixed by the limiting isomorphism.  Corollary~\ref{cor:reservoir-model} gives the stated reservoir class.
For an original block head in a general reservoir, use the native
tail-triple transport argument at the end of the proof of
Theorem~\ref{thm:split-fusion}: isolate that head first, identify
only its complementary reservoir tail with the interval model,
transport the desired marking through the invertible tail
inclusions, and extend by the identity on the head.
\end{proof}

\subsection{Completion of the classification theorem}
\label{subsec:classification-completion}

We combine general fusion with marked nuclear-core replacement
to complete the classification and morphism-realization assertions.

\begin{theorem}[Nuclear factors: classification and morphism realization]
\label{thm:nuclear-classification-main}
Let $D$ be a reservoir in the sense of Definition~\ref{def:reservoir}.
Let $\mathcal A$ and $\mathcal B$ be nonempty finite or countable
families of separable unital nuclear UCT algebras equipped with
faithful selected tracial states. Put
\[
 A=P_{\mathcal A}\fr D,\qquad B=P_{\mathcal B}\fr D.
\]
Each of $A$ and $B$ is isomorphic to $F\fr D$ for a nuclear UCT
core $F$ with the same aggregate selected pointed graded
$\mathrm K$-groups and pairing.
Every unital morphism $\eta:\Ell(A)\to\Ell(B)$ is induced by a
unital trace-preserving embedding $\Phi:A\to B$.
If $\eta$ is an isomorphism, $\Phi$ may be chosen to be a
$^*$-isomorphism. In particular,
\[
 A\cong B\quad\Longleftrightarrow\quad\Ell(A)\cong\Ell(B).
\]
\end{theorem}
\begin{proof}
Proposition~\ref{prop:product-elliott} gives countable aggregate
pointed groups $(G_0,u,G_1)$ and a selected pairing
$\rho:G_0\to\R$ with $\rho(u)=1$ for each family.
Equip $G_0$ with the strict trace cone
$\{0\}\cup\{g:\rho(g)>0\}$.
As in the proof of Lemma~\ref{lem:marked-range}, this is a simple
weakly unperforated scaled ordered group with unique normalized
state $\rho$. The nuclear range theorem
\cite[Theorem~2.8]{EGLN} supplies a simple, infinite-dimensional,
unital, monotracial, nuclear UCT, $\mathcal Z$-stable core with
these data. For the exceptional data $(\Z,m,0)$ one may use
$\mathrm M_m(\mathcal Z)$, as in that lemma.
Theorem~\ref{gf:general-fusion} therefore gives trace-preserving
isomorphisms
\[
 \Theta_A:A\longrightarrow F_{\mathcal A}\fr D,\qquad
 \Theta_B:B\longrightarrow F_{\mathcal B}\fr D
\]
with the chosen aggregate markings. This proves the core-model
assertion.

Given $\eta$, transport it to the nuclear-core stabilizations:
\[
 \widetilde\eta
   =\Ell(\Theta_B)\circ\eta\circ\Ell(\Theta_A)^{-1}.
\]
The canonical core inclusions identify the graded $\mathrm K$-groups,
units and selected pairings of these stabilizations with those of
$F_{\mathcal A}$ and $F_{\mathcal B}$, by
Proposition~\ref{prop:product-elliott}.
Thus the graded part of $\widetilde\eta$ consists of unit- and
selected-pairing-preserving homomorphisms between the two cores.
Theorem~\ref{thm:marked-core} realizes these maps by a unital
trace-preserving embedding
\[
 \Psi:F_{\mathcal A}\fr D\longrightarrow F_{\mathcal B}\fr D.
\]
Its Elliott-invariant morphism is $\widetilde\eta$: the graded
maps agree by construction, and the trace map between the one-point
trace simplices is unique. Consequently
\[
 \Phi=\Theta_B^{-1}\circ\Psi\circ\Theta_A:A\longrightarrow B
\]
induces $\eta$. If $\eta$ is an isomorphism, both transported
group maps are isomorphisms, so Theorem~\ref{thm:marked-core}
allows $\Psi$, and hence $\Phi$, to be chosen as a $^*$-isomorphism.
The converse implication in the classification assertion follows
from functoriality of the Elliott invariant.
\end{proof}

\begin{corollary}[Classification of $D$-stable products]
\label{cor:classification-d-stable}
Let $\mathcal A,\mathcal B$ be families as in
Theorem~\ref{thm:nuclear-classification-main}. Suppose that
$P_{\mathcal A}$ and $P_{\mathcal B}$ each absorb $D$ by a
trace-preserving isomorphism. Then both products are simple and
monotracial, and
\[
 P_{\mathcal A}\cong P_{\mathcal B}
 \quad\Longleftrightarrow\quad
 \Ell(P_{\mathcal A})\cong\Ell(P_{\mathcal B}).
\]
Every unital morphism $\Ell(P_{\mathcal A})\to\Ell(P_{\mathcal B})$
is induced by a unital trace-preserving embedding. If the morphism is
an isomorphism, the realizing map may be chosen to be a $^*$-isomorphism.

The invariant range is the range in Theorem~\ref{thm:nuclear-range},
with representatives $F\fr\mathcal Z$ as described there.
The absorption hypothesis holds, in particular, if each family
contains two distinct factors admitting Haar unitaries.
\end{corollary}
\begin{proof}
Choose trace-preserving absorption isomorphisms
$\alpha_{\mathcal A}:P_{\mathcal A}\to P_{\mathcal A}\fr D$ and
$\alpha_{\mathcal B}:P_{\mathcal B}\to P_{\mathcal B}\fr D$.
Lemma~\ref{lem:interval-regularity} gives simplicity and uniqueness
of trace. For a prescribed invariant morphism $\chi$, apply
Theorem~\ref{thm:nuclear-classification-main} to
\[
 \Ell(\alpha_{\mathcal B})\circ\chi\circ\Ell(\alpha_{\mathcal A})^{-1}.
\]
If $\Phi$ is its realizing embedding, then
$\alpha_{\mathcal B}^{-1}\circ\Phi\circ\alpha_{\mathcal A}$
realizes $\chi$. An invariant isomorphism permits $\Phi$ to be
chosen as an isomorphism. The converse follows from functoriality.
Theorem~\ref{thm:nuclear-range} gives the range and its $D$-stable
representatives. For the last assertion apply
Theorem~\ref{thm:absorption-without-reservoir}\textup{(i)} to the
two Haar-unitary factors in each family.
\end{proof}

\begin{corollary}[Closure under unital strong Morita equivalence]
\label{cor:unital-morita-closure}
The class in Theorem~\ref{thm:nuclear-classification-main} is closed,
up to unital isomorphism, under strong Morita equivalence among
unital $\mathrm C^*$-algebras.
More precisely, let $A=P_{\mathcal A}\fr D$ be in that class,
let $n\geq1$, and let $p\in\mathrm M_n(A)$ be a nonzero projection.
There is a simple, infinite-dimensional, unital, monotracial,
nuclear UCT, $\mathcal Z$-stable algebra $E$ such that
\[
 p\mathrm M_n(A)p\cong E\fr D.
\]
The isomorphism preserves the normalized corner trace.
Under the Morita identifications of the ordered $\mathrm K$-groups,
if $g=[p]\in\mathrm K_0(A)$ and $\rho=(\tau_A)_*$, the corner has
unit class $g$ and trace pairing $\rho/\rho(g)$.
\end{corollary}
\begin{proof}
As in the proof of Theorem~\ref{thm:nuclear-classification-main},
choose a simple, infinite-dimensional, unital, monotracial,
nuclear UCT, $\mathcal Z$-stable core $N$ and a trace-preserving
isomorphism $A\cong N\fr D$. We use this identification throughout.
The core inclusion $\iota:N\to A$ induces an isomorphism on
$\mathrm K_0$, preserving the trace pairing, by
Proposition~\ref{prop:product-elliott}.
Consequently
\[
 h=\iota_*^{-1}([p])\in\mathrm K_0(N),
 \qquad
 (\tau_N)_*(h)=(\Tr_n\otimes\tau_A)(p)>0.
\]
The order on $\mathrm K_0(N)$ is the strict trace order.
Indeed, strict comparison in $N$ gives this assertion by the
projection argument of Lemma~\ref{lem:positive-cone}.
Thus $h$ is represented by a projection
$q\in\mathrm M_l(N)$ for some $l$.

The images of $p$ and $q$ have the same $\mathrm K_0(A)$ class.
Stable rank one of $A$ gives cancellation of projections, so,
after zero padding, they are Murray--von Neumann equivalent in
$\mathrm M_{2^k}(A)$ for some $k$. In particular,
\[
 p\mathrm M_n(A)p\cong q\mathrm M_{2^k}(A)q.
\]
Iterating Lemma~\ref{lem:gaussian} gives a trace-preserving
isomorphism
\[
 \mathrm M_{2^k}(N\fr D)
       \cong \mathrm M_{2^k}(N)\fr D
\]
fixing $\mathrm M_{2^k}(N)$ pointwise.
The projection $q$ is full in the simple algebra
$\mathrm M_{2^k}(N)$. Lemma~\ref{lem:gaussian-full-corner}
therefore gives
\[
 q\mathrm M_{2^k}(A)q
   \cong q\bigl(\mathrm M_{2^k}(N)\fr D\bigr)q
   \cong \bigl(q\mathrm M_{2^k}(N)q\bigr)\fr D.
\]
Take $E=q\mathrm M_{2^k}(N)q$. This full corner retains the
listed properties of $N$, and its trace is the normalized
restriction of $\Tr_{2^k}\otimes\tau_N$.
All displayed isomorphisms preserve the normalized corner traces.

For completeness, every unital algebra strongly Morita equivalent
to $A$ is a full corner of a finite matrix algebra over $A$;
compare \cite[Theorem~1.2]{BGR}.
Indeed, its imprimitivity module is finitely generated projective
as a right $A$-module, since its algebra of compact operators is
unital. The preceding argument therefore proves the asserted
closure.
Finally, the Morita identification sends the corner unit to $[p]$,
and its normalized trace is
\[
 \tau_p(x)=\frac{(\Tr_n\otimes\tau_A)(x)}
                         {(\Tr_n\otimes\tau_A)(p)}.
\]
This gives the stated pairing $\rho/\rho(g)$, with the positive
cone and both $\mathrm K$-groups unchanged under those
identifications.
\end{proof}

\begin{corollary}[Commutative factors: classification and morphism realization]
\label{cor:commutative-classification-main}
Let $D$ be a reservoir. Let $\mathcal X$ and $\mathcal Y$ be
nonempty finite or countable families of nonempty compact metrizable
spaces equipped with full-support Borel probability measures.
Put $A=P_{\mathcal X}\fr D$ and $B=P_{\mathcal Y}\fr D$.
Every unital morphism $\Ell(A)\to\Ell(B)$ is induced by a unital
trace-preserving embedding $A\to B$. If the invariant morphism is
an isomorphism, the realizing map may be chosen to be a
$^*$-isomorphism. In particular,
\[
 A\cong B\quad\Longleftrightarrow\quad\Ell(A)\cong\Ell(B).
\]
If each family has two distinct factors admitting Haar unitaries,
these assertions also hold without adjoining $D$.
In particular, this applies when each family contains at least
two nonatomic measures.
\end{corollary}
\begin{proof}
For a nonempty compact metrizable space $X$, the algebra
$\mathrm C(X)$ is separable, unital, nuclear and UCT.
A full-support probability measure defines a faithful tracial state.
Theorem~\ref{thm:nuclear-classification-main} gives the stabilized
assertions, and Corollary~\ref{cor:classification-d-stable} gives
the assertions under the Haar-unitary hypothesis.
If each family has two nonatomic measures,
Lemma~\ref{lem:continuous-Haar} supplies Haar unitaries in the
corresponding two distinct factors of each product.
\end{proof}

\begin{remark}[Commutative factors and concrete models]
Theorem~\ref{gf:general-fusion} includes arbitrary finite or
countable mixtures of faithfully tracial nuclear UCT factors.
For commutative factors, evaluation at a point splits the unit;
their sharper invariant range and standard models are described in
Section~\ref{sec:models}.

The AF and real-rank-zero constructions in
Subsection~\ref{subsec:af-models} give concrete models for selected
pointed data. An AF product may have torsion even
though each individual AF factor has torsion-free $\mathrm K_0$.
These model descriptions are useful in examples, while general
fusion permits arbitrary unit classes and both $\mathrm K$-groups.
The matrix enlargement in its proof retains the exact minimal
projection class before compression, including its torsion part.
Section~\ref{sec:models} gives standard free-product models for the
whole nuclear-factor class by adding homogeneous Moore-space corners
to the commutative list.
\end{remark}

\begin{remark}[Arbitrary inductive limits need not be classified]
\label{rem:inductive-limit-counterexample}
The classification theorem does not extend to arbitrary inductive
limits of the free products considered here, even when the connecting
maps are unital, injective and trace preserving and the limits absorb
$D$.  Take $D$ in its interval presentation.  Let
$U=\mathrm M_{2^\infty}$, with its unique trace, and use
the minimal tensor product to define
\[
 X=(D\otimes U)\fr D,\qquad Y=U\fr D.
\]
For $n\geq1$, put
\[
 E_n=\mathrm M_{2^n}(D)\fr D,\qquad
 F_n=\mathrm M_{2^n}(\C)\fr D.
\]
Let $\alpha_n:\mathrm M_{2^n}(D)\to\mathrm M_{2^{n+1}}(D)$
be $x\mapsto\operatorname{diag}(x,x)$, and let $\mu_n$ be the
same embedding on scalar matrices.  Taking their free products
with $\id_D$ gives
\[
 X\cong\varinjlim(E_n,\alpha_n\fr\id_D),\qquad
 Y\cong\varinjlim(F_n,\mu_n\fr\id_D).
\]
Indeed, the nested core algebras have dense unions in $D\otimes U$
and $U$, respectively, and finite reduced words give the asserted
continuity.  Equation~\eqref{gf:eq-GA} and reindexing $D\fr D\cong D$
give trace-preserving isomorphisms
\[
 E_n\cong\mathrm M_{2^n}(D\fr D)
     \cong\mathrm M_{2^n}(D)\cong F_n.
\]
In particular, all stages belong to the class of
Theorem~\ref{thm:nuclear-classification-main}.  Both invariant diagrams
have $\mathrm K_0=\Z$, positive cone $\Z_{\geq0}$, unit $2^n$,
pairing $m\mapsto m/2^n$, zero $\mathrm K_1$, and connecting
$\mathrm K_0$-maps $m\mapsto2m$.
The stages, and hence their limits, are simple and monotracial.
Continuity of ordered $\mathrm K$-theory therefore gives
\[
 \Ell(X)\cong\Ell(Y)
 \cong\bigl(\Z[1/2],\Z[1/2]_{\geq0},1,0,\{\tau\},\rho\bigr),
 \qquad \rho(q)=q.
\]
Moreover, associativity gives $X\fr D\cong X$ and $Y\fr D\cong Y$.

Nevertheless, $X\not\cong Y$.  In the respective tracial GNS
representations, write $M=D''\cong L(\mathbb F_\infty)$ and
$R=U''$, the hyperfinite $\mathrm{II}_1$ factor.  Their von Neumann
closures are
\[
 X''\cong(M\,\overline\otimes\,R)*M,\qquad
 Y''\cong R*M\cong L(\mathbb F_\infty),
\]
where the last identification is \cite[Theorem~4.1]{Dykema}.
Recall that a finite von Neumann algebra is \emph{solid} if the
relative commutant of every diffuse von Neumann subalgebra is
amenable.  The factor $Y''$ is solid by
\cite[Theorem~1 and Proposition~2]{OzawaSolid}; here
$\mathbb F_\infty$ is a subgroup of $\mathbb F_2$.
In $X''$, however, the diffuse algebra $1\otimes R$ in the first
free factor has relative commutant containing the nonamenable
algebra $M\otimes1$.  Every von Neumann subalgebra of a finite
von Neumann algebra admits a trace-preserving conditional
expectation, so this relative commutant cannot be amenable.
Thus $X''$ is not solid.  A $\mathrm C^*$-isomorphism $X\cong Y$
would preserve the unique traces and extend to an isomorphism of
these von Neumann closures, a contradiction.

Thus isomorphic stages and isomorphic Elliott-invariant diagrams
do not determine the limit, even with free $D$-absorption.
By Corollary~\ref{cor:finite-standard-model}, the stages may even
be replaced by finite standard free products, transporting the
connecting maps through the chosen isomorphisms.
The connecting maps in Theorem~\ref{thm:standard-models} are
chosen more specifically: after marked identifications
$C_n\cong N_n\fr D$, they have the form $\nu_n\fr\id_D$ with
nuclear core maps $\nu_n:N_n\to N_{n+1}$.
This identifies that limit as $(\varinjlim N_n)\fr D$ and permits
application of the nuclear-core theorem.
\end{remark}

\section{Range and models}\label{sec:range}\label{sec:models}\label{sec:fusion-models}

We first give the full range for nuclear cores and the products of
Theorem~\ref{thm:nuclear-classification-main}, then determine the sharper
commutative range. We then describe models built from AF factors
and construct standard free-product models for the whole class,
retaining commutative models for the commutative subclass.

\subsection{The general nuclear range}

\begin{theorem}[Range for nuclear cores]\label{thm:nuclear-range}
The Elliott invariants of the $D$-stabilized products in
Theorem~\ref{thm:nuclear-classification-main} are precisely
\[
 \bigl(G_0,G_0^+,u,G_1,\{\tau\},\rho\bigr),
 \qquad G_0^+=\{0\}\cup\{g:\rho(g)>0\},
\]
where $G_0,G_1$ are countable abelian groups, $u\in G_0$, and
$\rho:G_0\to\R$ is a homomorphism with $\rho(u)=1$.
Every such invariant is already attained, without adjoining $D$,
by a reduced free product of two simple infinite-dimensional
monotracial nuclear UCT $\mathcal Z$-stable algebras, one of which
is $\mathcal Z$. The unit of the other algebra need not split.
\end{theorem}
\begin{proof}
Necessity follows from separability and
Proposition~\ref{prop:product-elliott}.
Conversely, the displayed strict trace cone makes $G_0$ a simple
weakly unperforated directed ordered group, with order unit $u$.
Its unique normalized state is $\rho$: rational upper and lower
comparisons with $u$ determine the value of every element.
The range theorem
\cite[Theorem 2.8]{EGLN} provides a simple infinite-dimensional
monotracial nuclear UCT $\mathcal Z$-stable algebra $F$ with these
pointed groups and pairing. The exceptional data
$(G_0,u,G_1)=(\Z,m,0)$, $m\geq1$, can be realized explicitly by
$\mathrm{M}_m(\mathcal Z)$.

The free-product $\mathrm{K}$-sequence identifies the pointed groups
and pairing of $F\fr\mathcal Z$ with those of $F$.
Both factors contain Haar unitaries by \cite[Corollary~B]{ThielHaar},
since they are simple, unital and infinite dimensional.
Semicircular absorption therefore gives
$F\fr\mathcal Z\cong(F\fr\mathcal Z)\fr D$.
The latter algebra has the displayed Elliott invariant by
Proposition~\ref{prop:product-elliott}. The representatives
$F\fr\mathcal Z$ belong to the class of
Theorem~\ref{thm:nuclear-classification-main} and realize every listed
invariant. This proves its range assertion as well.
\end{proof}

\subsection{The commutative range}

For commutative factors, rank imposes a restriction on the selected
pairing beyond splitting the unit.

\begin{lemma}[Rank decomposition]\label{lem:rank-decomposition}
Let $X$ be a nonempty compact metrizable space.  There are a finite
or countable set $J_X$ and a countable abelian group $H_X$ such that
\[
 \mathrm{K}_0(\mathrm{C}(X))
 \cong \Z[1]\oplus\Z^{(J_X)}\oplus H_X,
 \qquad H_X=\ker(\operatorname{rk}),
\]
where $\operatorname{rk}:\mathrm{K}_0(\mathrm{C}(X))
\to\mathrm{C}(X,\Z)$ is the pointwise rank map.
For every Borel probability measure $\mu$ on $X$, its pairing
vanishes on $H_X$ and is integration on the first two summands,
identified with $\mathrm{C}(X,\Z)$.
\end{lemma}
\begin{proof}
The rank map has the canonical group section
\[
 f\longmapsto\sum_{k\in f(X)} k[1_{f^{-1}(k)}].
\]
Here $f(X)$ is finite, and the characteristic functions are
projections in $\mathrm{C}(X)$.  Thus
\[
 \mathrm{K}_0(\mathrm{C}(X))
 \cong\mathrm{C}(X,\Z)\oplus\ker(\operatorname{rk}).
\]
There are only countably many clopen subsets of $X$: every such
subset is compact and therefore is a finite union of elements of
a fixed countable base.  Enumerating them gives nested finite
clopen partitions $\mathcal P_n$, starting with $\mathcal P_0=\{X\}$,
whose locally constant integer-valued functions exhaust
$\mathrm{C}(X,\Z)$.  A refinement replaces the characteristic
function of each old atom by the sum of those of its nonempty
subatoms.  This is a split inclusion between finitely generated
free groups, with free cokernel: in each refinement block the
diagonal vector $(1,\ldots,1)$ extends to an integral basis.
Choosing a complement at every step therefore gives
\[
 \mathrm{C}(X,\Z)=\Z 1_X\oplus\Z^{(J_X)}.
\]
Finally, for a projection $p$ in a matrix algebra over
$\mathrm{C}(X)$,
\[
 (\tau_\mu\otimes\operatorname{Tr})(p)
 =\int_X\operatorname{rank}(p(x))\,d\mu(x).
\]
This proves the pairing assertion.  All groups are countable
because $\mathrm{C}(X)$ is separable.
\end{proof}

\begin{lemma}[Connected pointed range]\label{lem:connected-pointed-range}
Let $H_0,H_1$ be countable abelian groups.  There is a connected
compact metrizable space $Y$ of dimension at most four, a point
$y_0\in Y$, and isomorphisms
\[
 \mathrm{K}_0(\mathrm{C}(Y))\cong\Z[1]\oplus H_0,
 \qquad \mathrm{K}_1(\mathrm{C}(Y))\cong H_1,
\]
under which evaluation at $y_0$ is $(k,h)\mapsto k$.
Consequently every probability measure on $Y$ pairs by this same
rank map.  The space can be chosen nontrivial and hence admits a
full-support nonatomic Borel probability measure.
If $H_0=0$, the dimension may be taken at most three.
\end{lemma}
\begin{proof}
We record the pointed construction to specify the unit and rank
splitting.  The existence of finite-dimensional connected spaces
with arbitrary prescribed countable $\mathrm{K}$-groups is also
used at the beginning of \cite[Section~6]{DPS}.

For a finitely generated abelian group
$L=\Z^r\oplus\bigoplus_a\Z/n_a\Z$, let
\[
 W(L)=\bigvee_{j=1}^r S^3\ \vee\
       \bigvee_a M(\Z/n_a\Z,2),
 \qquad M(\Z/n\Z,2)=S^2\cup_{\deg n}D^3.
\]
Take $W(0)$ to be a point.  We use complex topological $\mathrm{K}$-theory in this proof.
These finite complexes satisfy
$\widetilde{\mathrm{K}}_0(W(L))=0$ and
$\mathrm{K}_1(W(L))\cong L$.
Every homomorphism $\alpha:L\to L'$ is induced contravariantly
by a pointed continuous map $W(L')\to W(L)$.
Here is an explicit verification.  On a spherical summand of the
domain, degree maps and the pinch map realize all the required
free coordinates.  On a summand $M(\Z/m\Z,2)$, collapsing its
two-skeleton gives the reduction map $\Z\to\Z/m\Z$ from a
target $S^3$-summand.  A homomorphism
$\Z/n\Z\to\Z/m\Z$ is multiplication by an integer $b$ with
$nb=ma$ for some integer $a$.  Degree $a$ on the lower two-sphere and degree $b$ on the
attaching two-sphere give a homotopy-commutative square of the
degree-$m$ and degree-$n$ maps. The resulting map of their
mapping cones induces the required homomorphism.  Finally, every domain summand is a suspension,
so its suspension pinch map combines the maps to the separate
target summands; on reduced $\mathrm{K}$-theory the resulting
homomorphisms add.

Choose increasing finitely generated subgroups $H_{i,n}$ with
union $H_i$, for $i=0,1$.  Realize their inclusions by pointed
maps $W(H_{i,n+1})\to W(H_{i,n})$ as above, and form the inverse
system
\[
 Y_n=\Sigma W(H_{0,n})\vee W(H_{1,n})\vee[0,1],
\]
with the suspended maps on the first summand and the identity
on the last.  Here $\Sigma$ is reduced suspension and $0$ is
the basepoint of the interval.  Each $Y_n$ is connected, is
at most four-dimensional, and has pointed groups
$\mathrm{K}_0(Y_n)=\Z[1]\oplus H_{0,n}$ and
$\mathrm{K}_1(Y_n)=H_{1,n}$.  All bonding maps preserve the
basepoint and fix the interval.  Their inverse limit $Y$ is
therefore nontrivial, connected, compact, metrizable and at most
four-dimensional; surjectivity of the bonding maps is not
needed. Indeed, in $\prod_{n\geq1}Y_n$ the set satisfying only
the first $m$ bonding relations is homeomorphic to
$Y_{m+1}\times\prod_{k>m+1}Y_k$ and is compact and connected.
Their nested intersection is $Y$, so it is connected; the fixed
interval makes it nontrivial. Finite covers refine to pullbacks
from a finite coordinate image, giving the dimension bound.
When $H_0=0$, the suspended summand is absent and each $Y_n$,
and hence $Y$, has dimension at most three.
The canonical map from the $\mathrm C^*$-inductive limit of the
pullback system $\mathrm C(Y_n)$ to $\mathrm C(Y)$ is an
isomorphism even when the bonding maps are not surjective:
the coordinate functions are dense by Stone--Weierstrass, and
the decreasing compact coordinate images give equality of norms.
Continuity of $\mathrm{K}$-theory therefore gives the displayed
pointed groups.  Since $Y$ is connected, the rank of a
projection is constant, proving the pairing assertion.
Finally, $Y$ has no isolated points.  For a countable basis
$(U_j)_{j\geq1}$ of nonempty open sets, choose a Cantor subset
of each $U_j$ and an atomless probability measure $\mu_j$
supported on it.  Then $\sum_{j\geq1}2^{-j}\mu_j$ is a
full-support nonatomic probability measure on $Y$.
\end{proof}

\begin{theorem}[Range of the Elliott invariant]\label{thm:invariant-range}
The Elliott invariants of the algebras $P_{\mathcal X}\fr D$ in
Corollary~\ref{cor:commutative-classification-main} are exactly the following invariants,
up to isomorphism:
\begin{align}
 \mathrm{K}_0&=\Z u\oplus\Z^{(J)}\oplus H_0,
 &\mathrm{K}_1&=H_1,\qquad [1]=u,\qquad \mathrm{T}=\{\tau\},
 \label{eq:invariant-range-groups}\\
 \rho\bigl(mu+(n_j),h\bigr)&=m+\sum_{j\in J}a_jn_j,
 &\mathrm{K}_0^+&=\{0\}\cup\{g:\rho(g)>0\}.
 \label{eq:invariant-range-pairing}
\end{align}
Here $J$ is finite or countable, $H_0,H_1$ are arbitrary countable
abelian groups, and the $a_j$ are arbitrary real numbers.
Equivalently, $\mathrm{K}_0$ has a decomposition $F\oplus H_0$ with
$F$ free abelian, the unit belonging to a basis of $F$, and
$\rho(H_0)=0$; apart from $\rho(u)=1$, the homomorphism on $F$
is unrestricted.

Every invariant in this list is realized by a finite or countable
reduced free product containing two Lebesgue-interval factors, with
all commutative spectra of dimension at most four. Thus the range
is the same for the $D$-stabilized class and the Haar-unitary class
in Corollary~\ref{cor:commutative-classification-main}.
\end{theorem}
\begin{proof}
For necessity, apply Lemma~\ref{lem:rank-decomposition} to each
commutative factor. Identifying their unit classes in
Proposition~\ref{prop:product-elliott} gives
\[
 \mathrm{K}_0(P_{\mathcal X})
 =\Z u\oplus\bigoplus_i\Z^{(J_{X_i})}
             \oplus\bigoplus_i H_{X_i}.
\]
The middle summand is countable free abelian and the pairing
vanishes on the last. The degree-one group is countable, and
adjoining $D$ preserves the groups, unit and pairing.
Regularity and Lemma~\ref{lem:positive-cone} give the positive
cone and the one-point trace simplex.

Conversely, fix the displayed data. By
Lemma~\ref{lem:connected-pointed-range}, choose a connected compact
metrizable $Y$ with pointed groups
$(\Z[1]\oplus H_0,H_1)$ and rank pairing, and give it a full-support
nonatomic probability measure. Its unit is marked by $u$, its
rank-zero summand by $H_0$, and its degree-one group by $H_1$.
For the standard generator $e_j$ of $\Z^{(J)}$, put
\[
 q_j=\lfloor a_j\rfloor,\qquad s_j=a_j-q_j\in[0,1).
\]
If $s_j>0$, take the space $[0,1]\sqcup[0,1]$ with component
masses $s_j,1-s_j$ and Lebesgue measure within each component.
Mark the first-component projection by $e_j-q_ju$.
If $s_j=0$, take $S^2$ with any full-support nonatomic measure
and mark its rank-zero Bott class by $e_j-q_ju$.
These choices realize exactly the required trace values.
The classes $u$ and $e_j-q_ju$, together with $H_0$, still give
the prescribed direct-sum decomposition.

Form $P$ from the factor $\mathrm{C}(Y)$, these finite or countably
many additional factors, and two Lebesgue-interval factors.
Proposition~\ref{prop:product-elliott} computes the data of $P\fr D$
as \eqref{eq:invariant-range-groups}--\eqref{eq:invariant-range-pairing}.
The two interval factors contain Haar unitaries, so semicircular
absorption gives $P\cong P\fr D$. This proves realization in both
classes. All spectra used have dimension at most four.
\end{proof}

\begin{example}[A split unit is not sufficient]\label{ex:range-obstruction}
Consider $\Z u\oplus\Z[1/2]$ with unit $u$ and normalized pairing
\[
 \rho(mu,q)=m+q.
\]
This pointed group and pairing cannot occur in
Corollary~\ref{cor:commutative-classification-main}. Indeed, every homomorphism from
$\Z[1/2]$ to a free abelian group is zero, since its image lies in
$\bigcap_{n\ge1}2^nF=\{0\}$. Consequently the pairing in
Theorem~\ref{thm:invariant-range} must vanish on this subgroup,
whereas the displayed pairing does not.
Theorem~\ref{thm:nuclear-range} realizes these same pointed data in
the larger class with nuclear factors.
\end{example}

\subsection{Models built from AF factors and split-unit fusion}
\label{subsec:af-models}

We say that $(A,\tau_A)$ \emph{admits a model built from AF factors}
if its pointed graded $\mathrm K$-groups and selected trace pairing
agree with those of a nonempty finite or countable
reduced free product of faithfully tracial separable unital AF algebras.
The original order and other traces are not prescribed. An AF-product
model may be nonnuclear. The replacement statements below start from
nuclear UCT factors.

\begin{corollary}[AF factors and an arbitrary nuclear root]
\label{af:block-fusion}\label{af:rooted-fusion}
Let $R$ be a nuclear UCT core with faithful selected trace, and let
$((B_i,\tau_i))$ be a finite or countable family of faithfully
tracial unital AF algebras. If $F$ is a nuclear UCT core carrying
the aggregate selected pointed data, then
\[
 \left(R\fr\left(\fr_i B_i\right)\right)\fr D\cong F\fr D.
\]
Every prescribed compatible full marking is realized, and a finite
external reservoir head may be fixed pointwise.
\end{corollary}
\begin{proof}
Apply Theorem~\ref{gf:general-fusion}. The empty AF family is
Theorem~\ref{thm:marked-core}.
\end{proof}

\begin{proposition}[Replacement by AF products]
\label{af:model-replacement}
Let $\mathcal A=((A_i,\tau_i))_{i\in S}$ be a nonempty finite or
countable family of nuclear UCT cores with faithful selected traces.
Suppose each $A_i$ admits a model built from AF factors, and choose
such a model
\[
 R_i=\fr_{j\in J_i}(F_{ij},\sigma_{ij})
\]
with a specified unit- and selected-pairing-preserving isomorphism
of graded $\mathrm K$-groups from $A_i$ to $R_i$.
For every reservoir $D$ and every faithfully tracial separable unital
algebra $Q$ there is a trace-preserving isomorphism
\begin{equation}\label{af:eq-model-replacement}
 (P_{\mathcal A}\fr Q)\fr D
 \cong\left(\left(\fr_{i\in S}R_i\right)\fr Q\right)\fr D.
\end{equation}
It fixes $Q$ pointwise and induces the specified maps on the
canonical factor $\mathrm K$-classes. A prescribed finite reservoir
prefix may also be fixed pointwise. The flattened family
$(F_{ij})_{i\in S,\,j\in J_i}$ is finite or countable.
\end{proposition}

\begin{proof}
First use the interval reservoir and partition its coordinates into
disjoint infinite banks so that
\[
 D_I=D_*\fr\left(\fr_{i\in S}D_i\right).
\]
Keep the prescribed finite head among the coordinates of $D_*$.
For each $i$, Corollary~\ref{af:block-fusion} supplies a nuclear UCT
range core $E_i$ and a marked trace-preserving isomorphism
$R_i\fr D_i\cong E_i\fr D_i$. The selected data of $A_i$ and
$E_i$ agree. Theorem~\ref{thm:marked-core} therefore gives a marked
trace-preserving isomorphism $A_i\fr D_i\cong E_i\fr D_i$.
Choose the markings so that their composition gives
\[
 \theta_i:A_i\fr D_i\longrightarrow R_i\fr D_i
\]
with the specified graded map on $A_i$.

Take the reduced free product of all the actual isomorphisms
$\theta_i$, together with the identities on $Q$ and $D_*$.
Associativity identifies its source and target as
\[
 \begin{aligned}
 \left(\fr_{i\in S}(A_i\fr D_i)\right)\fr Q\fr D_*
     &\cong(P_{\mathcal A}\fr Q)\fr D_I,\\
 \left(\fr_{i\in S}(R_i\fr D_i)\right)\fr Q\fr D_*
     &\cong\left(\left(\fr_{i\in S}R_i\right)\fr Q\right)\fr D_I.
 \end{aligned}
\]
Trace preservation identifies the reduced completions, and the free
product of the inverse maps is the inverse. In the countable case
this is a single countable free product of actual isomorphisms;
no infinite composition or limiting choice of corrections is used.
The free-product $\mathrm K$-sequence gives the stated factor
markings. The algebra $Q$ and the head in $D_*$ are fixed exactly.

For a general reservoir, keep its protected finite prefix external
and use Corollary~\ref{cor:reservoir-model} to identify the
complementary tail with $D_I$. Apply the preceding construction
there and extend by the identity on the prefix. This proves all
the assertions.
\end{proof}

\begin{proposition}[Models with real rank zero and torsion-free $\mathrm K_0$]
\label{af:rr0-model}
Let $(A,\tau)$ be a separable unital nuclear UCT algebra with faithful
selected trace. Suppose its selected pointed data agree with those of
a separable unital nuclear UCT algebra $(B,\tau_B)$ with faithful trace,
real rank zero and torsion-free $\mathrm K_0(B)$.
There are a faithfully tracial unital AF algebra $F$ and a nuclear
UCT core $C$ with faithful selected trace such that
\[
 \begin{aligned}
 (\mathrm K_0(F),[1_F],\tau_{F*})
     &\cong(\mathrm K_0(A),[1_A],\tau_*),\qquad \mathrm K_1(F)=0,\\
 \mathrm K_0(C)&=\Z[1_C],\qquad \mathrm K_1(C)=\mathrm K_1(A).
 \end{aligned}
\]
For every reservoir $D$ there is a trace-preserving isomorphism
\begin{equation}\label{af:eq-rr0-replacement}
 A\fr D\cong(F\fr C)\fr D
\end{equation}
with the corresponding selected-data marking. Every compatible full
marking may be realized, and a finite reservoir prefix may be fixed
pointwise.
\end{proposition}

\begin{proof}
Identify the selected data of $B$ with those of $A$.
For finite-dimensional $B$, the selected data give $\mathrm K_1(A)=0$.
Take $F=B$ and $C=\C$;
Theorem~\ref{thm:marked-core} gives the assertion.
Assume now that $B$ is infinite dimensional. Real rank zero gives,
for every $N$, mutually orthogonal nonzero
projections $p_1,\ldots,p_N$ in $B$: take orthogonal nonzero positive
elements and a nonzero projection in each of their hereditary
subalgebras. Faithfulness gives $\tau_B(p_j)>0$, while
$\sum_j\tau_B(p_j)\leq1$. Thus the common selected trace range has
arbitrarily small positive elements and is dense in $\R$.

Put $G=\mathrm K_0(A)$, $u=[1_A]$ and $\rho=\tau_*$, and give $G$
the cone
\[
 G^+=\{0\}\cup\{g:\rho(g)>0\}.
\]
This is a directed unperforated ordered group with order unit $u$;
unperforation uses torsion-freeness also when $ng=0$.
For interpolation, suppose $a_i\leq b_j$ for $i,j=1,2$.
If $\max_i\rho(a_i)<\min_j\rho(b_j)$, density supplies an element
with trace strictly between the two bounds. If the bounds agree,
a maximizing $a_i$ and a minimizing $b_j$ have equal trace and are
comparable, hence are equal; this common element interpolates.
The countable AF range theorem \cite{EHS} therefore gives a unital
AF algebra $F$ with this ordered pointed group and a faithful trace
inducing $\rho$. Indeed the displayed order is simple and every
normalized state equals $\rho$, by rational comparisons with $u$,
so $F$ may be chosen simple and monotracial. The original order and
other traces of $A$ are not part of the selected data.

The nuclear range theorem \cite[Theorem~2.8]{EGLN} supplies $C$ with
the displayed groups and pairing $\tau_{C*}(k[1_C])=k$.
The scalar unit relation gives
\[
 \mathrm K_0(F\fr C)
   =\frac{G\oplus\Z[1_C]}{\langle u-[1_C]\rangle}\cong G,
 \qquad \mathrm K_1(F\fr C)=\mathrm K_1(A),
\]
with unit $u$ and selected pairing $\rho$.
Since $C$ has split unit, Theorem~\ref{thm:split-fusion} gives
$(F\fr C)\fr D\cong E\fr D$ for a nuclear range core $E$ with
these data. Theorem~\ref{thm:marked-core} compares $A\fr D$ with
$E\fr D$, proving \eqref{af:eq-rr0-replacement} with the specified
markings and finite-prefix control.
\end{proof}

\subsection{Commutative standard blocks}
We first describe the commutative blocks and their finite realization
property. For general nuclear factors we will add homogeneous corners
over finite wedges of Moore spaces. The spaces needed at each finite
stage come from fixed types, although the connecting homomorphisms
need not come from maps between the spaces. Put
\[
 T_{\mathrm{II},n}=S^1\cup_{z\mapsto z^n}D^2,
 \qquad
 T_{\mathrm{III},n}=S^2\cup_{\deg n}D^3
 \quad(n\geq2).
\]
These are the usual two- and three-dimensional Moore spaces,
used as standard model spaces in the classification of AH
algebras in \cite[Section~4.2]{EGModels}.  In particular,
$T_{\mathrm{III},n}$ is not a three-dimensional real projective space.
Our list consists of
\begin{equation}\label{eq:standard-spaces}
 [0,1],\quad [0,1]\sqcup[0,1],\quad S^1,\quad S^2,
 \quad T_{\mathrm{II},n},\quad T_{\mathrm{III},n}\quad(n\geq2).
\end{equation}
For the connected spaces, use the complement to the unit given by
the kernel of evaluation at a point.  For the disconnected interval
space, use the complement generated by the characteristic function
$e$ of its first component.  The last column records the pairing
on these chosen complements, using the displayed generator lifts:
\[
\begin{array}{c|c|c|c}
 X&\mathrm{K}_0/\Z[1]
       &\mathrm{K}_1
       &\rho\text{ on the complement}\\ \hline
 {}[0,1]&0&0&0\\
 {}[0,1]\sqcup[0,1]&\Z[e]&0&[e]\mapsto t\\
 S^1&0&\Z&0\\
 S^2&\Z&0&0\\
 T_{\mathrm{II},n}&\Z/n\Z&0&0\\
 T_{\mathrm{III},n}&0&\Z/n\Z&0
\end{array}
\]
Here $t\in(0,1)$ is the mass of the first interval.  Every space
has a full-support nonatomic probability measure; on a disconnected
interval space use normalized Lebesgue measure on each component
with the indicated masses.  For connected spaces the pairing is
the rank map and is independent of this choice.

\begin{lemma}[Finite standard realization]\label{lem:finite-standard-model}
Let $H_0,H_1$ be finitely generated abelian groups, let
$H_0=\Z u\oplus H'_0$, and let $\rho:H_0\to\R$ satisfy
$\rho(u)=1$.  There is a finite reduced free product $C$ of measured
commutative algebras over the spaces in \eqref{eq:standard-spaces},
including two Lebesgue-interval factors, and isomorphisms
\[
 \eta_i:\mathrm{K}_i(C)\longrightarrow H_i\quad(i=0,1)
\]
such that $\eta_0([1_C])=u$ and $\tau_{C*}=\rho\eta_0$.
Its positive cone corresponds to
$\{0\}\cup\{h\in H_0:\rho(h)>0\}$.
\end{lemma}
\begin{proof}
Write $H'_0=\bigoplus_{s=1}^r\Z h_s\oplus
\bigoplus_a\Z/n_a\Z$ and
$H_1=\Z^q\oplus\bigoplus_b\Z/m_b\Z$.
Set $k_s=\lfloor\rho(h_s)\rfloor$ and
$t_s=\rho(h_s)-k_s$.  If $t_s>0$, use a disconnected interval
factor of first-component mass $t_s$ and mark its class $[e]$ by
$h_s-k_su$.  If $t_s=0$, use $S^2$ and mark its Bott class by
$h_s-k_su$.  Use $T_{\mathrm{II},n_a}$ for the cyclic summands of
$H'_0$, and $q$ circle factors and the spaces
$T_{\mathrm{III},m_b}$ for $H_1$.  Torsion has zero pairing with
$\R$.  Finally add two Lebesgue-interval factors.  They add no
reduced $\mathrm{K}$-theory and ensure the Haar-unitary hypothesis.
The free-product calculation in
Proposition~\ref{prop:product-elliott} gives exactly the stated
markings and pairing.  Haar absorption identifies $C$ with
$C\fr D$, so regularity and the strict trace cone follow from
Lemmas~\ref{lem:interval-regularity} and~\ref{lem:positive-cone}.
\end{proof}

\subsection{Homogeneous corners and arbitrary unit classes}

In addition to the commutative blocks in~\eqref{eq:standard-spaces},
allow the homogeneous blocks
\begin{equation}\label{eq:standard-corner}
 B(X,V)=\Gamma(X,\operatorname{End}V)
       \cong p\,\mathrm M_N(\mathrm C(X))p,
 \qquad V=\mathbf1^{\,d-1}\oplus L,\quad d\geq1.
\end{equation}
Here $X$ is a finite wedge of two-dimensional Moore spaces
$T_{\mathrm{II},n}$, with a point allowed for the empty wedge,
and $L$ is a complex line bundle on $X$.
The full projection $p$ represents $V$ and has constant rank $d$.
Equip this block with the faithful selected trace
\[
 \tau_{X,V}(b)=\int_X\frac1d\operatorname{Tr}(b(x))\,d\mu(x),
\]
where $\mu$ is any full-support Borel probability measure.
For $X$ a point these blocks are the full matrix algebras.
Let $\mathscr B$ denote this enlarged collection of faithfully
tracial blocks, together with the commutative blocks above.

\begin{lemma}[Finite type I realization with an arbitrary unit]
\label{lem:finite-general-model}
Let $H_0$ be a finitely generated abelian group of free rank $r$,
let $H_1$ be any countable abelian group, let $u\in H_0$, and let
$\rho:H_0\to\R$ satisfy $\rho(u)=1$.
There is a reduced free product $C$ of at most $\max\{2,r\}$
separable unital type~I algebras with faithful selected traces,
and isomorphisms
\[
 \eta_i:\mathrm K_i(C)\longrightarrow H_i\quad(i=0,1),
 \qquad \eta_0([1_C])=u,\qquad \tau_{C*}=\rho\eta_0.
\]
Its positive cone corresponds to
$\{0\}\cup\{h\in H_0:\rho(h)>0\}$.
All factors have central Haar unitaries.
At most one factor is noncommutative; it is a homogeneous corner
over a connected compact metrizable space of dimension at most three.
The other factors are commutative, with spectra of dimension at most two.
If the unit splits off a copy of $\Z$, all factors may be chosen
commutative.

If $H_1$ is also finitely generated, one may alternatively use a finite
product of blocks in $\mathscr B$, containing two Lebesgue-interval
factors and at most one corner from~\eqref{eq:standard-corner}.
In this alternative the split-unit case uses only the commutative
blocks in~\eqref{eq:standard-spaces}.
\end{lemma}
\begin{proof}
Put $T_0=\operatorname{Tor}H_0$.
The unit has infinite order, so $r\geq1$.
An integral change of basis in $H_0/T_0$ gives
\[
 H_0=\Z v\oplus\bigoplus_{s=1}^{r-1}\Z h_s\oplus T_0,
 \qquad u=dv+t_0,\quad d\geq1,\quad t_0\in T_0.
\]
Thus $\rho(v)=1/d$, since $\rho$ vanishes on $T_0$.
Write $T_0=\bigoplus_a\Z/n_a\Z$ and take
$X_0=\bigvee_aT_{\mathrm{II},n_a}$, using a point when $T_0=0$.
Rank and first Chern class give
$\mathrm K_0(\mathrm C(X_0))=\Z[1]\oplus T_0$ and
$\mathrm K_1(\mathrm C(X_0))=0$.
The first Chern class classifies complex line bundles on $X_0$.
Choose $L_0$ with $c_1(L_0)=t_0$ under $H^2(X_0;\Z)\cong T_0$.
The rank and first Chern class calculation then gives $[L_0]-[1]=t_0$.
Homogeneous corners over $T_{\mathrm{II},n}$ are standard AH building
blocks \cite[Section~1]{DG}.

Apply Lemma~\ref{lem:connected-pointed-range} with $H_0=0$ to
obtain a pointed connected compact metrizable space $Y$ of dimension
at most three with
\[
 \mathrm K_0(\mathrm C(Y))=\Z[1],\qquad
 \mathrm K_1(\mathrm C(Y))=H_1.
\]
Set $X=X_0\vee Y\vee[0,1]$ and pull $L_0$ back under the
retraction $X\to X_0$, obtaining $L$.
The split evaluation sequences for the wedge give
\[
 \mathrm K_0(\mathrm C(X))=\Z[1]\oplus T_0,\qquad
 \mathrm K_1(\mathrm C(X))=H_1.
\]
Take $V=\mathbf1^{\,d-1}\oplus L$ and let
$B=\Gamma(X,\operatorname{End}V)$.
This is a unital homogeneous type~I algebra, and Morita equivalence gives
\[
 \mathrm K_0(B)=\Z v\oplus T_0,\qquad
 [1_B]\longmapsto dv+t_0,\qquad
 \mathrm K_1(B)=H_1.
\]
Give $X$ a full-support nonatomic probability measure and use the
integrated normalized fibre trace on $B$.
Since $X$ is connected, its pairing is rank divided by $d$,
so it is $1/d$ on $v$ and zero on $T_0$.
Lemma~\ref{lem:continuous-Haar} supplies a Haar unitary in the
central copy of $\mathrm C(X)$.

For each remaining generator $h_s$, put
$k_s=\lfloor\rho(h_s)\rfloor$ and
$\lambda_s=\rho(h_s)-k_s\in[0,1)$.
If $\lambda_s>0$, use the disconnected interval block with
first-component mass $\lambda_s$, marking its first-component
class by $h_s-k_su$.
If $\lambda_s=0$, use $S^2$ with a full-support nonatomic measure,
marking its rank-zero Bott class by $h_s-k_su$.
These commutative factors have Haar unitaries by
Lemma~\ref{lem:continuous-Haar}.
Replacing $h_s$ by $h_s-k_su$ is an invertible triangular change
of generators fixing $\Z v\oplus T_0$, even when $t_0\ne0$.
Proposition~\ref{prop:product-elliott} therefore gives exactly the
required pointed groups and pairing.

There are $r$ factors when $r\geq2$.
When $r=1$, add one Lebesgue-interval factor, giving two factors.
Thus there are at most $\max\{2,r\}$ factors and two distinct
factors with Haar unitaries.
Haar absorption gives $C\cong C\fr D$.
The factors are nuclear and UCT, so
Lemmas~\ref{lem:interval-regularity} and~\ref{lem:positive-cone}
give a unique trace and the asserted positive cone.
If the unit splits, choose $v=u$, $d=1$ and $t_0=0$.
Then $B=\mathrm C(X)$, giving the commutative assertion.

For the final assertion, when $H_1$ is finitely generated use
$B(X_0,\mathbf1^{\,d-1}\oplus L_0)$ in place of $B$.
Realize $H_1$ by finitely many circle and
$T_{\mathrm{III},m}$ factors, and add two Lebesgue-interval factors.
The same group and pairing calculation proves the assertion.
When the unit splits,
Lemma~\ref{lem:finite-standard-model} gives the stated standard blocks.
\end{proof}

\subsection{Type I free-product models for the general class}

\begin{theorem}[Type I models and standard-block refinements]
\label{thm:standard-models}
Let $\mathcal A$ satisfy Theorem~\ref{thm:nuclear-classification-main},
and put $A=P_{\mathcal A}\fr D$.
Choose increasing finitely generated subgroups
$G_{0,n}\subset\mathrm K_0(A)$ containing $[1_A]$ and exhausting
$\mathrm K_0(A)$, and put $r_n=\operatorname{rank}G_{0,n}$.
There is an inductive system
\[
 C_1\xrightarrow{\varphi_1}C_2
 \xrightarrow{\varphi_2}\cdots
\]
with unital trace-preserving injective connecting homomorphisms,
such that each $C_n$ is a reduced free product of at most
$\max\{2,r_n\}$ separable unital type~I algebras with faithful traces,
and
\[
 A\cong\varinjlim(C_n,\varphi_n).
\]
The marked stage groups are
$\mathrm K_0(C_n)=G_{0,n}$ and
$\mathrm K_1(C_n)=\mathrm K_1(A)$; the connecting maps induce
the inclusions in degree zero and the identity in degree one.
At most one factor at each stage is noncommutative, of the
homogeneous form in Lemma~\ref{lem:finite-general-model}.
All factor spectra have dimension at most three.
If $[1_A]$ splits off a copy of $\Z$ in $\mathrm K_0(A)$, the
subgroups and models may be chosen so that all factors are commutative.
The limiting isomorphism may realize any prescribed
Elliott-invariant identification.

There is also a standard-block refinement: one may instead choose
every $C_n$ as a finite reduced free product of blocks in
$\mathscr B$, including two Lebesgue-interval factors and at most
one homogeneous Moore-space corner.
In this refinement both graded groups are exhausted by finitely
generated stage groups.
If all original factors are commutative, only the commutative
blocks in~\eqref{eq:standard-spaces} are needed.
If two distinct original factors admit Haar unitaries, the same
conclusions hold for $P_{\mathcal A}$ without adjoining $D$.
\end{theorem}
\begin{proof}
Apply Lemma~\ref{lem:finite-general-model} to
$G_{0,n},[1_A],\mathrm K_1(A)$ and the restricted selected pairing.
This gives the primary stage models with the asserted factor bound.
If the unit splits, write
$\mathrm K_0(A)=\Z[1_A]\oplus H$ and choose
$G_{0,n}=\Z[1_A]\oplus H_n$ for increasing finitely generated
subgroups $H_n$ exhausting $H$.
The commutative assertion follows from the same lemma.

The stage products absorb $D$, and their factors are nuclear and UCT.
Choose simple monotracial nuclear UCT $\mathcal Z$-stable cores
$N_n$ by the strict-cone range construction in the proof of
Lemma~\ref{lem:marked-range}.
The products $C_n$ and $N_n\fr D$ both absorb $D$, the latter by
reindexing two copies of $D$. Since $N_n\fr D$ is a product of the
nuclear UCT core $N_n$ with the reservoir blocks,
Corollary~\ref{cor:classification-d-stable} gives an isomorphism
\[
 \psi_n:C_n\longrightarrow N_n\fr D
\]
inducing the prescribed identification of the stage Elliott invariants.
The inclusions $G_{0,n}\subset G_{0,n+1}$ preserve the unit and
the strict trace cones.  Together with the identity in degree one
and the unique trace maps, they define Elliott-invariant morphisms
carrying the target trace to the faithful source trace.
Nuclear embedding existence, with compatible full-invariant
extensions \cite[Theorem~B and Theorem~3.9]{CGSTW}, supplies
unital embeddings $\nu_n:N_n\to N_{n+1}$ inducing the inclusions
on $\mathrm K_0$ and the identity on $\mathrm K_1$.
Their unique traces agree. Define
\[
 \varphi_n=\psi_{n+1}^{-1}\circ
          (\nu_n\fr\id_D)\circ\psi_n.
\]
These maps are injective and trace preserving. Identifying the
$N_n$ with nested subalgebras of $N=\varinjlim(N_n,\nu_n)$,
finite reduced words give
\[
 \varinjlim(C_n,\varphi_n)\cong N\fr D.
\]
The algebra $N$ is simple, separable, unital, nuclear and UCT, with
faithful unique trace. Its groups, unit and pairing are those of $A$
by continuity. A class of positive trace is positive at any stage
representing it. Conversely, a nonzero positive class comes from a
nonzero positive class at a stage and hence has positive trace.
Theorem~\ref{thm:nuclear-classification-main} also identifies
$A$ with $F_\infty\fr D$ for a nuclear core with these data.
Theorem~\ref{thm:marked-core} compares $N\fr D$ and $F_\infty\fr D$
with any prescribed Elliott marking.
This proves the primary model assertion.

For the standard-block refinement, choose in addition increasing
finitely generated subgroups $G_{1,n}$ exhausting $\mathrm K_1(A)$.
Use the final assertion of Lemma~\ref{lem:finite-general-model}
and the same construction, now inducing subgroup inclusions in
both degrees. For the commutative subclass, evaluation on the
original factors splits the common unit; use
$G_{0,n}=\Z[1_A]\oplus H_n$ and
Lemma~\ref{lem:finite-standard-model}.
Haar absorption proves the last assertion.
\end{proof}

\begin{corollary}[A bound on the number of type I factors]
\label{cor:finite-standard-model}
If $\mathrm K_0(A)$ in Theorem~\ref{thm:standard-models} is finitely
generated with free rank $r$, then $A$ is isomorphic to a reduced
free product of at most $\max\{2,r\}$ faithfully tracial separable
unital type~I factors. No finite-generation assumption on
$\mathrm K_1(A)$ is needed.
At most one factor is noncommutative, of the homogeneous form in
Lemma~\ref{lem:finite-general-model}.
If $[1_A]$ splits off a copy of $\Z$, all factors may be chosen
commutative; in particular this holds for commutative original factors.
The isomorphism may realize any prescribed Elliott-invariant
identification.

If both groups are finitely generated, there is also a finite
standard-product model using $\mathscr B$, with at most one
homogeneous Moore-space corner and two Lebesgue-interval factors.
For commutative original factors this alternative uses only
the blocks in~\eqref{eq:standard-spaces}.
Under the Haar-unitary hypothesis the same conclusions hold for
$P_{\mathcal A}$.
\end{corollary}
\begin{proof}
Apply Lemma~\ref{lem:finite-general-model} to the full invariant.
The resulting product $C$ absorbs $D$ and has the same Elliott
invariant as $A$. The $D$-stabilized product $A$ also absorbs $D$
by reindexing the reservoir, so Corollary~\ref{cor:classification-d-stable}
supplies the desired marked isomorphism. The standard-block refinement
follows from the final assertion of the lemma. The same corollary
applies to $P_{\mathcal A}$ under the Haar-unitary hypothesis.
\end{proof}

\begin{remark}[The unit class and the corner]
\label{rem:standard-corner-unit}
The projection in~\eqref{eq:standard-corner} retains the entire unit
class, including its torsion component.
For example, if $L$ is the nontrivial complex line bundle on
$\mathbb{RP}^2=T_{\mathrm{II},2}$, the block
$\operatorname{End}(\mathbf1\oplus L)$ has
$\mathrm K_0=\Z v\oplus\Z/2\Z\,t$ and unit $2v+t$.
This unit is not divisible by two, whereas the unit of
$\mathrm M_2(\mathrm C(\mathbb{RP}^2))$ is.
The corner is used here as an individual nuclear free factor;
no absorption identity for an arbitrary projection corner of a
free product is needed.
Elliott--Thomsen point--line factors may also be allowed, but the
realizations above already give the stated type~I and standard-block models.
\end{remark}

\begin{remark}[Type I factors: unrestricted $\mathrm K_1$ and a pointed $\mathrm K_0$ obstruction]
\label{rem:type-I-obstruction}
Arbitrary unital type~I factors impose no restriction on the
countable $\mathrm K_1$-group.
For any countable abelian group $G$, Lemma~\ref{lem:connected-pointed-range}
gives a connected compact metrizable space $Y$ with
$\mathrm K_0(\mathrm C(Y))=\Z[1]$ and
$\mathrm K_1(\mathrm C(Y))=G$.
Choose a full-support probability measure on $Y$ and take its reduced
free product with two Lebesgue-interval algebras.
All three factors are commutative and faithfully tracial, and
Proposition~\ref{prop:product-elliott} gives $\mathrm K_1=G$.

For the full invariant, however, no finite or countable number of
unital type~I factors with faithful traces suffices in general.
In fact, the pointed group
\[
                 (\mathrm K_0,[1])=(\Z[1/2],1)
\]
cannot occur for such a free product, even though it occurs for
$\mathrm M_{2^\infty}\fr D$, whose degree-one group is zero and
whose selected pairing is the inclusion $\Z[1/2]\subset\R$.

To prove the obstruction, suppose $P=\fr_i(A_i,\tau_i)$ has this
pointed group, and write $G_i=\mathrm K_0(A_i)$ and $u_i=[1_{A_i}]$.
The faithful traces make every $u_i$ have infinite order.
The free-product formula gives
\[
 \mathrm K_0(P)=
 \left(\bigoplus_iG_i\right)/\langle u_i-u_j\rangle,
 \qquad
 \mathrm K_0(P)/\Z[1_P]\cong\bigoplus_iG_i/\Z u_i.
\]
Each $G_i$ injects into $\mathrm K_0(P)$.
Indeed, fix $i$ and present the relations as
$u_j-u_i$ for $j\ne i$.
If an element of $G_i$ is a finite sum of these relations, its
$j$-coordinate is an integer multiple of $u_j$ and must vanish
for every $j\ne i$; infinite order forces all coefficients to be zero.
Thus each $G_i$ is a subgroup of $\Z[1/2]$ containing $1$.

Every nonzero unital type~I algebra has a finite-dimensional
irreducible quotient: quotient by a maximal proper closed ideal;
the resulting simple unital type~I algebra is an algebra of compact
operators and hence a full matrix algebra.
Consequently there is a homomorphism
$G_i\to\Z$ sending $u_i$ to an integer $n_i>0$.
A subgroup of $\Z[1/2]$ containing $1$ is either
$2^{-k_i}\Z$ for some $k_i\geq0$, or all of $\Z[1/2]$.
The latter is impossible, since every homomorphism
$\Z[1/2]\to\Z$ is zero.
Hence every $G_i/\Z u_i$ is a finite cyclic $2$-group.
Their direct sum is not $2$-divisible unless it is zero.
But $\Z[1/2]/\Z$ is a nonzero divisible group, a contradiction.

Thus the inductive-limit formulation in
Theorem~\ref{thm:standard-models} remains necessary for the full
general range, even after allowing arbitrary type~I factors.
Corollary~\ref{cor:finite-standard-model} gives a finite bound when
$\mathrm K_0$ is finitely generated; the bound is not asserted to be optimal.
For the fixed standard list $\mathscr B$, there is also the simpler
degree-one obstruction: a single countable free product of those
individual blocks has $\mathrm K_1$ a direct sum of cyclic groups,
so it cannot have $\Z[1/2]$ as its degree-one group.
\end{remark}

The range of the actual commutative free products is exactly that of
Theorem~\ref{thm:invariant-range}; arbitrary inductive systems of
standard products need not have invariants in this range.
The connecting maps constructed above need not be pullbacks of
continuous maps between the spectra.
\section{Exponential length and exponential rank}\label{sec:exponential}

We compute the exponential rank and exponential length using only
the trace range in the classification data. We first obtain the required upper estimates
directly from Theorem~\ref{thm:embedding}; no additional classification
theorem for nonnuclear algebras is needed.

\subsection{Definitions and interval lifts}

For a unital algebra $B$ and $u\in \mathrm{U}_0(B)$, put
\[
 \cel_B(u)=\inf\left\{\sum_{j=1}^r\|h_j\|:
 u=\prod_{j=1}^r e^{ih_j},\ h_j=h_j^*\in B\right\},\qquad
 \cel(B)=\sup_{u\in \mathrm{U}_0(B)}\cel_B(u).
\]
This agrees with the infimum of lengths of rectifiable unitary paths
from $1$ to $u$ \cite{Ringrose}. Here
\[
 \CU(B)=\overline{[\mathrm{U}_0(B),\mathrm{U}_0(B)]},\qquad
 \cucu(B)=\sup_{u\in\CU(B)}\cel_B(u).
\]
The paths defining $\cucu$ may run through all of $\mathrm{U}_0(B)$.
The assertion $\cer(B)\leq1+\varepsilon$ means that single
exponentials are norm dense in $\mathrm{U}_0(B)$; equality means additionally
that some connected unitary has no self-adjoint logarithm.
The symbol $\varepsilon$ in this rank notation is formal.

We use two elementary facts. If $\|u-e^{ih}\|<\delta<2$, the
residual unitary has a logarithm of norm at most
$2\arcsin(\delta/2)$. Thus uniform approximation with
$\|h\|\leq L$ implies $\cel_B(u)\leq L$.
Also, if $h=h^*$ vanishes on every tracial state, then
$e^{ih}\in\CU(B)$, by the trace-commutator result of \cite{CP}
and the Trotter product formula.

\begin{proposition}[Bounded approximate logarithms]\label{prop:bounded-logarithms}
Let $T$ be an admissible target, with unique trace $\tau$, and put
$H=\tau_*(\mathrm{K}_0(T))$. Every $u\in \mathrm{U}_0(T)$ is a norm limit of
exponentials $e^{ih}$ with $h=h^*$ and $\|h\|\leq2\pi$.
Consequently
\[
 \cer(T)\leq1+\varepsilon,\qquad
 \cel(T)\leq2\pi.
\]
\end{proposition}
\begin{proof}
If the spectrum of $u$ omits a point of the circle, a branch of the
argument gives a continuous logarithm of norm at most $2\pi$.
We may therefore suppose that
$\operatorname{sp}(u)=\mathbb T$.
Let $\mu$ be its spectral distribution for $\tau$; it has full support.
Let $\nu_0$ be the lift of $\mu$ by a Borel argument in $[0,2\pi)$,
regarded as a measure on $[0,2\pi]$, and put
\[
 a=\frac1{2\pi}\int s\,d\nu_0(s).
\]
Every open interval at either endpoint contains the lift of an open
arc, so $\nu_0$ has full support on $[0,2\pi]$.
The Thomsen sequence identifies the connected unitary class of $u$
with a determinant
$\Delta_T(u)\in\R/\overline H$. Choose a representative $d\in\R$.
Choose $t\in[0,1)$ such that $a-t-d\in\Z$, and let
\[
 \nu_t=(1-t)\nu_0+t(s\mapsto s-2\pi)_*\nu_0.
\]
For $t=0$ use $Y=[0,2\pi]$; otherwise use $Y=[-2\pi,2\pi]$.
Then $\nu_t$ is faithful on $\mathrm{C}(Y)$, its exponential pushforward is
$\mu$, and its normalized mean is $a-t$.

Send the unit of $\mathrm{K}_0(\mathrm{C}(Y))=\Z$ to $[1_T]$, and use $\nu_t$ for
the trace component. These data extend to a compatible full-invariant
morphism \cite[Theorem 3.9]{CGSTW}.
Theorem~\ref{thm:embedding} gives a unital embedding
$\phi:\mathrm{C}(Y)\to T$ with $\tau\phi=\nu_t$.
For $h=\phi(s\mapsto s)$ and $v=e^{ih}$ we have
\[
 \|h\|\leq2\pi,\qquad
 \tau(f(v))=\tau(f(u))\quad(f\in \mathrm{C}(\mathbb T)),\qquad
 \Delta_T(v)=a-t=d\pmod{\overline H}.
\]
Here $\Z\subseteq H$, since the trace is normalized.

The two circle embeddings defined by $u$ and $v$ have the same full
invariant. Indeed, their $\mathrm{K}_0$-maps are unital and their $\mathrm{K}_1$-maps
are zero. The free $\mathrm{K}$-groups of $\mathrm{C}(\mathbb T)$ imply equality on
all coefficient groups. Their trace components agree by the displayed
spectral-distribution identity. Finally, the Thomsen kernel and the
class of the coordinate unitary $z$ generate
$\Uquot{\mathrm{C}(\mathbb T)}$, because its $\mathrm{K}_1$-quotient is
$\Z[z]$. The trace action determines the map on that kernel, and
the matched determinants give equality on $[z]$.
Thus Theorem~\ref{thm:embedding} makes $u$ and $v$ approximately
unitarily conjugate. Conjugation preserves the norm of $h$.

The small-residual estimate above gives the length bounds.
\end{proof}

\subsection{Sharpness for a discrete range}

\begin{lemma}\label{lem:short-loop-exp}
A rectifiable unitary loop based at $1$ of length less than $2\pi$
is null-homotopic. Its trace determinant integral is zero.
\end{lemma}
\begin{proof}
If $-1\in\operatorname{sp}(u)$, a faithful representation provides
unit vectors $\xi_j$ with $\|u\xi_j+\xi_j\|\to0$. Applying any
unitary path from $1$ to $u$ to $\xi_j$ gives a path on the Hilbert
unit sphere of length bounded by that of the unitary path. The
spherical distance between its endpoints tends to $\pi$, so the
unitary path has length at least $\pi$. Every point on the loop
under consideration has a subarc from $1$ of length less than $\pi$.
Its spectrum therefore omits $-1$.
The continuous principal logarithm contracts the whole loop to $1$.
\end{proof}

\begin{proposition}\label{prop:discrete-length-lower}
Let $B$ be unital with a unique faithful trace $\tau$, and suppose it
contains a Haar unitary. If $H=\tau_*(\mathrm{K}_0(B))$ is discrete, then
$\cucu(B)\geq2\pi$.
\end{proposition}
\begin{proof}
For a piecewise smooth unitary path $\gamma$, use the normalized
determinant integral
\[
 D_\tau(\gamma)=\frac1{2\pi i}
 \int_0^1\tau(\gamma'(s)\gamma(s)^*)\,ds.
\]
It is invariant under endpoint-preserving homotopies, has absolute
value at most $\operatorname{Length}(\gamma)/(2\pi)$, and has periods
in $H$ on loops \cite{HS}.

Let $b$ be a Haar unitary and choose $0<t<1/2$ with $t\notin H$.
Choose an arc of Haar measure $t$ and its Fej\'er approximants $f_n$,
so $0\leq f_n\leq1$, $\int f_n=t$, and $f_n$ tends in $L^2$ to
the indicator of the arc. Put
\[
 h_n=2\pi(t1-f_n(b)),\qquad u_n=e^{ih_n},\qquad
 \lambda=e^{2\pi it}.
\]
Then $\tau(h_n)=0$, so $u_n\in\CU(B)$, and
\[
 \cel_B(u_n)\leq2\pi(1-t),\qquad
 \|u_n-\lambda1\|_{2,\tau}\longrightarrow0.
\]
We show that the first upper bound is asymptotically sharp.
Otherwise, by the definition of exponential length, a subsequence
admits piecewise exponential paths $\gamma_n$ from $1$ to $u_n$
with lengths at most some $\ell<2\pi(1-t)$.
Comparing with the trace-zero exponential path gives
$D_\tau(\gamma_n)\in H$. These numbers are bounded. Discreteness
of $H$ permits a subsequence on which they are a fixed $c\in H$.

Choose $\eta>0$ with $\ell+\eta+2\pi t<2\pi$.
Fix $0<\delta<1$ so small that
\[
 2\arcsin(r/2)\leq\left(1+\frac{\eta}{\ell+1}\right)r
 \qquad(0\leq r\leq\delta).
\]
Parametrize each $\gamma_n$ by arclength and subdivide it into
subpaths of length at most $\delta$. The number of subpaths is at
most $N=\lceil\ell/\delta\rceil+1$, independently of $n$.
After left translation by its initial unitary, each subpath lies
in the norm-$\delta$ neighborhood of $1$, where the principal
logarithm is continuous. Interpolating its self-adjoint logarithm
with the linear path to the endpoint logarithm gives a homotopy
relative to the endpoints to the principal-logarithm path.
Thus this replacement preserves the determinant integral.
If the original subpath has length $l_{j,n}$ and endpoint increment
$e^{ik_{j,n}}$, with $k_{j,n}$ its principal self-adjoint logarithm,
then
\[
 \|k_{j,n}\|\leq2\arcsin(l_{j,n}/2)
 \leq\left(1+\frac{\eta}{\ell+1}\right)l_{j,n}.
\]
Consequently the replacements have total length at most
$\ell+\eta$. Pad the list of logarithms by zeros to obtain exactly
$N$ terms. This fixed number of uniformly bounded logarithms allows
the paths and their determinant integrals to pass to the tracial
ultrapower.
Their images $k_j=(k_{j,n})_\omega$ in the tracial ultrapower of
$\pi_\tau(B)''$ satisfy
\[
 \sum_j\|k_j\|\leq\lim_\omega\sum_j\|k_{j,n}\|\leq\ell+\eta,
 \qquad
 \frac1{2\pi}\sum_j\tau_\omega(k_j)=\lim_\omega c=c.
\]
They therefore define a genuine piecewise exponential path from
$1$ to $\lambda1$, with determinant integral $c$ and length at most
$\ell+\eta$.
Appending the reverse scalar path of length $2\pi t$ gives a loop
of length less than $2\pi$ with integral $c-t\ne0$.
This contradicts Lemma~\ref{lem:short-loop-exp}.
Hence $\cel_B(u_n)\to2\pi(1-t)$ along the original sequence.
Letting $t\downarrow0$ outside $H$ proves the proposition.
\end{proof}

\subsection{A unitary with no ambient logarithm}

\begin{lemma}\label{lem:interval-no-log}
Let $P$ be faithfully tracial and let $A=P\fr D$.
There is a Haar unitary $w\in\CU(A)$ such that
\[
 A\cap\{w\}'=\mathrm{C}^*(w).
\]
In particular $w$ has no self-adjoint logarithm in $A$.
\end{lemma}
\begin{proof}
By Corollary~\ref{cor:reservoir-model}, identify $D$ with its
Lebesgue-interval presentation. Let $x_1,x_2$ be the first two
interval coordinates, put $a_j=e^{2\pi ix_j}$, and set $w=a_1a_2$.
The two-interval subalgebra has a unique trace by the Avitzour
condition \cite{Avitzour}.
Since $x_j-1/2$ is traceless there,
\[
 w=e^{2\pi i(x_1-1/2)}e^{2\pi i(x_2-1/2)}\in\CU(A).
\]
Every nonzero power of $w$ is a centered reduced word, so $w$ is Haar.

In the tracial von Neumann closure $M$ of $A$, each $a_j$ generates
the corresponding interval von Neumann algebra by measurable
functional calculus. These are free Haar unitaries, free from the
von Neumann closure of $P$. The free-group substitution
$(a_1,a_2,a_3,\ldots)\mapsto(w,a_2,a_3,\ldots)$ gives
$M=\mathrm{W}^*(w)*R$ for a finite tracial von Neumann algebra $R$.
The reduced-word decomposition identifies
$L^2(M)\ominus L^2(\mathrm{W}^*(w))$ with
\[
 \ell^2(\Z)\otimes\mathcal H\otimes\ell^2(\Z)
\]
where $\mathcal H$ is the Hilbert space of reduced words beginning
and ending with centered $R$-letters. On this space conjugation by $w$
shifts the two integer indices by $(m,n)\mapsto(m+1,n-1)$.
Every orbit is infinite, so there are no nonzero fixed vectors.
Consequently $M\cap\{w\}'=\mathrm{W}^*(w)$.

Let $\E_w:M\to \mathrm{W}^*(w)$ be the trace-preserving expectation.
For centered functions $f_j$ in the first interval and $g_j$ in
the second, reduced-word orthogonality gives
\[
 \E_w(f_1g_1\cdots f_rg_r)
 =\left(\prod_{j=1}^r\tau(a_1^*f_j)\tau(a_2^*g_j)\right)w^r.
\]
The reverse pattern has expectation a scalar multiple of $w^{-r}$.
Every other centered reduced-word pattern has expectation zero,
including a word with a letter from $P$ or a later interval.
Thus $\E_w$ maps the algebraic free product into Laurent polynomials
in $w$. Norm density and contractivity give $\E_w(A)\subseteq \mathrm{C}^*(w)$,
and hence $A\cap\{w\}'=\mathrm{C}^*(w)$.

If $w=e^{ih}$ with $h=h^*\in A$, then $h$ commutes with $w$ and
therefore equals $f(w)$ for a continuous real function $f$ on the
circle. But $e^{if(z)}=z$ is impossible by winding number.
\end{proof}

\subsection{The trace-pairing dichotomy}

\begin{theorem}\label{thm:exponential-tail}
Let $(P,\tau_P)$ be a separable unital exact faithfully tracial algebra,
put $A=P\fr D$, and let $H=\tau_{P*}(\mathrm{K}_0(P))$.
Then
\[
 \cel(A)=\cucu(A)=
 \begin{cases}
 2\pi,&H\text{ is discrete},\\
 \pi,&H\text{ is dense in }\R,
 \end{cases}
 \qquad \cer(A)=1+\varepsilon.
\]
All three values are unchanged on passing to $\mathrm{M}_d(A)$, $d\geq1$.
\end{theorem}
\begin{proof}
Lemma~\ref{lem:interval-regularity} makes $A$ an admissible target.
For an arbitrary exact $P$ in the statement, apply the scalar
free-product $\mathrm K$-sequence
\cite[Theorem~3.1 and Corollary~4.12]{FG} successively to
$P$ and the reservoir blocks. The unit relation is injective
at every step by the normalized traces, and each block has
$\mathrm K_0=\Z[1]$ and $\mathrm K_1=0$. Continuity therefore
makes $\mathrm K_*(P)\to\mathrm K_*(A)$ an isomorphism preserving
the unit and pairing. Thus $\tau_*(\mathrm{K}_0(A))=H$.
For the subsequent length estimates, use the Lebesgue-interval
presentation of $D$ from Corollary~\ref{cor:reservoir-model}. A subgroup of $\R$ containing $1$
is either discrete or dense.

For discrete $H$, Proposition~\ref{prop:bounded-logarithms} and
Proposition~\ref{prop:discrete-length-lower} give both lengths
equal to $2\pi$. For dense $H$, strict comparison makes
$\mathrm W(A)$ almost unperforated. Admissibility supplies exactness,
simplicity and stable rank one, and the unique trace identifies
$\Aff(\mathrm T(A))$ with $\R$.
Thus \cite[Proposition~7.1]{RordamRR} gives real rank zero directly.
The real-rank-zero approximation theorem of \cite{LinRR} approximates
each unitary in $\mathrm{U}_0(A)$ by single exponentials $e^{ih}$.
Real rank zero then approximates $h$ by a finite-spectrum
self-adjoint element. Reducing its finitely many eigenvalues modulo
$2\pi$ into $[-\pi,\pi]$ leaves its exponential unchanged and gives
an approximate logarithm of norm at most $\pi$.
The resulting short correction paths give $\cel(A)\leq\pi$.
The unitary $v=e^{2\pi i(x_1-1/2)}$ belongs to $\CU(A)$ and has
$-1$ in its spectrum. The path-length observation in
Lemma~\ref{lem:short-loop-exp} gives $\cucu(A)\geq\pi$.
The approximate logarithms prove $\cer(A)\leq1+\varepsilon$;
Lemma~\ref{lem:interval-no-log} excludes rank at most $1$.

Matrix algebras over $A$ remain admissible targets, and their normalized
trace range is $d^{-1}H$. Real rank zero passes to matrix algebras,
so the same upper bounds apply.
For the discrete lower bound use the Haar unitary
$\operatorname{diag}(b,\ldots,b)$, where $b$ is an interval Haar unitary.
For the dense lower bound use $v\oplus1_{d-1}\in\CU(\mathrm{M}_d(A))$.
Finally, $w\oplus1_{d-1}$ has no logarithm. A self-adjoint logarithm
would commute with this unitary. Since $w$ is Haar, $w-1$ has zero
kernel in the faithful tracial von Neumann representation; hence the
off-diagonal entries between the first coordinate and its complement
must vanish. The first diagonal entry would be a logarithm of $w$,
contradicting Lemma~\ref{lem:interval-no-log}.
\end{proof}

\begin{remark}[Abstract hypotheses and free-product structure]
The length dichotomy in Theorem~\ref{thm:exponential-tail} holds for
every admissible target $T$.\label{rem:abstract-exponential}
More precisely, if
$H=\tau_*(\mathrm{K}_0(T))$, then
\[
 \cel(T)=\cucu(T)=
 \begin{cases}
 2\pi,&H\text{ is discrete},\\
 \pi,&H\text{ is dense in }\R,
 \end{cases}
 \qquad \cer(T)\leq1+\varepsilon.
\]
Proposition~\ref{prop:bounded-logarithms} already gives the abstract
upper bounds. For discrete $H$, every admissible target contains a
Haar unitary by \cite[Corollary~B]{ThielHaar}, so
Proposition~\ref{prop:discrete-length-lower} gives sharpness.
For dense $H$, strict comparison makes $\mathrm{W}(T)$ almost
unperforated, and \cite[Proposition~7.1]{RordamRR} gives real rank zero;
that criterion requires neither nuclearity nor
$\mathcal Z$-stability. The finite-spectrum approximation argument
above, using \cite{LinRR}, gives the upper bound $\pi$.
A trace-preserving embedding of the uniform interval into $T$, supplied
by Theorem~\ref{thm:embedding}, gives a trace-zero exponential with
$-1$ in its spectrum, proving the reverse bound as above.
These length formulas also hold for all matrix amplifications, whose
normalized trace ranges are $d^{-1}H$.

The exact rank equality $\cer(P\fr D)=1+\varepsilon$ in our proof
additionally uses the concrete free-product structure:
Lemma~\ref{lem:interval-no-log} produces a connected unitary with no
self-adjoint logarithm in the ambient algebra, by computing its
commutant using reduced words. The abstract argument above establishes
the upper bound on exponential rank. In the free-product applications,
the concrete structure also verifies admissibility and identifies $H$
from the trace pairings of the factors.
\end{remark}

\begin{proof}[Proof of Theorem~\ref{thm:exponential-intro}]
Apply Theorem~\ref{thm:exponential-tail} to $P=P_{\mathcal A}$.
Proposition~\ref{prop:product-elliott} identifies its trace range
with $H=\tau_{A*}(\mathrm{K}_0(A))$.
Under the Haar-unitary hypothesis, $P_{\mathcal A}\cong P_{\mathcal A}\fr D$,
so all the conclusions, including those for matrix amplifications,
transfer to the original free product.
In the commutative subclass, if every $X_i$ is connected, each projection has constant integral
rank, so $H=\Z$.
\end{proof}

\section{Examples and consequences}\label{sec:examples}

We compute models for products of commutative and finite-dimensional
factors, including reduced group algebras of free products of finite
groups. The criteria in Subsection~\ref{subsec:unstabilized-examples}
identify products to which Corollary~\ref{cor:classification-d-stable}
applies.

\subsection{Commutative factors}

\begin{example}[Circles and annuli]
Let $Y$ be a compact planar annulus with normalized area measure $\mu$,
and give $S^1$ its normalized Haar measure $\lambda_{S^1}$.
For a faithful tracial state $\sigma$ on $\C\oplus\C$, put
\[
 B_\sigma=(\C\oplus\C,\sigma)^{*_{\mathrm r}\infty}.
\]
Then, for $R=D$ or $R=B_\sigma$, there is a trace-preserving
isomorphism\label{ex:hp-annulus}
\[
 (\mathrm C(S^1),\tau_{\lambda_{S^1}})\fr R
 \cong (\mathrm C(Y),\tau_\mu)\fr R.
\]
For $R=D$ and, respectively, $R=B_\sigma$ with
$\sigma(1,0)$ irrational, these are the comparisons asked for by
Hirshberg and Phillips in \cite[Question~4.9(2)(b) and~(1)(b)]{HP}.
The conclusion for $B_\sigma$ also holds when $\sigma(1,0)$ is rational.
Moreover, for $2\leq n\leq\infty$,
\[
 (\mathrm C(Y),\tau_\mu)^{*_{\mathrm r}n}
 \cong \mathrm C^*_{\red}(\mathbb F_n).
\]
\end{example}
\begin{proof}
The annulus deformation retracts onto a circle. Thus both root
factors have $\mathrm K_0=\Z[1]$ and $\mathrm K_1=\Z$, and their
degree-zero trace pairings are rank. Appending the same tail $R$
gives identical aggregate pointed groups and selected pairings.
Corollary~\ref{cor:commutative-classification-main} therefore identifies the two
products after adjoining $D$.

Both tails freely absorb $D$. This is immediate for $R=D$ by
reindexing. For $R=B_\sigma$, put $p=\sigma(1,0)\in(0,1)$ and
$I=(\mathrm C([0,1]),\tau_\lambda)$. The atom parameter of each
two-point factor is $\max(p,1-p)<1$, so
\cite[Theorem~3.1]{HP} gives $I\fr B_\sigma\cong B_\sigma$.
Since $B_\sigma^{*_{\mathrm r}\infty}\cong B_\sigma$ by reindexing,
grouping countably many copies gives
\[
 D\fr B_\sigma
 \cong (I\fr B_\sigma)^{*_{\mathrm r}\infty}
 \cong B_\sigma.
\]
These isomorphisms preserve the unique free-product traces, so they
remove the adjoined reservoir and prove the first assertion.
For the free powers, apply the two-nonatomic-factor clause of
Corollary~\ref{cor:commutative-classification-main} to the annulus factors and the
circle factors. Their aggregate invariants agree, and the circle
free power is $\mathrm C^*_{\red}(\mathbb F_n)$.
\end{proof}

\begin{example}[Two spheres]
For faithful nonatomic sphere and wedge measures,
\begin{equation}\label{eq:note-3-3}
 (\mathrm{C}(S^2),\mu_1)\fr(\mathrm{C}(S^2),\mu_2)\fr D
 \cong (\mathrm{C}(S^2\vee S^2),\nu)\fr D .
\end{equation}
The two sphere Bott classes map to the two wedge Bott classes.
\end{example}

\begin{example}[A torsion example]
For faithful nonatomic measures on the real projective plane
and its wedge,
\[
 (\mathrm{C}(\mathbb{RP}^2),\mu_1)\fr
 (\mathrm{C}(\mathbb{RP}^2),\mu_2)\fr D
 \cong (\mathrm{C}(\mathbb{RP}^2\vee\mathbb{RP}^2),\nu)\fr D.
\]
Both sides have pointed groups
$\mathrm{K}_0=\Z[1]\oplus(\Z/2)^2$, $\mathrm{K}_1=0$; the trace pairing
is $1$ on the unit and zero on the two torsion summands.
\end{example}

\begin{example}[Solenoids]
Let $\Sigma_p=\varprojlim(S^1,z\mapsto z^p)$ be the $p$-adic
solenoid, with its Haar probability measure. Then, for any faithful
selected probability $\nu$ on the wedge,
\[
 (\mathrm{C}(\Sigma_2),\mu_2)\fr(\mathrm{C}(\Sigma_3),\mu_3)\fr D
 \cong (\mathrm{C}(\Sigma_2\vee\Sigma_3),\nu)\fr D.
\]
Here $\mathrm{K}_0=\Z[1]$ and
$\mathrm{K}_1=\Z[1/2]\oplus\Z[1/3]$. Neither summand in degree one
is finitely generated or free abelian. The left-hand side may also
be written without $D$.
\end{example}
\begin{proof}
Continuity for
$\mathrm{C}(\Sigma_p)=\varinjlim(\mathrm{C}(S^1),f(z)\mapsto f(z^p))$ gives
$\mathrm{K}_0=\Z[1]$ and $\mathrm{K}_1=\Z[1/p]$. Apply
Theorem~\ref{thm:fusion}. The Haar measures are nonatomic,
so \cite[Theorem 5.2]{Shlyakhtenko} absorbs $D$ on the left.
\end{proof}

\begin{example}[The $\mathrm{K}$-theory of a point]\label{cor:point}
Let $P=\fr_{j\in S}(E_j,\tau_{\mu_j})$, with $E_j=\mathrm{C}(X_j)$,
be as in Corollary~\ref{cor:commutative-classification-main}. Suppose each factor
satisfies $\mathrm{K}_0(E_j)=\Z[1]$ and $\mathrm{K}_1(E_j)=0$.
Then $P\fr D\cong D$. If two factors contain Haar unitaries,
then $P\cong D$.
\end{example}
\begin{proof}
The aggregate data are $(\Z,1,0,\id)$, which are the data of $D$.
Apply Corollary~\ref{cor:commutative-classification-main} and the reindexing
$D\fr D\cong D$.
\end{proof}

\begin{example}[Odd spheres]\label{cor:odd-spheres}
Let $d\geq1$ and let $\mu$ be a nonatomic full-support probability
measure on $S^{2d+1}$. For every $n\in\{2,3,\ldots,\infty\}$,
\[
 (\mathrm{C}(S^{2d+1}),\tau_\mu)^{\fr n}
 \cong \mathrm{C}^*_{\red}(\mathbb F_n).
\]
\end{example}
\begin{proof}
The pointed graded groups of $\mathrm{C}(S^{2d+1})$ and $\mathrm{C}(S^1)$ are
$(\Z[1],\Z)$, and their selected trace pairings in degree zero
are rank. Their corresponding free powers therefore have the same
aggregate invariant. Use Corollary~\ref{cor:commutative-classification-main} with Haar
measure on $S^1$; the free power on the circle side is
$\mathrm{C}^*_{\red}(\mathbb F_n)$.
\end{proof}

\begin{example}[The pairing for disconnected spaces]\label{ex:trace-pairing}
Let $X=[0,1]\sqcup[0,1]$, let $0<t<1$, and give $X$ the measure
$\mu_t=t\lambda\sqcup(1-t)\lambda$. Put
$P_t=(\mathrm{C}(X),\tau_{\mu_t})\fr(\mathrm{C}(X),\tau_{\mu_t})$.
All these measures are faithful and nonatomic. The pointed groups
are $\mathrm{K}_0(P_t)=\Z[1]\oplus\Z^2$ and $\mathrm{K}_1(P_t)=0$.
With the two reduced generators represented by the first-component
projections in the two free factors, the trace pairing is
\[
 \tau_{t*}(m[1]+ng_1+kg_2)=m+t(n+k).
\]
Its image is $\Z+t\Z$. For $t=1/2$ this is the discrete group
$\tfrac12\Z$; for irrational $t$ it is dense in $\R$.
Consequently these two types of free products are nonisomorphic,
although their pointed graded groups agree. Their traces are unique,
so any isomorphism would preserve the displayed trace range.
Theorem~\ref{thm:exponential-intro} also gives
\[
 \cel(P_t)=\cucu(P_t)=
 \begin{cases}
 2\pi,&t\in\mathbb Q,\\
 \pi,&t\notin\mathbb Q,
 \end{cases}
 \qquad \cer(P_t)=1+\varepsilon.
\]
Thus the trace pairing detects an exact change in exponential length
even though the pointed graded groups are independent of $t$.
\end{example}

\subsection{Finite-dimensional factors and their models}

\begin{example}[Finite-dimensional and finite-group products]
\label{ex:finite-group-products}
Let $A_i=\bigoplus_{\pi}\mathrm M_{d_{i,\pi}}$ be nonscalar
finite-dimensional factors with faithful selected traces, and let
$\alpha_{i,\pi}$ be the trace of the unit of the indicated block.
Write $a_{i,\pi}$ for its minimal projection class. For
$P=\mathop{\fr}_{i=1}^r A_i$, Proposition~\ref{prop:product-elliott}
gives $\mathrm K_1(P)=0$ and
\[
 \mathrm K_0(P)=
 \frac{\bigoplus_{i,\pi}\Z a_{i,\pi}}
      {\langle u_i-u_j:i,j\rangle},\qquad
 u_i=\sum_\pi d_{i,\pi}a_{i,\pi},\qquad
 \rho(a_{i,\pi})=\frac{\alpha_{i,\pi}}{d_{i,\pi}} .
\]
For $r\geq2$, Theorem~\ref{thm:absorption-without-reservoir}\textup{(ii)}
gives absorption precisely when $P$ is simple, equivalently when
$\sum_i(1-\max_\pi\alpha_{i,\pi}/d_{i,\pi}^2)>1$.
Corollary~\ref{cor:classification-d-stable} classifies the absorbing
products by these data.

For nontrivial finite groups $G_i$, take $A_i=\mathrm C^*_{\red}(G_i)$ with
the canonical trace. The reduced-word rule identifies $P$ with
$\mathrm C^*_{\red}(G_1*\cdots*G_r)$. The irreducible representations
$\pi$ index the matrix blocks and $\alpha_{i,\pi}=d_{i,\pi}^2/|G_i|$
by their multiplicities in the regular representation. The criterion
becomes $\sum_i(1-1/|G_i|)>1$. Thus, for $r\geq2$, the sole
nonabsorbing group case is the infinite dihedral group $\Z/2*\Z/2$,
whose reduced algebra has a character.

\label{ex:modular-group}For the modular group
$\mathrm{PSL}_2(\Z)\cong\Z/2*\Z/3$~\cite{NewmanModular}, the unit relation
is $e_1+e_2=f_1+f_2+f_3$. The basis $e_1-f_1,e_2-f_2,f_1,f_2$ sums
to the unit and pairs with the trace as $(1/6,1/6,1/3,1/3)$.
Since the group product absorbs $D$, Theorem~\ref{thm:nuclear-classification-main}
identifies it with the stabilized four-point model
\[
 \mathrm C^*_{\red}(\mathrm{PSL}_2(\Z))
 \cong (\C^4,(1/6,1/6,1/3,1/3))\fr D
\]
trace preservingly.
\end{example}

\begin{example}[Finite products of full matrix algebras]
\label{af:moore-model}
Let $\mathbf n=(n_1,\ldots,n_r)$, where $r\geq2$ and $n_i\geq2$,
and let $P_{\mathbf n}=\mathop{\fr}_{i=1}^r\mathrm M_{n_i}$ with the
normalized matrix traces. Put
\[
 L_{\mathbf n}=\operatorname{lcm}(n_1,\ldots,n_r),\qquad
 G_{\mathbf n}=\frac{\bigoplus_{i=1}^r\Z a_i}
                   {\langle n_i a_i-n_j a_j:i,j\rangle},
 \qquad T_{\mathbf n}=\Tor G_{\mathbf n},
\]
where $a_i$ is the minimal projection class from the $i$th factor.
Then $P_{\mathbf n}$ absorbs $D$, and its invariant is
\[
 \mathrm K_0(P_{\mathbf n})=\Z g\oplus T_{\mathbf n},\qquad
 [1]=L_{\mathbf n}g,\qquad
 \rho(g)=1/L_{\mathbf n},\qquad
 \rho(T_{\mathbf n})=0,\qquad \mathrm K_1(P_{\mathbf n})=0,
\]
with the strict trace order. Two such products are isomorphic
exactly when their least common multiples agree and their torsion
groups are isomorphic. A unital embedding $P_{\mathbf n}\to
P_{\mathbf m}$ exists exactly when $L_{\mathbf n}\mid L_{\mathbf m}$.

There is a finite wedge $Y$ of Moore complexes with
$H^1(Y;\Z)=0$ and $H^2(Y;\Z)\cong T_{\mathbf n}$ such that
\[
 P_{\mathbf n}\cong \mathrm M_{L_{\mathbf n}}(\mathrm C(Y))\fr D,
\]
with the matrix $\mathrm K_0$ markings described below.
Take $Y$ to be a point if $T_{\mathbf n}=0$. For two factors of
sizes $m,n$, one may take $Y=S^1\cup_d e^2$ with $d=\gcd(m,n)$,
using a point for $d=1$ and $\mathbb{RP}^2$ for $d=2$.
\end{example}

\begin{proof}
Any two of the factors absorb $D$ by
Theorem~\ref{thm:absorption-without-reservoir}\textup{(ii)}, since the sum of their
inverse squared sizes is less than one. This extends to $P_{\mathbf n}$
by associativity. Proposition~\ref{prop:product-elliott} gives
$G_{\mathbf n}$ and the pairing $\rho(a_i)=1/n_i$.
Write $L=L_{\mathbf n}$, $T=T_{\mathbf n}$ and $r_i=L/n_i$.
Since $\gcd(r_1,\ldots,r_r)=1$, choose $k_i\in\Z$ with
$\sum_i k_i r_i=1$ and take $g=\sum_i k_i a_i$. The common unit
$u=n_i a_i$ satisfies $Lg=u$. The elements $t_i=a_i-r_i g$
satisfy $n_i t_i=0$ and generate $T$. In the two-factor case,
the unit relation yields $T\cong\Z/\gcd(m,n)$.

By Corollary~\ref{cor:classification-d-stable}, the trace range
$\frac{1}{L}\Z$ and $T$ give the isomorphism criterion. Uniqueness of the traces makes
inclusion of the trace ranges necessary for an embedding. Conversely,
if $L_{\mathbf n}\mid L_{\mathbf m}$, the map sending $g_{\mathbf n}$
to $(L_{\mathbf m}/L_{\mathbf n})g_{\mathbf m}$ and killing the
torsion preserves the unit, trace and order, so the corollary realizes it.

For the Moore-space model, rank and first Chern class give the same
pointed groups and pairing for $\mathrm M_L(\mathrm C(Y))$ with any
full-support probability measure and normalized fibre trace.
Let $L_{t_i}$ be the line bundle with first Chern class $t_i$ and
put $V_i=\unit^{r_i-1}\oplus L_{t_i}$. Then $V_i^{\oplus n_i}$
has rank $L$ and zero first Chern class, hence is trivial on $Y$.
The decomposition gives a unital matrix embedding into
$\mathrm M_L(\mathrm C(Y))$ with minimal class $r_i g+t_i$.
The classification isomorphism can realize these $\mathrm K_0$
markings.
\end{proof}

For example, using \eqref{gf:eq-GA}, one obtains
\[
 \begin{aligned}
 \mathrm M_2\fr\mathrm M_3 &\cong \mathrm M_6(D),\\
 \mathrm M_2\fr\mathrm M_6 &\cong
 \mathrm M_2\fr\mathrm M_2\fr\mathrm M_3
 \cong \mathrm M_6(\mathrm C(\mathbb{RP}^2)\fr D).
 \end{aligned}
\]
The first two-factor product has trivial torsion and the second
has torsion $\Z/2$. They embed unitally into one another, since
both have least common multiple $6$.

\begin{example}[Two matrix factors and the projective plane]
\label{ex:matrix-projective-plane}
Give $E=\mathrm M_2(\mathrm C(\mathbb{RP}^2))$ the normalized
fibre trace integrated against any full-support probability measure.
Example~\ref{af:moore-model} gives
$\mathrm M_2\fr\mathrm M_2\cong E\fr D$.
If $a$ is the trivial line class and $t=[L]-[\unit]$ for the
nontrivial complex line bundle $L$, the two minimal matrix classes
correspond to $a,a+t$. The trivial decomposition
$\unit\oplus\unit$ gives a unital matrix embedding into $E$
with minimal class $a$, and a trivialization of $L\oplus L$
gives one with minimal class $a+t$.

The algebra $E$ admits a model built from AF factors, namely the
two matrix algebras above. Proposition~\ref{af:model-replacement}
therefore permits simultaneous replacement of any finite or countable
family of these $E$ factors by their paired matrix models after
adjoining $D$, while fixing the remaining factors.
Theorem~\ref{gf:general-fusion} applies to arbitrary additional nuclear
UCT factors with faithful selected traces.
\end{example}

\begin{example}[Non-AF factors with matrix models]
\label{af:matrix-models}
Let $X$ be a nonempty compact metrizable contractible space, let
$n\geq1$, and give $A=\mathrm M_n(\mathrm C(X))$ any faithful
tracial state. Under Morita identification, its selected data are
\[
 \mathrm K_0(A)=\Z g,\qquad [1_A]=ng,\qquad
 \tau_{A*}(g)=1/n,\qquad \mathrm K_1(A)=0.
\]
These are the selected data of $\mathrm M_n$ with normalized trace.
The same data are carried by $\mathrm M_n(\mathcal Z)$.
Proposition~\ref{af:model-replacement} permits simultaneous replacement
of any finite or countable family of these factors by the corresponding
matrix factors after adjoining $D$. Theorem~\ref{gf:general-fusion}
allows arbitrary additional nuclear UCT factors with faithful selected
traces. For example, it applies to
arbitrarily many factors $\mathrm M_{n_i}(\mathrm C([0,1]))$,
although these interval algebras are not AF and their unit classes
do not split when $n_i>1$.
\end{example}

\begin{example}[A unit-twisted Moore corner and a matrix factor]
\label{af:twisted-moore-example}
Let $L$ be the nontrivial complex line bundle over $\mathbb {RP}^2$,
and put
\[
 A=\Gamma\bigl(\operatorname{End}(\unit\oplus L)\bigr),
 \qquad P_4=S^1\cup_4D^2.
\]
Equip $A$ with the normalized fibre trace integrated against any
full-support probability measure. Under Morita identification,
\[
 \mathrm K_0(A)=\Z a\oplus(\Z/2)t,\qquad
 [1_A]=2a+t,\qquad \tau_*(a)=\tfrac12,\quad
 \tau_*(t)=0,\qquad \mathrm K_1(A)=0.
\]
Corollary~\ref{af:rooted-fusion} applies with $A$ as the arbitrary
root and $\mathrm M_2$ as its AF leaf. If $b$ is the minimal matrix
class and $d=b-a$, the unit relation gives
\[
 t=2d,\qquad 4d=0,\qquad
 \mathrm K_0(A\fr\mathrm M_2)=\Z b\oplus(\Z/4)d,
 \quad [1]=2b,\quad \mathrm K_1=0.
\]
The selected pairing sends $b$ to $1/2$ and $d$ to zero.
These are exactly the selected pointed data of
$\mathrm M_2(\mathrm C(P_4))$, for any full-support probability
measure on $P_4$. Marked nuclear-core replacement therefore gives
\[
 (A\fr\mathrm M_2)\fr D
 \cong \mathrm M_2(\mathrm C(P_4))\fr D.
\]
The markings may send the original classes $a,t,b$ to $b-d,2d,b$,
respectively. The original order-two class becomes twice the new
order-four class; replacing $A$ by the trivial bundle corner would
lose its unit marking. The proof uses no freeness assertion inferred
from approximate conjugacy of the whole root.
\end{example}

\begin{example}[Two unit-twisted Moore corners]
\label{ex:two-twisted-moore-corners}
Let $L$ be the nontrivial complex line bundle over
$\mathbb{RP}^2$ and put
\[
 A=\Gamma\bigl(\operatorname{End}(\unit\oplus L)\bigr),
 \qquad P_4=S^1\cup_{z\mapsto z^4}D^2,
 \qquad Y=\mathbb{RP}^2\vee P_4.
\]
Give each copy of $A$ its normalized fibre trace integrated against
a full-support probability measure. Write its two pointed groups as
\[
 \mathrm K_0(A_1)=\Z a\oplus(\Z/2)t,\quad [1_{A_1}]=2a+t,
 \qquad
 \mathrm K_0(A_2)=\Z b\oplus(\Z/2)s,\quad [1_{A_2}]=2b+s.
\]
Both degree-one groups vanish. Setting $d=b-a$ in the free-product
unit relation gives $s=t-2d$ and $4d=0$.
Thus $s=t+2d$ as well, and
\[
 \mathrm K_0(A_1\fr A_2)
 =\Z a\oplus(\Z/4)d\oplus(\Z/2)t,\qquad
 [1]=2a+t,\quad \rho(a)=\tfrac12,
 \quad \rho(d)=\rho(t)=0.
\]
Let $L_t$ be the pullback of $L$ through the collapse
$Y\to\mathbb{RP}^2$, and set
$F=\Gamma(\operatorname{End}(\unit\oplus L_t))$, with the
normalized fibre trace integrated against any full-support probability
measure on $Y$. Rank and first Chern class on the two-dimensional complex $Y$
give the same pointed groups and pairing for $F$, with
$\mathrm K_1(F)=0$. Theorem~\ref{gf:general-fusion} gives
\[
                         (A_1\fr A_2)\fr D\cong F\fr D
\]
with factor markings $a\mapsto a$, $t\mapsto t$,
$b\mapsto a+d$ and $s\mapsto t+2d$.
The exact unit $2a+t$ and the order-four class are both retained.
The two factors have nonsplit units and torsion in degree zero.
\end{example}

\subsection{Absorption of the reservoir}
\label{subsec:unstabilized-examples}

We identify $D$ with its semicircular model using
Corollary~\ref{cor:reservoir-model}. We say that a product $P$ absorbs
$D$ if there is a trace-preserving isomorphism $P\to P\fr D$.
When studying a pair $A\fr B$, a marked absorption isomorphism fixes
$B$ pointwise and acts on $A$ by conjugation with a unitary in the
target. We use such isomorphisms to extend absorption from a scalar
corner to the original factors.

\begin{theorem}[Criteria for $D$-absorption]
\label{thm:absorption-without-reservoir}
Let $P_{\mathcal A}$ be the reduced free product of a nonempty finite or
countable family $\mathcal A=(A_i,\tau_i)_i$ of separable unital
$\mathrm C^*$-algebras with faithful selected tracial states.
\begin{enumerate}
\item If $\mathcal A$ has two disjoint nonempty subfamilies whose reduced
free products contain Haar unitaries, then $P_{\mathcal A}$ absorbs $D$.
In particular, this holds under any of the following conditions.
\begin{enumerate}[label=\textup{(\alph*)},leftmargin=1.8em]
\item Two distinct factors have diffuse selected traces. This includes
two simple infinite-dimensional factors.
\item Four distinct factors contain unitaries $z_i$ with $\tau_i(z_i)=0$.
\item The family is countably infinite and
$\sum_i(1-\operatorname{At}(A_i,\tau_i))=\infty$.
\end{enumerate}

\item Suppose $\mathcal A$ consists of $r\geq2$ nonscalar finite-dimensional
algebras. Write $A_i=\bigoplus_j\mathrm M_{n_{ij}}$ and let
$\alpha_{ij}$ be the trace of the unit of the $j$th block. Put
$\nu_i=\max_j\alpha_{ij}/n_{ij}^2$. Then
\[
 P_{\mathcal A}\text{ absorbs }D
 \quad\Longleftrightarrow\quad P_{\mathcal A}\text{ is simple}
 \quad\Longleftrightarrow\quad \sum_{i=1}^r(1-\nu_i)>1.
\]
For $r=2$, the absorption isomorphism can fix either factor and
conjugate the other by a unitary arbitrarily close to $1$.
\end{enumerate}
\end{theorem}

\begin{proof}[Proof of part \textup{(i)}]
For \textup{(i)}, group the first distinguished subfamily into $B_1$
and the second together with the remaining factors into $B_2$. Then
$P_{\mathcal A}=B_1\fr B_2$, and each $B_j$ contains a Haar unitary.
Semicircular absorption \cite[Theorem~5.2]{Shlyakhtenko} gives
$P_{\mathcal A}\cong P_{\mathcal A}\fr D$ trace preservingly.

In \textup{(a)}, diffuse selected traces admit Haar unitaries
\cite[Theorem~A]{ThielHaar}; traces on simple unital infinite-dimensional
algebras are diffuse by \cite[Corollary~B]{ThielHaar}.
In \textup{(b)}, the unitaries $z_1z_2$ and $z_3z_4$ are Haar in the
two disjoint paired products since every nonzero power is a centered
alternating word \cite[Remark~5.5]{Shlyakhtenko}.
In \textup{(c)}, enumerate the factors by $\N$ and choose
$a_i=a_i^*\in A_i$ whose selected spectral law $\mu_i$ has largest
atom weight $\beta_i<\operatorname{At}(A_i,\tau_i)+2^{-i}$.
Then $\sum_i(1-\beta_i)=\infty$. Choose disjoint finite sets $S_1,S_2$
with $\sum_{i\in S_j}(1-\beta_i)>1$ for $j=1,2$. By the atom formula
for free additive convolution \cite{BercoviciVoiculescu}, the law of
$b_j=\sum_{i\in S_j}a_i$ is atomless, since a tuple of atoms contributing
to an atom of that law would have to have total weight greater than
$|S_j|-1$. This is the grouping argument in \cite[Lemma~2.6]{HP}.
If $H_j$ is the distribution function of $b_j$, then
$\exp(2\pi iH_j(b_j))$ is Haar. Countably many copies of a fixed
nonscalar factor satisfy \textup{(c)}, since a nonscalar selfadjoint
element has a non-point-mass law for a faithful trace.
\end{proof}

We use two operations on tracial amalgamated free products with their
trace-preserving expectations. For expected inclusions $C\subset E\subset L$
and $C\subset N$, one has
\[
 L*_E(E*_C N)=L*_C N.
\]
To verify this, expand an $E$-centered letter from $E*_C N$ into
$C$-reduced words containing an $N$-letter. An $E$-centered letter
from $L$, multiplied on either side by elements of $E$, is still
$C$-centered. The reduced-word test gives the identity, first on
words and then in the tracial closure. In particular, attaching
$N_i$ to a common algebra $M$ over expected subalgebras $E_i\subset M$
can be done in either order, since both iterated products equal
$(M*_{E_1}N_1)*_M(M*_{E_2}N_2)$.

Suppose next that $\theta:M\to M'$ is a trace-preserving isomorphism
and, on a specified common subalgebra $E$, agrees with $\Ad W$ for a
unitary $W\in M'$. It extends from $M*_E N$ to $M' *_E N$
by acting on $N$ as $\Ad W$. Indeed, postcompose $\theta$ by
$\Ad W^*$ to fix $E$ pointwise, extend by the identity on $N$ using
the trace-preserving expectations, and conjugate back by $W$.
We apply these rules to tracial closures. An isomorphism whose
images of the original generators lie in and generate the specified
target $\mathrm C^*$-algebra restricts to an isomorphism there.

\begin{lemma}[Absorption in a common scalar corner]
\label{lem:scalar-corner-absorption}
Let $L_1,L_2$ generate a separable faithfully tracial
$\mathrm C^*$-algebra $P$, and suppose their tracial closures
$M_i=L_i''$ are free over a common von Neumann subalgebra $E$.
Let $p\in L_1\cap L_2\cap E$ be a nonzero central projection of $E$
with $pE=\C p$. If each $pL_i p$ contains a Haar unitary for its
normalized trace, there is a trace-preserving isomorphism
$P\to P[p,D]$ fixing $L_2$ and acting on $L_1$ by conjugation
with a unitary in $P[p,D]$ arbitrarily close to $1$.
If $p$ is full in $P$, the supported target can be replaced by
$P\fr D$, with the same marks.
\end{lemma}

\begin{proof}
Put $K_i=pM_i p$ and $H_i=(1-p)E\oplus K_i$, and form
$H=H_1*_E H_2=(1-p)E\oplus(K_1*K_2)$, where the corner product
uses normalized traces. Then $P''$ is obtained by attaching
$M_i$ to $H$ over $H_i$. By Lemma~\ref{lem:gaussian-full-corner},
$P[p,D]''$ is obtained by attaching the supported semicircular
family over $\C p+\C(1-p)$. Doing this attachment first replaces
$H$ by $(1-p)E\oplus(K_1*K_2*D_p'')$, where $D_p$ denotes the
compressed copy of $D$.

Choose Haar unitaries $u_i\in pL_ip$. Apply
\cite[Lemma~5.1 and Theorem~5.2]{Shlyakhtenko} to the scalar free
product of the separable $\mathrm C^*$-algebras $pL_1p$ and $pL_2p$,
then extend the trace-preserving isomorphism to the tracial closures.
It fixes $K_2$, conjugates $K_1$ by
$U\in\mathrm C^*(u_1,u_2,D_p)$ arbitrarily close to $p$, and satisfies
$\mathrm C^*(Uu_1U^*,u_2)=\mathrm C^*(u_1,u_2,D_p)$.
The unitary $W=(1-p)+U$ commutes with $E$. Extend the middle map by
$\Ad W$ on $M_1$ and by the identity on $M_2$, using reattachment.
Its image on $P$ contains $p$, $Uu_1U^*$ and $u_2$, hence $U$,
$W$ and $D_p$. It therefore contains $L_1$ and $L_2$ and equals
$P[p,D]$. When $p$ is full, the isomorphism $P[p,D]\cong P\fr D$
from Lemma~\ref{lem:gaussian-full-corner} fixes $P$ pointwise, so
it retains the marks and the unitary norm bound.
\end{proof}

\begin{lemma}\label{lem:three-balanced-pairs}
The reduced free product of three equally weighted two-point algebras
absorbs $D$.
\end{lemma}
\begin{proof}
Let $p,q,s$ be their trace-$1/2$ projections. Put $E=\mathrm C^*(p)$,
$L_1=\mathrm C^*(p,q)$ and $L_2=\mathrm C^*(p,s)$. Reassociation
gives $P=L_1 *_E L_2$. Both $pL_1p$ and $pL_2p$ have atomless
normalized traces by the balanced two-projection formula
\cite[Proposition~2.7]{DykemaStable}, and hence have Haar unitaries.
Lemma~\ref{lem:scalar-corner-absorption} gives $P\cong P[p,D]$.
For fullness, $(2p-1)(2q-1)$ is Haar, and its first two powers
are centered and trace orthogonal. Together with $2s-1$ these
verify the Avitzour condition \cite{Avitzour}, so $P$ is simple.
Now apply Lemma~\ref{lem:gaussian-full-corner}.
\end{proof}

\begin{proof}[Proof of part \textup{(ii)}]
Write $P=P_{\mathcal A}$. The stabilized product is simple by Lemma~\ref{lem:interval-regularity}.
For two finite-dimensional factors, Ivanov's theorem
\cite[Theorem~2.6]{IvanovFD} gives simplicity exactly when
$\nu_1+\nu_2<1$, which is the displayed inequality for $r=2$.
For the product of two two-point algebras, each maximal weight is
at least $1/2$ and the product has a character by
\cite[Proposition~2.7]{DykemaStable}.

We first prove a sufficient condition for any simple reduced
free product $A\fr B$. Suppose $A$ contains a nontrivial projection
$p$ and $B$ has a partition of $1$ into at least three nonzero
projections. Let $1-q$ be a least-trace member, with trace $b$,
and write $q=e+f$ as a sum of orthogonal projections of traces at
least $b$. Put $F=\mathrm C^*(p,q)$ and $R=\mathrm C^*(q)$.
The two-projection formula \cite[Proposition~2.7]{DykemaStable}
gives the unit $r$ of the interval summand of $qFq$, with
\[
 0<\tau(r)=\min\{\tau(p),1-\tau(p),b,1-b\}\leq b.
\]
The trace on $rFr$ is atomless and $rFr$ has a Haar unitary.
Put $R'=\mathrm C^*(q,r)$ and $L=\mathrm C^*(R',B)$.
Reassociation identifies $L=R' *_R B$ and
$\mathrm C^*(F,B)=F*_{R'}L$. In $qLq$, the projections $r,e$
are scalar free for the normalized trace, since $qRq=\C q$.
In the normalized $r$-corner, the two-projection formula gives
an atom of $rer$ at zero only if $\tau(r)>\tau(e)$, and at one
only if $\tau(r)>\tau(f)$. These are its only possible atoms.
Thus $rLr$ contains a Haar unitary, obtained by applying the
continuous distribution function and then the circle exponential
to $rer$. Apply Lemma~\ref{lem:scalar-corner-absorption} over $R'$
to fix $F$ and conjugate $L$, then reattach $A$ over
$\mathrm C^*(p)$ by the preceding amalgamated reattachment rule.
This gives $P\cong P[r,D]$. Simplicity makes $r$ full, and
Lemma~\ref{lem:gaussian-full-corner} identifies $P[r,D]$ with
$P\fr D$ over $P$. This proves the sufficient condition with the
indicated factor markings.

For a simple product of two finite-dimensional factors this leaves only $\mathrm M_2\fr\mathrm M_2$
and $\mathrm M_2\fr\C^2$ among the simple products. Write $A=\mathrm M_2$,
take $p=e_{11}$ and choose a projection $q\in B$ of trace
$s\in(1/4,1/2]$. For $\mathrm M_2$ use a minimal projection.
For $\C^2$ use the smaller-weight summand, whose trace is $>1/4$ by the binary simplicity test. With $F=\mathrm C^*(p,q)$, set
\[
 J=pFp\oplus(1-p)F(1-p),\qquad
 N=J*_{\mathrm C^*(p)}\mathrm M_2.
\]
Reassociation gives $\mathrm M_2\fr\mathrm C^*(q)=N*_J F$.
By \eqref{gf:eq-MT}, $N=\mathrm M_2(H)$ for $H=C_1\fr C_2$,
where $C_1=pFp$ and $C_2$ is $(1-p)F(1-p)$ transported into
the first corner. The algebra $J$ is diagonal with entries
$C_1,C_2$. In each $C_i$ the interval summand has weight
$\beta=2s>1/2$ by the two-projection formula.

For $\beta=1$, put $r=1_H$. Marked absorption
\cite[Theorem~5.2]{Shlyakhtenko}, applied with the factors reversed,
gives $H\cong H[r,D]$ fixing $C_1$ and conjugating $C_2$
by a unitary $W$ arbitrarily close to $1_H$. For $\beta<1$, write $d_i$ for the interval
units and $E=\mathrm C^*(d_1,d_2)$. The projection $r=d_1\wedge d_2$
belongs to $E$ and gives a scalar central summand of weight
$2\beta-1$. Put $L_i=\mathrm C^*(E,C_i)$, so $H=L_1*_E L_2$.
The corner computation \cite[Proposition~2.8]{DykemaStable} gives
\[
 d_iL_i d_i\cong (d_iC_i d_i)\fr (d_iEd_i).
\]
The first factor is diffuse abelian and the second contains
the proper nonzero projection $r$. Thus the displayed corner is
simple and infinite dimensional by \cite[Theorem~2]{DykemaStable},
and its nonzero $r$-corner has a Haar unitary by
\cite[Corollary~B]{ThielHaar}. Lemma~\ref{lem:scalar-corner-absorption}
now gives $H\cong H[r,D]$, fixing $C_1$ and conjugating $C_2$
by a unitary $W$ close to $1$.

In both cases amplify the marked map to $N\cong N[r,D]$,
viewing $r$ in the first matrix corner. The matrix units recover
the supported extension from its first corner $H[r,D]$.
On $J$ the amplified map is conjugation by
$\operatorname{diag}(1,W)$. Conjugating the target by its adjoint
fixes $J$ and conjugates the original $\mathrm M_2$ by a unitary
close to $1$. Reattach $F$ over $J$ and $B$ over $\mathrm C^*(q)$,
fixing both attached factors. This gives $P\cong P[r,D]$.
The projection $r$ is full in $P$, so
Lemma~\ref{lem:gaussian-full-corner} gives the result over $P$.
Postcomposing a marked map with the inverse inner automorphism
interchanges which factor is fixed.

For $r\geq3$, apply Ivanov's decomposition successively, putting
the new finite-dimensional factor in the first slot and the
preceding product in the second. The part left after the matrix
summands contains a unital diffuse abelian subalgebra: the diffuse
abelian subalgebras on the compressed block units of the first
factor in the proof of \cite[Theorem~2.6]{IvanovFD} sum to such
a subalgebra. The exceptional initial product of two two-point
algebras has an interval matrix-field summand, with largest scalar
weight $\max\{0,\nu_1+\nu_2-1\}$ by
\cite[Proposition~2.7]{DykemaStable}. At each stage let $v_i$ be
the largest matrix weight divided by its squared size, with
$v_i=0$ if there are no matrix summands. Ivanov's weight formula
gives $v_i=\max\{0,v_{i-1}+\nu_i-1\}$, with $v_1=\nu_1$.
Thus $v_i=\max\{0,1-\sum_{j=1}^i(1-\nu_j)\}$.
At the last step simplicity is equivalent to $v_{r-1}+\nu_r<1$,
since equality also gives a quotient. This gives the claimed
strict inequality.

For absorption, assume $r\geq3$ and $P$ is simple. If some factor
has a partition into three nonzero projections, group the others
together and apply the three-projection argument above. If every factor
has at most two nonzero orthogonal projections, each is $\mathrm M_2$
or $\C^2$. Two copies of $\mathrm M_2$,
or $\mathrm M_2$ and an equally weighted $\C^2$, give an absorbing
pair. If there is an unequally weighted $\C^2$, use its nontrivial
projection and a nontrivial projection from any other factor.
Their two-projection algebra has a nonzero scalar summand and
a nonzero interval matrix-field summand. The first projection
and its complement split the unit of the interval summand into
two nonzero pieces, so this paired block has a partition into
three nonzero projections. The three-projection argument now
applies by grouping the other factors together. Finally, three
equally weighted $\C^2$ factors absorb $D$ by
Lemma~\ref{lem:three-balanced-pairs}.
\end{proof}

\end{document}